\documentclass[10pt]{article}
\usepackage{songmei}
\usepackage[toc,page]{appendix}

\usepackage[mathscr]{euscript}

\DeclareSymbolFont{rsfs}{U}{rsfs}{m}{n}
\DeclareSymbolFontAlphabet{\mathscrsfs}{rsfs}

\def\bxi{{\boldsymbol \xi}}
\def\cA{{\mathcal A}}
\def\cO{{\mathcal O}}
\def\hE{{\hat{\mathbb E}}}
\def\bU{{\boldsymbol U}}
\def\bV{{\boldsymbol V}}
\def\bz{{\boldsymbol z}}
\def\bOmega{{\boldsymbol \Omega}}
\def\bfe{{\boldsymbol e}}
\def\bG{{\boldsymbol G}}
\def\bbm{{\boldsymbol m}}

\def\bh{{\boldsymbol h}}

\def\bB{{\boldsymbol B}}
\def\beps{{\boldsymbol \varepsilon}}
\def\bH{{\boldsymbol H}}
\def\bK{{\boldsymbol K}}
\def\bQ{{\boldsymbol Q}}
\def\bv{{\boldsymbol v}}
\def\bDelta{{\boldsymbol \Delta}}
\def\bE{{\boldsymbol E}}
\def\bX{{\boldsymbol X}}
\def\bY{{\boldsymbol Y}}
\def\bw{{\boldsymbol w}}
\def\bx{{\boldsymbol x}}
\def\by{{\boldsymbol y}}
\def\bW{{\boldsymbol W}}
\def\hba{\hat{\boldsymbol a}}
\def\ba{{\boldsymbol a}}

\def\bT{{\boldsymbol T}}
\def\bDelta{{\boldsymbol \Delta}}

\def\bu{{\boldsymbol u}}
\def\bg{{\boldsymbol g}}
\def\bA{{\boldsymbol A}}
\def\btheta{{\boldsymbol \theta}}
\def\bTheta{{\boldsymbol \Theta}}

\def\bbeta{{\boldsymbol \beta}}
\def\bJ{{\boldsymbol J}}
\def\bC{{\boldsymbol C}}
\def\boldf{{\boldsymbol f}}
\def\bM{{\boldsymbol M}}
\def\obM{\overline{\boldsymbol M}}

\def\bS{{\boldsymbol S}}

\def\bD{{\boldsymbol D}}

\def\bF{{\boldsymbol F}}
\def\bZ{{\boldsymbol Z}}
\def\bsigma{{\boldsymbol \sigma}}
\def\bq{{\boldsymbol q}}
\def\bXi{{\boldsymbol \Xi}}
\def\bfeta{{\boldsymbol \eta}}
\def\bR{{\boldsymbol R}}
\def\imagunit{{\mathrm i}}
\def\Res{{\boldsymbol\Xi}}
\def\oRes{{\boldsymbol \Pi}}
\def\Restrict{{\mathsf R}}

\def\cT{{\mathcal T}}
\def\cQ{{\mathcal Q}}
\def\cF{{\mathcal F}}
\def\cE{{\mathcal E}}

\def\cH{{\mathcal H}}
\def\bcH{{\boldsymbol {\mathcal H}}}

\def\sHone{{{\sf H}_1}}
\def\sFone{{{\sf F}_1}}
\def\sFtwo{{{\sf F}_2}}

\def\ok{{\overline k}}

\def\oR{{\overline R}}

\def\reals{{\mathbb R}}
\def\complex{{\mathbb C}}
\def\integers{{\mathbb N}}

\def\tm{{\tilde m}}

\def\tsigma{{\tilde \sigma}}
\def\err{{\rm err}}
\def\disk{{\mathbb D}}
\def\bsF{\textbf{\textsf{F}}}

\def\ratio{{\zeta}}
\def\tbtheta{{\tilde \btheta}}
\def\tbx{{\tilde \bx}}
\def\tbB{{\tilde \bB}}
\def\tbQ{{\tilde \bQ}}
\def\tbJ{{\tilde \bJ}}
\def\tbH{{\tilde \bH}}
\def\tQ{{\tilde Q}}
\def\tJ{{\tilde J}}
\def\tH{{\tilde H}}
\def\bfone{{\mathbf 1}}
\def\tz{{\tilde z}}

\def\ob{\mu}
\def\Tr{{\rm Tr}}
\def\normal{{\mathsf N}}
\def\de{{\rm d}}
\def\Unif{{\rm Unif}}
\def\proj{{\mathsf P}}
\def\He{{\rm He}}

\def\Poly{{\rm Poly}}
\def\Coeff{{\rm Coeff}}
\def\RF{{\rm RF}}

\def\diag{{\rm diag}}

\def\Log{{\rm Log}}

\def\normf{F}

\def\hbbeta{\hat{\boldsymbol \beta}}
\def\bSigma{{\boldsymbol\Sigma}}

\def\lsamp{\mbox{\tiny\rm lsamp}}
\def\wide{\mbox{\tiny\rm wide}}
\def\rless{\mbox{\tiny\rm rless}}
\def\sNL{\mbox{\tiny\rm NL}}

\def\hf{\hat{f}}

\def\olambda{\overline{\lambda}}

\def\cuR{\mathscrsfs{R}}
\def\cuE{\mathscrsfs{E}}
\def\cuB{\mathscrsfs{B}}
\def\cuV{\mathscrsfs{V}}

\usepackage{hyperref}
\hypersetup{
    colorlinks,
    linkcolor={blue!80!black},
    citecolor={green!50!black},
}
\colorlet{linkequation}{blue}

\begin{document}

\title{The generalization error of random features regression:\\
 Precise asymptotics and double descent curve}

\author{Song Mei\thanks{Institute for Computational and Mathematical
    Engineering, Stanford University} \;\,\, and \,\, Andrea Montanari\thanks{Department of Electrical Engineering and Department of Statistics, Stanford University}}

\maketitle

\begin{abstract}
Deep learning methods operate in regimes that defy the traditional statistical mindset. Neural network architectures often contain more parameters than training samples, and are so rich that they can interpolate the observed labels, even if the latter are replaced by pure noise. Despite their huge complexity, the same architectures achieve small generalization error on real data. 

This phenomenon has been rationalized in terms of a so-called `double descent' curve. As the model complexity increases, the test error follows the usual U-shaped curve at the beginning, first decreasing and then peaking around the interpolation threshold (when the model achieves vanishing training error). However, it descends again as model complexity exceeds this threshold. The global minimum of the test error is found above the interpolation threshold, often in the  extreme overparametrization regime in which the number of parameters is much larger than the number of samples. Far from being a peculiar property of deep neural networks, elements of this behavior have been demonstrated in much simpler settings, including linear regression with random covariates. 

In this paper we consider the problem of learning an unknown function over the $d$-dimensional sphere $\S^{d-1}$, from $n$ i.i.d. samples $(\bx_i, y_i)\in \S^{d-1} \times  \reals$,  $i\le n$. We perform ridge regression on $N$ random features of the form $\sigma(\bw_a^{\sT}\bx)$, $a\le N$. This can be equivalently described as a two-layers neural network with random first-layer weights.  We compute the precise asymptotics of the test error, in the limit $N,n,d\to \infty$ with $N/d$ and $n/d$ fixed. This provides the first  analytically tractable model that captures all the features of the double descent phenomenon without assuming ad hoc misspecification structures. In particular, above a critical value of the signal-to-noise ratio, minimum test error is achieved by extremely overparametrized interpolators, i.e., networks that have a number of parameters much larger than the sample size, and vanishing training error. 
\end{abstract}

\tableofcontents

\section{Introduction}

Statistical lore recommends not to use models that have too many parameters since this will lead to `overfitting' and poor generalization. Indeed, a plot of the test error as a function of the model complexity often reveals a U-shaped curve. The test error first decreases because the model is less and less biased, but then increases because of a variance explosion \cite{Hastie}. In particular, the interpolation threshold, i.e., the threshold in model complexity above which the training error vanishes (the model completely interpolates the data), corresponds to a large test error. It seems wise to keep the model complexity well below this threshold in order to obtain a small generalization error.

These classical prescriptions are in stark contrast with the current practice in deep learning. The number of parameters of modern neural networks can be much larger than the number of training samples, and the resulting models are often so complex that they can perfectly interpolate the data. Even more surprisingly, they can interpolate the data when the actual labels are replaced by pure noise \cite{zhang2016understanding}. Despite such a large complexity, these models have small test error and can outperform others trained in the classical underparametrized regime. 

This behavior has been rationalized in terms of a so-called `double-descent' curve \cite{belkin2018understand,belkin2018reconciling}. A plot of the test error as a function of the model complexity follows the traditional U-shaped curve until the interpolation threshold. However, after a peak at the interpolation threshold, the test error decreases, and attains a global minimum in the overparametrized regime. In fact, the minimum error often appears to be `at infinite complexity': the more overparametrized is the model, the smaller is the error. It is conjectured that the good generalization behavior in this highly overparametrized regime is due to the implicit regularization induced by gradient descent learning: among all interpolating models, gradient descent selects the simplest one, in a suitable sense. An example of double descent curve is plotted in Fig.~\ref{fig:DoubleDescentFirst}. The main contribution of this paper is to describe a natural, analytically tractable model leading to this generalization curve, and to derive precise formulae for the same curve, in a suitable asymptotic regime. 
\begin{figure}[!ht]
\centering
\includegraphics[width = 0.48\linewidth]{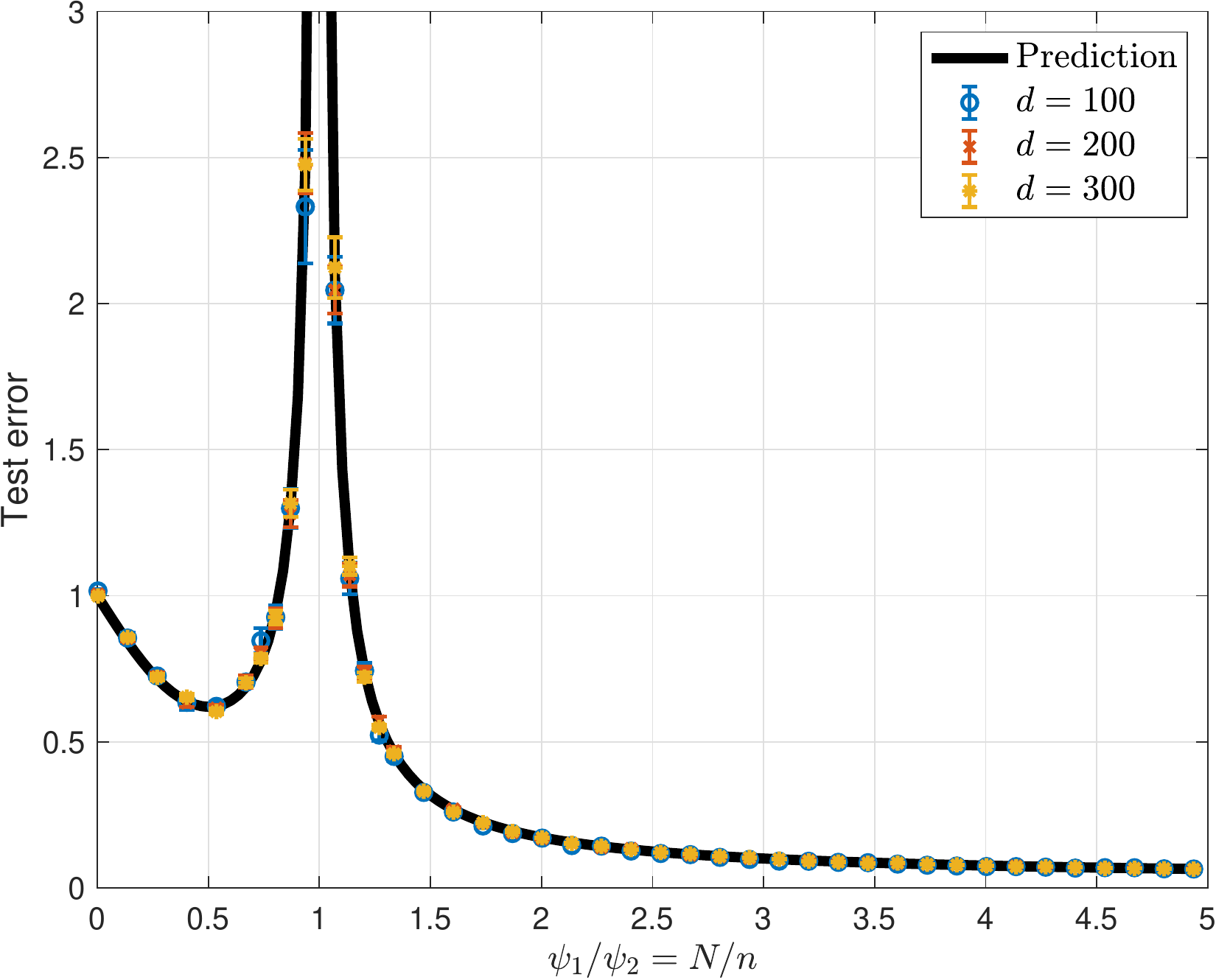}
\includegraphics[width = 0.48\linewidth]{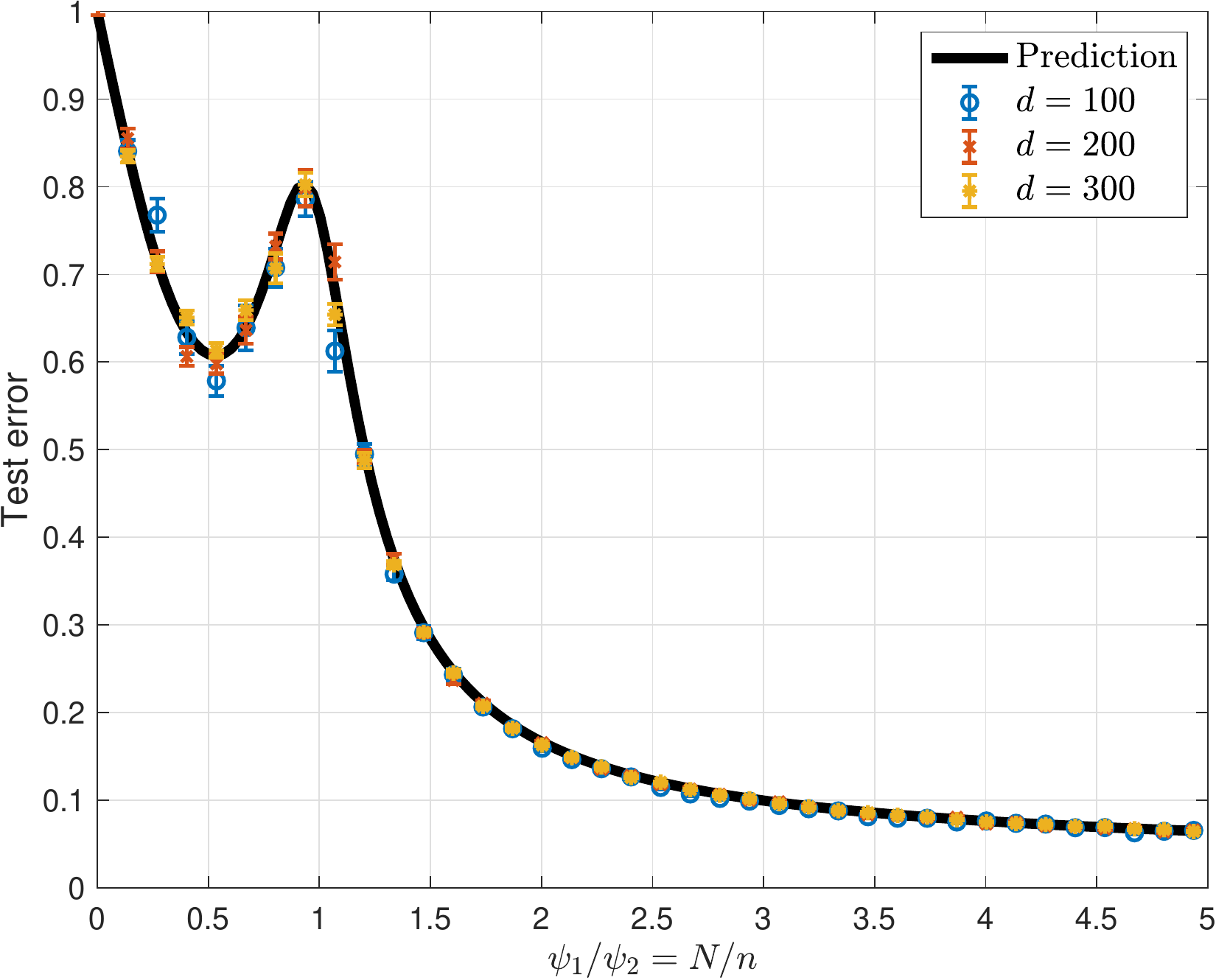}
\caption{Random features ridge regression with ReLU activation ($\sigma = \max\{x, 0\}$). Data are generated via $y_i=\<\bbeta_1,\bx_i\>$ (zero noise) with $\|\bbeta_1\|_2^2 =1$, and $\psi_2 = n / d = 3$. Left frame: regularization $\lambda = 10^{-8}$ (we didn't set $\lambda = 0$ exactly for numerical stability). 
Right frame: $\lambda = 10^{-3}$. The continuous black line is our theoretical prediction, and the colored symbols are numerical results for several dimensions $d$. Symbols are averages over $20$ instances and the error bars report the standard error of the means over these $20$ instances. 
}\label{fig:DoubleDescentFirst}
\end{figure}

The double-descent scenario is far from being specific to neural networks, and was instead demonstrated empirically in a variety of models including random forests and random features models \cite{belkin2018reconciling}. Recently, several elements of this scenario were established analytically in simple least square regression, with certain probabilistic models for the random covariates \cite{advani2017high,hastie2019surprises,belkin2019two}. These papers consider a setting in which we are given i.i.d. samples $(y_i,\bx_i)\in \reals\times \reals^d$, $i\le n$, where $y_i$ is a response variable which depends on covariates $\bx_i$ via $y_i = \<\bbeta,\bx_i\>+\eps_i$, with $\E(\eps_i) = 0$ and $\E(\eps_i^2) = \tau^2$; or in matrix notation, $\by = \bX\bbeta+\beps$. The authors study the test error of `ridgeless least square regression' $\hbbeta = \bX^\dagger \by$ (where $\bX^\dagger$ stands for the pseudoinverse of $\bX$), and use random matrix theory to derive its precise asymptotics in the limit $n,d\to\infty$ with $d/n = \gamma$ fixed, when $\bx_i = \bSigma^{1/2}\bz_i$ with $\bz_i$ a vector with i.i.d. entries.

Despite its simplicity, this random covariates model captures several features of the double descent scenario. In particular, the asymptotic generalization curve is U-shaped for $\gamma<1$, diverging at the interpolation threshold $\gamma=1$, and descends again exceeding that threshold. The divergence at $\gamma=1$ is explained by an explosion in the variance, which is in turn related to a divergence of the condition number of the random matrix $\bX$. At the same time, this simple model misses some interesting features that are observed in more complex settings: $(i)$~In the Gaussian covariates model, the global minimum of the test error is achieved in the underparametrized regime $\gamma<1$, unless ad-hoc misspecification structure is assumed; $(ii)$~The number of parameters is tied to the covariates dimension $d$ and hence the effects of overparametrization are not isolated from the effects of the ambient dimensions; $(iii)$~Ridge regression, with some regularization $\lambda>0$, is always found to outperform the ridgeless limit $\lambda\to 0$. Moreover, this linear model is not directly connected to actual neural networks, which are highly nonlinear in the covariates $\bx_i$.

In this paper, we study the random features model of Rahimi and Recht \cite{rahimi2008random}. The random features model can be viewed either as a randomized approximation to kernel ridge regression, or as a two-layers neural networks with random first layer wights. We compute the precise asymptotics of the test error and show that it reproduces all the qualitative features of the double-descent scenario.

More precisely, we consider the problem of learning a function  $f_d \in L^2(\S^{d-1}(\sqrt d))$ on the $d$-dimensional sphere. (Here and below $\S^{d-1}(r)$ denotes the sphere of radius $r$ in $d$ dimensions, and we set $r=\sqrt{d}$ without loss of generality.) We are given i.i.d. data $\{(\bx_i, y_i)\}_{i \le n} \sim_{iid} \P_{\bx, y}$, where $\bx_i \sim_{iid} \Unif(\S^{d-1}(\sqrt d))$ and $y_i = f_d(\bx_i) + \eps_i$, with $\eps_i \sim_{iid} \P_\eps$ independent of $\bx_i$. The noise distribution satisfies $\E_\eps(\eps_1) = 0$, $\E_\eps(\eps_1^2) = \tau^2$, and $\E_\eps(\eps_1^4) < \infty$. We fit these training data using the random features ($\RF$) model, which is defined as the function class
\begin{align}
\cF_{\RF}(\bTheta) = \Big\{ f(\bx;\ba,\bTheta) \equiv&~ \sum_{i=1}^N a_i \sigma(\< \btheta_i, \bx\>/\sqrt{d}): \;\;\; a_i \in \R\,\, \forall i \in [N] \Big\}\, . \label{eq:RF-Def}
\end{align}
Here, $\bTheta \in \R^{N \times d}$ is a matrix whose $i$-th row is the vector $\btheta_i$, which is chosen randomly, and independent of the data. In order to simplify some of the calculations below, we will assume the normalization $\|\btheta_i\|_2=\sqrt{d}$, which justifies the factor $1/\sqrt{d}$ in the above expression, yielding $\< \btheta_i, \bx_j\>/\sqrt{d}$ of order one. As mentioned above, the functions in $\cF_{\RF}(\bTheta)$ are two-layers neural networks, except that the first layer is kept constant. A substantial literature draws connections between random features models,  fully trained neural networks, and kernel methods. We refer to Section \ref{sec:Related}
for a summary of this line of work.

We learn the coefficients $\ba = (a_i)_{i\le N}$ by performing ridge regression
\begin{align}
\hba(\lambda) = \argmin_{\ba\in\reals^N} \left\{\frac{1}{n}\sum_{j=1}^n \Big( y_j- \sum_{i=1}^N a_i \sigma(\< \btheta_i, \bx_j\> / \sqrt d) \Big)^2  +  \frac{N\lambda}{d}\, \| \ba \|_2^2\right\}\, . \label{eq:Ridge}
\end{align}
The choice of ridge penalty is motivated by the connection to kernel ridge regression, of which this method can be regarded as a finite-rank approximation. Further, the ridge regularization path is naturally connected to the path of gradient flow with respect to the mean square error $\sum_{i\le n}(y_i-f(\bx_i;\ba,\bTheta))^2$, starting at $\ba=0$. In particular, gradient flow converges to the ridgeless limit ($\lambda\to 0$) of $\hba(\lambda)$, and there is a correspondence between positive $\lambda$, and early stopping in gradient descent \cite{yao2007early}.

We are interested in the `prediction' or `test' error (which we will also call `generalization error,' with a slight abuse of terminology), that is the mean square error on predicting $f_d(\bx)$ for $\bx\sim \Unif(\S^{d-1}(\sqrt{d}))$ a fresh sample independent of the training data $\bX = (\bx_i)_{i\le n}$, noise $\beps = (\eps_i)_{i\le n}$, and the random features $\bTheta = (\btheta_a)_{a \le N}$:
\begin{align}\label{eqn:prediction_risk_first_definition}
R_\RF(f_d, \bX, \bTheta, \lambda) = \E_{\bx}\Big[ \Big(f_d(\bx) - f(\bx;\hba(\lambda),\bTheta)\Big)^2\Big]\, .
\end{align} 
Notice that we do not take expectation with respect to the training data $\bX$, the random features $\bTheta$ or the data noise $\beps$.
This is not very important, because we will show that  $R_\RF(f_d, \bX, \bTheta, \lambda)$ concentrates around the expectation 
$\oR_\RF(f_d, \lambda) \equiv \E_{\bX,\bTheta, \beps}R_\RF(f_d, \bX, \bTheta, \lambda)$.
We study the following setting
\begin{itemize}
\item  The random features are uniformly distributed on a sphere: $(\btheta_i)_{i\le N}\sim_{iid}\Unif(\S^{d-1}(\sqrt{d}))$.
\item  $N,n,d$ lie in a proportional asymptotics regime. Namely, $N,n,d\to\infty$ with $N/d\to \psi_1$, $n/d\to \psi_2$ for some $\psi_1,\psi_2\in (0,\infty)$.
\item We consider two models for the regression function $f_d$: $(1)$ A linear model: $f_d(\bx) = \beta_{d, 0} +\<\bbeta_{d, 1}, \bx\>$, where $\bbeta_{d, 1} \in\reals^d$ is  arbitrary with $\| \bbeta_{d, 1} \|_2^2 = \normf_1^2$;
$(2)$ A nonlinear model: $f_d(\bx) =\beta_{d, 0} +\<\bbeta_{d, 1},\bx\> + f_d^{\sNL}(\bx)$ where the nonlinear component  $f_d^{\sNL}(\bx)$ is a centered isotropic Gaussian process indexed by $\bx \in \S^{d-1}(\sqrt d)$. 
%
%
(Note that the linear model is a special case of the nonlinear one, but we prefer to keep the former distinct since it is purely deterministic.)
\end{itemize}

Within this setting, we are able to determine the precise asymptotics of the prediction error, as an explicit function of the dimension parameters $\psi_1,\psi_2$, the noise level $\tau^2$, the activation function $\sigma$, the regularization parameter $\lambda$, and the power of linear and nonlinear components of $f_d$:  $\normf_1^2$ and $\normf_\star^2\equiv\lim_{d\to\infty}\E\{f_d^{\sNL}(\bx)^2\}$. 
The resulting formulae are somewhat complicated, and we defer them to Section \ref{sec:Main}, limiting ourselves to give the general form of our result for the linear model. 

\begin{theorem}{(Linear truth, formulas omitted)}\label{thm:MainIntro}
Let $\sigma: \R \to \R$ be weakly differentiable, with $\sigma'$ a weak derivative of $\sigma$. Assume $\vert \sigma(u)\vert, \vert \sigma'(u)\vert \le c_0 e^{c_1 \vert u \vert}$ for some constants $c_0, c_1 < \infty$. Define the parameters $\ob_0$, $\ob_1$, $\ob_\star$, $\zeta$, and the signal-to-noise ratio $\rho \in [0, \infty]$, via
\begin{align}
\ob_0 = \E[\sigma(G)], ~~~~ \ob_1 = \E[G \sigma(G)], ~~~~ \ob_\star^2 = \E[\sigma(G)^2]- \ob_0^2- \ob_1^2, ~~~~ \zeta \equiv \ob_1^2/\ob_\star^2\, , ~~~~ \rho\equiv F_1^2/\tau^2\, ,
\end{align}
where expectation is taken with respect to $G\sim\normal(0,1)$. Assume $\ob_0, \ob_1, \ob_\star \neq 0$. 

Then, for  linear $f_d$ in the setting described above, for any $\lambda>0$, the prediction risk convergences in probability
\begin{align}
R_\RF(f_d, \bX, \bTheta, \lambda) \overset{p}{\to} (F_1^2+\tau^2) \, \cuR(\rho,\zeta,\psi_1,\psi_2,\lambda/\ob_\star^2) \, ,\label{eq:MainIntro}
\end{align}
where $\cuR(\rho,\zeta,\psi_1,\psi_2,\olambda)$ is explicitly given in Definition \ref{def:formula_ridge}.
\end{theorem}

Section \ref{subsec:main_results} also contains an analogous statement for the nonlinear model.
\begin{remark}\label{rmk:Ridgeless}
Theorem \ref{thm:MainIntro} and its generalizations stated below require $\lambda>0$ fixed as $N,n,d\to\infty$. We can then consider the ridgeless 
limit by taking $\lambda\to 0$. Let us stress that this does not necessarily yield the prediction risk of the min-norm least square estimator that is also given by the limit
$\hba(0+) \equiv \lim_{\lambda\to 0} \hba(\lambda)$ at $N,n,d$
fixed. Denoting by $\bZ = \sigma(\bX\bTheta^{\sT}/\sqrt{d}) / \sqrt d$
the design matrix, the latter is given by
 $\hba(0+) = (\bZ^{\sT}\bZ)^\dagger \bZ^{\sT}\by / \sqrt d$.
While we conjecture that indeed this is the same as taking $\lambda\to 0$  in the asymptotic expression of Theorem \ref{thm:MainIntro}, establishing this rigorously would require proving that the limits $\lambda\to 0$ and $d\to\infty$ can be exchanged. We leave this to future work.
\end{remark}

\begin{remark}
As usual, we can decompose the risk $R_\RF(f_d, \bX, \bTheta, \lambda) = \|f_d-\hf\|^2_{L^2}$ (where $\hf(\bx) = f(\bx;\hba(\lambda),\bTheta)$) into a variance component $\|\hf-\E_{\beps}(\hf)\|_{L^2}^2$, and a bias component $\|f_d-\E_{\beps}(\hf)\|_{L^2}^2$.
The asymptotics of the variance component in the $\lambda \to 0+$ limit was concurrently computed in \cite[Section 8]{hastie2019surprises}. Notice that the variance calculation only
requires to consider a pure noise model in which $\by = \beps\sim\normal(0,\tau^2\id_n)$, and indeed \cite{hastie2019surprises}
does not mention the nonparametric model $y_i = f_d(\bx_i)+\eps_i$. The pure noise ridgeless ($\lambda\to 0$) setting
captures the divergence of the risk at $N=n$ but misses most phenomena that are interesting from a statistical viewpoint:
the optimality of vanishing regularization, the optimality of large overparametrization, the disappearance
of double descent for optimally regularized models.

Our work is the first one to provide a complete treatment of
the  nonparametric model in the proportional asymptotics and to establish those phenomena.
From a mathematical viewpoint, the calculation of the test error can be reduced to studying a block-structured
kernel random matrix, with a more intricate structure than the one of \cite{hastie2019surprises}.
The reduction itself is novel in the present context, and goes through  the log determinant of this random matrix,
while the variance computation of \cite{hastie2019surprises} is directly connected to the resolvent. 
\end{remark}

Figure \ref{fig:DoubleDescentFirst} reports numerical results for learning a linear function $f_d(\bx) = \<\bbeta_1,\bx\>$, $\|\bbeta_1\|_2^2 =1$ with $\E[\eps^2] = 0$ using ReLU activation function $\sigma(x) = \max\{ x, 0\}$ and $\psi_2 = n / d = 3$.  We use minimum $\ell_2$-norm least squares (the $\lambda\to 0$ limit of Eq.~\eqref{eq:Ridge}, left figure) and regularized least squares with $\lambda = 10^{-3}$ (right figure), and plot the prediction error as a function of the number of parameters per dimension $\psi_1=N/d$. We compare the numerical results with the asymptotic formula $\cuR(\infty,\zeta,\psi_1,\psi_2,\lambda/\mu_\star^2)$. The agreement is excellent and displays all the key features of the double descent phenomenon, as discussed in the next section.

The proof of Theorem \ref{thm:MainIntro} builds on ideas from random matrix theory. A careful look at these arguments unveils an 
interesting phenomenon. While the random features $\{\sigma(\<\btheta_i,\bx\>/\sqrt{d})\}_{i\le d}$ are highly non-Gaussian, it is possible to
construct a Gaussian covariates model with the same asymptotic prediction error as for the  random features model.
Apart from being mathematically interesting, this finding provides additional intuition for the behavior of random features models, and
opens the way to some interesting future directions. In particular, \cite{montanari2019maxmargin} uses this Gaussian covariates proxy to analyze maximum margin
classification using random features. 

The rest of the paper is organized as follows:
\begin{itemize}
\item In Section \ref{sec:Insights} we summarize the main insights that can be extracted from the asymptotic theory,
and illustrate them through plots.
\item Section \ref{sec:Related} provides a succinct overview of related work.
\item Section \ref{sec:notations} introduces the notations that are used in this paper. 
\item Section \ref{sec:Main} contains formal statements of our main results, which is the asymptotics of prediction error as in Theorem \ref{thm:main_theorem}. This section also presents some special cases of the asymptotic formula. 
\item Section \ref{sec:training} contains the statements of the asymptotics of the training error as in Theorem \ref{thm:training_asymptotics}. 
\item Section \ref{sec:Gaussian_covariates} presents an interesting phenomenon which is that the random features model has the same asymptotic prediction error as a simpler model with Gaussian covariates. 
\item In Section \ref{sec:main_proof} we present the proof of main results. The main results will use several propositions that are proved in the following sections and in the appendices. 
\end{itemize}

\section{Results and insights: An informal overview}
\label{sec:Insights}

Before explaining in detail our technical results ---which we will do in Section \ref{sec:Main}--- it is useful to
pause and describe some consequences of the exact asymptotic formulae that we prove. Our focus here will be 
on insights that have a chance to hold more generally, beyond the specific setting studied here.

\vspace{0.25cm}

\noindent\emph{Bias term also exhibits a singularity at the interpolation threshold.} A prominent feature of the double descent curve is the peak in test error at the interpolation threshold
which, in the present case, is located at $\psi_1=\psi_2$. In the linear regression model of 
\cite{advani2017high,hastie2019surprises,belkin2019two}, this phenomenon is entirely explained by a peak in the variance of the estimator (that diverges in the ridgeless limit $\lambda\to 0$), while its bias is
monotone increasing across to this threshold. 

In contrast, in the random features model studied here, both variance and bias have a peak at the interpolation threshold, diverging there when $\lambda\to 0$. 
This is apparent from Figure \ref{fig:DoubleDescentFirst} which was obtained for $\tau^2=0$, and therefore in a setting in which the error is entirely due to bias.
The fact that the double descent scenario persists in the noiseless limit is particularly important, especially in view of the fact that many machine learning
tasks are usually considered nearly noiseless. 

\vspace{0.25cm}

\begin{figure}[!t]
\centering
\includegraphics[width = 0.48\linewidth]{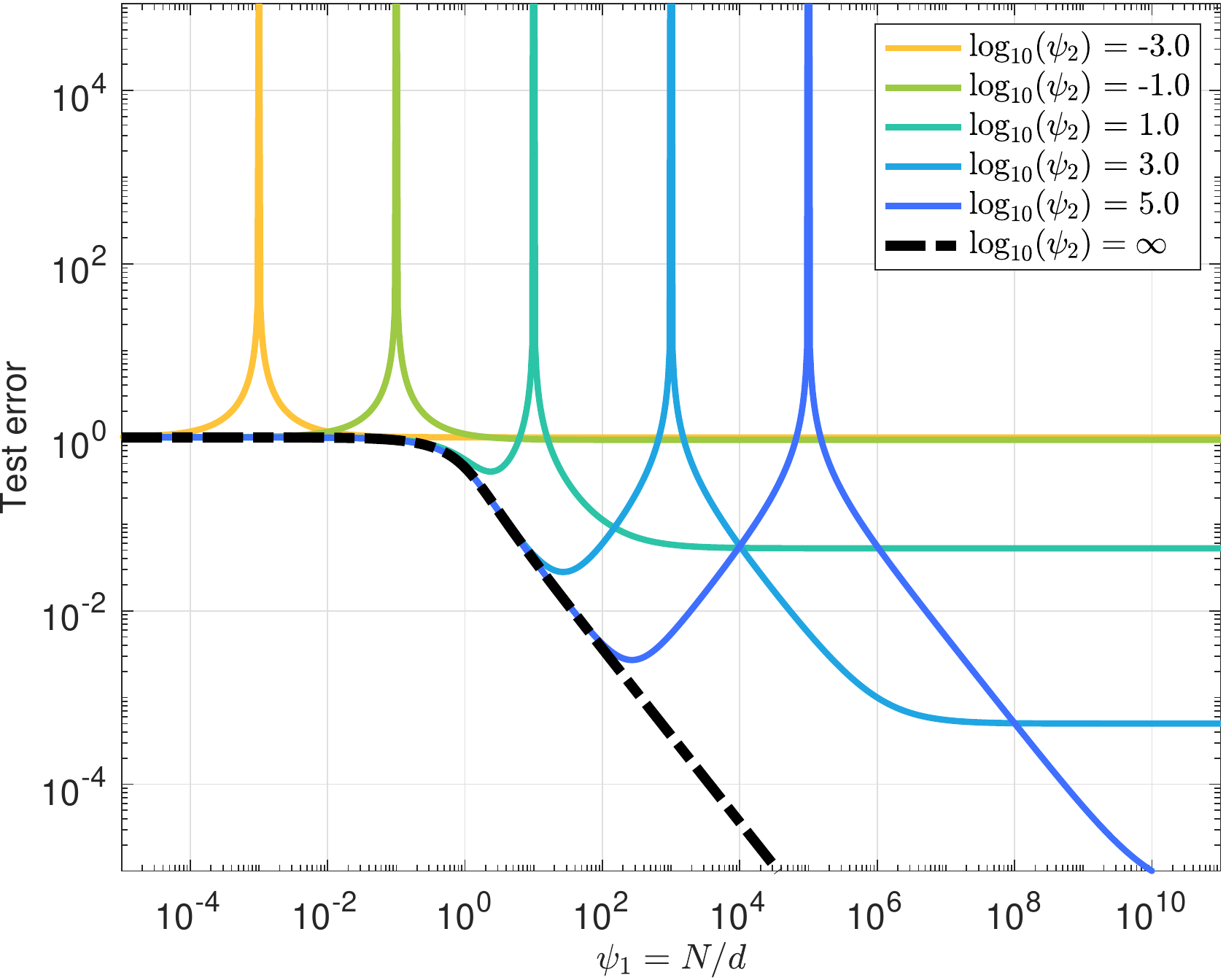}
\includegraphics[width = 0.48\linewidth]{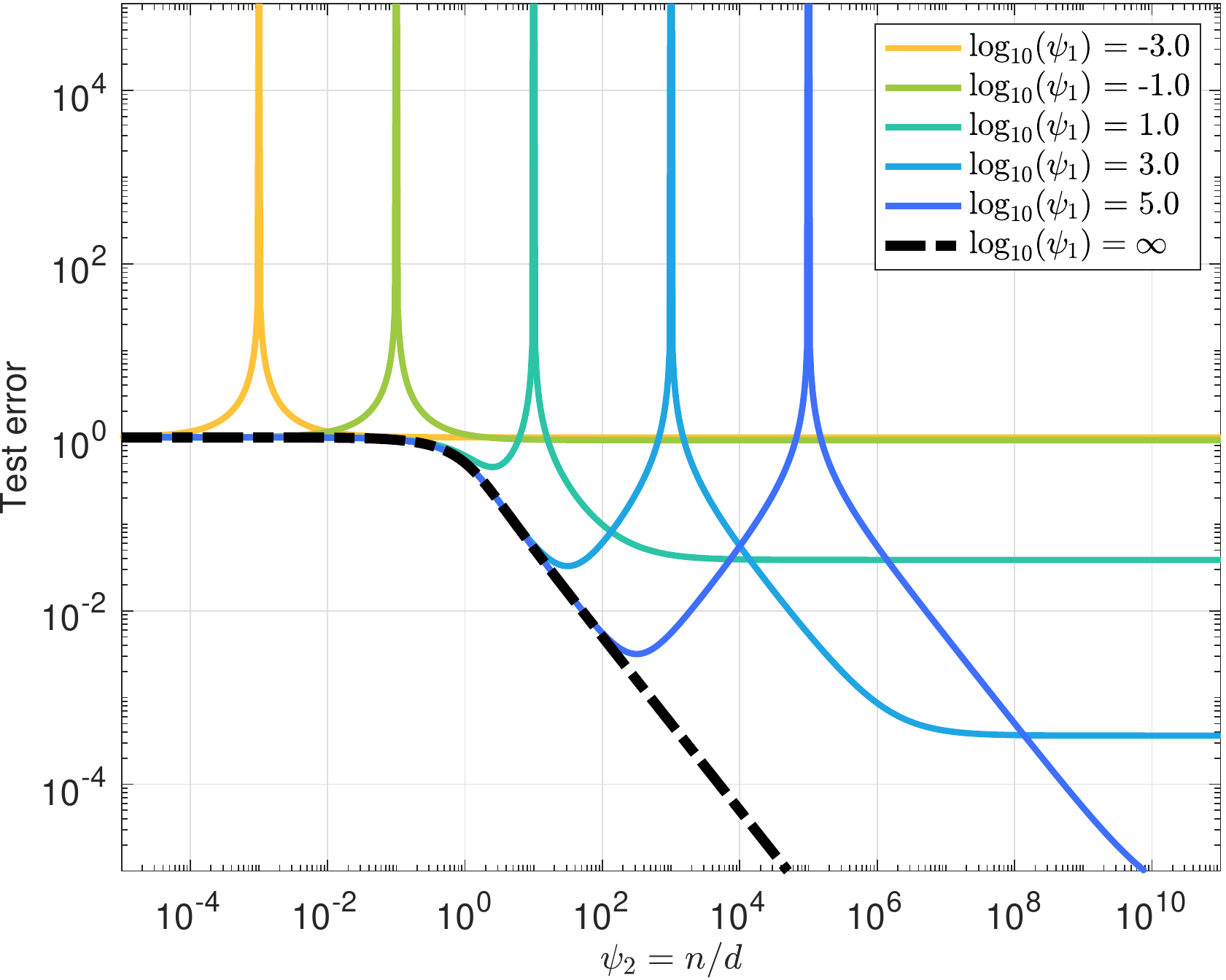}
\caption{Analytical predictions for the test error of learning a linear function $f_d(\bx) = \<\bbeta_1,\bx\>$ with $\| \bbeta_1 \|_2^2 = 1$ using random features with ReLU activation function  $\sigma(x) = \max\{ x, 0 \}$. 
Here we  perform ridgeless regression ($\lambda \to 0$). The signal-to-noise ratio is $\|\bbeta_1\|_2^2/\tau^2 \equiv \rho=2$. In the left figure, we plot the test error as a function of $\psi_1 = N/d$, and different curves correspond to different sample sizes ($\psi_2 = n/d$). In the right figure, we plot the test error as a function of $\psi_2 = n/d$, and different curves correspond to different number of features ($\psi_1 = N/d$). 
}\label{fig:zero_lambda_increased_psi1}
\end{figure}

\begin{figure}[!t]
\centering
\includegraphics[width = 0.48\linewidth]{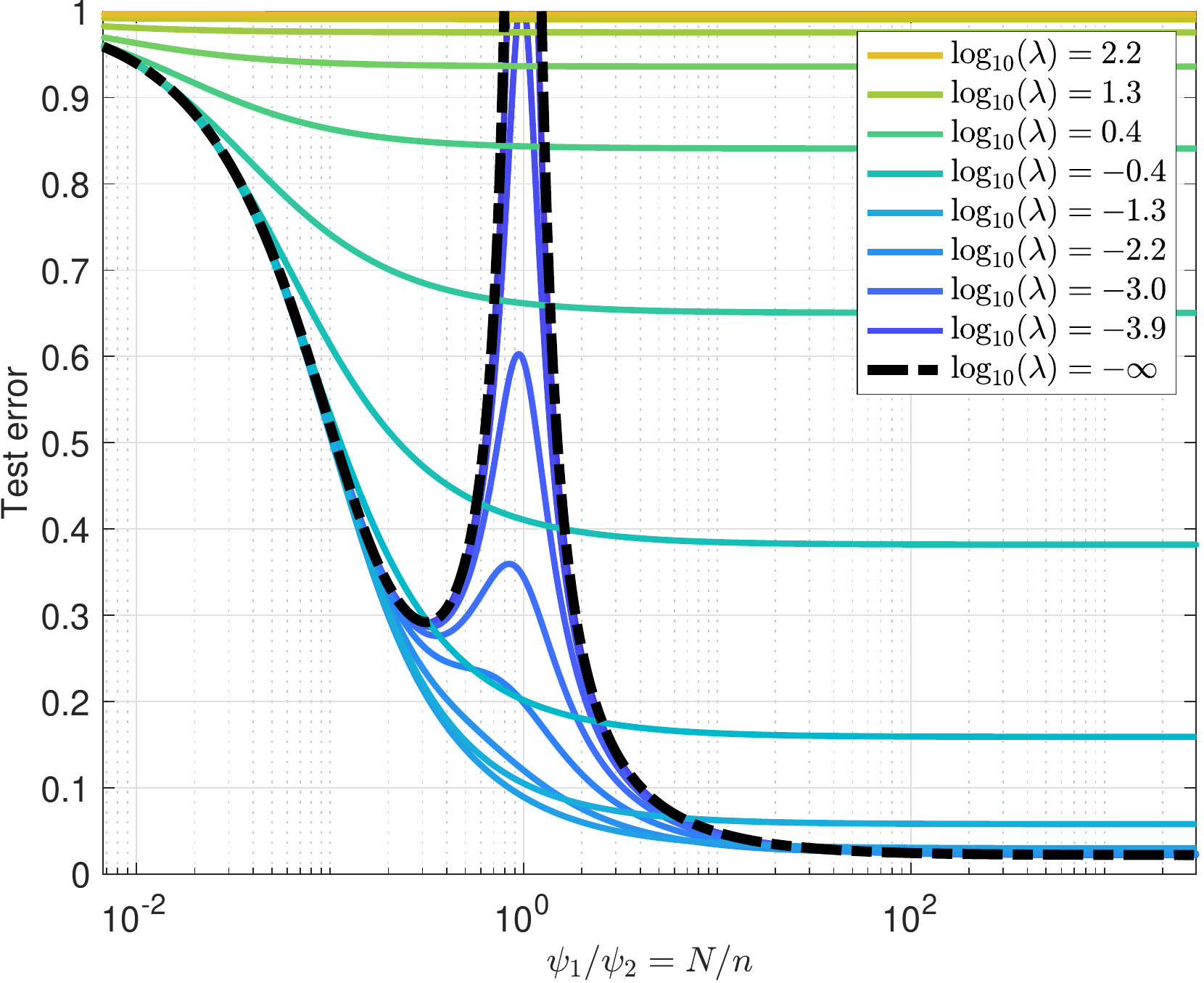}
\includegraphics[width = 0.48\linewidth]{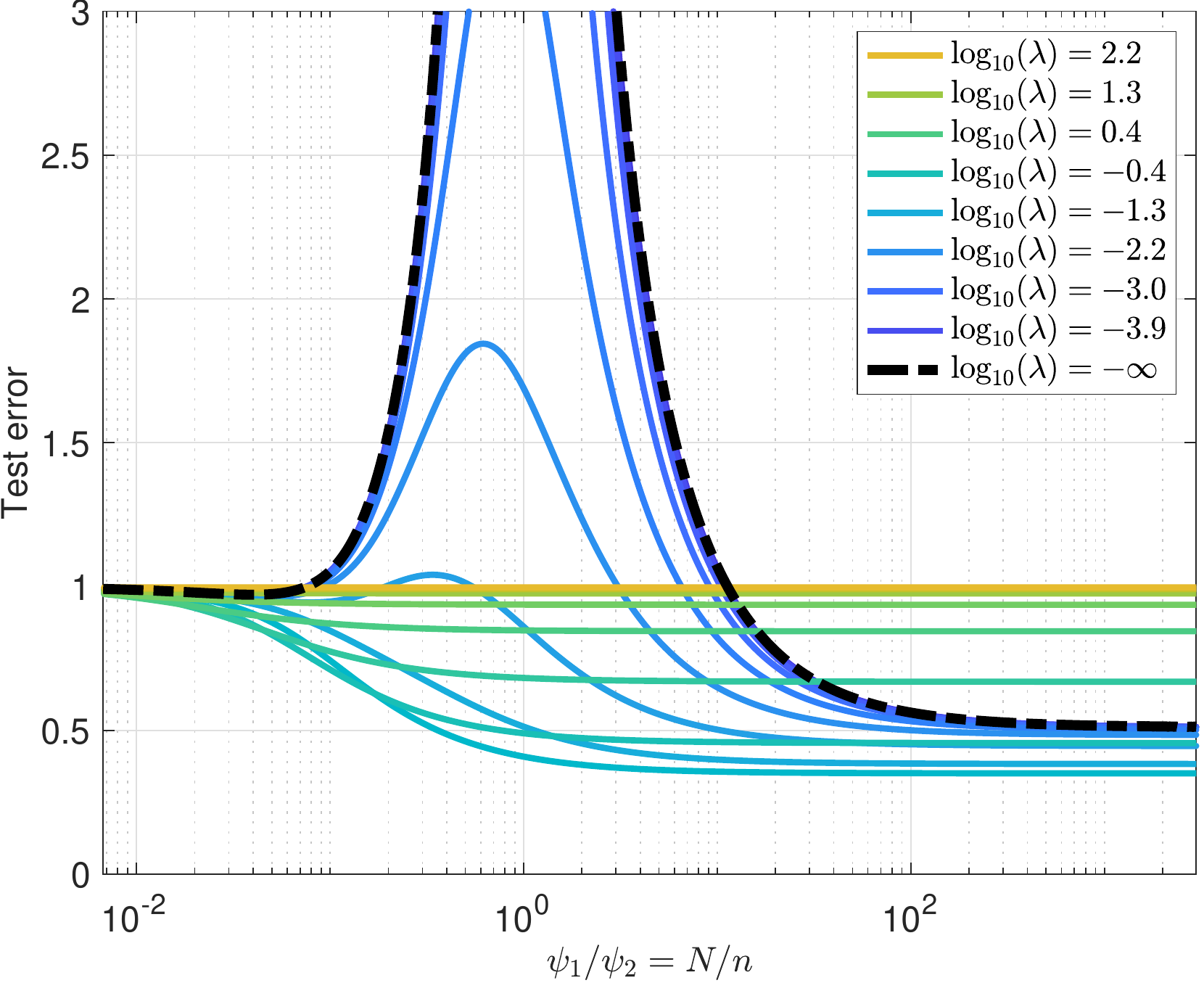}
\caption{Analytical predictions for the test error of learning a linear function $f_d(\bx) = \<\bbeta_1,\bx\>$ with $\| \bbeta_1 \|_2^2 = 1$ using random features with ReLU activation function  $\sigma(x) = \max\{x, 0\}$. 
The rescaled sample size is fixed to $n/d\equiv \psi_2=10$. Different curves are for different values of the regularization $\lambda$.
On the left: high SNR  $\|\bbeta_1\|_2^2/ \tau^2 \equiv \rho=5$.  On the right: low SNR  $\rho=1/5$. }\label{fig:fixed_lambda_increased_psi1}
\end{figure}

\noindent\emph{Optimal prediction error is achieved in the highly overparametrized regime.} Figure \ref{fig:zero_lambda_increased_psi1} (left) reports the predicted test error in the ridgeless limit $\lambda\to 0$ (for a case with non-vanishing noise, $\tau^2>0$) as a function of $\psi_1=N/d$ , for several values of $\psi_2=n/d$.
Figure \ref{fig:fixed_lambda_increased_psi1} plots the predicted test error as a function of $\psi_1/\psi_2=N/n$, for fixed $\psi_2$, 
several values of $\lambda>0$, and two values of the SNR.
We repeatedly observe that: $(i)$ For a fixed $\lambda$, the minimum of test error (over $\psi_1$) is in the highly overparametrized regime $\psi_1\to\infty$; $(ii)$ The global minimum (over $\lambda$ and $\psi_1$) of test error is achieved at a value of $\lambda$ that depends on the SNR, but always at $\psi_1\to\infty$; $(iii)$ In the ridgeless limit $\lambda\to 0$, the generalization curve is monotonically decreasing in $\psi_1$ when $\psi_1>\psi_2$.

To the best of our knowledge, this is the first natural and analytically tractable model which satisfies the following requirements:
$(1)$~Large overparametrization is necessary to achieve optimal prediction; $(2)$~No special misspecification structure needs to be postulated.

\vspace{0.25cm}

\noindent\emph{Optimal regularization eliminated the double-descent. } Figure \ref{fig:fixed_lambda_increased_psi1} reports
the asymptotic prediction for the test error as a function of the overparametrization ratio $N/n$
for various values of the regularization parameter $\lambda$. The peak at the interpolation threshold $N=n$ is
apparent, but it becomes less prominent as the regularization increases. In particular, if we consider the optimal regularization (the lower envelope of these curves), the test error becomes monotone decreasing in the number of parameters: regularization compensates overparametrization.

\begin{figure}[t]
\centering
\includegraphics[width = 0.48\linewidth]{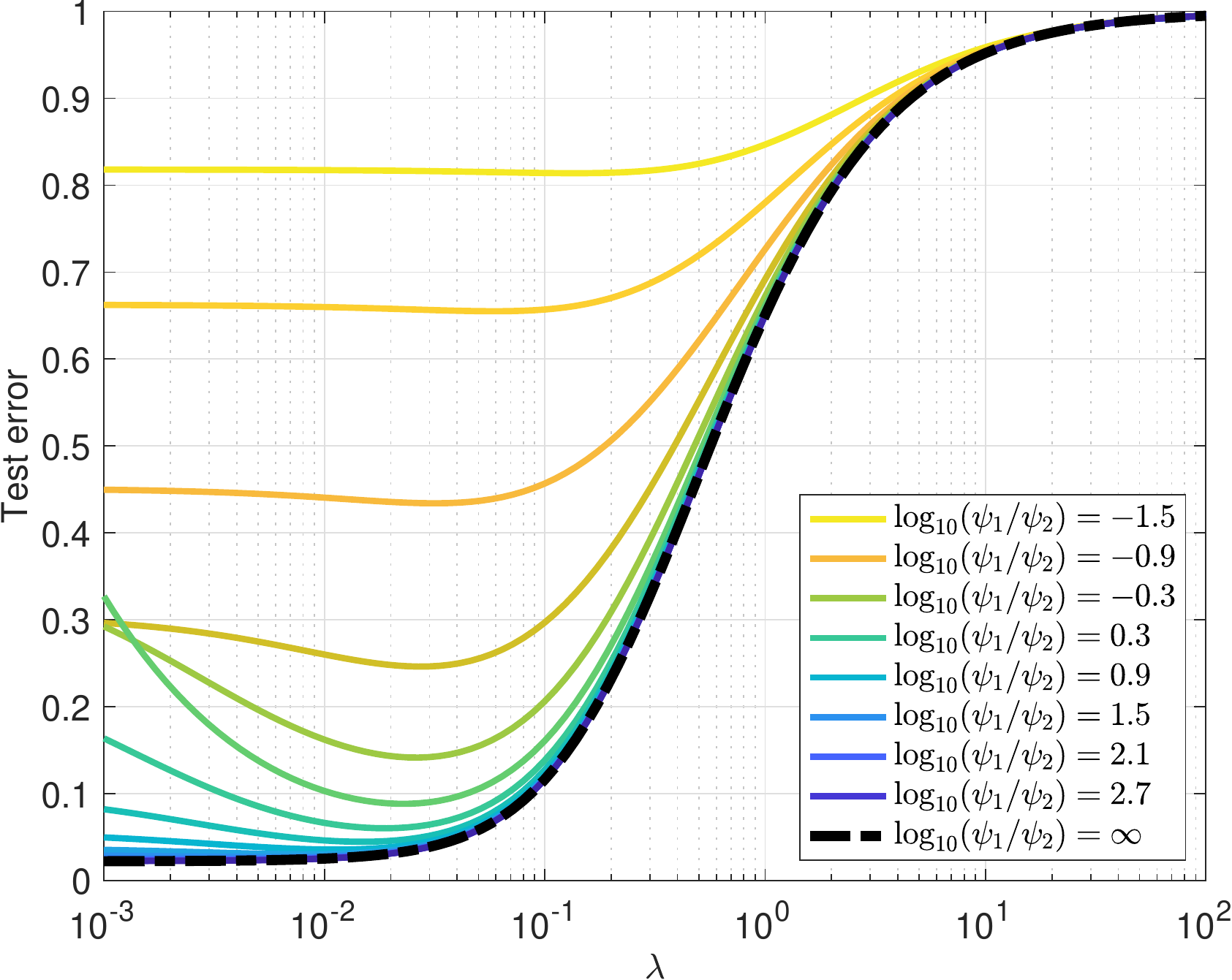}
\includegraphics[width = 0.48\linewidth]{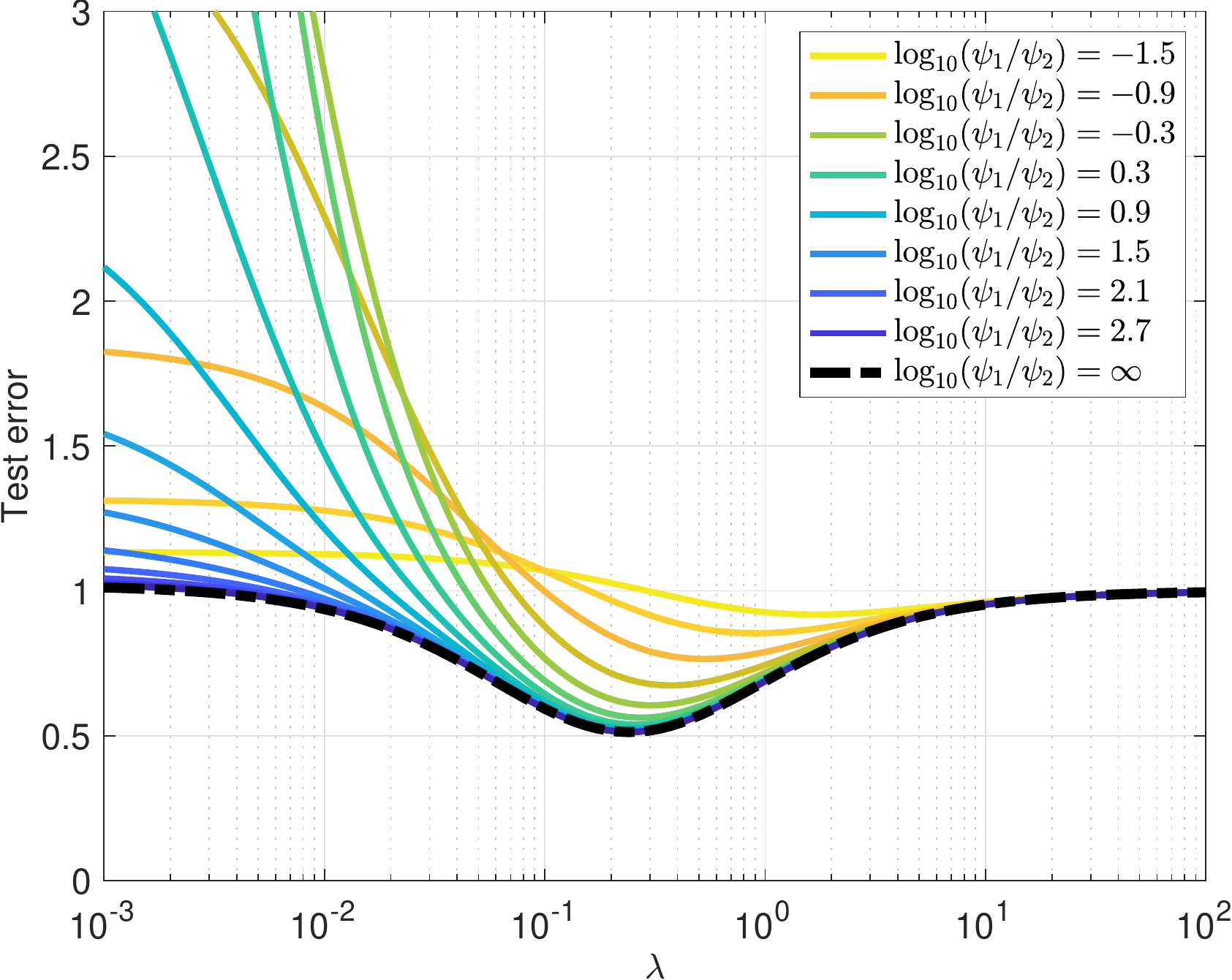}
\caption{Analytical predictions for the test error of learning a linear function $f_d(\bx) = \<\bbeta_1,\bx\>$ with $\| \bbeta_1 \|_2^2 = 1$ using random features with ReLU activation function  $\sigma(x) = \max\{x, 0\}$.
The rescaled sample size is fixed to $\psi_2 = n/d = 10$.
Different curves are for different values of the number of neurons $\psi_1 = N/d$.
On the left: high SNR  $\|\bbeta_1\|_2^2/ \tau^2 \equiv \rho=5$.  On the right: low SNR  $\rho=1/10$. }\label{fig:fixed_psi1_increased_lambda}
\end{figure}

\vspace{0.25cm}

\noindent\emph{Non-vanishing regularization can hurt (at high SNR).} Figure \ref{fig:fixed_psi1_increased_lambda} plots the predicted test error as a function of $\lambda$,
for several values of $\psi_1$, with $\psi_2$ fixed. The lower envelope of these curves is given by the curve at $\psi_1\to\infty$, confirming that the optimal error
is achieved in the highly overparametrized regime. However the dependence of this lower envelope on $\lambda$  changes qualitatively, depending on the SNR. 
For small SNR, the global minimum is achieved as some $\lambda>0$: regularization helps.  However, for a large SNR the minimum error is achieved as 
$\lambda\to 0$. The optimal regularization is vanishingly small. 

These  two noise regime are separated by a phase transition at a critical SNR which we denote by $\rho_{\star}$.
A characterization of this critical value is given in Section  \ref{sec:HighOver}.

Notice that, in the overparametrized regime, the training error vanishes as $\lambda\to 0$, 
and the resulting model is a `near-interpolator'.
We therefore conclude that  highly overparametrized (near)
interpolators\footnote{We cannot prove it is an exact interpolator because here we take $\lambda\to 0$ after $d\to\infty$.
Following Remark \ref{rmk:Ridgeless},  we expect the minimum-$\ell_2$ norm interpolator 
also to achieve asymptotically minimum error.} are statistically optimal 
when the SNR is above the critical value $\rho_{\star}$.

\vspace{0.25cm}

\noindent\emph{Self-induced regularization.} What is the mechanism underlying the optimality of the ridgeless limit
$\lambda\to 0$? An intuitive explanation can be obtained by considering the (random) kernel associated to the ridge regression
\eqref{eq:Ridge}, namely
\begin{align}
  \cH_N(\bx, \bx') = \frac{1}{N} \sum_{i=1}^N \sigma(\< \bx, \btheta_i\>/\sqrt d) \sigma(\< \bx', \btheta_i \> / \sqrt d)\, .
  \label{eq:HN}
\end{align}
The diagonal elements of the empirical kernel $\bcH_{N,n} = (\cH_N(\bx_i,\bx_j))_{i,j\le n}\in\reals^{n\times n}$
concentrate around the value $\E \cH_N(\bx_i,\bx_i)\approx \E\{\sigma(G)^2\}$ (here $G\sim\normal(0,1)$),
while the out-of-diagonal terms are equal to a constant $\E\{\sigma(G)\}^2$ plus fluctuations
of order $1/\sqrt{N}$. One would naively expect that these diagonal elements are equivalent to
a regularization (that we call `self-induced') $\lambda_0$ of order $\Var(\sigma(G))$. The reality is
more complicated because out-of-diagonals are random and not negligible, 
However, this intuition is essentially correct in the wide limit $N/d \to\infty$
(after $N,n,d\to\infty$), see Section \ref{sec:HighOver}.

\section{Related literature}
\label{sec:Related}

\subsection{Learning via interpolation}

A recent stream of papers studied the generalization behavior of machine learning models in the interpolation regime.
An incomplete list of references includes \cite{belkin2018understand, belkin2018does, liang2018just, belkin2018reconciling, rakhlin2018consistency}. The starting point of this line of work were the experimental results in \cite{zhang2016understanding, belkin2018understand}, which showed that deep neural networks as well as kernel methods can generalize even if the prediction function interpolates all the data. It was proved that several machine learning models including kernel regression \cite{belkin2018does} and kernel ridgeless regression \cite{liang2018just} can generalize under certain conditions.  

The double descent phenomenon, which is our focus in this paper, was first discussed in general terms in  \cite{belkin2018reconciling}. The same phenomenon was also observed in \cite{advani2017high, geiger2019scaling}.
The paper \cite{kobak2018implicit} observes that the optimal amount of ridge regularization is sometimes vanishing,
and provides an explanation in terms of noisy features. Analytical predictions confirming this scenario were
obtained, within the linear regression model, in two concurrent papers \cite{hastie2019surprises, belkin2019two}.
In particular, \cite{hastie2019surprises} derives the precise high-dimensional asymptotics of the prediction error,
for a general model with correlated covariates. On the other hand, \cite{belkin2019two} gives exact formula for
any finite dimension, for a model with i.i.d. Gaussian covariates. The same papers also compute the double
descent curve within other models, including over-specified linear model \cite{hastie2019surprises}, and a
Fourier series model \cite{belkin2019two}.

As mentioned in the introduction, \cite[Section 8]{hastie2019surprises} also calculates the variance term of the prediction error in the random features model in the ridgeless limit $\lambda \to 0$. Both the simple linear regression models of
\cite{hastie2019surprises, belkin2019two}, and the variance calculation of \cite[Section 8]{hastie2019surprises}
capture the peak of the test error at the interpolation threshold. However, these calculations do not elucidate
several crucial statistical phenomena, which are instead the main contribution of our work (see Section \ref{sec:Insights}):
optimality of large overparametrization; optimality of interpolators at high SNR ($\lambda\to 0$ limit);
the role of self-induced regularization; disappearance of the double descent at optimal overparametrization.

Rate-optimal bounds on the generalization error of overparametrized linear models were recently derived in
\cite{bartlett2019benign} (see also \cite{muthukumar2019harmless} for a different perspective).  

\subsection{Random features and kernels}

The random features model has been studied in considerable depth since the original work in \cite{rahimi2008random}. A 
classical viewpoint suggests that $\cF_{\RF}(\bTheta)$ should be regarded as random approximation of the reproducing kernel Hilbert space $\cF_H$ defined by the kernel 
\begin{align}
\cH(\bx, \bx') = \E_{\btheta \sim \Unif(\S^{d-1}(\sqrt d))}[\sigma(\< \bx, \btheta\>/\sqrt d) \sigma(\< \bx', \btheta\>/\sqrt d)]. \label{eq:H-Kernel}
\end{align}
Indeed $\cF_{\RF}(\bTheta)$ is an RKHS defined by the finite-rank approximation of this kernel
defined in Eq.~\eqref{eq:HN}.
The paper \cite{rahimi2008random} showed the pointwise convergence of the empirical kernel $H_N$ to $H$. Subsequent work \cite{bach2017equivalence} showed the convergence of the empirical kernel matrix to the population kernel in terms of operator norm and derived bound on the approximation error
(see also \cite{bach2013sharp, alaoui2015fast, rudi2017generalization} for related work). 

The setting in the present paper is quite different, since we take the limit of a large number neurons $N\to\infty$, together with large dimension $d\to\infty$. Our focus on this high-dimensional regime is partially motivated by
\cite{rakhlin2018consistency}, which emphasizes that optimality of interpolators is somewhat un-natural in low dimension.

It is well-known that approximation using two-layers network suffers from the curse of dimensionality,
in particular when first-layer weights are not trained 
\cite{devore1989optimal, bach2017breaking, vempala2018gradient, ghorbani2019linearized}.
The recent paper \cite{ghorbani2019linearized} studies random features regression in a setting similar to ours,
by considering two different regimes: $(1)$ the population limit $n=\infty$, with $N$ scaling as a
polynomial of $d$; $(2)$ the wide limit $N=\infty$, with  $n$ scaling as a
polynomial of $d$. In particular, \cite{ghorbani2019linearized}  proves that, if $d^{k + \delta} \le N \le d^{k+1-\delta}$ and $n=\infty$ or $d^{k + \delta} \le n \le d^{k+1-\delta}$ and $N=\infty$ then a random features model can only fit the projection of the true function $f_d$ onto degree-$k$ polynomials.

Here, we consider $N, n = \Theta_d(d)$,  and therefore \cite{ghorbani2019linearized} only implies
that the test error of the random feature model is (asymptotically) lower bounded by the norm of
the nonlinear component of the target function
$F_{\star}^2=\lim_{d\to\infty}\E(f^{\sNL}_{d}(\bx)^2)$. 
The present results are of course much more precise: we confirm this lower bound which is achieved in the limit $N/d,n/d\to\infty$, but also derive the precise asymptotics of the test error for finite $n/d$, $N/d$.
The connection between neural networks and random features models was pointed out originally in \cite{neal1996priors, williams1997computing} 
and has attracted significant attention recently \cite{hazan2015steps, matthews2018gaussian, lee2017deep, novak2018bayesian, garriga2018deep}. 
The papers \cite{daniely2016toward, daniely2017sgd} showed that, for a certain initialization, gradient descent training of overparametrized neural networks 
learns a function in an RKHS, which corresponds to the random features kernel. A recent line of work 
\cite{jacot2018neural, li2018learning, du2018gradient, du2018gradientb, allen2018convergence, allen2018learning, arora2019fine, zou2018stochastic, oymak2019towards} 
studied the training dynamics of overparametrized neural networks under a second type of initialization, and showed that it learns a function in a different but comparable RKHS, 
which corresponds to the ``neural tangent kernel". A concurrent approach \cite{mei2018mean, rotskoff2018neural, chizat2018global, sirignano2019mean, javanmard2019analysis, nguyen2019mean, rotskoff2019neuron, araujo2019mean} studies the training dynamics of overparametrized neural networks under a third type of initialization, 
and showed that the dynamics of empirical distribution of weights follows Wasserstein gradient flow of a risk functional. The connection
between neural tangent theory and Wasserstein gradient flow  was studied in \cite{chizat2018note, dou2019training, mei2019mean}.

\subsection{Technical contribution}

We use methods from random matrix theory.  The general class of matrices we need to consider are kernel inner
product random matrices, namely matrices of the form $\sigma(\bW\bW^\sT / \sqrt d)$, where $\bW$
is a random matrix with i.i.d. entries, or similar ($\sigma: \R \to \R$ is a scaler function and for a matrix $\bA \in \R^{m \times n}$, $\sigma(\bA) \in \R^{m \times n}$ is a matrix that formed by applying $\sigma$ to $\bA$ elementwisely). The paper \cite{el2010spectrum} studied the spectrum
of random kernel matrices when $\sigma$ can be well approximated by a linear function and hence the
spectrum converges to a scaled Marchenko-Pastur law. In the nonlinear regime, the spectrum
was shown to converge to the free convolution of a Marchenko-Pastur  and a
scaled semi-circular law \cite{cheng2013spectrum}. The extreme eigenvalues of the same random
matrix were studied in \cite{fan2019spectral}. The random matrix we need to consider
is an asymmetric kernel matrix $\bZ = \sigma(\bX \bTheta^\sT / \sqrt d)/\sqrt{d}$, whose asymptotic singular values distribution was calculated in \cite{pennington2017nonlinear} (see also \cite{louart2018random} for $\bX$ deterministic).

The asymptotic singular values distribution of $\bZ$ is not sufficient to compute the asymptotic prediction error,
which also depends on the singular vectors of $\bZ$.
The  paper \cite{hastie2019surprises} addresses this challenge for what concerns the variance
term of the error, and only in the limit $\lambda\to 0$. Notice that the variance term is given (up to constants)
by $\Tr((\bZ^{\sT}\bZ)^{\dagger}\bSigma)$. It is quite straightforward to express this
quantity in terms of the \textit{Stieltjes transform} of a certain block random matrix,  and
\cite{hastie2019surprises}  use the leave-one-out method to characterize
the asymptotics of this Stieltjes transform. 

Unfortunately, the approach of  \cite{hastie2019surprises}  cannot be pushed to compute the full test error (i.e.
both the bias and variance terms): the latter cannot be expressed in terms of  the Stieltjes transform
of the same matrix. A key observation of the present paper is that the full prediction error can be expressed in terms
of derivatives of the \textit{log-determinant} of a different block-structured random
matrix.
In order to compute the asymptotics of this log-determinant, we use leave-one-out arguments
(e.g. \cite[Chapter 3.3]{bai2010spectral}) to derive fixed point equations for the Stieltjes transform of this random matrix, and then integrate this Stieltjes transform.

One further difference from \cite{hastie2019surprises} is that we consider the full nonparametric model
$y_i = f_d(\bx_i)+\eps_i$  while for  the  calculation of \cite{hastie2019surprises}  does not model the target function.
As mentioned above, our setting is similar  to the one of \cite{ghorbani2019linearized}.
However, the main technical content of \cite{ghorbani2019linearized} is to prove that, under polynomial scalings of $n$ and $d$ (at $N=\infty$) or $N$ and $d$ (at $n=\infty$),
the kernel matrix is near isometric. In contrast, here we study a regime in which it is not
true that the same matrix is a near isometry, and we characterize its spectral distribution (alongside those
properties of the eigenvectors that determine the test error).

\section{Notations}\label{sec:notations}

Let $\R$ denote the set of real numbers, $\C$ the set of complex numbers, and $\N = \{ 0, 1, 2, \ldots\}$ the set of natural numbers. For $z \in \C$, let $\Re z$ and $\Im z$ denote the real part and the imaginary part of $z$ respectively. We denote by $\C_+ = \{ z \in \C: \Im z > 0 \}$ the set of complex numbers with positive imaginary part. We denote by $\imagunit = \sqrt{-1}$ the imaginary unit. We denote by $\S^{d-1}(r) = \{ \bx \in \R^d: \| \bx \|_2 = r \}$ the set of $d$-dimensional vectors with radius $r$. For an integer $k$, let $[k]$ denote the set $\{ 1, 2, \ldots, k\}$. 

Throughout the proofs, let $O_d(\, \cdot \, )$  (respectively $o_d (\, \cdot \,)$, $\Omega_d(\, \cdot \,)$) denote the standard big-O (respectively little-o, big-Omega) notation, where the subscript $d$ emphasizes the asymptotic variable. We denote by $O_{d,\P} (\, \cdot \,)$ the big-O in probability notation: $h_1 (d) = O_{d,\P} ( h_2(d) )$ if for any $\eps > 0$, there exists $C_\eps > 0 $ and $d_\eps \in \Z_{>0}$, such that
\[
\begin{aligned}
\P ( |h_1 (d) / h_2 (d) | > C_{\eps}  ) \le \eps, \qquad \forall d \ge d_{\eps}.
\end{aligned}
\]
We denote by $o_{d,\P} (\, \cdot \,)$ the little-o in probability notation: $h_1 (d) = o_{d,\P} ( h_2(d) )$, if $h_1 (d) / h_2 (d)$ converges to $0$ in probability. We write $h(d) = O_d(\Poly(\log d))$, if there exists a constant $k$, such that $h(d) = O_d((\log d)^k)$. 

Throughout the paper, we use bold lowercase letters $\{\bx, \by, \bz, \ldots\}$ to denote vectors and bold uppercase letters $\{\bA, \bB, \bC, \ldots\}$ to denote matrices. We denote by $\id_n \in \R^{n \times n}$ the identity matrix, by $\ones_{n \times m} \in \R^{n \times m}$ the all-ones matrix, and by $\bzero_{n \times m} \in \R^{n \times m}$ the all-zero matrix.  

For a matrix $\bA \in \R^{n \times m}$, we denote by $\| \bA \|_F = (\sum_{i \in [n], j \in [m] } A_{ij}^2 )^{1/2}$ the Frobenius norm of $\bA$, $\| \bA \|_\star$ the nuclear norm of $\bA$, $\| \bA \|_{\op}$ the operator norm of $\bA$, and $\| \bA \|_{\max} = \max_{i \in [n], j \in [m]} \vert A_{ij} \vert$ the maximum norm of $\bA$. Further, we denote by $\bA^\dagger \in \R^{m \times n}$ the Moore–Penrose inverse of matrix $\bA \in \R^{n \times m}$. For a measurable function $h: \R \to \R$ and a matrix $\bA \in \R^{n \times m}$, we denote $h(\bA) = (h(A_{ij}))_{i \in [n], j \in [m]} \in \R^{n \times m}$. For a matrix $\bA \in \R^{n \times n}$, we denote by $\Tr(\bA) = \sum_{i=1}^n A_{ii}$ the trace of $\bA$. For two integers $a$ and $b$, we denote by $\Tr_{[a, b]}(\bA) = \sum_{i=a}^b A_{ii}$ the partial trace of $\bA$. For two matrices $\bA, \bB \in \R^{n \times m}$, let $\bA \odot \bB$ denote the element-wise product of $\bA$ and $\bB$.

Let $\mu_G$ denote the standard Gaussian measure. Let $\gamma_d$ denote the uniform probability distribution on $\S^{d-1}(\sqrt d)$. We denote by $\mu_d$ the distribution of $\< \bx_1, \bx_2\> / \sqrt d$ when $\bx_1, \bx_2 \sim \normal(\bzero, \id_d)$, $\tau_d$ the distribution of $\< \bx_1, \bx_2\> / \sqrt d$ when $\bx_1, \bx_2 \sim \Unif(\S^{d-1}(\sqrt d))$, and $\tilde \tau_d$ the distribution of $\< \bx_1, \bx_2\>$ when $\bx_1, \bx_2 \sim \Unif(\S^{d-1}(\sqrt d))$.

\section{Main results}
\label{sec:Main}

We begin by stating our assumptions and notations for the activation
function $\sigma$. It is straightforward to check that these are satisfied by all commonly-used activations,
including ReLU and sigmoid functions. 
\begin{assumption}\label{ass:activation}
Let $\sigma: \R \to \R$ be weakly differentiable, with weak derivative $\sigma'$. 
Assume $\vert \sigma(u)\vert, \vert \sigma'(u)\vert \le c_0 e^{c_1 \vert u \vert}$ for some constants $c_0, c_1 < \infty$. 
Define
\begin{align}
\ob_0 \equiv \E\{ \sigma(G)\}, ~~~~ \ob_1\equiv\E\{G\sigma(G)\},~~~~
\ob_\star^2 \equiv\E\{\sigma(G)^2\}- \ob_0^2- \ob_1^2 \, ,
\end{align}
where expectation is with respect to $G\sim\normal(0,1)$. 
Assuming $0 < \ob_0^2, \ob_1^2, \ob_{\star}^2 < \infty$, define  $\ratio$ by
\begin{align}
\ratio \equiv \frac{\ob_1}{ \ob_\star}\, . 
\end{align}
\end{assumption}
We will consider sequences of parameters $(N,n,d)$ that diverge proportionally to each other. When necessary, we can
think such sequences to be indexed by $d$, with $N=N(d)$, $n=n(d)$ functions of $d$.
\begin{assumption}\label{ass:linear}
Defining $\psi_{1, d} = N/d$ and $\psi_{2, d} = n / d$, we assume that the following limits exist in $(0,\infty)$:
\begin{align}
\lim_{d \to \infty}  \psi_{1, d} = \psi_1, ~~~~~~~~~
\lim_{d \to \infty} \psi_{2, d} = \psi_2\, .
\end{align}
\end{assumption}

Our last assumption concerns the distribution of the data $(y,\bx)$, and, in particular,  the regression function $f_d(\bx) = \E[y\vert \bx]$.
As stated in the introduction, we take $f_d$ to be the sum of a deterministic linear component, and a nonlinear component that we assume to be random and isotropic.
\begin{assumption}\label{ass:ground_truth}
We assume $y_i=f_d(\bx_i)+\eps_i$, where $(\eps_i)_{i\le n} \sim_{iid} \P_\eps$ independent of $(\bx_i)_{i\le n}$,
 with $\E_\eps(\eps_1) = 0$, $\E_{\eps}(\eps_1^2) = \tau^2$, $\E_{\eps}(\eps_1^4) < \infty$. Further 
\begin{align}
 f_d(\bx) =&~ \beta_{d, 0} +\<\bbeta_{d, 1},\bx\> + f_d^{\sNL}(\bx)\, , 
\end{align}
where $\beta_{d, 0}\in\reals$ and $\bbeta_{d, 1} \in\reals^d$ are deterministic with  $\lim_{d\to\infty}\beta_{d, 0}^2 =F_0^2$, $\lim_{d\to\infty}\|\bbeta_{d, 1}\|_2^2=F_1^2$.
The nonlinear component  $f_d^{\sNL}(\bx)$ is a centered Gaussian process indexed by $\bx \in \S^{d-1}(\sqrt d)$, with covariance 
\begin{align}
\E_{f_d^{\sNL}}\{f_d^{\sNL}(\bx_1) f_d^{\sNL}(\bx_2)\} = \Sigma_d(\<\bx_1,\bx_2\> / d)\,
\end{align}
satisfying $\E_{\bx \sim \Unif(\S^{d-1}(\sqrt d))} \{\Sigma_d( x_1 / \sqrt d)\}= 0$, $\E_{\bx \sim \Unif(\S^{d-1}(\sqrt d))}\{\Sigma_d( x_1 / \sqrt d) x_1 \}=0$, and $\lim_{d \to \infty} \Sigma_d(1) = \normf_\star^2$.
We define the signal-to-noise ratio parameter $\rho$ by
\begin{align}
\rho= \frac{\normf_1^2}{\normf_\star^2 + \tau^2} \, . 
\end{align}
\end{assumption}
\begin{remark}
The last assumption covers, as a special case, deterministic linear functions $f_d(\bx) = \beta_{d, 0}+\<\bbeta_{d, 1},\bx\>$, but also a large class
of  random non-linear functions. As an example, let $\bG = (G_{ij})_{i,j\le d}$, where $(G_{ij})_{i,j\le d} \sim_{iid}\normal(0,1)$, and consider the
random quadratic function 
\begin{align}
f_d(x) = \beta_{d, 0} +\<\bbeta_{d, 1},\bx\> +\frac{F_\star}{d} \big[\<\bx,\bG\bx\>-\Tr(\bG)\big]\, ,
\end{align}
for some fixed $F_\star \in\reals$. It is easy to check that this $f_d$ satisfies Assumption \ref{ass:ground_truth}, where the covariance function gives
\[
\Sigma_d(\< \bx_1, \bx_2\> / d) = \frac{\normf_\star^2}{d^2}  \Big( \< \bx_1, \bx_2\>^2 - d \Big).
\]
Higher order polynomials
can be constructed analogously (or using the expansion of $f_d$ in spherical harmonics).  

We also emphasize that that the nonlinear part  $f_d^{\sNL}(\bx_2)$, although being random, is the same for all 
samples, and hence should not be confused with additive noise $\eps$.
\end{remark}

We finally introduce the formula for the asymptotic prediction error, denoted by $\cuR(\rho,\zeta,\psi_1,\psi_2,\lambda)$ in Theorem \ref{thm:MainIntro}.
\begin{definition}[Formula for the prediction error of random features regression]\label{def:formula_ridge}~
Let the functions $\nu_1, \nu_2: \C_+ \to \C_+$ be be uniquely defined by the following conditions: $(i)$ $\nu_1$, $\nu_2$ are analytic on $\C_+$;
$(ii)$ For $\Im(\xi)>0$, $\nu_1(\xi)$, $\nu_2(\xi)$ satisfy the following equations
\begin{equation}\label{eqn:nu_definition}
\begin{aligned}
\nu_1 =&~ \psi_1\Big(-\xi -  \nu_2 - \frac{\ratio^2 \nu_2}{1- \ratio^2 \nu_1\nu_2}\Big)^{-1}\, ,\\
\nu_2 =&~ \psi_2\Big(-\xi - \nu_1 - \frac{\ratio^2 \nu_1}{1- \ratio^2 \nu_1\nu_2}\Big)^{-1}\, ;
\end{aligned}
\end{equation}
$(iii)$  $(\nu_1(\xi), \nu_2(\xi))$ is the unique solution of these equations with $\vert \nu_1(\xi)\vert\le \psi_1/\Im(\xi)$, $\vert \nu_2(\xi) \vert \le \psi_2/\Im(\xi)$ for $\Im(\xi) > C$, with $C$ a sufficiently large constant.

Let  
\begin{equation}\label{eqn:definition_chi_main_formula}
\chi \equiv \nu_1(\imagunit ( \psi_1 \psi_2 \olambda)^{1/2}) \cdot \nu_2(\imagunit ( \psi_1 \psi_2 \olambda)^{1/2}),
\end{equation}
and
\begin{equation}\label{eq:E012def}
\begin{aligned}
\cuE_0(\ratio, \psi_1, \psi_2, \olambda) \equiv&~  - \chi^5\ratio^6 + 3\chi^4 \ratio^4+ (\psi_1\psi_2 - \psi_2 - \psi_1 + 1)\chi^3\ratio^6 - 2\chi^3\ratio^4 - 3\chi^3\ratio^2 \\
&+ (\psi_1 + \psi_2 - 3\psi_1\psi_2 + 1)\chi^2\ratio^4 + 2\chi^2\ratio^2+ \chi^2+ 3\psi_1\psi_2\chi\ratio^2 - \psi_1\psi_2\, ,\\
\cuE_1(\ratio, \psi_1, \psi_2,\olambda)  \equiv&~ \psi_2\chi^3\ratio^4 - \psi_2\chi^2\ratio^2 + \psi_1\psi_2\chi\ratio^2 - \psi_1\psi_2\, , \\
\cuE_2(\ratio, \psi_1, \psi_2, \olambda) \equiv&~ \chi^5\ratio^6 - 3\chi^4\ratio^4+ (\psi_1 - 1)\chi^3\ratio^6 + 2\chi^3\ratio^4 + 3\chi^3\ratio^2 + (- \psi_1 - 1)\chi^2\ratio^4 - 2\chi^2\ratio^2 - \chi^2\,.\\
\end{aligned}
\end{equation}
We then define
\begin{align}
\cuB(\ratio, \psi_1, \psi_2, \olambda) \equiv&~  \frac{\cuE_1(\ratio, \psi_1, \psi_2, \olambda) }{\cuE_0 (\ratio, \psi_1, \psi_2, \olambda)  }\, , \label{eqn:main_bias}\\
\cuV(\ratio, \psi_1, \psi_2, \olambda) \equiv&~  \frac{\cuE_2(\ratio, \psi_1, \psi_2, \olambda) }{\cuE_0 (\ratio, \psi_1, \psi_2, \olambda)  }\, , \label{eqn:main_var}\\
\cuR(\rho, \ratio, \psi_1, \psi_2, \olambda) \equiv&~ \frac{\rho}{1 + \rho} \, \cuB(\ratio, \psi_1, \psi_2, \olambda) + \frac{1}{1 + \rho} \,  \cuV(\ratio, \psi_1, \psi_2, \olambda)\, .
\label{eq:RTheoryDef}
\end{align}
\end{definition}

The formula for the asymptotic risk can be easily evaluated numerically. In order to gain further insight, it can be simplified in some interesting special cases, as shown in Section \ref{sec:SimplifyFormula}.

\subsection{Statement of main result}\label{subsec:main_results}

We are now in position to state our main theorem, which generalizes Theorem \ref{thm:MainIntro} to the case in which
$f_d$ has a  nonlinear component $f_d^{\sNL}$.

\begin{theorem}\label{thm:main_theorem}
Let $\bX = (\bx_1, \ldots, \bx_n)^\sT \in \R^{n \times d}$ with $(\bx_i)_{i \in [n]} \sim_{iid} \Unif(\S^{d-1}(\sqrt d))$ and $\bTheta
= (\btheta_1, \ldots \btheta_N)^\sT \in \R^{N \times d}$ with $(\btheta_a)_{a\in [N]} \sim_{iid} \Unif(\S^{d-1}(\sqrt d))$ independently. 
Let the activation function $\sigma$ satisfy Assumption \ref{ass:activation}, and consider proportional asymptotics $N/d\to\psi_1$, $n/d\to\psi_2$,  
as per Assumption \ref{ass:linear}. Finally, let the regression function  $\{f_d \}_{d \ge 1}$ and the response variables $(y_i)_{i \in [n]}$ satisfy 
Assumption \ref{ass:ground_truth}.

Then for any value of the regularization parameter $\lambda > 0$, the asymptotic prediction error of random features ridge regression satisfies
\begin{align}\label{eqn:main}
\E_{\bX, \bTheta, \beps, f_d^{\sNL}} \Big\vert R_\RF(f_d, \bX, \bTheta, \lambda) - \Big[  \normf_1^2 \cuB(\ratio, \psi_1, \psi_2, \lambda/\ob_{\star}^2)+ 
(\tau^2 + \normf_\star^2) \cuV(\ratio, \psi_1, \psi_2, \lambda/\ob_{\star}^2) + 
\normf_\star^2  \Big] \Big\vert = o_{d}(1)\, ,
\end{align}
where $\E_{\bX, \bTheta, \beps, f_d^{\sNL}}$ denotes expectation with respect to data covariates $\bX$, feature vectors $\bTheta$, data noise $\beps$, and $f_d^{\sNL}$ the nonlinear part of the true regression function (as a Gaussian process), as per Assumption \ref{ass:ground_truth}.
The functions $\cuB, \cuV$ are given in Definition \ref{def:formula_ridge}. 
\end{theorem}

\begin{remark}
If the regression function $f_d(\bx)$ is linear (i.e., $f^{\sNL}_d(\bx)=0$), we recover Theorem \ref{thm:MainIntro}, where $\cuR$ is defined as per Eq.~\eqref{eq:RTheoryDef}.
Numerical experiments suggest that Eq.~\eqref{eqn:main} holds for any deterministic nonlinear functions $f_d$ as well, and that the convergence in Eq.~\eqref{eqn:main} is uniform over $\lambda$ in compacts. We defer the study of these stronger properties to future work. 
\end{remark}

\begin{remark}
Note that the formula for a nonlinear truth, cf. Eq.~\eqref{eqn:main}, is almost identical to the one for a linear truth in Eq.~\eqref{eq:MainIntro}.
In fact, the only difference is that the  the prediction error increases by a term $F_\star^2$, and the noise level $\tau^2$ is replaced by $\tau^2+F_{\star}^2$.

Recall that the parameter $F_{\star}^2$ is the variance of the nonlinear part $\E(f^{\sNL}_{d}(\bx)^2)\to F_{\star}^2$. Hence, these changes
can be interpreted by saying that random features regression (in $N, n, d$ proportional regime) only estimates the linear component of $f_d$ and the nonlinear component behaves similar to random noise. This finding is consistent with the results of \cite{ghorbani2019linearized}
which imply, in particular, $R_\RF(f_d, \bX, \bTheta, \lambda)\ge F_\star^2+o_{d, \P}(1)$ for any $n$ and for $N = o_d(d^{2-\delta})$ for any $\delta>0$. 
\end{remark}
Figure \ref{fig:DoubleDescentNonlinear} illustrates the last remark. We report the simulated and predicted test error as a function of $\psi_1 / \psi_2 = N/n$, for three different choices of the function $f_d$ and noise level $\tau^2$. In all the settings, the total power of nonlinearity and noise is $\normf_\star^2 + \tau^2 = 0.5$, while the power of the linear component is
 $\normf_1^2 = 1$. The test errors in these three settings appear to be very close, as predicted by our theory.

\begin{figure}[!ht]
\centering
\includegraphics[width = 0.48\linewidth]{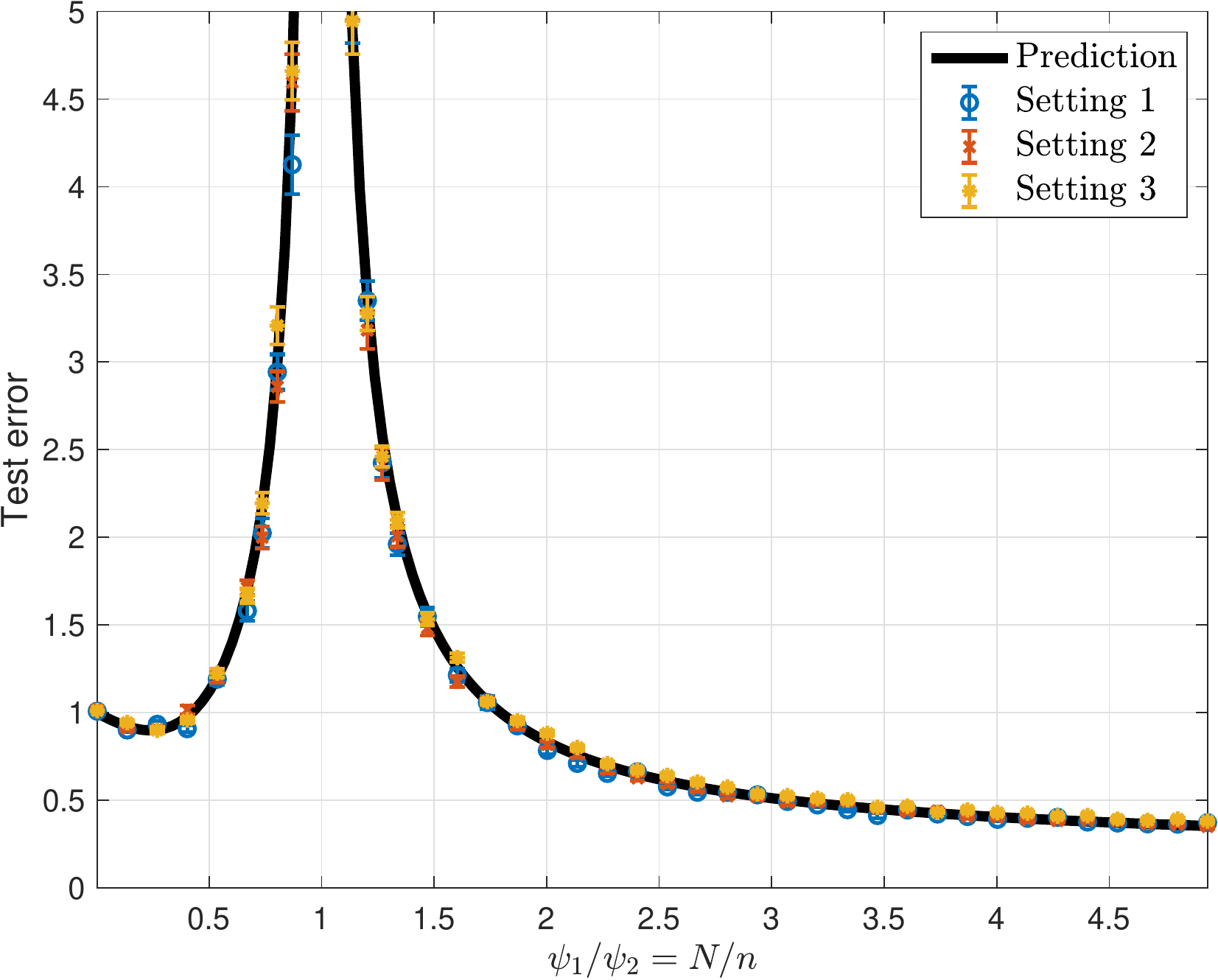}
\includegraphics[width = 0.48\linewidth]{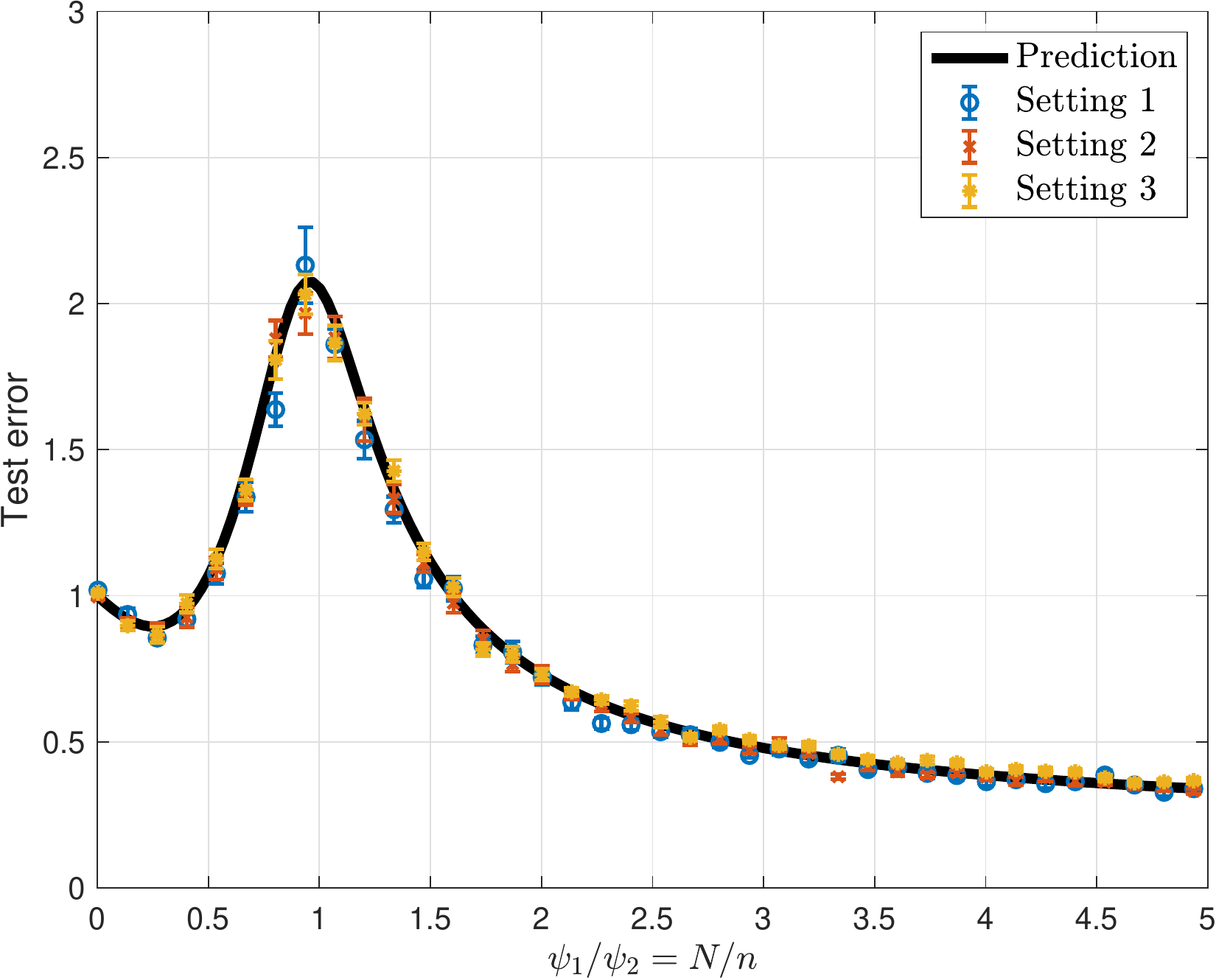}
\caption{Random features regression with ReLU activation ($\sigma = \max\{x, 0\}$). Data are generated according to one of three settings: 
$(1)$ $f_d(\bx) = x_1$ and $\E[\eps^2] = 0.5$; $(2)$ $f_d(\bx) = x_1 + (x_1^2 - 1)/2$ and $\E[\eps^2] = 0$; $(3)$ $f_d(\bx) = x_1 + x_1 x_2 / \sqrt 2$ and $\E[\eps^2] = 0$.
Within any of these settings, the total  power of nonlinearity and noise is $F_{\star}^2 + \tau^2 =0.5$, while the power of linear part is $F_1^2 = 1$.
Left frame: $\lambda = 10^{-8}$. Right frame: $\lambda = 10^{-3}$.  Here $n=300$, $d=100$. The continuous black line is our theoretical prediction, and the 
colored symbols are numerical results.  Symbols are averages over $20$ instances and the error bars report the standard error of the means over these $20$ instances.
}\label{fig:DoubleDescentNonlinear}
\end{figure}

\begin{remark}
The terms $\cuB$ and $\cuV$ in Eq.~\eqref{eqn:main} correspond to the the limits of the bias and variance of the estimated function $f(\bx;\hba(\lambda),\bTheta)$, when the ground truth function $f_d$ is linear. That is, for $f_d$ to be a linear function, we have
\begin{align}
\E_{\bx}\big\{\big[f_{d}(\bx)-\E_{\beps}f(\bx;\hba(\lambda),\bTheta)\big]^2\big\} =&~ \cuB(\ratio, \psi_1, \psi_2, \lambda/\ob_{\star}^2) \normf_{1}^2 +o_{d, \P}(1)\, ,\\
\E_{\bx}\Var_{\beps}\big(f(\bx;\hba(\lambda),\bTheta)\big) =&~ \cuV(\ratio, \psi_1, \psi_2, \lambda/\ob_{\star}^2)\, \tau^2+o_{d, \P}(1)\, .
\end{align}
\end{remark}

\subsection{Simplifying the asymptotic risk in special cases}
\label{sec:SimplifyFormula}

In order to gain further insight into the formula for the asymptotic risk $\cuR(\rho, \ratio, \psi_1, \psi_2, \olambda)$, we consider here
three special cases that are particularly interesting:
\begin{enumerate}
\item The ridgeless limit $\lambda\to 0+$.
\item The highly overparametrized regime $\psi_1\to\infty$ (recall that $\psi_1 = \lim_{d\to\infty}N/d$).
\item The large sample limit $\psi_2\to\infty$ (recall that $\psi_2 = \lim_{d\to\infty}n/d$).
\end{enumerate}
Let us emphasize that these limits are taken \emph{after} the limit $N,n,d\to\infty$ with $N/d\to\infty$ and $n/d\to\infty$.
Hence, the correct interpretation of the highly overparametrized regime is not that the width $N$ is infinite, but rather much larger than $d$ (more precisely,
larger than any constant times $d$). Analogously, the large sample limit does not coincide with infinite sample size $n$, but 
instead sample size that is much larger than $d$.

\subsubsection{Ridgeless limit}

The ridgeless limit $\lambda\to 0+$ is important because it captures the asymptotic behavior the min-norm interpolation predictor
(see also Remark \ref{rmk:Ridgeless}.)
\begin{theorem}\label{thm:ridgeless_limit}
Under the assumptions of  Theorem \ref{thm:main_theorem}, set  $\psi \equiv \min\{ \psi_1, \psi_2\}$ and define
\begin{align}
\chi \equiv - \frac{[(\psi \ratio^2 - \ratio^2 - 1)^2 + 4  \ratio^2 \psi]^{1/2} + (\psi \ratio^2 - \ratio^2 - 1)}{ 2 \ratio^2}, \label{eq:chidef}
\end{align}
and
\begin{equation}\label{eqn:definition_cuE_ridgeless}
\begin{aligned}
\cuE_{0, \rless}(\ratio, \psi_1, \psi_2)\equiv&~   - \chi^5\ratio^6 + 3\chi^4 \ratio^4+ (\psi_1\psi_2 - \psi_2 - \psi_1 + 1)\chi^3\ratio^6 - 2\chi^3\ratio^4 - 3\chi^3\ratio^2\\
&+ (\psi_1 + \psi_2 - 3\psi_1\psi_2 + 1)\chi^2\ratio^4 + 2\chi^2\ratio^2+ \chi^2+ 3\psi_1\psi_2\chi\ratio^2 - \psi_1\psi_2\, ,\\
\cuE_{1, \rless}(\ratio, \psi_1, \psi_2)\equiv&~ \psi_2 \chi^3 \ratio^4 - \psi_2 \chi^2 \ratio^2 + \psi_1 \psi_2 \chi \ratio^2 - \psi_1\psi_2, \\
\cuE_{2, \rless}(\ratio, \psi_1, \psi_2)\equiv&~ \chi^5\ratio^6 - 3\chi^4\ratio^4 + (\psi_1 - 1)\chi^3\ratio^6 + 2\chi^3\ratio^4 + 3\chi^3\ratio^2 + (- \psi_1 - 1)\chi^2\ratio^4 - 2\chi^2\ratio^2 - \chi^2, \\
\end{aligned}
\end{equation}
and
\begin{align}
\cuB_{\rless}(\ratio, \psi_1, \psi_2)\equiv&~  \cuE_{1, \rless} / \cuE_{0, \rless} ,\label{eq:BiasRless}\\
\cuV_{\rless}(\ratio, \psi_1, \psi_2)\equiv&~  \cuE_{2, \rless} / \cuE_{0, \rless}. \label{eq:VarRless}
\end{align}
Then the asymptotic prediction error of random features ridgeless regression is given by
\begin{align}
\lim_{\lambda \to 0} \lim_{d \to \infty} \E[ R_\RF(f_d, \bX, \bTheta,  \lambda)] =\normf_1^2 \cuB_{\rless}(\ratio, \psi_1, \psi_2)+ (\tau^2 + \normf_\star^2) \cuV_{\rless}(\ratio, \psi_1, \psi_2)
+ \normf_\star^2  \, .
\end{align}
\end{theorem}
The proof of this result can be found in Section \ref{sec:simplification}. 

The next proposition establishes the main qualitative properties of the ridgeless limit.
\begin{proposition}\label{prop:limiting_behavior_ridgeless}
Recall the bias and variance  functions $\cuB_{\rless}$ and $\cuV_{\rless}$ defined in Eq.~\eqref{eq:BiasRless} and \eqref{eq:VarRless}.
Then, for any $\ratio \in (0, \infty)$ and fixed $\psi_2 \in (0, \infty)$, we have 
\begin{enumerate}
\item Small width limit $\psi_1\to 0$:
\begin{align}
\lim_{\psi_1 \to 0} \cuB_{\rless} (\ratio, \psi_1, \psi_2) = 1,\;\;\;\;\;\; \lim_{\psi_1 \to 0} \cuV_{\rless}(\ratio, \psi_1, \psi_2) = 0.
\end{align}
\item Divergence at the interpolation threshold $\psi_1 = \psi_2$: 
\begin{align}
\cuB_{\rless} (\ratio, \psi_2, \psi_2) = \infty, \;\;\;\;\;\;\cuV_{\rless}(\ratio, \psi_2, \psi_2) = \infty. 
\end{align}
\item Large width limit $\psi_1 \to \infty$ (here $\chi$ is defined as per Eq.~\eqref{eq:chidef}): 
\begin{align}
\lim_{\psi_1 \to \infty} \cuB_{\rless}(\ratio, \psi_1, \psi_2) =&~ (\psi_2 \chi \ratio^2 - \psi_2) / (( \psi_2 - 1)\chi^3\ratio^6 +(1  - 3 \psi_2)\chi^2 \ratio^4  + 3 \psi_2 \chi \ratio^2 - \psi_2),\\
\lim_{\psi_1 \to \infty} \cuV_{\rless}(\ratio, \psi_1, \psi_2) =&~ (\chi^3\ratio^6 - \chi^2 \ratio^4) / ((\psi_2 - 1)\chi^3\ratio^6 +(1  - 3 \psi_2)\chi^2 \ratio^4  + 3 \psi_2 \chi \ratio^2 - \psi_2)\, .
\end{align}
\item Above the interpolation threshold (i.e. for $\psi_1 \ge \psi_2$), the function $\cuB_{\rless}(\ratio, \psi_1, \psi_2)$ and $\cuV_{\rless}(\ratio, \psi_1, \psi_2)$ 
are strictly decreasing in the rescaled number of neurons $\psi_1$. 
\end{enumerate}
\end{proposition}
The proof of this proposition is presented in Section \ref{sec:proof_limiting_behavior_ridgeless}. 

As anticipated, point $2$ establishes an important difference with respect to the random covariates linear regression model
of \cite{advani2017high,hastie2019surprises,belkin2019two}.  While in those models the peak in prediction error
is entirely due to a variance divergence, in the present setting both variance and bias diverge. 

Another important difference is established in point $4$: both bias and variance are monotonically decreasing above the
interpolation threshold. This, again, contrasts with the behavior of simpler models, in which bias increases after the interpolation threshold, or after a somewhat larger point in the number of parameters per dimension (if misspecification is added).

This monotone decrease of the bias is crucial, and is at the origin of the observation that highly overparametrized models
outperform underparametrized or moderately overparametrized ones. See Figure \ref{fig:BiasVariance} for an illustration. 

\begin{figure}[!ht]
\centering
\includegraphics[width = 0.48\linewidth]{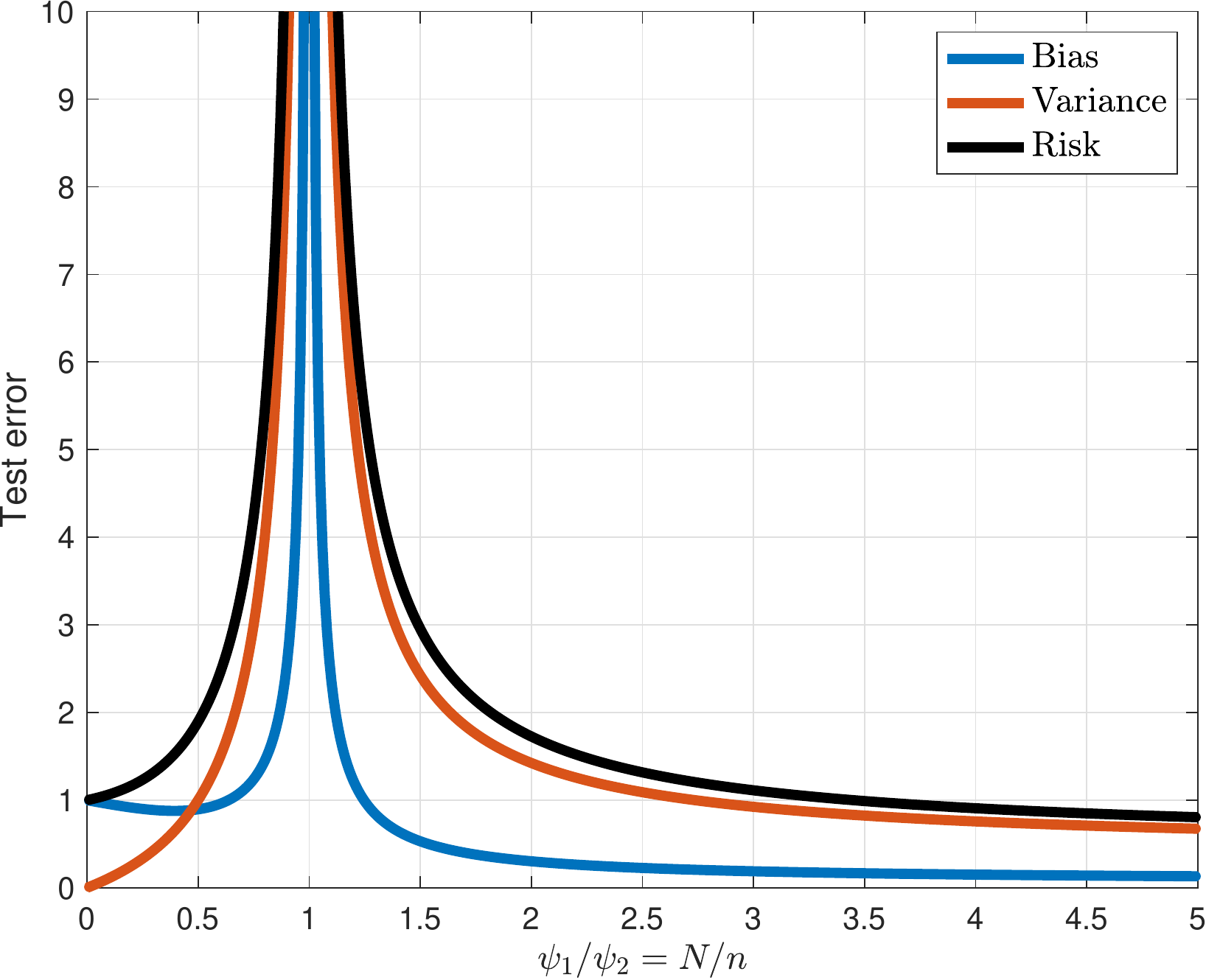}
\caption{Analytical predictions of learning a linear function $f_d(\bx) = \< \bx, \bbeta_1\>$  with ReLU activation ($\sigma = \max\{x, 0\}$) in the ridgeless limit ($\lambda \to 0$). We take $\| \bbeta_1 \|_2^2 = 1$ and $\E[\eps^2] = 1$ . We fix $\psi_2 = 2$ and plot the bias, variance, and the test error as functions of $\psi_1/\psi_2$. Both the bias and the variance term diverge when $\psi_1 = \psi_2$, and decrease in $\psi_1$ when $\psi_1 > \psi_2$. }\label{fig:BiasVariance}
\end{figure}

\subsubsection{Highly overparametrized regime}
\label{sec:HighOver}

As the number of neurons $N$ diverges (for fixed dimension $d$), random features ridge regression is known to approach kernel ridge regression with respect to
the kernel \eqref{eq:H-Kernel}. It is therefore interesting what happens when $N$ and $d$ diverge together, but $N$ is larger than any constant times $d$.
\begin{theorem}\label{thm:overparametrized_limit}
Under the assumptions of  Theorem \ref{thm:main_theorem}, define
\begin{align}
\omega \equiv - \frac{[( \psi_2\ratio^2 - \ratio^2 -\olambda\psi_2 -1)^2 + 4\psi_2\ratio^2(\olambda\psi_2 + 1)]^{1/2} +  (\psi_2\ratio^2 - \ratio^2 -\olambda\psi_2 -1)}
{2(\olambda\psi_2 + 1)}, 
\end{align}
and
\begin{align}
\cuB_{\wide}(\ratio,  \psi_2, \olambda) =&~  \frac{\psi_2\omega - \psi_2}{ (\psi_2 - 1 )\omega^3 +(1 - 3 \psi_2) \omega^2  + 3 \psi_2 \omega - \psi_2},\\
\cuV_{\wide}(\ratio,  \psi_2, \olambda)  =&~  \frac{\omega^3 - \omega^2}{ ( \psi_2 - 1 )\omega^3 +(1 - 3 \psi_2) \omega^2  + 3 \psi_2 \omega - \psi_2}\, .
\end{align}
Then the asymptotic prediction error of random features ridge regression, in the large width limit is given by
\begin{align}
\lim_{\psi_1 \to \infty}& \lim_{d \to \infty} \E[ R_\RF(f_d, \bX, \bTheta,  \lambda)] =
\normf_1^2 \cuB_{\wide}(\ratio, \psi_2,\lambda/\ob_\star^2)+ (\tau^2 + \normf_\star^2) \cuV_{\wide}(\ratio, \psi_2,\lambda/\ob_\star^2)+ \normf_\star^2  \label{eq:WideStatement}\, .
\end{align}
\end{theorem}
The proof of this result can be found in Section \ref{sec:simplification}. Note that, as expected,
the risk remains lower bounded by $\normf_\star^2$, even in the limit $\psi_1\to \infty$. Naively, one could have expected to recover kernel ridge regression in this limit, and hence a method that can fit nonlinear functions. However,  as shown in
 \cite{ghorbani2019linearized}, random features methods can only learn linear functions for  $N=O_d(d^{2-\delta})$.

As observed in Figures \ref{fig:zero_lambda_increased_psi1} to \ref{fig:fixed_psi1_increased_lambda}
(which have been obtained by applying Theorem \ref{thm:main_theorem}), the minimum prediction error is often achieved by 
highly overparametrized networks $\psi_1\to\infty$. 
It is natural to ask what is the effect of regularization on such networks. Somewhat surprisingly (and as anticipated in Section \ref{sec:Insights}),
we find that regularization does not always help. Namely, there exists a critical value $\rho_{\star}$ of the signal-to-noise ratio, such
that vanishing regularization is optimal for $\rho>\rho_\star$, and is not optimal for $\rho<\rho_{\star}$.

In order to  state formally this result, we define the following quantities
\begin{align}
\cuR_{\wide}(\rho,\ratio, \psi_2,\olambda)\equiv&~ \frac{\rho}{1 + \rho} \cuB_{\wide}(\ratio, \psi_2,\olambda)+\frac{1}{1 + \rho} \cuV_{\wide}(\ratio, \psi_2,\olambda)\, ,\\
\omega_0(\ratio, \psi_2) \equiv&~ - \frac{[( \psi_2\ratio^2 - \ratio^2 -1)^2 + 4\psi_2\ratio^2]^{1/2} +  (\psi_2\ratio^2 - \ratio^2  -1)}{2},\\
\rho_\star(\ratio, \psi_2)  \equiv&~ \frac{\omega_0^2 - \omega_0}{(1 - \psi_2)\omega_0 + \psi_2}\, .
\end{align}
Notice in particular that $\cuR_{\wide}(\rho,\ratio, \psi_2,\lambda/\ob_\star^2)$ 
is the limiting value of the prediction error (right-hand side of \eqref{eq:WideStatement}) up to an additive constant and an multiplicative constant.
\begin{proposition}\label{prop:limiting_behavior_wide}
Fix $\ratio, \psi_2 \in (0, \infty)$ and $\rho \in (0, \infty)$. Then the function 
$\olambda\mapsto \cuR_{\wide}(\rho,\ratio, \psi_2,\olambda)$ is either strictly increasing in $\olambda$, or strictly decreasing first and then strictly increasing. 

Moreover, we have
\begin{align}
\rho < \rho_\star(\ratio, \psi_2) &\;\;\Rightarrow \;\;\argmin_{\olambda \ge 0} \cuR_{\wide}(\rho,\ratio, \psi_2,\olambda) = 0\, ,\\
\rho > \rho_\star(\ratio, \psi_2) &\;\;\Rightarrow \;\;\argmin_{\olambda \ge 0} \cuR_{\wide}(\rho,\ratio, \psi_2,\olambda)= \olambda_\star(\ratio, \psi_2, \rho)>0\, .
\end{align}
\end{proposition}
The proof of this proposition is presented in Section \ref{sec:proof_limiting_behavior_wide}, which also provides further information about this phase transition (and, in particular, an explicit expression for $\olambda_\star(\ratio, \psi_2, \rho)$).

\subsubsection{Large sample limit}

As the number of sample $n$ goes to infinity, both training error (minus $\tau^2$) and test error\footnote{The difference between training error and test
error is due  to the fact that we define  the former as $\hE_n\{(y-\hf(\bx))^2\}$ and the latter as  $\E\{(f(\bx)-\hf(\bx))^2\}$.}
converge to the approximation error using random features class to fit the true function $f_d$.  It is therefore interesting what happens when $n$ and $d$ diverge together, but $n$ is larger than any constant times $d$.

\begin{theorem}\label{thm:large_sample_limit}
Under the assumptions of  Theorem \ref{thm:main_theorem}, define
\begin{align}
\omega \equiv - \frac{[( \psi_1\ratio^2 - \ratio^2 -\olambda\psi_1 -1)^2 + 4\psi_1\ratio^2(\olambda\psi_1 + 1)]^{1/2} +  (\psi_1\ratio^2 - \ratio^2 -\olambda\psi_1 -1)}{2(\olambda\psi_1 + 1)}, 
\end{align}
and 
\begin{align}
\cuB_{\lsamp}(\ratio, \psi_1, \olambda) = \frac{ (\omega^3 -  \omega^2) /\ratio^2  + \psi_1 \omega - \psi_1}{( \psi_1 - 1)\omega^3+(1 - 3\psi_1)\omega^2 + 3 \psi_1 \omega - \psi_1}. 
\end{align}
Then the asymptotic prediction error of random features ridge regression, in the large width limit is given by
\begin{align}
\lim_{\psi_2 \to \infty} \lim_{d \to \infty} \E[ R_\RF(f_d, \bX, \bTheta,  \lambda)] = \normf_1^2 \cuB_{\lsamp}(\ratio, \psi_2,\lambda/\ob_\star^2) + \normf_\star^2\, .  \label{eq:LSampleStatement}
\end{align}
\end{theorem}
The proof of this result can be found in Section \ref{sec:simplification}.

\section{Asymptotics of the training error}\label{sec:training}
\def\cuL{\mathscrsfs{L}}
\def\cuA{\mathscrsfs{A}}

Theorem \ref{thm:main_theorem} establishes the exact asymptotics of the test error in the random features model. However, the technical results obtained in the proofs allow us to characterize several other quantities of interest. Here we consider the behavior of the training error and of the norm of the parameters. We define the regularized training error by
\begin{equation}\label{eqn:regularized_training_loss}
\begin{aligned}
L_{\RF}(f_d, \bX, \bTheta, \lambda) =&~ \min_{\ba}\Big\{  \frac{1}{n} \sum_{i=1}^n \Big( y_i - \sum_{j = 1}^N a_j \sigma(\<\btheta_j, \bx_i \> / \sqrt d) \Big)^2 + \frac{N \lambda}{d}\| \ba \|_2^2\Big\} \,.
\end{aligned}
\end{equation}
We also recall that $\hba(\lambda)$ denotes the minimizer in the last expression, cf. Eq.~\eqref{eq:Ridge}
The next definition presents the asymptotic formulas for these quantities.
\begin{definition}[Asymptotic formula for training error of random features regression]\label{def:training_error}
Let the functions $\nu_1, \nu_2: \C_+ \to \C_+$ be uniquely defined by the following conditions: $(i)$ $\nu_1$, $\nu_2$ are analytic on $\C_+$; $(ii)$ For $\Im(\xi)>0$, $\nu_1(\xi)$, $\nu_2(\xi)$ satisfy the following equations
\begin{equation}\label{eqn:nu_definition_training}
\begin{aligned}
\nu_1 =&~ \psi_1\Big(-\xi -  \nu_2 - \frac{\ratio^2 \nu_2}{1- \ratio^2 \nu_1\nu_2}\Big)^{-1}\, ,\\
\nu_2 =&~ \psi_2\Big(-\xi - \nu_1 - \frac{\ratio^2 \nu_1}{1- \ratio^2 \nu_1\nu_2}\Big)^{-1}\, ;
\end{aligned}
\end{equation}
$(iii)$  $(\nu_1(\xi), \nu_2(\xi))$ is the unique solution of these equations with $\vert \nu_1(\xi)\vert\le \psi_1/\Im(\xi)$, $\vert \nu_2(\xi) \vert \le \psi_2/\Im(\xi)$ for $\Im(\xi) > C$, with $C$ a sufficiently large constant.

Let  
\begin{equation}\label{eqn:definition_chi_training}
\chi \equiv \nu_1(\imagunit ( \psi_1 \psi_2 \olambda)^{1/2}) \cdot \nu_2(\imagunit ( \psi_1 \psi_2 \olambda)^{1/2}),
\end{equation}
and
\begin{equation}
\begin{aligned}
\cuL =&~ - \imagunit \nu_2(\imagunit ( \psi_1 \psi_2 \olambda)^{1/2}) \cdot \Big( \frac{\olambda \psi_1}{\psi_2}\Big)^{1/2} \cdot \Big[ \frac{\rho}{1 + \rho} \cdot \frac{1}{1 - \chi \ratio^2}+ \frac{1}{1 + \rho}  \Big], \\
\cuA_1 =&~ \frac{\rho}{1 + \rho} \Big[ - \chi^2 (\chi \ratio^4 - \chi \ratio^2 + \psi_2 \ratio^2 + \ratio^2 - \chi \psi_2 \ratio^4 + 1)\Big] \\
&~ + \frac{1}{ 1 + \rho} \Big[ \chi^2 (\chi \ratio^2 - 1) (\chi^2 \ratio^4 - 2 \chi \ratio^2 + \ratio^2 + 1) \Big], \\
\cuA_0 =&~ - \chi^5\ratio^6 + 3\chi^4\ratio^4 + (\psi_1\psi_2 - \psi_2 - \psi_1 + 1)\chi^3\ratio^6 - 2\chi^3\ratio^4 - 3\chi^3\ratio^2\\
&~+ (\psi_1 + \psi_2 - 3\psi_1\psi_2 + 1)\chi^2\ratio^4 + 2\chi^2\ratio^2 + \chi^2 + 3\psi_1\psi_2\chi\ratio^2 - \psi_1\psi_2, \\
\cuA =&~ \cuA_1 / \cuA_0. 
\end{aligned}
\end{equation}
\end{definition}

We next state our asymptotic characterization of  $L_\RF(f_d, \bX, \bTheta, \lambda)$ and $\| \hba(\lambda) \|_2^2$.
\begin{theorem}\label{thm:training_asymptotics}
Let $\bX = (\bx_1, \ldots, \bx_n)^\sT \in \R^{n \times d}$ with $(\bx_i)_{i \in [n]} \sim_{iid} \Unif(\S^{d-1}(\sqrt d))$ and $\bTheta
= (\btheta_1, \ldots \btheta_N)^\sT \in \R^{N \times d}$ with $(\btheta_a)_{a\in [N]} \sim_{iid} \Unif(\S^{d-1}(\sqrt d))$ independently. 
Let the activation function $\sigma$ satisfy Assumption \ref{ass:activation}, and consider proportional asymptotics $N/d\to\psi_1$,  $N/d\to\psi_2$,  
as per Assumption \ref{ass:linear}. Finally, let the regression function  $\{f_d \}_{d \ge 1}$ and the response variables $(y_i)_{i \in [n]}$ satisfy Assumption \ref{ass:ground_truth}.

Then for any value of the regularization parameter $\lambda > 0$, the asymptotic regularized training error and norm square of its minimizer satisfy
\begin{equation}\label{eqn:main_training}
\begin{aligned}
\E_{\bX, \bTheta, \beps, f_d^{\sNL}} \Big\vert L_\RF(f_d, \bX, \bTheta, \lambda) - (\normf_1^2 + \normf_\star^2 + \tau^2) \cuL \Big\vert =&~ o_{d}(1), \\
\E_{\bX, \bTheta, \beps, f_d^{\sNL}} \Big\vert \ob_\star^2 \| \hba(\lambda) \|_2^2- (\normf_1^2 + \normf_\star^2 + \tau^2) \cuA \Big\vert =&~ o_{d}(1), \\
\end{aligned}
\end{equation}
where $\E_{\bX, \bTheta, \beps, f_d^{\sNL}}$ denotes expectation with respect to data covariates $\bX$, feature vectors $\bTheta$, data noise $\beps$, and $f_d^{\sNL}$ the nonlinear part of the true regression function (as a Gaussian process), as per Assumption \ref{ass:ground_truth}. The functions $\cuL$ and $\cuA$ are given in Definition \ref{def:training_error}. 
\end{theorem}

The proof of Theorem \ref{thm:training_asymptotics} is similar to the proof of Theorem \ref{thm:main_theorem}. We will give a sketch of proof of Theorem \ref{thm:training_asymptotics} in Section \ref{sec:training_proof_sketch}.

\subsection{Numerical illustrations}

\begin{figure}[!h]\label{fig:numerical_train_test}
\centering
\includegraphics[width = 0.55\linewidth]{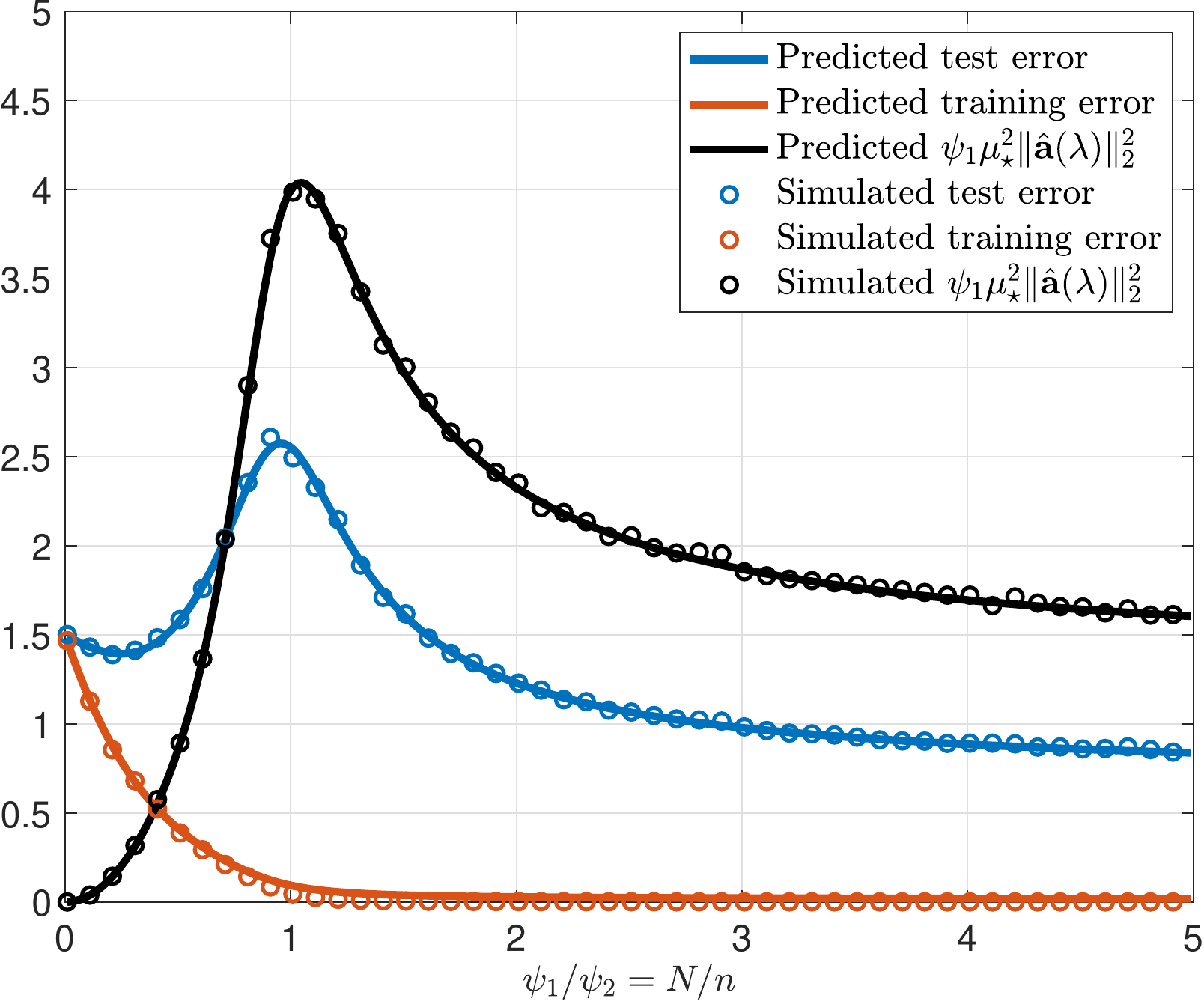}
\caption{Analytical predictions and numerical simulations for the test error and regularized training error. 
Data are generated according to $y_i = \< \bbeta_1, \bx_i \> + \eps_i$ with $\| \bbeta_1 \|_2^2 = 1$ and $\eps_i\sim \normal(0,\tau^2)$,  $\tau^2=0.5$. We fit a random features model with ReLU activations ($\sigma(x) = \max\{ x, 0\}$) and ridge regularization parameter $\lambda = 10^{-3}$. In simulations we use $d = 100$ and $n = 300$. We add $\tau^2=0.5$ to the test error to make it comparable with training error. Symbols are averages over $20$ instances.}
\end{figure}

In this section, we illustrate Theorem \ref{thm:training_asymptotics} through numerical simulations. Figure \ref{fig:numerical_train_test} 
reports the theoretical prediction and numerical results  for the regularized training error, the test error, and the norm of the coefficients $\hba(\lambda)$. We use a small non-zero value of the regularization parameter $\lambda = 10^{-3}$, fix the number of samples per dimension $\psi_2 = n/d$, and follow these quantities as a function of the overparameterization ratio $\psi_1/\psi_2=N/n$. 

As expected, the behavior of the training error strikingly different from the one of the test error. The training error is monotone decreasing in the overparameterization ratio $N/n$, and is close to zero in the overparameterized regime $N/n>1$ (it is not exactly vanishing because we use a small $\lambda>0$). In other words, the fitted model is nearly interpolating the data, and the peak in test error matches the interpolation threshold.

On the other hand, the penalty term $\psi_1 \| \hba(\lambda) \|_2^2$ is non-monotone: it increases up to the interpolation threshold, then decreases  for $N/n>1$, and converges to a constant as $\psi_1 \to \infty$. If we take this as a proxy for the model complexity,  the behavior of $\psi_1 \| \hba(\lambda) \|_2^2$  provides useful intuition about descent of the generalization error. As the number of parameters increases beyond the interpolation threshold, the model complexity decreases instead of increasing.

\begin{figure}[!h]\label{fig:train_test_optimal}
\centering
\includegraphics[width = 0.55\linewidth]{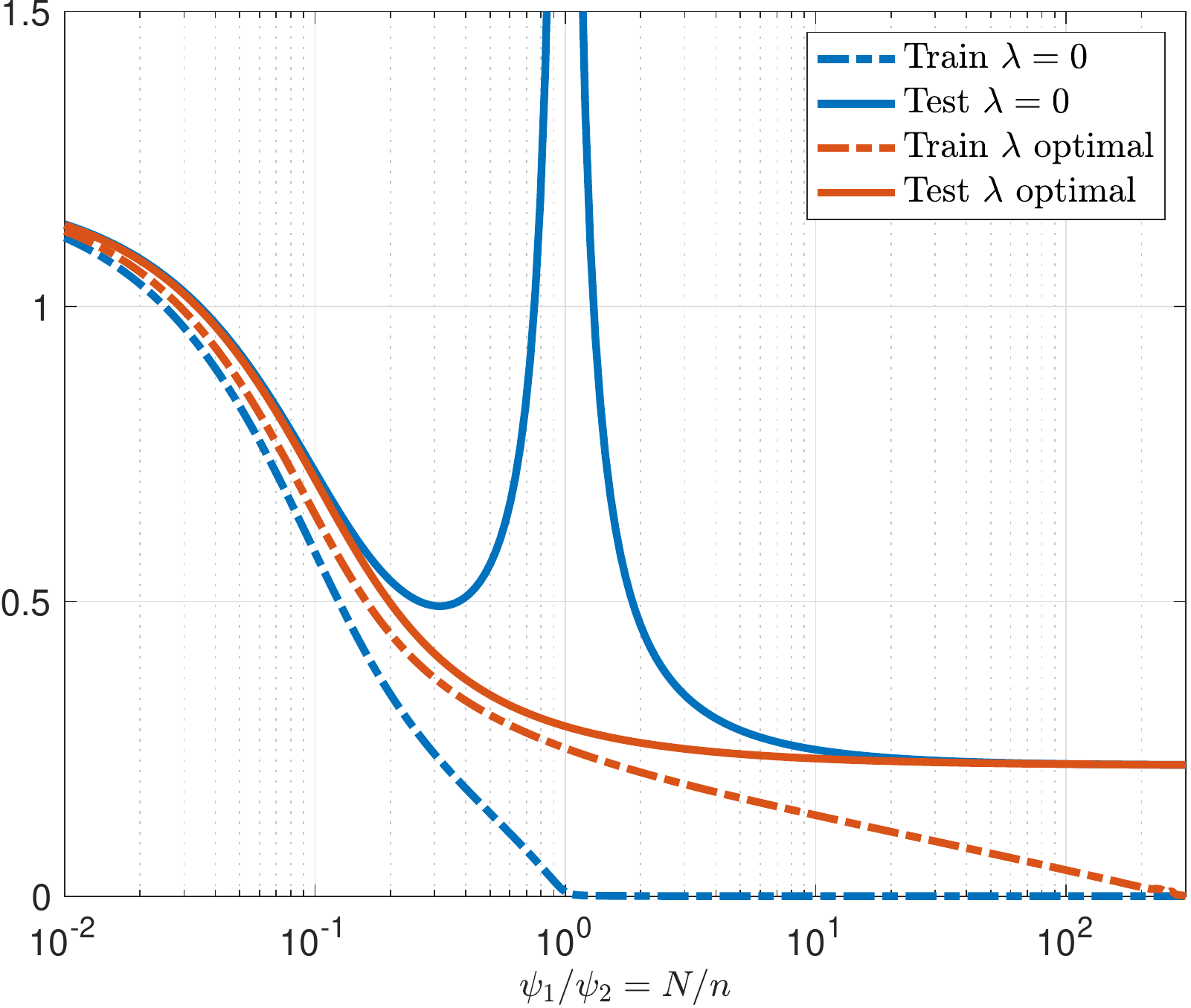}
\caption{Analytical predictions and numerical simulations results for the test error and
  the regularized training error. 
Data are generated according to $y_i = \< \bbeta_1, \bx_i \> + \eps_i$ with $\| \bbeta_1 \|_2^2 = 1$ and $\eps_i\sim \normal(0,\tau^2)$,  $\tau^2=0.2$.
We fit a random features model with ReLU activations ($\sigma(x) = \max\{ x, 0\}$). We fix $\psi_2 = n/d = 10$. We add $\tau^2 = 0.2$ to the test error 
make it comparable with training error. In the optimal ridge setting, we choose $\lambda$ for each value of $\psi_1$ as to minimize the asymptotic test error. }
\end{figure}

We can confirm the intuition that the double descent of the test error is driven by the behavior of the model complexity $\psi_1 \| \hba(\lambda) \|_2^2$, by selecting $\lambda$ in an optimal way. Following \cite{hastie2019surprises}, we expect that the optimal regularization should produce a smaller value of $\psi_1  \| \hba(\lambda) \|_2^2$, and hence eliminate or reduce the double descent phenomenon. Indeed, this is illustrated in Figure \ref{fig:train_test_optimal} demonstrates the prediction of the regularized training error and the test error for two choices of $\lambda$: $\lambda = 0$, and an optimal $\lambda$ such that the test error is minimized. When we choose an optimal $\lambda$, the test error becomes strictly decreasing as $\psi_1 = N / d$ increases. We expect this is a generic phenomenon that also holds in other interesting models.

\section{An equivalent Gaussian covariates  model}\label{sec:Gaussian_covariates}
\def\GC{{\rm GC}}

An examination of the proof of our main result (Theorem \ref{thm:main_theorem}) reveals an interesting phenomenon. The random features model has the same asymptotic prediction error as a simpler model with Gaussian covariates and response that is linear in these covariates, 
provided we use a special covariance and signal structure. 

The construction of the Gaussian covariates model proceeds as follows.
Fix $\bbeta_1 \in \reals^d$,  $\| \bbeta_1 \|_2^2 = \normf_1^2$ and $\bTheta = (\btheta_1, \ldots, \btheta_N)^\sT$ with 
$(\btheta_j)_{j \in [N]} \sim_{iid} \Unif(\S^{d-1}(\sqrt d))$. The joint distribution of 
$(y, \bx, \bu) \in \R \times \R^d \times \R^N$ conditional on $\bTheta$ is defined by the following procedure: 
\begin{enumerate}
\item  Draw $\bx \sim \normal(0, \id_d)$, $\eps\sim\normal(\bzero,\tau^2)$, and $\bw \sim \normal(\bzero, \id_N)$ independently,  conditional on $\bTheta$.
\item Let $y =  \< \bbeta_1, \bx\> + \eps$.
\item Let $\bu = (u_1, \ldots, u_N)^\sT$, $u_j = \ob_0 + \ob_1 \< \btheta_j, \bx\> / \sqrt d + \ob_\star w_j$, for some $0 < \vert \ob_0 \vert, \vert \ob_1\vert, \vert \ob_\star\vert < \infty$.
\end{enumerate}
We will denote by $\P_{y, \bx, \bu \vert \bTheta}$ the probability distribution thus defined. 
As anticipated,  this is a Gaussian covariates model.
Indeed, the covariates vector $\bu\sim\normal(\bzero,\bSigma)$ is Gaussian, with covariance 
$\bSigma = \mu_0^2\bfone\bfone^{\sT}+\mu_1^2\bTheta\bTheta^{\sT}/d+\mu_\star^2\id_N$. Also $(y,\bu)$ are jointly Gaussian and we can 
therefore write $y = \<\tilde{\bbeta}_1,\bu\>+\tilde{\eps}$, for some new vector of coefficients $\tilde{\bbeta}_1$, and noise $\tilde{\eps}$
which is independent of $\bu$.

Let $[ \{ (y_i, \bx_i, \bu_i) \}_{i \in [n]} \vert \bTheta] \sim_{iid} \P_{y, \bx, \bu \vert \bTheta}$. 
We learn a regression function $\hat f(\bx; \ba, \bTheta) = \< \bu, \ba\>$, by performing ridge regression
\begin{align}
\hat \ba(\lambda) = \argmin_{\ba\in\R^N} \left\{\frac{1}{n}\sum_{i=1}^n ( y_i- \< \bu_i, \ba\> )^2  +  
\frac{N\lambda}{d}\, \| \ba \|_2^2\right\}\,. 
\end{align}
The prediction error is defined by
\begin{align}
R_\GC(f_d, \bX, \bTheta, \lambda) = \E_{\bx, \bz\vert \bTheta}[(f_d(\bx) - \< \bu, \hat \ba(\lambda)\> )^2]\, .
\end{align} 

Remarkably, in the proportional asymptotics $N,n,d\to\infty$ with $N / d \to \psi_1, n / d \to \psi_2$,
the behavior of this model is the same as the one of the nonlinear random features model studied in the rest of the paper.
In particular, the asymptotic prediction error $\cuR$ is given by the same formula as in Definition \ref{def:formula_ridge}. 
\begin{theorem}{(Gaussian covariates prediction model)}\label{thm:Gaussian_covariates}
Define $\zeta$ and the signal-to-noise ratio $\rho \in [0, \infty]$ as
\begin{align}
\zeta \equiv \ob_1^2/\ob_\star^2\, , ~~~~ \rho\equiv F_1^2/\tau^2\, ,
\end{align}
and assume $\ob_0, \ob_1, \ob_\star \neq 0$. Then, in the Gaussian covariates model described above, for any $\lambda>0$, we have
\begin{align}
R_\GC(f_d, \bX, \bTheta, \lambda) = (F_1^2+\tau^2) \, \cuR(\rho,\zeta,\psi_1,\psi_2,\lambda/\ob_\star^2)+o_{d, \P}(1) \, ,
\end{align}
where $\cuR(\rho,\zeta,\psi_1,\psi_2,\olambda)$ is explicitly given in Definition \ref{def:formula_ridge}.
\end{theorem}
The proof of Theorem \ref{thm:Gaussian_covariates} is is almost the same as the  one of Theorem \ref{thm:main_theorem}
(with several simplifications, because of the greater amount of independence). To avoid repetitions, we will not present a proof here.

\begin{figure}[!t]
\centering
\includegraphics[width = 0.6\linewidth]{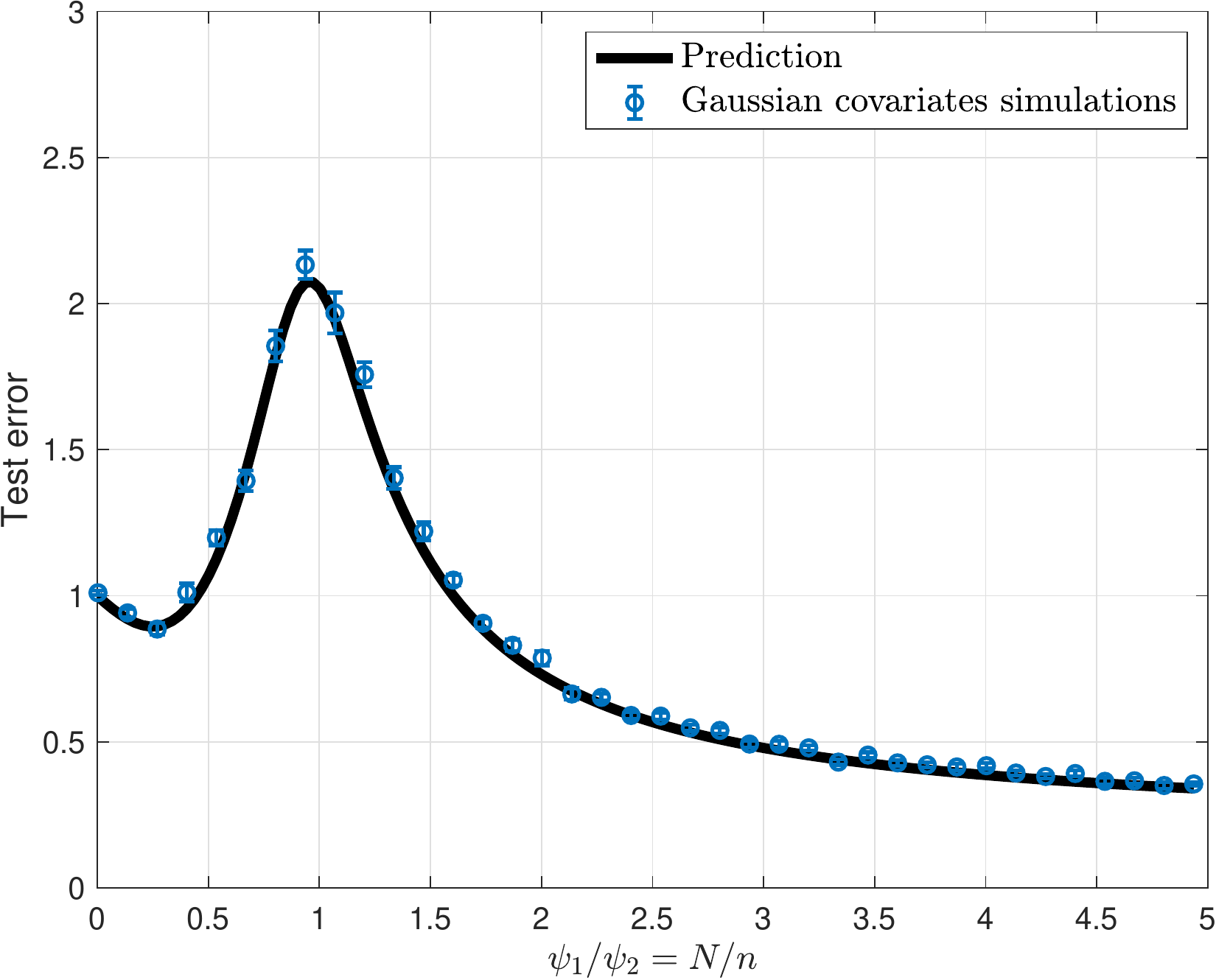}
\caption{Predictions and numerical simulations for the test error of the Gaussian covariates model. 
We fit $y_i = \< \bbeta_1, \bx_i \> + \eps_i$ with $\| \bbeta_1 \|_2^2 = 1$ and $\tau^2 = \E[\eps_i^2] = 0.5$, and parameters 
$\ob_1 = 0.5$, $\ob_\star = \sqrt{(\pi - 2)/ (4 \pi)}$, and $\lambda = 10^{-3}$. This choice of parameters $\ob_1$ and $\ob_\star$ 
matches the corresponding parameters for ReLU activations. Here $n=300$, $d=100$. The continuous black line is our theoretical prediction, and the 
colored symbols are numerical results. Symbols are averages over $20$ instances and the error bars report the standard error of the means over $20$ instances. }\label{fig:G_covariates}
\end{figure}
Figure \ref{fig:G_covariates} illustrates the content of Theorem \ref{thm:Gaussian_covariates} via numerical simulations.
 We report the simulated and predicted test error as a function of $\psi_1 / \psi_2 = N/n$. The theoretical prediction here is exactly the same 
as the one reported in Figure \ref{fig:DoubleDescentNonlinear}. However, numerical simulations were carried out with the
Gaussian covariates model instead of random features. The agreement is excellent, as predicted by Theorem \ref{thm:Gaussian_covariates}.

Why do the $\RF$ and $\GC$ models result in the same asymptotic prediction error?
It is useful to provide a heuristic explanation of this interesting phenomenon. Consider an activation function 
$\sigma: \R \to \R$, with $\ob_k = \E[\He_k(G) \sigma(G)]$ and $\ob_\star^2 = \E[\sigma^2(G)] - \ob_0^2 - \ob_1^2$ for $G \sim \normal(0, 1)$. 
Define the nonlinear component of the activation function by $\sigma^{\perp}(x) \equiv \sigma(x)- \ob_0- \ob_1x$.
Note that we have
\[
\begin{aligned}
\sigma(\< \bx_i, \btheta_j \> / \sqrt d) =&~ \ob_0 + \ob_1 \< \bx_i, \btheta_j\> / \sqrt d +  \ob_\star \tilde{w}_{ij}\, ,\;\;\;\;\;
\tilde{w}_{ij} \equiv \frac{1}{\ob_{\star}}\sigma^{\perp}(\< \bx_i, \btheta_j\> / \sqrt d), \\
u_j =&~ \ob_0 + \ob_1 \< \bx_i, \btheta_j\> / \sqrt d + \ob_\star w_{ij}, 
\end{aligned}
\]
where $(w_{ij})_{i \in [n], j \in [N]} \sim_{iid} \normal(0, 1)$ independent of $\bX$ and $\bTheta$. Note that the first two moments
of $\tilde{w}_{ij}$ match those of $w_{ij}$, i.e. $\E_{\bx|\bTheta}\tilde{w}_{ij}=0$, $\E_{\bx|\bTheta}(\tilde{w}_{ij}^2)=1$. Further, for $i\neq l$,
$\tilde{w}_{ij}$, $\tilde{w}_{il}$ are nearly uncorrelated: $\E_{\bx|\bTheta}\{\tilde{w}_{ij}\tilde{w}_{il}\} = O((\<\btheta_j,\btheta_l\>/d)^2) = O_\P(1/d)$
It is therefore not unreasonable to imagine that they should behave as independents.
The same  intuition also appears in the analysis of the spectrum of kernel random matrices in \cite{cheng2013spectrum, pennington2017nonlinear}.

\clearpage

\section{Proof of Theorem \ref{thm:main_theorem}}\label{sec:main_proof}

This section presents the  proof strategy of Theorem \ref{thm:main_theorem}, deferring a detailed
proof of technical  propositions to the next sections. Throughout the proof, we let $\bX = (\bx_1, \ldots, \bx_n)^\sT \in \R^{n \times d}$ with $(\bx_i)_{i \in [n]} \sim_{iid} \Unif(\S^{d-1}(\sqrt d))$, $\bTheta = (\btheta_1, \ldots, \btheta_n)^\sT \in \R^{N \times d}$ with $(\btheta_a)_{a\in [N]} \sim_{iid} \Unif(\S^{d-1}(\sqrt d))$ independently of $\bX$. Further, we let Assumptions \ref{ass:activation}, \ref{ass:linear}, \ref{ass:ground_truth} hold, and $\lambda > 0$ is kept fixed.

We begin by observing that the minimizer of the training error \eqref{eq:Ridge} is given by 
\[
\hat \ba(\lambda) = \frac{1}{\sqrt d}(\bZ^\sT \bZ + \lambda \psi_{1, d} \psi_{2, d} \id_N)^{-1} \bZ^\sT \by.
\]
It is useful to introduce the following resolvent matrix $\Res\in\reals^{N\times N}$
\begin{align}\label{eqn:resolvent_Z'Z}
\Res \equiv  (\bZ^\sT \bZ +   \lambda \psi_{1, d} \psi_{2, d} \id_N)^{-1}.
\end{align}
Then $\hat \ba(\lambda)$ can be written in a simpler form $\hat \ba(\lambda) =  \Res \bZ^\sT \by / \sqrt d$. After a simple calculation, we obtain
\begin{equation}\label{eqn:prediction_risk_decomposition_first}
\begin{aligned}
R_\RF(f_d, \bX, \bTheta, \lambda) &=\E_{\bx}[f_d(\bx)^2] - 2 \by^\sT \bZ\Res \bV  / \sqrt d +\by^\sT \bZ  \Res \bU \Res  \bZ^\sT \by / d.
\end{aligned}
\end{equation}
Here 
$\bsigma(\bx) =~ (\sigma(\< \btheta_1, \bx\> / \sqrt d), \ldots, \sigma(\< \btheta_N, \bx\> / \sqrt d) )^\sT \in \R^N$,
$\by =~ (y_1, \ldots, y_n)^\sT = \boldf + \beps \in \R^n$, $\boldf = (f_d(\bx_1), \ldots, f_d(\bx_n))^\sT \in \R^n$,
$\beps =~ (\eps_1, \ldots, \eps_n)^\sT \in \R^n$,
and $\bV = (V_1, \ldots, V_N)^\sT \in \R^N$, $\bU = (U_{ij})_{ij \in [N]} \in \R^{N \times N}$, are defined by
\begin{equation}\label{eqn:def_Vi_Uij}
\begin{aligned}
V_i =&~ \E_\bx[f_d(\bx) \sigma(\< \btheta_i, \bx\> / \sqrt d)], \\
U_{ij} =&~ \E_\bx [\sigma(\<\btheta_i, \bx\> /\sqrt d) \sigma(\< \btheta_j, \bx\> / \sqrt d)].  
\end{aligned}
\end{equation}
Our first step is to replace the exact expression \eqref{eqn:prediction_risk_decomposition_first}
by a simpler one involving traces of combinations of $\Res$ and the following four random matrices:
\begin{equation}\label{eqn:QHZZ1_in_decomposition}
\begin{aligned}
\bQ =&~ \frac{1}{d} \bTheta \bTheta^\sT, \, \;\;\;\;\;\;\; \;\;\;\;\;\;\; \;\;\;\;\;\;\;\;\;
\bH =~ \frac{1}{d} \bX \bX^\sT, \\
\bZ =&~ \frac{1}{\sqrt d} \sigma\Big( \frac{1}{\sqrt d} \bX \bTheta^\sT \Big), \, \;\;\;\;\;\;\;
\bZ_1 =~ \frac{\ob_1}{d} \bX \bTheta^\sT . 
\end{aligned}
\end{equation}
\begin{proposition}[Decomposition] \label{prop:decomposition}
We have
\begin{equation}\label{eqn:main_decomposition}
\E_{\bX, \bTheta, \beps, f_d^{\sNL}} \Big\vert R_\RF(f_d, \bX, \bTheta, \lambda) - \Big[ \normf_1^2 (1 - 2 \Psi_1 + \Psi_2) + (\normf_\star^2 + \tau^2) \Psi_3 + \normf_\star^2 \Big] \Big\vert = o_{d}(1),
\end{equation}
where
\begin{equation}\label{eqn:Psi_definitions_in_decomposition}
\begin{aligned}
\Psi_1 =&~ \frac{1}{d} \Tr [ \bZ_1^\sT  \bZ \Res ], \\
\Psi_2 =&~ \frac{1}{d}\Tr[\Res (\ob_1^2 \bQ + \ob_\star^2 \id_N) \Res \bZ^\sT \bH \bZ], \\
\Psi_3 =&~ \frac{1}{d}\Tr[ \Res (\ob_1^2 \bQ + \ob_\star^2 \id_N) \Res  \bZ^\sT \bZ ]. 
\end{aligned}
\end{equation}
\end{proposition}
The proof of this proposition is deferred to Section \ref{sec:decomposition} and is based on the following main steps:
\begin{itemize}
\item As a preliminary remark,
we show that by invariance of the distributions of $(\btheta_j)_{j\le N}$ and $(\bx_i)_{i\le n}$ under rotations in $\reals^d$,
we  can replace the deterministic vector $\bbeta_{d,1}$ by a uniformly random vector on the sphere with radius
$\|\bbeta_{d,1}\|_2 = F_{d,1}$.
\item Second, we compute the expectation $\E_{\bbeta,\beps}[R_\RF(f_d, \bX, \bTheta, \lambda)]$, and simplify this expression, in particular by proving that a negligible error is incurred by replacing the kernel matrix $\bU$ by $\mu^2_1\bQ+\mu_{\star}^2\id_N$. 
\item Finally, we show that $R_\RF(f_d, \bX, \bTheta, \lambda)$ concentrates around its expectation with respect to $f_d$ (i.e. the coefficients $\{\bbeta_{d,k}\}_{k\ge 1}$) and $\beps$. 
\end{itemize}

In order to  compute the traces $\Psi_j$ appearing in the last proposition, we introduce a block-structured matrix $\bA\in\reals^{M\times M}$, $M=N+n$
as follows. For $\bq = (s_1,s_2, t_1,t_2, p) \in \R^5$, we define
\begin{align}\label{eqn:matrix_A}
\bA =\bA(\bq) := \left[\begin{matrix}
s_1\id_N+s_2\bQ & \bZ^\sT  + p \bZ_1^{\sT}\\
\bZ + p \bZ_1 & t_1\id_n+t_2\bH
\end{matrix}\right]\,.
\end{align}
For  $\xi \in \complex_+$ and $\bq \in \R^5$, we define the Stieltjes transform of $\bA$
(denoted by $m_d$) and its log-determinant (denoted by $G_d$)  via
\begin{equation}\label{eqn:Stieltjes_A}
\begin{aligned}
m_{d}(\xi; \bq)  =&~ \E[M_{d}(\xi; \bq)], \;\;\;\;\; M_{d}(\xi; \bq)
=~ \frac{1}{d} \Tr[(\bA - \xi \id_M)^{-1}]\, ,\\
G_d(\xi; \bq) &= \frac{1}{d} \sum_{i = 1}^M \Log(\lambda_i(\bA(\bq)) - \xi). 
\end{aligned}
\end{equation}
Here $\Log$ is the complex logarithm with branch cut on the negative real axis and $\{ \lambda_i(\bA)\}_{ i \in [M]}$ is
the set of eigenvalues of $\bA$ in non-increasing order.

The next proposition connects the quantities $\Psi_j$ to the transforms $G_d$ and $M_d$. 
\begin{proposition}\label{prop:G_and_M}
For $\xi \in \C_+$ and $\bq \in \R^5$, we have 
\begin{equation}\label{eqn:connection_G_M}
\frac{\de \phantom{\xi}}{\de \xi} G_d(\xi; \bq) = - \frac{1}{d} \sum_{i = 1}^M (\lambda_i(\bA) - \xi)^{-1} = - M_d(\xi; \bq)\, ,
\end{equation}
%
and
\begin{equation}\label{eqn:Psi_j_and_G_d}
\begin{aligned}
\Psi_1 =&~ \frac{1}{2}\partial_p G_d(\imagunit (\psi_1 \psi_2 \lambda)^{1/2}; \bzero), \\
\Psi_2 =&~ - \ob_\star^2 ~ \partial_{s_1, t_2} G_d(\imagunit (\psi_1 \psi_2 \lambda)^{1/2}; \bzero)  - \ob_1^2 ~ \partial_{s_2, t_2} G_d(\imagunit (\psi_1 \psi_2 \lambda)^{1/2}; \bzero), \\
\Psi_3 =&~ - \ob_\star^2~ \partial_{s_1, t_1} G_d(\imagunit (\psi_1 \psi_2 \lambda)^{1/2}; \bzero)  - \ob_1^2~ \partial_{s_2, t_1} G_d(\imagunit (\psi_1 \psi_2 \lambda)^{1/2}; \bzero). 
\end{aligned}
\end{equation}
\end{proposition}
The proof of Proposition \ref{prop:G_and_M} follows by basic calculus and linear algebra, and we defer its proof to Appendix \ref{sec:proof_prop_G_and_M}. Despite its simplicity, this statement provides the basic scheme of our proof. We
will determine the asymptotics of $M_d(\xi; \bq)$ using a leave-one-out argument; then extract the behavior
of $G_d(\xi; \bq)$ using Eq.~\eqref{eqn:connection_G_M}; finally we characterize the test error using Eqs.~\eqref{eqn:Psi_j_and_G_d}
and Proposition \ref{prop:decomposition}.
\begin{remark}
  The construction of the matrix $\bA(\bq)$ is related to the linear pencil method in free probability, see e.g., \cite{helton2018applications}. 
  A significantly simpler construction was used in \cite[Section 8]{hastie2019surprises} to calculate the variance part of the
  risk $R_{\RF}$ in the limit $\lambda \to 0$ (in special cases). The approach of \cite{hastie2019surprises} amounts to computing the
  Stieltjes transform of $\bA$ for $p=t_1=t_2=0$, in the limit $\xi\to 0$: unfortunately this quantity is not sufficient to
  extract the prediction error. We overcome this difficulty by considering a more complex block-structured matrix and expressing the risk in terms of derivatives of   the log determinant $G_d(\xi; \bq)$.
\end{remark}

In order to compute the Stieltjes transform of $\bA$, we derive a set of two non-linear equations for the partial
transform $m_{1, d}(\xi; \bq)  = (N/d)\E\{ [(\bA - \xi \id_{M})^{-1}]_{N,N} \}$,
$m_{2, d}(\xi; \bq)  =(n/d)\E\{ [( \bA - \xi \id_{M})^{-1}]_{M, M} \}$ corresponding to the two blocks in the definition
of $\bA$. The starting point is the Schur complement formula with respect to entry $(N,N)$ of matrix  $\bA - \xi \id_{M}$
\begin{align}
m_{1,d}=\frac{N}{d}\,  \E\Big\{\Big(-\xi+s_1+s_2\| \btheta_N\|_2^2/d-\bA_{\cdot, N}^{\sT}(\bB-\xi \id_{M-1})^{-1}\bA_{\cdot, N} \Big)^{-1}\Big\}\, .\label{eq:SchurFirst}
\end{align}
An analogous formula for $m_{2,d}$ is obtained by taking the complement of entry $(M,M)$. Here $\bA_{\cdot, N}\in\reals^{M-1}$ is the $N$-th column of $\bA$, with the $N$-th entry removed and $\bB \in \reals^{(M-1)\times (M-1)}$ is the matrix obtained from $\bA$ by removing the $N$-th column and $N$-th row. As usual in random matrix theory, we aim at expressing the right-hand side as an
explicit deterministic function of $m_{1,d}$, $m_{2,d}$, plus a small error.
Unlike in more standard random matrix models, the matrix $\bB$ is not independent of
the vector $\bA_{\cdot, N}$: both are functions of $(\btheta_a)_{a<N}$ and $(\bx_i)_{i\le n}$. In order to overcome this difficulty,
we decompose these vectors in the components along $\btheta_N$ and the ones orthogonal to $\btheta_N$: the first one carries
most of the dependence and can be treated explicitly, while for the second we can leverage independence.

Unfortunately, even conditional on $\btheta_N$, the projections of $(\btheta_{a})_{a<N}$ and $(\bx_i)_{i\le n}$
along $\btheta_N$ and orthogonal to it
are not independent (because of the sphere constraint). To overcome this problem we replace these by Gaussian vectors
$(\overline \btheta_{a})_{a<N}$ and $(\overline\bx_i)_{i\le n}$ and prove that the two distributions
yield the same asymptotics of the Stieltjes transform. The decomposition of these Gaussian vectors takes the form
\begin{align}
\overline \btheta_a =&~ \eta_a \frac{\overline \btheta_N}{\| \overline \btheta_N\|_2}+\tbtheta_a\, , \;\;\; \<\overline \btheta_N,\tbtheta_a\> = 0, \;\;\; &a&\in [N-1]\,,\\
\overline \bx_i =&~ u_i \frac{\overline \btheta_N}{\| \overline \btheta_N\|} +\tbx_i \, , \;\;\; \<\overline \btheta_N,\tbx_i\> = 0, \;\;\; &i& \in [n]\, .
\end{align}
Note that $\{\eta_a\}_{a \in [N-1] },\{u_i\}_{i \in [n]}\sim_{iid} \normal(0,1)$ are independent of $\overline \btheta_N/\| \overline \btheta_N\|_2$, $\{\tbtheta_a\}_{a\in [N-1]}$, and $\{\tbx_i\}_{i\in [n]}$.
Further, the vector $\overline\bA_{\cdot, N}$ (the equivalent of $\bA_{\cdot,N}$ for the Gaussian model)
only depends on the $\eta_a$'s and $u_i$'s. While the matrix $\overline\bB$ (the equivalent of $\bB$)
depends on all of $\eta_a$'s, $u_i$'s, $\tbtheta_a$'s, $\tbx_i$'s, we show it can be approximated
by $\tilde\bB+\bDelta$ where $\tilde\bB$ only depends on  the $\tbtheta_a$'s, $\tbx_i$'s,
and $\bDelta$ is a low-rank matrix depending only on $\eta_a$'s, $u_i$'s. We thus get:
\begin{align}
m_{1,d} = \frac{N}{d}\E\Big\{\Big(-\xi+s_1+s_2- \overline  \bA_{\cdot, N}^{\sT}( \tbB+\bDelta-\xi\id_{M-1})^{-1} \overline \bA_{\cdot, N} \Big)^{-1} \Big\}+{\sf err}(d)\, .  
\end{align}
At this point independence can be exploited to obtain concentration results on the right-hand side.
Let us emphasize that, while these paragraphs outline the main elements of the leave-one-out
argument, several technical subtleties make the actual proof significantly longer, see
Section \ref{sec:Stieltjes} for details.

We next state the asymptotic characterization of the Stieltjes transform which is obtained by this argument.
Define  $\cQ\subseteq \reals^5$ via
\begin{equation}\label{eqn:definition_of_cQ}
\cQ = \{ (s_1, s_2, t_1, t_2, p): \vert s_2 t_2\vert \le \ob_1^2 (1 + p)^2 / 2\}\, ,
\end{equation}
and two functions
$\sFone(\,\cdot\, ,\,\cdot\,;\xi;\bq, \psi_1,\psi_2, \ob_1, \ob_\star), \sFtwo(\,\cdot\, ,\,\cdot\,;\xi;\bq, \psi_1,\psi_2, \ob_1, \ob_\star):\complex\times\complex \to \complex$ via:
\begin{equation}\label{eq:Fdef}
\begin{aligned}
\sFone(m_1,m_2;\xi;\bq, \psi_1,\psi_2, \ob_1, \ob_\star) \equiv&~  \psi_1\left(-\xi+s_1 - \ob_\star^2 m_2 +\frac{(1+t_2m_2)s_2- \ob_1^2 (1 + p)^2 m_2}{(1+s_2m_1)(1+t_2m_2)- \ob_1^2 (1 + p)^2 m_1m_2}\right)^{-1}\,, \\
\sFtwo(m_1,m_2;\xi;\bq, \psi_1,\psi_2, \ob_1, \ob_\star) \equiv&~  \psi_2\left(-\xi+t_1 - \ob_\star^2 m_1 +\frac{(1+s_2 m_1)t_2- \ob_1^2 (1 + p)^2 m_1}{(1+t_2m_2)(1+s_2m_1)- \ob_1^2 (1 + p)^2 m_1m_2}\right)^{-1}\, .
\end{aligned}
\end{equation}
\begin{proposition}[Stieltjes transform]\label{prop:Stieltjes}
Let $m_1(\,\cdot\, ;\bq)$ $m_2(\,\cdot\, ;\bq):\complex_+\to\complex_+$ be defined, for $\Im(\xi)\ge C$ a sufficiently large constant, as the unique solution of 
the equations
\begin{align}
m_{1} = \sFone(m_1,m_2;\xi;\bq, \psi_1,\psi_2, \ob_1, \ob_\star) \, , \;\;\;\;\;\; m_{2} = \sFtwo(m_1, m_2;\xi;\bq,\psi_1,\psi_2, \ob_1, \ob_\star)\, , \label{eq:FixedPoint}
\end{align}
subject to the condition $\vert m_1\vert \le \psi_1/\Im(\xi)$, $\vert m_2\vert \le \psi_2/\Im(\xi)$. Extend this definition to $\Im(\xi) >0$ by requiring $m_1,m_2$ to be analytic functions in $\complex_+$. Define $m(\xi; \bq) = m_1(\xi; \bq) + m_2(\xi; \bq)$. Then for any $\xi \in \C_+$ with $\Im \xi > 0$,  and
any compact set $\Omega \subseteq \C_+$, we have 
\begin{align}\label{eqn:weak_version_M_m_convergence}
&\lim_{d\to\infty} \E[ \vert M_d(\xi;\bq) - m(\xi;\bq) \vert] = 0\, ,\\
\label{eqn:strongest_in_compact}
&\lim_{d \to \infty} \E \Big[ \sup_{\xi \in \Omega} \vert M_d(\xi; \bq) - m(\xi; \bq) \vert\Big] = 0.
\end{align}
\end{proposition}
The proof of Proposition \ref{prop:Stieltjes} is presented in Section \ref{sec:Stieltjes}.
The fixed point equations \eqref{eq:FixedPoint} arise as a consequence of Eq.~\eqref{eq:SchurFirst}
(and the analogous equation for $m_{2,d}$). Indeed the proof also shows that the solution $(m_1,m_2)$ of these equations gives
the limit of $(m_{1,d},m_{2,d})$ as $n,N,d\to\infty$. 

Recall that,  by Proposition \ref{prop:G_and_M}, we have $M_d(\xi; \bq) = - \de G_d(\xi; \bq) / \de \xi$. We can therefore
derive an asymptotic formula for $G_d(\xi; \bq)$, by integrating the expression for $m(\xi; \bq)$ in
Proposition \ref{prop:Stieltjes} over a path in the $\xi$ plane. Namely, we integrate over a path in $\complex_+$ between
$\xi$ and $\imagunit  K$, and let $K\to\infty$. A priori, one could expect this integral not to have a closed form.
Instead, we obtain a relatively explicit expression given below.
\begin{proposition}\label{prop:expression_for_log_determinant}
Define
\begin{equation}\label{eqn:log_determinant_variation}
\begin{aligned}
\Xi(\xi, z_1, z_2; \bq) \equiv&~ \log[(s_2 z_1 + 1)(t_2 z_2 + 1) - \ob_1^2 (1 + p)^2 z_1 z_2] - \ob_\star^2 z_1 z_2 \\
 &+ s_1 z_1 +  t_1 z_2  - \psi_1 \log (z_1 / \psi_1) - \psi_2 \log (z_2 / \psi_2)  - \xi (z_1 + z_2) - \psi_1 - \psi_2.
\end{aligned}
\end{equation}
For $\xi \in \C_+$ and $\bq \in \cQ$ (c.f. Eq. (\ref{eqn:definition_of_cQ})), let $m_1(\xi; \bq), m_2(\xi; \bq)$ be defined as the analytic continuation of solution of Eq. (\ref{eq:FixedPoint}) as defined in Proposition \ref{prop:Stieltjes}. Define
\begin{equation}\label{eqn:formula_g}
g(\xi; \bq) = \Xi(\xi, m_1(\xi; \bq), m_2(\xi; \bq); \bq). 
\end{equation}
Consider proportional asymptotics $N/d\to\psi_1$,  $N/d\to\psi_2$,  
as per Assumption \ref{ass:linear}. Then for any fixed $\xi \in \C_+$ and $\bq \in \cQ$, we have
\begin{equation}\label{eqn:expression_for_log_determinant}
\begin{aligned}
\lim_{d \to \infty} \E[ \vert G_d(\xi; \bq) -  g(\xi; \bq) \vert] = 0.
\end{aligned}
\end{equation}
Moreover, for any fixed $u \in \R_+$, we have
\begin{align}
\lim_{d \to \infty}  \E[\| \partial_\bq G_d(\imagunit u; \bzero) - \partial_\bq g(\imagunit u; \bzero) \|_2 ] =&~ 0, \label{eqn:convergence_of_derivatives}\\
\lim_{d \to \infty}  \E[\| \nabla_{\bq}^2 G_d(\imagunit u; \bzero) - \nabla_{\bq}^2 g(\imagunit u; \bzero) \|_{\op} ] =&~ 0. \label{eqn:convergence_of_second_derivatives}
\end{align}
\end{proposition}
For a complete proof of this proposition we refer to Section \ref{sec:PropoIntegrateXi}.

We can now use  
Eqs. (\ref{eqn:convergence_of_derivatives}), (\ref{eqn:convergence_of_second_derivatives}), and (\ref{eqn:Psi_j_and_G_d})
in Proposition \ref{prop:decomposition}, to  get 
\begin{equation}\label{eqn:main_expression_main_proof}
\E_{\bX, \bTheta, \beps, f_d^{\sNL}} \Big\vert R_\RF(f_d, \bX, \bTheta, \lambda)  - \overline \cuR \Big\vert = o_d(1),
\end{equation}
where
\begin{align}
\overline \cuR =&~ \normf_1^2 \cuB + (\normf_\star^2 + \tau^2) \cuV + \normf_\star^2, \label{eqn:risk_main_proof} \\
\cuB =&~ 1 -  \partial_p g(\imagunit (\psi_1 \psi_2 \lambda)^{1/2}; \bzero) - \ob_\star^2\, \partial_{s_1, t_2} g(\imagunit (\psi_1 \psi_2 \lambda)^{1/2}; \bzero)  - \ob_1^2 \, \partial_{s_2, t_2} g(\imagunit (\psi_1 \psi_2 \lambda)^{1/2}; \bzero)\, , \label{eqn:bias_main_proof} \\
\cuV =&~ - \ob_\star^2\, \partial_{s_1, t_1} g(\imagunit (\psi_1 \psi_2 \lambda)^{1/2}; \bzero)  - \ob_1^2\, \partial_{s_2, t_1} g(\imagunit (\psi_1 \psi_2 \lambda)^{1/2}; \bzero)\,. \label{eqn:var_main_proof}
\end{align}

The last display provides the desired asymptotics of bias and variance. However, these expressions
involve  derivatives of $g$ that are very inconvenient to evaluate.
We conclude by proving more explicit expressions for these quantities.
The key remark here is that the expression $g(\xi; \bq)$ in Proposition \ref{prop:expression_for_log_determinant}
has a special property: the fixed point equations \eqref{eq:FixedPoint} imply  that $(m_1(\xi; \bq), m_2(\xi; \bq))$ is a
stationary point of the function $\Xi(\xi, \,\cdot\, , \, \cdot\, ; \bq)$. This simplifies the calculation of derivatives with respect to
$\bq$. In particular, the first derivative are obtained by computing the partial derivatives of $\Xi$ with respect to $\bq$
and evaluating it at $m_1,m_2$.
\begin{lemma}[Formula for derivatives of $g$]\label{lem:formula_derivatives_g}
For fixed $\xi \in \C_+$ and $\bq \in \R^5$, let $m_1(\xi; \bq), m_2(\xi; \bq)$ be defined as the analytic continuation of solution of Eq. (\ref{eq:FixedPoint}) as defined in Proposition \ref{prop:Stieltjes}. Recall the definition of $\Xi$ and $g$ given in Eq. (\ref{eqn:log_determinant_variation}) and (\ref{eqn:formula_g}).
Defining 
\begin{align}\label{eqn:m0_definition}
m_0 = m_0(\xi) \equiv m_1(\xi; \bzero) \cdot m_2(\xi; \bzero), 
\end{align}
we have 
\begin{equation}\label{eqn:derivatives_formula_g_main_proof}
\begin{aligned}
\partial_p g(\xi; \bzero) =&~ 2m_0 \ob_1^2/(m_0 \ob_1^2 - 1), \\
\partial_{s_1, t_1}^2 g(\xi; \bzero) =&~ [m_0^5\ob_1^6\ob_\star^2 - 3m_0^4\ob_1^4\ob_\star^2 + m_0^3\ob_1^4 + 3m_0^3\ob_1^2\ob_\star^2 - m_0^2\ob_1^2 - m_0^2\ob_\star^2
] / S,\\
\partial_{s_1, t_2}^2 g(\xi; \bzero) =&~ [(\psi_2 - 1)m_0^3\ob_1^4 + m_0^3\ob_1^2\ob_\star^2 + (- \psi_2 - 1)m_0^2\ob_1^2 - m_0^2\ob_\star^2]/S,\\
\partial_{s_2, t_1}^2 g(\xi; \bzero) =&~ [(\psi_1 - 1)m_0^3\ob_1^4 + m_0^3\ob_1^2\ob_\star^2 + (- \psi_1 - 1)m_0^2\ob_1^2 - m_0^2\ob_\star^2]/S,\\
\partial_{s_2, t_2}^2 g(\xi; \bzero) =&~[- m_0^6\ob_1^6\ob_\star^4 + 2m_0^5\ob_1^4\ob_\star^4 + (\psi_1\psi_2 - \psi_2 - \psi_1 + 1)m_0^4\ob_1^6 \\
&- m_0^4\ob_1^4\ob_\star^2 - m_0^4\ob_1^2\ob_\star^4 + (2 - 2\psi_1\psi_2)m_0^3\ob_1^4 \\
&+ (\psi_1 + \psi_2 + \psi_1\psi_2 + 1)m_0^2\ob_1^2 + m_0^2\ob_\star^2] / [(m_0 \ob_1^2 - 1) S],
\end{aligned}
\end{equation}
where
\begin{equation}\label{eqn:derivatives_formula_g_S_main_proof}
\begin{aligned}
S =&~ m_0^5\ob_1^6\ob_\star^4 - 3m_0^4\ob_1^4\ob_\star^4 + (\psi_1 + \psi_2 - \psi_1\psi_2 - 1)m_0^3\ob_1^6 \\
&+ 2m_0^3\ob_1^4\ob_\star^2 + 3m_0^3\ob_1^2\ob_\star^4 + (3\psi_1\psi_2 - \psi_2 - \psi_1 - 1)m_0^2\ob_1^4 \\
&- 2m_0^2\ob_1^2\ob_\star^2 - m_0^2\ob_\star^4 - 3\psi_1\psi_2m_0\ob_1^2 + \psi_1\psi_2.\\
\end{aligned}
\end{equation}
\end{lemma}
The proof of this lemma follows by simple calculus and can be found in  Appendix \ref{sec:proof_lemma_formula_derivatives_g}.

Define 
\begin{equation}\label{eqn:definition_nu_in_main_proof}
\begin{aligned}
\nu_1(\imagunit \xi) \equiv&~ m_1(\imagunit \xi \ob_\star; \bzero) \cdot \ob_\star, \\
\nu_2(\imagunit \xi) \equiv&~ m_2(\imagunit \xi \ob_\star; \bzero) \cdot \ob_\star. 
\end{aligned}
\end{equation}
By the definition of analytic functions $m_1$ and $m_2$ (satisfying Eq. (\ref{eq:FixedPoint}) and (\ref{eq:Fdef}) with $\bq = \bzero$ as defined in Proposition \ref{prop:Stieltjes}), the definition of $\nu_1$ and $\nu_2$ in Eq. (\ref{eqn:definition_nu_in_main_proof}) above is equivalent to its definition in Definition \ref{thm:MainIntro} (as per Eq. (\ref{eqn:nu_definition})). Moreover, for $\chi$ defined in Eq. (\ref{eqn:definition_chi_main_formula}) with $\olambda = \lambda / \ob_\star^2$ and $m_0$ defined in Eq. (\ref{eqn:m0_definition}), we have
\begin{equation}\label{eqn:chi_m0_relationship}
\begin{aligned}
\chi =&~ \nu_1(\imagunit (\psi_1 \psi_2 \lambda / \ob_\star^2)^{1/2}) \nu_2(\imagunit (\psi_1 \psi_2 \lambda / \ob_\star^2)^{1/2})\\
 =&~ m_1(\imagunit (\psi_1 \psi_2 \lambda)^{1/2}; \bzero) m_2(\imagunit (\psi_1 \psi_2 \lambda)^{1/2}; \bzero)\cdot \ob_\star^2 \\
 =&~ m_0(\imagunit (\psi_1 \psi_2 \lambda)^{1/2}) \cdot \ob_\star^2. 
\end{aligned}
\end{equation}
Plugging in Eq. (\ref{eqn:derivatives_formula_g_main_proof}) and (\ref{eqn:derivatives_formula_g_S_main_proof}) into Eq. (\ref{eqn:bias_main_proof}) and (\ref{eqn:var_main_proof}) and using Eq. (\ref{eqn:chi_m0_relationship}), we can see that the expressions for $\cuB$ and $\cuV$ defined in Eq. (\ref{eqn:bias_main_proof}) and (\ref{eqn:var_main_proof}) coincide with Eq. (\ref{eqn:main_bias}) and (\ref{eqn:main_var}) where $\cuE_0, \cuE_1, \cuE_2$ are provided in Eq. (\ref{eq:E012def}). Combining with Eq. (\ref{eqn:main_expression_main_proof}) and (\ref{eqn:risk_main_proof}) proves the theorem.

\section{Proof of Proposition \ref{prop:decomposition}}\label{sec:decomposition}

Throughout the proof of Proposition \ref{prop:decomposition}, we write that $\psi_1 = \psi_{1, d} = N/d$ and
$\psi_2 = \psi_{2, d} = n/d$ for notation simplicity. Throughout this section, we will denote by $B(d,k)$ the
dimension of the space of spherical harmonics of degree $k$ on $\S^{d-1}(\sqrt{d})$, and by $(Y_{kl}^{(d)})_{l\le B(d,k)}$
a basis for this space. We refer to Appendix \ref{sec:Background} for further background.

As a useful preliminary remark, we note that the Gaussian process $f_d^{\sNL}$ defined in Assumption \ref{ass:ground_truth}
can be explicitly represented as a sum of spherical harmonics with Gaussian coefficients. The
following lemma is standard (see, e.g., \cite[Proposition 6.11]{marinucci2011random}). For the reader's convenience,
we present a simple  proof in Appendix \ref{sec:additional_decomposition}. 
\begin{lemma}\label{lem:decomposition1}
  For any kernel function $\Sigma_d$ satisfying Assumption \ref{ass:ground_truth}, we can always find a sequence $(\normf_{d, k}^2 \in \R_+)_{k \ge 2}$ satisfying: $(1)$  $\sum_{k \ge 2} \normf_{d, k}^2  = \Sigma_d(1)$, $\lim_{d \to \infty} \sum_{k \ge 2} \normf_{d, k}^2 = \normf_\star^2$; $(2)$ There exist a sequence of independent random vectors
  $\bbeta_{d, k} \sim \normal(\bzero, [\normf_{d, k}^2 / B(d, k)] \id_{B(d, k)})$ such that
  \begin{align}
    f_d^{\sNL}(\bx) = \sum_{k \ge 2} \sum_{l \in [B(d, k)]}  (\bbeta_{d, k})_l Y_{kl}^{(d)}(\bx)\, .\label{eq:RepresentationSpherical}
  \end{align}
\end{lemma}

By exploiting the symmetry in the problem, the next lemma shows that, to show Eq. (\ref{eqn:main_decomposition}), instead of considering a fixed sequence of $\{ \bbeta_{d, 1} \}_{d \ge 2}$, we can consider to take $\{ \bbeta_{d, 1} \sim \Unif(\S^{d-1}(\normf_{d, 1})) \}_{d \ge 2}$. We defer the proof of this lemma to Section \ref{sec:additional_decomposition}. 

\begin{lemma}\label{lem:randomize_beta}
Let us write the random variable in the left hand side of Eq. (\ref{eqn:main_decomposition}) as a function of $\bbeta_{d, 1}$ and de-emphasize its dependence on other variables, i.e., 
\[
\cE(\bbeta_{d, 1}) \equiv \Big\vert R_\RF(f_d, \bX, \bTheta, \lambda) - \Big[ \normf_1^2 (1 - 2 \Psi_1 + \Psi_2) + (\normf_\star^2 + \tau^2) \Psi_3 + \normf_\star^2 \Big] \Big\vert. 
\]
Let $\bX$, $\bTheta$, $\beps$, and $f_d^{\sNL}$ be distributed as in the statement of Proposition \ref{prop:decomposition}. Then, for any fixed $\bbeta_{d, 1} \in \S^{d-1}(\normf_{d, 1})$, we have 
\[
\E_{\bX, \bTheta, \beps, f_d^{\sNL}} [ \cE(\bbeta_{d, 1})]  = \E_{\tilde \bbeta_{d, 1} \sim \Unif(\S^{d-1}(\normf_{d, 1}))}\E_{\bX, \bTheta, \beps, f_d^{\sNL}} [\cE(\tilde \bbeta_{d, 1})]. 
\]
\end{lemma}

By Lemma \ref{lem:decomposition1}, we can represent the Gaussian process $f_d^{\sNL}$ as
per Eq.~\eqref{eq:RepresentationSpherical}. 
By Lemma \ref{lem:randomize_beta}, we can replace the expectation over $f_d^{\sNL}$ by expectation over 
$\bbeta_{d, 1} \sim \Unif(\S^{d-1}(\normf_{d, 1}))$ and the Gaussian vectors $\{ \bbeta_{d, k} \sim \cN(\bzero, [\normf_{d, k}^2 / B(d, k)] \id_{B(d, k)}) \}_{k \ge 2}$.

In the remaining of this section, we write $\E_\bbeta$ as a shorthand for this expectation.
To simplify our expressions, we sometimes write $\bbeta_k \equiv \bbeta_{d, k}$. It is furthermore useful to introduce two resolvent matrices $\Res \in \reals^{N\times N}$ and $\oRes \in \reals^{n \times n}$ ($\Res$ is the same as defined in Eq. (\ref{eqn:resolvent_Z'Z}) except that we are keeping $\psi_{1, d}$ and $\psi_{2, d}$ fixed here)
\begin{equation}
\begin{aligned}\label{eqn:resolvent_Z'Z_in_proof}
\Res \equiv&~  (\bZ^\sT \bZ +  \psi_1 \psi_2 \lambda \id_N)^{-1}, \\
\oRes\equiv&~  (\bZ \bZ^\sT +  \psi_1 \psi_2 \lambda \id_n)^{-1}.
\end{aligned}
\end{equation}

Next, we state three lemmas that are used in the proof of Proposition \ref{prop:decomposition}. 
\begin{lemma}[Decomposition]\label{lem:reformulation}
Let $\lambda_{d, k}(\sigma)$ be the Gegenbauer coefficients of the function $\sigma$, i.e., we have 
\begin{equation}\label{eqn:expansion_sigma_lem_reformulation}
\sigma(x) = \sum_{k = 0}^\infty \lambda_{d, k}(\sigma) B(d, k) Q_k(\sqrt d \cdot x). 
\end{equation}
Under the assumptions of Proposition \ref{prop:decomposition}, for any $\lambda > 0$, we have
\begin{align}
\E_{\bbeta, \beps}[R_\RF(f_d, \bX, \bTheta, \lambda)] =  \sum_{k=0}^\infty \normf_{d, k}^2(1 - 2 S_{1k} +S_{2k})  + \tau^2 S_3,
\end{align}
where
\begin{equation}\label{eqn:S_definitions_in_decomposition}
\begin{aligned}
S_{1k} =&~ \frac{1}{\sqrt d}\lambda_{d, k}(\sigma) \Tr [ Q_k(\bTheta \bX^\sT)  \bZ  \Res ], \\
S_{2k} =&~ \frac{1}{d} \Tr[  \Res \bU  \Res \bZ^\sT Q_k(\bX \bX^\sT) \bZ ], \\
S_3 =&~ \frac{1}{d} \Tr[  \Res \bU  \Res  \bZ^\sT \bZ],
\end{aligned}
\end{equation}
where $\bU = (U_{ij})_{i, j \in [N]} \in \R^{N \times N}$ is a matrix whose elements is as defined in Eq. (\ref{eqn:def_Vi_Uij}), 
$\bZ$ is given by Eq. (\ref{eqn:QHZZ1_in_decomposition}) and $\Res$ is given by Eq. (\ref{eqn:resolvent_Z'Z_in_proof}). 
\end{lemma}

\begin{lemma}\label{lem:expression_bound_S}
Under the same definitions and assumptions of Proposition \ref{prop:decomposition} and Lemma \ref{lem:reformulation}, for any $\lambda > 0$, we have ($\E$ is the expectation taken with respect to the randomness in $\bX$ and $\bTheta$)
\begin{align}
\E \vert 1 - 2 S_{10} + S_{20}\vert =&~ o_d(1), \label{eq:S012diff}\\
\E\Big[ \sup_{k \ge 2} \vert S_{1k}\vert \Big]  =&~ o_{d}(1), \label{eq:S1k}\\
\E\Big[ \sup_{k \ge 2} \vert S_{2k} - S_3 \vert \Big]=&~ o_{d}(1), \label{eq:S2k} \\
\E \vert S_{11} - \Psi_1\vert =&~ o_d(1), \label{eq:S11}\\
\E \vert S_{21} - \Psi_2\vert  =&~ o_d(1), \label{eq:S21}\\\
\E \vert S_{3} - \Psi_3\vert  =&~ o_d(1), \label{eq:S31}
\end{align}
where $S_{1k}, S_{2k}, S_3$ are given by Eq. (\ref{eqn:S_definitions_in_decomposition}), and $\Psi_1, \Psi_2, \Psi_3$ are given by Eq. (\ref{eqn:Psi_definitions_in_decomposition}). 
\end{lemma}

\begin{lemma}\label{lem:concentration_beta}
Under the assumptions of Proposition \ref{prop:decomposition}, we have
\begin{equation}\label{eqn:concentration_beta}
\E_{\bX, \bTheta} \Big[ \Var_{\bbeta, \beps} \Big( R_\RF(f_d, \bX, \bTheta, \lambda) \Big\vert \bX, \bTheta \Big)^{1/2} \Big] = o_d(1). 
\end{equation}
\end{lemma}
We defer the proofs of these three lemmas to the following subsections and show here that they imply Proposition \ref{prop:decomposition}.  Indeed, we have
\[
\begin{aligned}
\E_{\bX, \bTheta}\Big\vert \E_{\beps, \bbeta}&[R_\RF(f_d, \bX, \bTheta, \lambda)] - \Big[ \normf_{d, 1}^2 (1 - 2 \Psi_1 + \Psi_2) + \Big(\tau^2 + \sum_{k=2}^\infty \normf_{d, k}^2 \Big)\Psi_3 + \sum_{k=2}^\infty \normf_{d, k}^2  \Big] \Big\vert \\
\stackrel{(a)}{\le}&~ \normf_{d, 0}^2 \cdot \E \vert 1 - 2S_{10} + S_{20} \vert + \normf_{d, 1}^2  \cdot \Big[ \E\vert S_{11} - \Psi_1 \vert + \E\vert S_{21} - \Psi_2 \vert \Big] \\
& + \Big(\sum_{k = 2}^\infty \normf_{d, k}^2\Big) \cdot  \sup_{k\ge2}\Big[2 \E\vert S_{1k} \vert + \E\vert S_{2k} - \Psi_3 \vert \Big] + \tau^2 \E\vert S_3 - \Psi_3\vert\\
\stackrel{(b)}{=}&~ o_d(1). 
\end{aligned}
\]
where $(a)$ follows by Lemma \ref{lem:reformulation} and triangular inequality, and $(b)$ from Lemma
\ref{lem:expression_bound_S}.

Combining with Lemma \ref{lem:concentration_beta} (and $\E[\Psi_1], \E[\Psi_2], \E[\Psi_3] = O_d(1)$) and Lemma \ref{lem:randomize_beta} concludes the proof of Proposition \ref{prop:decomposition}. 
In the remaining
of this section, we will prove Lemma \ref{lem:reformulation}, \ref{lem:expression_bound_S}, and \ref{lem:concentration_beta}.

\subsection{Proof of Lemma \ref{lem:reformulation}}\label{subsec:decomposition_lemma}

Recall the expression \eqref{eqn:prediction_risk_decomposition_first} for the risk.
Taking expectation with respect to $\bbeta$ and $\beps$, we get
\begin{align*}
\E_{\bbeta, \beps}[R_\RF(f_d, \bX, \bTheta, \lambda)] =&~   \sum_{k \ge 0} \normf_{d, k}^2 - 2 T_1 + T_2 + T_3,
\end{align*}
where
\begin{align*}
T_1 = \frac{1}{\sqrt d}\E_{\bbeta}[\boldf^\sT \bZ  \Res \bV], ~~~~~~T_2 = \frac{1}{d} \E_{\bbeta}[\boldf^\sT \bZ   \Res \bU  \Res  \bZ^\sT \boldf],~~~~~ T_3 = \frac{1}{d} \E_{\beps}[\beps^\sT \bZ   \Res \bU  \Res  \bZ^\sT \beps ]. 
\end{align*}
The proof of the lemma follows by evaluating each of these three
terms. It is useful to introduce the matrices  $\bY_{k, \bx}$, $\bY_{k, \btheta}$, which denotes the evaluations of spherical harmonics of degree $k$ at the points $\{\bx_i\}_{i\le n}$ and $\{\btheta_a\}_{a\le N}$
(c.f.  Appendix \ref{sec:Background}): 
\begin{equation}\label{eqn:def_Y_k_x_theta}
\begin{aligned}
\bY_{k, \bx} =&~ (Y_{kl}(\bx_i))_{i \in [n], l \in [B(d, k)]} \in \R^{n \times B(d,k)}, \\
\bY_{k, \btheta} =&~ (Y_{kl}(\btheta_a))_{a \in [N], l \in [B(d, k)]} \in \R^{N \times B(d, k)}.\\
\end{aligned}
\end{equation}
With these notations we have
\begin{equation}\label{eqn:expansion_boldf_V}
\boldf = \sum_{k = 0}^\infty \bY_{k, \bx} \bbeta_k \in \R^n,~~~ \bV = \sum_{k = 0}^\infty \lambda_{d, k}(\sigma) \bY_{k, \btheta} \bbeta_k \in \R^N.
\end{equation}
Since $\bbeta_{k} \sim \normal(\bzero, \normf_{d, k}^2 \id_{B(d, k)} / B(d, k))$ for $k \ge 2$  and
$\bbeta_1 \sim \Unif(\S^{d-1}(\normf_{d, 1}))$ independently, we have
\begin{align}
  \E_{\bbeta}[\bV\boldf^{\sT}] &= \sum_{k = 0}^\infty \normf_{d, k}^2 \lambda_{d, k}(\sigma) Q_k(\bTheta \bX^\sT) \, ,\\
  \E_{\bbeta}[\boldf\boldf^{\sT}] &= \sum_{k = 0}^\infty \normf_{d, k}^2 Q_k(\bX \bX^\sT) \, .
\end{align}
Using these expressions, we can evaluate terms $T_1$ and $T_2$:
\begin{align*}
T_1 =&~ \frac{1}{\sqrt d} \sum_{k = 0}^\infty \normf_{d, k}^2 \lambda_{d, k}(\sigma)  \cdot \Tr\Big[ Q_k(\bTheta \bX^\sT) \bZ  \Res \Big] \, ,\\
T_2 =&~\frac{1}{d} \sum_{k = 0}^\infty \normf_{d, k}^2  \cdot \Tr\Big[  \Res \bU  \Res \bZ^\sT Q_k(\bX \bX^\sT) \bZ \Big]\, .
\end{align*}
We proceed analogously for term $T_3$. By the assumption
$\eps_i \sim_{iid} \P_\eps$ with $\E_\eps(\eps) = 0$ and $\E_\eps(\eps_1^2) = \tau^2$, we have
\begin{align*}
T_3 =&~ \frac{1}{d}\E_{\beps}[ \Tr(\beps \beps^\sT  \Res \bU  \Res  \bZ^\sT \bZ)] =\frac{\tau^2}{d} \cdot \Tr[ \Res \bU  \Res  \bZ^\sT \bZ ]. 
\end{align*}
Combining the above formulas for $T_1$, $T_2$, $T_3$ proves Lemma \ref{lem:reformulation}.

\subsection{Proof of Lemma \ref{lem:expression_bound_S}}\label{subsec:bounds_of_S_terms}

The next two lemmas will be used in the proofs of Lemma \ref{lem:expression_bound_S} and Lemma \ref{lem:concentration_beta},
and hold under the same assumptions. The first of these lemmas will be used to establish Eq.~\eqref{eq:S012diff} (but notice
that its statement does not coincide with that equation), and the second
will be used to control several terms in those proofs. 
The proofs of these lemmas are given in Section \ref{subsec:bounds_key_constant_term}.
\begin{lemma}\label{lem:constant_term_one}
Define
\begin{align}
A_1 \equiv&~  \frac{\lambda_{d, 0}(\sigma)}{\sqrt d} \Tr [ \ones_N \ones_n^\sT  \bZ  \Res ], \label{eq:A1def}\\
A_2 \equiv&~  \frac{\lambda_{d, 0}(\sigma)^2}{d} \Tr [  \Res \ones_N \ones_N^\sT  \Res \bZ^\sT \ones_n \ones_n^\sT \bZ ]. \label{eq:A2def}
\end{align}
Then for any $\lambda > 0$, we have
\[
\begin{aligned}
\E\vert 1 - 2 A_1 + A_2 \vert = o_d(1). 
\end{aligned}
\]
\end{lemma}

\begin{lemma}\label{lem:constant_term_two}
  Let $(\obM_\alpha)_{\alpha \in \cA} \in \R^{n \times n}$ be a collection of symmetric random matrices with $\E[\sup_{\alpha \in \cA} \| \obM_\alpha \|_{\op}^2]^{1/2} = O_{d}(1)$. Define
\begin{align}\label{eqn:Bl_constant_term_two}
B_\alpha \equiv&~  \frac{\lambda_{d, 0}(\sigma)^2}{d} \Tr [  \Res \ones_N \ones_N^\sT  \Res \bZ^\sT \obM_\alpha \bZ ]. 
\end{align}
Then for any $\lambda > 0$, we have
\[
\begin{aligned}
\E\Big[ \sup_{\alpha \in \cA} \vert B_\alpha \vert \Big] =&~ o_d(1). \\
\end{aligned}
\]
\end{lemma}

We will now use these lemmas to prove Lemma \ref{lem:expression_bound_S}. We begin by recalling
a few  facts that are used several times in the proof. Since $\lambda > 0$, there exists a constant $C < \infty$
depending on $(\lambda, \psi_1, \psi_2)$ such that deterministically 
\begin{equation}\label{eqn:Z(ZZ)-1_bound}
\| \bZ  \Res \|_{\op} = \| \bZ (\bZ^\sT \bZ +  \psi_1 \psi_2 \lambda \id_N)^{-1} \|_{\op} \le C, ~~~~~  \| \Res \|_{\op} = \| (\bZ^\sT \bZ +  \psi_1 \psi_2 \lambda \id_N)^{-1} \|_{\op} \le C.
\end{equation}
By operator norm bounds on Wishart matrices \cite{Guionnet}, we have (the definition of these matrices are given in Eq. (\ref{eqn:QHZZ1_in_decomposition}))
\begin{equation}\label{eqn:HQZ1_bound}
\E[\| \bH \|_{\op}^2], \E[\| \bQ \|_{\op}^2], \E[\| \bZ_1 \|_{\op}^2] = O_{d}(1). 
\end{equation}
Finally we need some simple operator norm bounds on the matrices $Q_k(\bX \bX^\sT) - \id_n$, $Q_k(\bTheta \bTheta^\sT) - \id_N$, and $Q_k(\bTheta \bX^\sT)$. Notice that $Q_k(\bX \bX^\sT)_{ii}=1$ (by the normalization condition of Gegenbauer polynomials) and the out-of-diagonal entries
of $Q_k(\bX \bX^\sT)$ have zero mean and typical size of order $1/d^{k/2}$ (see Appendix \ref{sec:Background}). This suggests the following estimates, which
are formalized in Lemma \ref{lem:gegenbauer_identity}, 
\begin{align}\label{eqn:QXX_bound}
\E\Big[\sup_{k \ge 2} \| Q_k(\bX \bX^\sT) - \id_n \|_{\op}^2 \Big] \vee \E\Big[\sup_{k \ge 2} \| Q_k(\bTheta \bTheta^\sT) - \id_N \|_{\op}^2 \Big] \vee \E\Big[ \sup_{k \ge 2} \| Q_k(\bTheta \bX^\sT)  \|_{\op}^2 \Big] = o_d(1). 
\end{align}
As a consequence of these estimates, we obtain a useful approximation result 
for the matrix $\bU\in\reals^N$ as defined in Eq.~\eqref{eqn:def_Vi_Uij}. In words, $\bU$ is
well approximated by a term that is linear in the weights covariance matrix $\bTheta\bTheta^{\sT}$, plus a
term that is proportional to the identity.
To see this, by the decomposition of $\sigma$ into Gegenbauer polynomials as in Eq. (\ref{eqn:expansion_sigma_lem_reformulation}) and the properties of Gegenbauer polynomials as in Appendix \ref{sec:Background}, we have 
\[
\bU = \sum_{k, l = 0}^{\infty} \lambda_{d, k}(\sigma) \lambda_{d, l}(\sigma) B(d, k) B(d, l) \E_{\bX}[Q_k(\bTheta \bX^\sT) Q_l(\bX \bTheta^\sT)] =  \sum_{k=0}^{\infty} \lambda_{d,k}(\sigma)^2 B(d,k)\, Q_k(\bTheta\bTheta^{\sT}).
\]
Since further $\lambda_{d,k}(\sigma)^2B(d,k)k! \to \ob_k(\sigma)^2$ as $d\to\infty$ (see Eq. (\ref{eqn:relationship_mu_lambda})), we have
\begin{align}\label{lem:U_decomposition_proof_in_lem:constant_term_two}
\bU =  \lambda_{d, 0}^2 \ones_N \ones_N^\sT + \ob_1^2 \bQ  + \ob_\star^2 (\id_N + \bDelta),\;\;\;\;
  \E[\| \bDelta \|_{\op}^2] = o_d(1)\,.
  \end{align}
(This estimate is stated formally in the appendices as  Lemma \ref{lem:decomposition_of_kernel_matrix}.)
It is also useful to introduce  the matrix 
\[
\bM \equiv \ob_1^2 \bQ  + \ob_\star^2 (\id_N + \bDelta),
\]
for which the above implies $\bU =  \lambda_{d, 0}^2 \ones_N \ones_N^\sT + \bM$ and $\E[\| \bM \|_{\op}^2] = O_{d}(1)$.

We have now finished presenting our preliminary estimates and can prove
Lemma \ref{lem:expression_bound_S}.

We begin by considering Eq.~\eqref{eq:S012diff}, where $S_{10}$ and $S_{20}$ are defined in Eq. (\ref{eqn:S_definitions_in_decomposition}).
By the approximate linearization of $\bU$ in Eq.~\eqref{lem:U_decomposition_proof_in_lem:constant_term_two}, we have
\begin{align*}
S_{10} =&~ \frac{\lambda_{d, 0}(\sigma)}{\sqrt d} \Tr(  \ones_N \ones_n^\sT \bZ  \Res ),  ~~~~~ S_{20} = \frac{\lambda_{d, 0}(\sigma)^2}{d} \Tr(  \Res \ones_N \ones_N^\sT  \Res \bZ^\sT \ones_n \ones_n^\sT \bZ ) + \frac{1}{d}\Tr(  \bZ \Res  \bM   \Res \bZ^\sT \ones_n \ones_n^\sT ). 
\end{align*}
Further recall the definition of $A_1$, and $A_2$ in Eqs.~\eqref{eq:A1def}, \eqref{eq:A2def} and the definition of $\Res$ and $\oRes$ as in Eq. (\ref{eqn:resolvent_Z'Z_in_proof}), and we define 
\begin{align*}
B \equiv&~ \frac{1}{d}\Tr(  \bZ \Res \bM \Res \bZ^\sT \ones_n \ones_n^\sT ) = \frac{1}{d}\Tr(  \bZ  \bM \bZ^\sT \oRes \ones_n \ones_n^\sT  \oRes). 
\end{align*}
Then we have $S_{10} = A_1$, $S_{20} = A_2 + B$ and by Lemma \ref{lem:constant_term_one}, we have
\begin{align}
  \E[\vert 1 - 2 S_{10} + S_{20}\vert] =&~ \E[\vert 1 - 2 A_{1} + A_{2} + B\vert] 
 \le \E[\vert 1 - 2 A_{1} + A_{2}\vert] + \E[\vert B\vert]
  \le \E[\vert B\vert] +o_d(1)\, . \label{eqn:constant_term_two_lem_decomposition_1}
\end{align}
By Lemma \ref{lem:constant_term_two} and the fact that $\E[\| \bM \|_{\op}^2] = O_d(1)$ as in Eq (\ref{lem:U_decomposition_proof_in_lem:constant_term_two}) (when applying Lemma \ref{lem:constant_term_two}, we change the role of $N$ and $n$, the role of $\Res$ and $\oRes$, and the role of $\bTheta$ and $\bX$; this can be done because the role of $\bTheta$ and $\bX$ is symmetric), we have 
\[
\E[\vert B\vert] = o_{d}(1). 
\]
Plugging these bounds into Eq. (\ref{eqn:constant_term_two_lem_decomposition_1}), we get
$\E[\vert 1 - 2 S_{10} + S_{20}\vert]  = o_d(1)$ as claimed. 

We next consider Eq.~\eqref{eq:S1k}, which requires to control  $S_{1k}$, defined in
Eq.~\eqref{eqn:S_definitions_in_decomposition}). By Eq.~\eqref{eqn:Z(ZZ)-1_bound}, we have
\begin{equation}\label{eqn:constant_term_two_lem_decomposition_4}
\begin{aligned}
\sup_{k \ge 2} \vert S_{1k} \vert \le&~ \sup_{k \ge 2} \Big[ \vert \sqrt d \lambda_{d, k}(\sigma)\vert \cdot \| Q_k(\bTheta \bX^\sT)  \bZ \Res \|_{\op} \Big] \\
\le&~ \sup_{k \ge 2} \Big[ C \cdot \vert \sqrt d \lambda_{d, k}(\sigma)\vert \cdot \| Q_k(\bTheta \bX^\sT)  \|_{\op} \Big]. 
\end{aligned}
\end{equation}
Further note $\| \sigma \|_{L^2(\tau_d)}^2 = \sum_{k \ge 0} \lambda_{d, k}(\sigma)^2 B(d, k) = O_d(1)$, $B(d, k) = \Theta(d^k)$, and for fixed $d$, $B(d, k)$ is non-decreasing in $k$ \cite[Lemma 1]{ghorbani2019linearized}. Therefore 
\[
\sup_{k \ge 2} \vert \lambda_{d, k}(\sigma) \vert \le \sup_{k \ge 2} \Big[ \| \sigma \|_{L^2(\tau_d)} / \sqrt{B(d, k)} \Big] = O_d(1/d). 
\]
Combining with Eq. (\ref{eqn:QXX_bound}) and Eq. (\ref{eqn:constant_term_two_lem_decomposition_4}), we get 
$\E \Big[ \sup_{k \ge 2} \vert S_{1k}\vert \Big]  = o_d(1)$.

We next consider Eq.~\eqref{eq:S2k}, whereby $S_{2k}$ and $S_3$ are defined as per Eq.~\eqref{eqn:S_definitions_in_decomposition}.
Recall that, by Eq.~\eqref{lem:U_decomposition_proof_in_lem:constant_term_two}, we have
$\bU =  \lambda_{d, 0}^2 \ones_N \ones_N^\sT +\bM$ where $\E[\|\bM\|^2_{\op}]=O_d(1)$.
We have therefore
\begin{equation}\label{eqn:constant_term_two_lem_decomposition_2}
\begin{aligned}
\sup_{k \ge 2} \vert S_{2k} - S_3 \vert \le I_1 + I_2,
\end{aligned}
\end{equation}
where
\[
\begin{aligned}
I_1 =&~ \sup_{k \ge 2} \Big\vert  \frac{\lambda_{d, 0}^2}{d} \Tr\Big[  \Res \ones_N \ones_N^\sT  \Res \bZ^\sT (Q_k(\bX \bX^\sT) - \id_n) \bZ \Big] \Big\vert,\\
I_2 =&~ \sup_{k \ge 2} \Big\vert  \frac{1}{d} \Tr\Big[ \Res \bM  \Res \bZ^\sT (Q_k(\bX \bX^\sT) - \id_n) \bZ \Big] \Big\vert.
\end{aligned}
\]
By Lemma \ref{lem:constant_term_two} (with $\obM_k= Q_k(\bX \bX^\sT) - \id_n$) and Eq.~(\ref{eqn:QXX_bound}), we get 
\[
\E[ I_1 ] = o_{d}(1).
\] 
Moreover, by Eqs.~\eqref{eqn:Z(ZZ)-1_bound}, \eqref{eqn:HQZ1_bound}, we have
\[
\begin{aligned}
\E[I_2 ] \le&~ \E\Big[ \sup_{k \ge 2}  \lambda_{d, 0}^2  \|  \bZ  \Res\|_{\op} \| \bM \|_{\op} \| \Res \bZ^\sT \|_{\op}  \| Q_k(\bX \bX^\sT) - \id_n \|_{\op} \Big] \\
\le&~ O_{d}(1) \cdot \E[\| \bM \|_{\op}^2]^{1/2} \cdot \E\Big[\sup_{k \ge 2} \| Q_k(\bX \bX^\sT) - \id_n \|_{\op}^2 \Big]^{1/2} = o_d(1). 
\end{aligned}
\]
Plugging these bounds into Eq. (\ref{eqn:constant_term_two_lem_decomposition_2}), we get
the desired bound $\E [ \sup_{k \ge 2} \vert S_{2k} - S_3 \vert ] = o_{d}(1)$.

We next consider Eq.~\eqref{eq:S11}, where we recall the definition of $S_{11}$ in Eq. (\ref{eqn:S_definitions_in_decomposition}) and the definition of $\Psi_1$ in Eq. (\ref{eqn:Psi_definitions_in_decomposition}). By observing that $\lim_{d \to \infty} \sqrt d \lambda_{1, d}(\sigma) = \ob_1$ (see Eq. (\ref{eqn:relationship_mu_lambda})) and that $\ob_1 Q_1(\bX \bTheta^\sT) = \ob_1 \bX \bTheta^\sT / d = \bZ_1$, we immediately get 
\[
\E \vert S_{11} - \Psi_1\vert = o_d(1)  \cdot \E[\vert \Psi_1 \vert] = o_d(1).
\] 

In order to prove Eq.~\eqref{eq:S21},
recall  the definition of $S_{21}$ in Eq. (\ref{eqn:S_definitions_in_decomposition}) and the definition of $\Psi_2$ in Eq. (\ref{eqn:Psi_definitions_in_decomposition}). By the decomposition of $\bU$ in Eq. (\ref{lem:U_decomposition_proof_in_lem:constant_term_two}) and
recalling that $Q_1(\bX \bX^\sT) = \bH$, we have 
\begin{equation}\label{eqn:constant_term_two_lem_decomposition_3}
\begin{aligned}
\vert S_{21} - \Psi_2 \vert \le I_3 + I_4, 
\end{aligned}
\end{equation}
where 
\[
\begin{aligned}
I_3 =&~ \Big  \vert  \frac{\lambda_{d, 0}(\sigma)^2}{d} \Tr\Big[ \Res \ones_N \ones_N^\sT  \Res \bZ^\sT \bH \bZ \Big] \Big \vert,\\
I_4 =&~ \Big  \vert \frac{\ob_\star^2}{d} \Tr\Big[  \Res \bDelta  \Res \bZ^\sT \bH \bZ \Big] \Big \vert. 
\end{aligned}
\]
By Lemma \ref{lem:constant_term_two} and Eq. (\ref{eqn:HQZ1_bound}), we get 
\[
\E[ I_3 ] = o_{d}(1).
\]
Moreover, by Eq. (\ref{eqn:Z(ZZ)-1_bound}), (\ref{eqn:HQZ1_bound}), we have
\[
\begin{aligned}
\E[ I_4 ] \le&~ \E\Big[\ob_\star^2 \| \bZ  \Res \|_{\op} \| \bDelta\|_{\op} \| \Res \bZ^\sT \|_{\op} \| \bH \|_{\op}\Big]\le O_d(1) \cdot \E[\| \bDelta \|_{\op}^2] \cdot \E[\| \bH \|_{\op}^2] = o_d(1). 
\end{aligned}
\]
Plugging these bounds into Eq.~\eqref{eqn:constant_term_two_lem_decomposition_3}, we get
the desired bound $\E \vert S_{21} - \Psi_2 \vert = o_d(1)$.

Finally,  Eq.~\eqref{eq:S31} is  proved analogously to Eq.~\eqref{eq:S21}: this completes the proof of the lemma.

\subsection{Proof of Lemma  \ref{lem:concentration_beta}}\label{subsec:concentration_beta}

Instead of taking $\bbeta_{d, 1} \sim \Unif(\S^{d-1}(\normf_{d, 1}))$,  in the proof we will assume $\bbeta_{d, 1} \sim \normal(\bzero,
[\normf_{d, 1}^2 / d] \id_d)$. Note for $\bbeta_{d, 1} \sim \normal(\bzero, [\normf_{d, 1}^2 / d] \id_d)$, we have $\normf_{d, 1}
\bbeta_{d, 1}/ \| \bbeta_{d, 1} \|_2 \sim \Unif(\S^{d-1}(\normf_{d, 1}))$. Moreover, in high dimension, $\| \bbeta_{d, 1}\|_2$ 
concentrates tightly around $\normf_{d, 1}$. Using these properties, it
is not hard to translate the proof from Gaussian  $\bbeta_{d, 1}$ to
spherical $\bbeta_{d, 1}$. 

To prove Lemma \ref{lem:concentration_beta}, we begin with by rewriting the prediction risk ---cf. Eq. (\ref{eqn:prediction_risk_decomposition_first})--- as  (note that $\by = \boldf + \beps$)
\begin{align*}
R_\RF(f_d, \bX, \bTheta, \lambda)
=&~  \sum_{k \ge 0} \normf_{d, k}^2 - 2 \Gamma_1 + \Gamma_2 + \Gamma_3 - 2 \Gamma_4 + 2 \Gamma_5, 
\end{align*}
where
\begin{align*}
\Gamma_1 =&~ \boldf^\sT \bZ  \Res \bV / \sqrt{d}, \\
\Gamma_2 =&~ \boldf^\sT \bZ   \Res \bU  \Res  \bZ^\sT \boldf / d, \\
\Gamma_3 =&~ \beps^\sT \bZ   \Res \bU  \Res  \bZ^\sT \beps / d, \\
\Gamma_4 =&~ \beps^\sT \bZ  \Res \bV / \sqrt{d}, \\
\Gamma_5 =&~ \beps^\sT \bZ   \Res \bU  \Res  \bZ^\sT \boldf / d, 
\end{align*}
and $\bV \in \R^N$ and $\bU  \in \R^{N \times N}$  given in Eq.~\eqref{eqn:def_Vi_Uij}. We will regard $\Gamma_1,\dots,\Gamma_5$
as quadratic forms in the vectors $\bbeta$, $\beps$, and bound their variances individually. Namely, we
claim that
\[
\E_{\bX, \bTheta}[\Var_{\bbeta, \beps} ( \Gamma_k )] = o_d(1) ,\;\;\; \forall k\le 5.
\]
This obviously implies the claims of the lemma.
In the rest of this proof, we show the variance bound for $\Gamma_1$, as the other bounds are very similar.

Recall the definition of $\bY_{k, \bx}$ and $\bY_{k, \btheta}$ in Eq. (\ref{eqn:def_Y_k_x_theta}), the definition of Gegenbauer coefficients $\lambda_{d, l} \equiv \lambda_{d, l}(\sigma)$ in Eq. (\ref{eqn:expansion_sigma_lem_reformulation}) and the expansion of $\boldf$ and $\bV$ vectors in Eq. (\ref{eqn:expansion_boldf_V}). 
We rewrite $\Gamma_1$ as
\[
\Gamma_1 = \frac{1}{ \sqrt d} \Big(\sum_{k=0}^\infty \bY_{k, \bx} \bbeta_{d, k} \Big)^\sT \bZ \Res \Big( \sum_{l = 0}^\infty \lambda_{d, l} \bY_{l, \btheta} \bbeta_{d, l} \Big) .
\] 

Calculating the variance of $\Gamma_1$ with respect to $\bbeta_{d, k} \sim \normal(\bzero, (\normf_{d, k}^2 / B(d, k)) \id)$ for $k \ge 1$ using Lemma \ref{lem:variance_calculations} (which follows from direct calculation), we get
\begin{align*}
\Var_{\bbeta}(\Gamma_1) 
=&~ \sum_{l \neq k} \frac{\lambda_{d, l}^2}{d} \Var_\bbeta \Big(   \bbeta_{d, k}^\sT\bY_{k, \bx}^\sT \bZ \Res  \bY_{l, \btheta} \bbeta_{d, l} \Big) + \sum_{k\ge 1} \frac{ \lambda_{d, k}^2}{d} \Var_\bbeta \Big(  \bbeta_{d, k}^\sT \bY_{k, \bx}^\sT \bZ \Res \bY_{k, \btheta} \bbeta_{d, k} \Big)\\
=&~ \sum_{l \neq k}  \normf_{d, l}^2 \normf_{d, k}^2  \frac{\lambda_{d, l}^2}{d}  \Tr\Big(  \Res \bZ^\sT Q_k(\bX \bX^\sT) \bZ \Res   Q_l(\bTheta \bTheta^\sT) \Big) \\
& + \sum_{k\ge 1} \normf_{d,k}^4    \frac{\lambda_{d, k}^2}{d} \Big[ \Tr\Big(  \Res \bZ^\sT Q_k(\bX \bX^\sT) \bZ \Res  Q_k(\bTheta \bTheta^\sT) \Big) + \Tr\Big(  \bZ \Res Q_k(\bTheta \bX^\sT) \bZ \Res Q_k(\bTheta \bX^\sT) \Big)\Big]. 
\end{align*}
Notice that we have  $\| \bZ \Res \|_{\op} \le C$  almost surely for some constant $C$, and recall
the bounds \eqref{eqn:QXX_bound} which imply $\E\{\sup_{k \ge 1} \| Q_k(\bX \bX^\sT) \|^2\} = O_d(1)$,
$\E\{\sup_{k \ge 1} \| Q_k(\bTheta \bX^\sT) \|^2\} = o_d(1)$, and $\E\{\sup_{k \ge 1} \| Q_k(\bTheta \bTheta^\sT) \|^2\} = O_d(1)$
(the case $k=1$ corresponds to standard Wishart matrices).

By taking expectation in the above expression, using $d^{-1}\Tr(\bA)\le C\|\bA\|_{\op}$ for $\bA\in\reals^{n\times n}$ or $\bA\in\reals^{N\times N}$, and Cauchy-Schwartz,  we obtain
\begin{align*}
\begin{aligned}
  \E_{\bX,\bTheta}\big[\Var_{\bbeta}(\Gamma_1) \big]  \le&~
  \sum_{k\ge1}  \normf_{d, 0}^2 \normf_{d, k}^2  \frac{\lambda_{d, 0}^2}{d}  \E_{\bX, \bTheta} \Tr\Big(  \Res \bZ^\sT Q_k(\bX \bX^\sT) \bZ \Res \ones_N\ones_N^\sT \Big)\\
&~  + \sum_{l>1}  \normf_{d, l}^2 \normf_{d, 0}^2  \frac{\lambda_{d, l}^2}{d} \E_{\bX, \bTheta}  \Tr\Big(  \Res \bZ^\sT \ones_n\ones_n^{\sT}
  \bZ \Res  Q_l(\bTheta \bTheta^\sT) \Big)+
  C \sum_{l \neq k\ge 1}  \normf_{d, l}^2 \normf_{d, k}^2  \lambda_{d, l}^2+
  C\sum_{k\ge 1} \normf_{d,k}^4    \lambda_{d, k}^2\, .
\end{aligned}
  \end{align*}
  Further note that $\| \sigma \|_{L^2(\tau_d)}^2 = \sum_{k \ge 0} \lambda_{d, k}^2 B(d, k) = O_d(1)$, $B(d, k) = \Theta(d^k)$, and for fixed $d$, $B(d, k)$ is non-decreasing in $k$ \cite[Lemma 1]{ghorbani2019linearized}. Therefore 
\[
\sup_{k \ge 1} \vert \lambda_{d, k}(\sigma) \vert \le \sup_{k \ge 1} \Big[ \| \sigma \|_{L^2(\tau_d)} / \sqrt{B(d, k)} \Big] = O_d(1/\sqrt d). 
\]
Substituting above we obtain, and using the fact that $\sum_{k\ge 1}F_{d,k}^2=O_d(1)$ by construction, we have
\begin{equation}\label{eq:BoundVarGamma1}
\begin{aligned}
  \E_{\bX,\bTheta}\big[\Var_{\bbeta}(\Gamma_1) \big] \le&~
  \sum_{k\ge 1}  \normf_{d, 0}^2 \normf_{d, k}^2  \frac{\lambda_{d, 0}^2}{d}  \E_{\bX, \bTheta} \Tr\Big(  \Res \bZ^\sT Q_k(\bX \bX^\sT) \bZ \Res \ones_N\ones_N^\sT\Big)\\
&~  + \sum_{l>1}  \normf_{d, l}^2 \normf_{d, 0}^2  \frac{\lambda_{d, l}^2}{d}  \E_{\bX, \bTheta} \Tr\Big(  \Res \bZ^\sT \ones_n\ones_n^{\sT} \bZ \Res  Q_l(\bTheta \bTheta^\sT) \Big)+o_d(1)\, .
\end{aligned}
\end{equation}
To bound the remaining two terms in this expression, note that 
\[
\begin{aligned}
&\sup_{k \ge 1} \E_{\bX, \bTheta}  \Big\vert \frac{1}{d}  \Tr\Big(  \Res \bZ^\sT \ones_n \ones_n^\sT \bZ \Res  Q_k(\bTheta \bTheta^\sT) \Big)  \Big\vert =\sup_{k \ge 1} \E_{\bX, \bTheta}  \Big\vert \frac{1}{d}  \Tr\Big(  \oRes  \ones_n \ones_n^\sT  \oRes  \bZ Q_k(\bTheta \bTheta^\sT)  \bZ^\sT \Big)  \Big\vert = o_d(1), 
\end{aligned}
\]
where the bound is implied by Lemma \ref{lem:constant_term_two} (when applying Lemma \ref{lem:constant_term_two}, we change the role of $N$ and $n$, the role of $\Res$ and $\oRes$, and the role of $\bTheta$ and $\bX$; this can be done because the role of $\bTheta$ and $\bX$ is symmetric), and by Eq. (\ref{eqn:HQZ1_bound}) and $\lambda_{d, 0}(\sigma) = \Theta_d(1)$ (by Assumption \ref{ass:activation} and note that $\mu_0(\sigma) = \lim_{d \to \infty} \lambda_{d, 0}(\sigma)$ by Eq. (\ref{eqn:relationship_mu_lambda})). This proves that
\begin{align*}
\sum_{l>1}  \normf_{d, l}^2 \normf_{d, 0}^2  \frac{\lambda_{d, l}^2}{d}  \E_{\bX, \bTheta}\Tr\Big(  \Res \bZ^\sT \ones_n\ones_n^{\sT} \bZ \Res  Q_l(\bTheta \bTheta^\sT) \Big) =o_d(1)\, .
\end{align*}
The bound on the first term in Eq.~\eqref{eq:BoundVarGamma1} is obtained analogously and we omit it for brevity.

\section{Proof of Proposition \ref{prop:Stieltjes}}\label{sec:Stieltjes}

This section is organized as follows. 
We collect the elements to prove Proposition \ref{prop:Stieltjes} in Sections \ref{subsec:lemma_Stieltjes}, \ref{subsec:property_fixed_point}, \ref{subsec:leave_one_out_proof}, \ref{subsec:equivalence_G_S}, and prove the proposition in Section \ref{subsec:proof_Stieltjes}.

More specifically, in Section \ref{subsec:lemma_Stieltjes} we state the key Lemma \ref{lemma:KeyFiniteD}:
the partial Stieltjes transforms of $\bA$  approximately satisfy the fixed point equation, when $\bx_i, \btheta_a$
are Gaussian vectors and  the activation function $\vphi$ is a polynomial with $\E_{G \sim \normal(0, 1)}[\vphi(G)] = 0$.
In Section \ref{subsec:property_fixed_point}, and Section \ref{subsec:leave_one_out_proof} we first establish some useful properties
of the fixed point equations and then prove Lemma \ref{lemma:KeyFiniteD}.
Finally, in Section \ref{subsec:equivalence_G_S}, we show that the Stieltjes transform
does not change significantly when changing the distribution of $\bx_i, \btheta_a$ from uniform on the sphere to Gaussian.

\subsection{The key lemma: partial Stieltjes transforms are approximate fixed point}\label{subsec:lemma_Stieltjes}

In this subsection, we state Lemma \ref{lemma:KeyFiniteD}, which is the key lemma that is used to prove Proposition \ref{prop:Stieltjes}. Lemma \ref{lemma:KeyFiniteD} studies $\overline m_{1, d}$ and $\overline m_{2, d}$, the partial Stieltjes transforms of the Gaussian counterparts of the matrix $\bA$ as defined in Eq. (\ref{eqn:matrix_A}). This lemma shows that these partial Stieltjes transforms $\overline m_{1, d}$ and  $\overline m_{2, d}$ approximately satisfy the fixed point equation which involves functions $\sFone$ and $\sFtwo$ as defined in Eq. (\ref{eq:Fdef}). We will prove Lemma \ref{lemma:KeyFiniteD} in Section \ref{subsec:leave_one_out_proof}. Later in Section \ref{subsec:equivalence_G_S}, we will show that the Gaussian counterpart of the Stieltjes transform shares the same asymptotics with its spherical version.

First let us define the Gaussian counterparts of the partial Stieltjes transforms. Let $(\overline \btheta_a)_{a\in [N]}\sim_{iid} \normal(0,\id_d)$ and $(\overline \bx_i)_{i\in[n]}\sim_{iid} \normal(0,\id_d)$. We denote by $\overline \bTheta\in \reals^{N\times d}$ the matrix whose $a$-th row is given by $\overline\btheta_a$, and by $\overline\bX\in\reals^{n\times d}$ the matrix whose $i$-th row is given by $\overline \bx_i$. We consider a polynomial activation functions $\vphi: \reals\to\reals$. Denote $\ob_k = \E[\vphi(G) \He_k(G)]$ and $\ob_\star^2 = \sum_{k\ge 2} \ob_k^2/k!$. We define the following matrices,
\begin{align}
  \overline \bQ =&~ \frac{1}{d} \overline \bTheta \overline \bTheta^\sT, \, \;\;\;\;\;\;\; \;\;\;\;\;\;\; \;\;\;\;\;\;\;\;\;
\overline \bH =~ \frac{1}{d}\overline  \bX \overline \bX^\sT, \label{eq:GaussMatrices1}\\
\overline \bJ =&~ \frac{1}{\sqrt d} \varphi\Big( \frac{1}{\sqrt d} \overline \bX \overline \bTheta^\sT \Big), \, \;\;\;\;\;\;\;
       \overline \bJ_1 =~ \frac{\ob_1}{d}\overline  \bX \overline \bTheta^\sT , \label{eq:GaussMatrices2}
\end{align}
as well as the block matrix $\overline \bA \in\reals^{M\times M}$, $M=N+n$, defined by
\begin{align}\label{eqn:matrix_A_bar}
\overline \bA =&~ \left[\begin{matrix}
s_1\id_N+s_2\overline \bQ & \overline \bJ^\sT  + p \overline \bJ_1^{\sT}\\
\overline \bJ + p \overline \bJ_1 & t_1\id_n+t_2\overline \bH
\end{matrix}\right]\,.
\end{align}
The matrix $\overline \bA$ is in parallel with its spherical version matrix $\bA$ defined as in Eq. (\ref{eqn:matrix_A}).

In what follows, we will write $\bq = (s_1,s_2, t_1,t_2, p)$. We would like to calculate the asymptotic behavior of the following partial Stieltjes transforms 
\begin{equation}\label{eqn:m_Gaussian_definition}
\begin{aligned}
\overline m_{1, d}(\xi; \bq)  =&~ \frac{N}{d}\E\{ (\overline \bA - \xi \id_{M})^{-1}_{11} \} = \E[\overline M_{1, d}(\xi; \bq)], \\
\overline m_{2, d}(\xi; \bq)  =&~ \frac{n}{d}\E\{ (\overline \bA - \xi \id_{M})^{-1}_{N+1, N+1}  \} = \E[\overline M_{2, d}(\xi; \bq)],
\end{aligned}
\end{equation}
where
\begin{equation}\label{eqn:M_Gaussian_definition}
\begin{aligned}
\overline M_{1, d}(\xi; \bq) =&~ \frac{1}{d} \Tr_{[1,N]}[(\overline \bA - \xi \id_M)^{-1}], \\
\overline M_{2, d}(\xi; \bq) =&~ \frac{1}{d} \Tr_{[N+1,N+n]}[(\overline \bA - \xi \id_M)^{-1}].
\end{aligned}
\end{equation}
Here, the partial trace notation $\Tr_{[\cdot, \cdot]}$ is defined as follows: for a matrix $\bK \in \C^{M \times M}$ and $1 \le a \le b \le M$, define
\[
\Tr_{[a, b]}(\bK) = \sum_{i=a}^b K_{ii}. 
\]

The crucial step consists in showing that the expected Stieltjes transforms $\overline m_{1,d}, \overline  m_{2,d}$ are approximate solutions of the fixed point equations \eqref{eq:FixedPoint}.

\begin{lemma}\label{lemma:KeyFiniteD}
Assume that $\vphi$ is a polynomial with $\E[\vphi(G)] = 0$ and $\ob_1 \equiv \E[\vphi(G) G] \neq 0$. Consider the linear regime Assumption \ref{ass:linear}. Then for any $\bq \in \cQ$ and for any $\xi_0 > 0$, there exists $C=C(\xi_0, \bq, \psi_1,\psi_2, \vphi)$ which is uniformly bounded when $(\bq, \psi_1,\psi_2)$ is in a compact set, and a function $\err(d)$ with $\lim_{d \to \infty} \err(d) \to 0$, such that for all $\xi \in \C_+$ with $\Im(\xi)>\xi_0$, we have
\begin{align}
\Big\vert \overline m_{1,d} - \sFone(\overline m_{1,d}, \overline m_{2,d};\xi;\bq,\psi_{1,d},\psi_{2,d}, \ob_1, \ob_\star)\Big\vert \le C  \, \cdot \err(d)\, ,\label{eqn:1_KeyFiniteD}\\
 \Big\vert \overline m_{2,d} - \sFtwo(\overline m_{1,d}, \overline m_{2,d};\xi;\bq,\psi_{1,d},\psi_{2,d}, \ob_1, \ob_\star)\Big\vert \le C \cdot \err(d)\,. \label{eqn:2_KeyFiniteD}
\end{align}
\end{lemma}

The proof of Lemma \ref{lemma:KeyFiniteD} uses a leave-one-out argument in deriving the fixed point equation for Stieltjes transform of random matrices, e.g. \cite[Chapter 3.3]{bai2010spectral} and \cite{cheng2013spectrum}. We will prove Lemma \ref{lemma:KeyFiniteD} in Section \ref{subsec:leave_one_out_proof}. In the next subsection, we will collect some lemmas that are used in this proof.

\subsection{Preliminaries of the proof of Lemma \ref{lemma:KeyFiniteD}: Stieltjes transforms and the fixed point equation}\label{subsec:property_fixed_point}

First we establish some useful properties of the fixed point characterization \eqref{eq:FixedPoint}, where $\sFone, \sFtwo$ is defined via Eq.~\eqref{eq:Fdef}.
For the sake of simplicity, we will write $\bbm = (m_1,m_2)$ and introduce the function $\bsF(\,\cdot\, ;\xi;\bq,\psi_1,\psi_2, \ob_1, \ob_\star):\complex\times\complex\to \complex\times\complex$ via
\begin{align}\label{eqn:def_bsF}
\bsF(\bbm;\xi;\bq,\psi_1,\psi_2, \ob_1, \ob_\star) = \left[
\begin{matrix}
\sFone(m_1,m_2;\xi;\bq,\psi_1,\psi_2, \ob_1, \ob_\star) \\
\sFtwo(m_1,m_2;\xi;\bq,\psi_1,\psi_2, \ob_1, \ob_\star) 
\end{matrix}\right]\, .
\end{align}
In the following lemma, we fix a $\bq \in \cQ$ (as defined in Eq. (\ref{eqn:definition_of_cQ})) and fix $0 < \psi_1,\psi_2, \ob_1, \ob_\star < \infty$. Since the parameters $\bq, \psi_1, \psi_2, \ob_1, \ob_\star$ are fixed, we will drop them from the argument of $\bsF$ unless necessary. 
In these notations,  Eq.~\eqref{eq:FixedPoint} reads
\begin{align}
\bbm = \bsF(\bbm;\xi)\, .\label{eq:FixedPointShort}
\end{align}
The following lemma shows that there exists a unique fixed point of the equation above in a certain domain provided $\Im \xi$ is large enough. 
%
%

\begin{lemma}\label{lemma:PropertiesMap}
Let $\disk(r) = \{z:\;  \vert z\vert <r\}$ be the disk of radius $r$ in the complex plane. There exists $\xi_0=\xi_0(\bq, \psi_1,\psi_2, \ob_1, \ob_\star)>0$ such that, for any $\xi \in \C_+$ with $\Im(\xi)\ge \xi_0$, $\bsF(\,\cdot\,; \xi)$ maps domain $\disk(2 \psi_1 / \xi_0)\times\disk(2 \psi_2 / \xi_0)$ into itself and further is $1/2$-Lipschitz continuous. As a result, Eq.~\eqref{eq:FixedPoint} admits a unique solution in $\disk(2 \psi_1 / \xi_0) \times \disk(2 \psi_2 / \xi_0)$. 
\end{lemma}

\begin{proof}[Proof of Lemma \ref{lemma:PropertiesMap}]

We rewrite the first equation in Eq.~\eqref{eq:Fdef} as
\begin{align}
\sFone(m_1,m_2;\xi) =&~ \frac{\psi_1}{-\xi+s_1+ \sHone(m_1,m_2)}\, , \label{eq:FHdef}\\
\sHone(m_1,m_2) =&~ - \ob_\star^2  m_2 +\frac{1}{m_1+\frac{1+t_2 m_2}{s_2 +(t_2s_2-\ob_1^2 (1 + p)^2)m_2}}\, . \label{eq:Hdef}
\end{align}
It is easy to see that, for $r_0=r_0(\bq, \ob_1, \ob_\star)$ small enough, $\vert \sHone(\bbm)\vert \le 2 + 2 \vert s_2 \vert$ for any $\bbm \in \disk(r_0)\times\disk(r_0)$. Therefore $\vert\sFone(\bbm;\xi)\vert \le \psi_1/(\Im(\xi)-2 - 2 \vert s_2 \vert) < 2 \psi_1 / \xi_0$ provided $\Im \xi \ge \xi_0> 4 + 4 \vert s_2 \vert$. Similarly, we have $\vert\sFtwo(\bbm;\xi)\vert < 2 \psi_2 / \xi_0$ provided $\Im \xi \ge \xi_0> 4 + 4 \vert t_2 \vert$. We enlarge $\xi_0$ so that $2 \max\{\psi_1, \psi_2 \} / \xi_0 \le r_0$. This shows that $\bsF$ maps domain $\disk(2 \psi_1 / \xi_0)\times\disk(2 \psi_2 / \xi_0)$ into itself.

In order to prove the Lipschitz continuity of $\bsF$ in this domain, notice that $\sFone$ is differentiable and
\begin{align}
\nabla_{\bbm}\sFone(\bbm;\xi) = \frac{\psi_1}{(-\xi+s_1+ \sHone(\bbm))^2} \, \nabla_{\bbm}\sHone(\bbm)\, .
\end{align}
By enlarging $\xi_0$, we can ensure $\|\nabla_{\bbm}\sHone(\bbm)\|_2\le C(\bq, \ob_1, \ob_\star)$ for all $\bbm\in \disk(2 \psi_1 / \xi_0)\times\disk(2 \psi_2 / \xi_0)$, whence in the same domain
$\|\nabla_{\bbm}\sFone(\bbm;\xi)\|_2\le C(\bq, \ob_1, \ob_\star) \psi_1/(\Im(\xi)-2 - 2 \vert s_2 \vert)^2$. This result similarly holds for $\sFtwo$. Therefore, by enlarging $\xi_0$, we get $\bsF$ is $1/2$-Lipschitz on $\disk(2 \psi_1 / \xi_0)\times\disk(2 \psi_2 / \xi_0)$. 
 
As a consequence, we have shown that $\bsF$ is a contraction on domain $\disk(2 \psi_1 / \xi_0)\times\disk(2 \psi_2 / \xi_0)$. The existence of a unique fixed point follows by Banach fixed point theorem. 
\end{proof}

Next, we establish some properties of the Stieltjes transforms as in Eq. (\ref{eqn:m_Gaussian_definition}). Notice that the functions $\xi \mapsto \overline m_{i,d}(\xi;\bq) / \psi_{i,d}$, $i\in\{1,2\}$ can be shown to be Stieltjes transforms of certain probability measures on the reals line $\reals$ \cite{hastie2019surprises}. As such, they enjoy several useful properties (see, e.g., \cite{Guionnet}). The next three lemmas are
standard, and already stated in \cite{hastie2019surprises}. For the reader's convenience, we reproduce them here without proof: although the present definition of the matrix $\overline\bA$ is slightly more general, the proofs are unchanged.

\begin{lemma}[Lemma 7 in \cite{hastie2019surprises}]\label{lemma:BasicProperties}
The functions $\xi\mapsto \overline m_{1,d}(\xi)$, $\xi\mapsto \overline m_{2,d}(\xi)$ have the following properties:
\begin{enumerate}
\item[$(a)$] $\xi\in\complex_+$, then $\vert \overline m_{i, d}\vert \le \psi_i /\Im(\xi)$ for $i \in \{ 1, 2\}$.
\item[$(b)$] $\overline m_{1,d}$, $\overline m_{2,d}$ are analytic on $\complex_+$ and map $\complex_+$ into $\complex_+$. 
\item[$(c)$] Let $\Omega\subseteq\complex_+$ be a set with an accumulation point. If $\overline m_{i, d}(\xi)\to m_i(\xi)$ for all $\xi\in\Omega$, then $m_i(\xi)$ has a unique analytic continuation to $\complex_+$ and  $\overline m_{i, d}(\xi)\to m_i(\xi)$ for all $\xi\in\complex_+$. Moreover, the convergence is uniform over compact sets $\Omega \subseteq \C_+$. 
\end{enumerate}
\end{lemma}

\begin{lemma}[Lemma 8 in \cite{hastie2019surprises}]\label{eq:BD-Diff}
Let $\bW\in\reals^{M \times M}$ be a symmetric matrix, and denote by $\bw_i$ its $i$-th column,
with the $i$-th entry set to $0$. Let $\bW^{(i)} \equiv \bW-\bw_i\bfe_i^{\sT}-\bfe_i\bw_i^{\sT}$, where $\bfe_i$ is the $i$-th element of the canonical basis
(in other words, $\bW^{(i)}$ is obtained from $\bW$ by zeroing all elements in the $i$-th row and column except on the diagonal). 
Finally, let $\xi\in\complex_+$ with $\Im(\xi)\ge\xi_0>0$. Then for any subset $S \subseteq [M]$, we have
\begin{align}
\Big\vert \Tr_{S}\big[(\bW-\xi\id_M)^{-1}\big]- \Tr_{S}\big[(\bW^{(i)}-\xi\id_M)^{-1}\big]\Big\vert \le \frac{3}{\xi_0}\, .
\end{align}
\end{lemma}

The next lemma establishes the concentration of Stieltjes transforms to its mean, whose proof is the same as the proof of Lemma 9 in \cite{hastie2019surprises}. 
\begin{lemma}[Concentration]\label{lemma:Concentration_Stieltjes}
Let $\Im(\xi) \ge \xi_0 > 0$ and consider the partial Stieltjes transforms $\overline M_{i, d}(\xi; \bq)$ as per Eq. (\ref{eqn:M_Gaussian_definition}). Then there exists $c_0= c_0(\xi_0)$ such that, for $i\in\{1,2\}$,
\begin{align}
\P\big(\big\vert \overline M_{i,d}(\xi; \bq) - \E \overline M_{i,d}(\xi; \bq)\big\vert \ge u\big) \le 2\, e^{-c_0 d u^2}, \label{eq:Concentration}
\end{align}
In particular, if $\Im(\xi) >0$, then $\vert \overline M_{i,d}(\xi; \bq) - \E \overline M_{i,d}(\xi;\bq)\vert \to 0$ almost surely and in $L^1$. 
\end{lemma}

\begin{lemma}[Lemma 5 in \cite{ghorbani2019linearized}]\label{lemma:square_integrable}
Assume $\sigma$ is an activation function with $\sigma(u)^2\le c_0 \, \exp(c_1\, u^2/2)$ for some constants $c_0 > 0$ and $c_1<1$ (this is implied by Assumption \ref{ass:activation}). 
Then 
\begin{enumerate}
\item[$(a)$] $\E_{G \sim \normal(0, 1)}[\sigma(G)^2] < \infty$. 
\item[$(b)$] Let $\|\bw\|_2=1$. Then there exists $d_0=d_0(c_1)$ such that, for $\bx \sim \Unif(\S^{d-1}(\sqrt  d))$,
\begin{align}
\sup_{d \ge d_0} \E_{\bx}[\sigma(\<\bw,\bx\>)^2] < \infty\, .
\end{align}
\item[$(c)$] Let $\|\bw\|_2=1$. Then there exists a coupling of $G\sim\normal(0,1)$ and $\bx\sim \Unif(\S^{d-1}(\sqrt  d))$ such that
\begin{align}\label{eq:ConvergenceExpMud}
\lim_{d \to \infty} \E_{\bx, G} [(\sigma(\<\bw,\bx\>) - \sigma(G))^2] =&~ 0.
\end{align}
\end{enumerate}
\end{lemma}

\subsection{Proof of Lemma \ref{lemma:KeyFiniteD}: Leave-one-out argument}\label{subsec:leave_one_out_proof}

Throughout the proof, we write $F(d) = O_d(G(d))$ if there exists a constant $C=C(\xi_0, \bq, \psi_1,\psi_2, \vphi)$ which is uniformly bounded when $(\xi_0, \bq, \psi_1,\psi_2)$ is in a compact set, such that $\vert F(d) \vert \le C\cdot \vert G(d)\vert$. We write $F(d) = o_d(G(d))$ if for any $\eps > 0$, there exists a constant $C = C(\eps, \xi_0, \bq, \psi_1,\psi_2, \vphi)$ which is uniformly bounded when $(\xi_0, \bq, \psi_1,\psi_2)$ is in a compact set, such that $\vert F(d) \vert \le \eps \cdot \vert G(d)\vert$ for any $d \ge C$. We use $C$ to denote generically such a constant, that can change from line to line. 

We write $F(d) = O_{d, \P}(G(d))$ if for any $\delta > 0$, there exists constant $K = K(\delta, \xi_0, \bq,\psi_1,\psi_2, \vphi), d_0 = d_0(\delta, \xi_0, \bq, \psi_1,\psi_2, \vphi)$ which are uniformly bounded when $(\xi_0, \bq, \psi_1,\psi_2)$ is in a compact set, such that $\P(\vert F(d)\vert > K \vert G(d) \vert ) \le \delta$ for any $d \ge d_0$. We write $F(d) = o_{d, \P}(G(d))$ if for any $\eps, \delta > 0$, there exists constant $d_0 = d_0(\eps, \delta, \xi_0, \bq, \psi_1,\psi_2, \vphi)$ which are uniformly bounded when $(\xi_0, \bq, \psi_1,\psi_2)$ is in a compact set, such that $\P(\vert F(d)\vert > \eps \vert G(d) \vert ) \le \delta$ for any $d \ge d_0$.

We will assume $p = 0$ throughout the proof. For $p \neq 0$, the lemma holds by viewing $\overline \bJ + p \overline \bJ_1 = \vphi_\star(\bX \bTheta^\sT / \sqrt d)/\sqrt d$ as a new kernel inner product matrix, with $\vphi_\star(x) = \vphi(x) + p \ob_1 x$. 

\noindent
{\bf Step 1. Calculate the Schur complement and define some notations. }

Let $\overline \bA_{\cdot, N}\in\reals^{M-1}$ be the $N$-th column of $\overline \bA$, with the $N$-th entry removed. We further denote by $\overline\bB \in \reals^{(M-1)\times (M-1)}$ be the the matrix obtained from $\overline \bA$ by removing the $N$-th column and $N$-th row. Applying Schur complement formula with respect to element $(N,N)$, we get 
\begin{align}
\overline m_{1,d}=\psi_{1,d}\,  \E\Big\{\Big(-\xi+s_1+s_2\|\overline \btheta_N\|_2^2/d- \overline \bA_{\cdot, N}^{\sT}(\overline \bB-\xi \id_{M-1})^{-1}\overline \bA_{\cdot, N} \Big)^{-1}\Big\}\, .
\end{align}
We decompose the vectors $\overline \btheta_a, \overline \bx_i$ in the components along $\overline \btheta_N$ and the orthogonal component:
\begin{align}
\overline \btheta_a =&~ \eta_a \frac{\overline \btheta_N}{\| \overline \btheta_N\|_2}+\tbtheta_a\, , \;\;\; \<\overline \btheta_N,\tbtheta_a\> = 0, \;\;\; &a&\in [N-1]\,,\\
\overline \bx_i =&~ u_i \frac{\overline \btheta_N}{\| \overline \btheta_N\|} +\tbx_i \, , \;\;\; \<\overline \btheta_N,\tbx_i\> = 0, \;\;\; &i& \in [n]\, .
\end{align}
Note that $\{\eta_a\}_{a \in [N-1] },\{u_i\}_{i \in [n]}\sim_{iid} \normal(0,1)$ are independent of all the other random variables, and 
$\{\tbtheta_a\}_{a\in [N-1]},\{\tbx_i\}_{i\in [n]}$ are conditionally independent given $\overline \btheta_N$, with 
$\tbtheta_a,\tbx_i\sim\normal(\bzero,\proj_{\perp})$, where $\proj_{\perp}$ is the projector orthogonal to $\overline \btheta_N$. 

With this decomposition we have
\begin{align}
\overline Q_{a,b}=&~\frac{1}{d}\Big(\eta_a\eta_b+\<\tbtheta_a,\tbtheta_b\>\Big)\, , & a&,b\in [N-1]\, ,\\
\overline J_{i,a} =&~ \frac{1}{\sqrt{d}}\varphi\left(\frac{1}{\sqrt{d}}\<\tbx_i,\tbtheta_a\> +\frac{1}{\sqrt{d}} u_i\eta_a\right) \, , & a&\in [N-1],\, i\in [n] \, ,\\
\overline H_{ij} =&~ \frac{1}{d}\Big(u_iu_j+\<\tbx_i,\tbx_j\>\Big)\, , & i&,j\in [n]\, .
\end{align}
Further we have $\overline \bA_{\cdot, N} = (\overline A_{1, N}, \ldots, \overline A_{M-1, N})^\sT \in \R^{M-1}$ with
\begin{align}
\overline  A_{i, N} = \begin{cases}
\frac{1}{d} s_2\eta_i\|\overline \btheta_N\|_2 & \mbox{ if $i\le N-1$,}\\
\frac{1}{\sqrt{d}} \varphi\Big(\frac{1}{\sqrt{d}} u_i\|\overline \btheta_N\|_2\Big) & \mbox{ if $i\ge N$.}\\
\end{cases}\label{eq:AcolDef}
\end{align}

We next write $\overline \bB$ as the sum of three terms:
\begin{align}
\overline \bB =&~ \tbB+ \bDelta+\bE_0 \in \R^{(M-1) \times (M-1)}\, ,\label{eq:DecompositionB}
\end{align}
where
\begin{align}\label{eq:Decomposition_B_tilde}
\tbB =&~ \left[\begin{matrix}
s_1\id_{N-1}+s_2\tbQ & \tbJ^{\sT}\\
\tbJ & t_1\id_n+t_2\tbH
\end{matrix}\right]\, ,\;\;\;\;\;
\bDelta  = \left[\begin{matrix}
\frac{s_2}{d}\bfeta\bfeta^{\sT}& \frac{\ob_1}{d}\bfeta\bu^{\sT}\\
\frac{\ob_1}{d}\bu\bfeta^{\sT} & \frac{t_2}{d}\bu\bu^{\sT}
\end{matrix}\right]\, ,\;\;\;\;\;\;\;
\bE_0  = \left[\begin{matrix}
\bzero& \bE_1^{\sT}\\
\bE_1 & \bzero
\end{matrix}\right]\, ,
\end{align}
and $\bfeta = (\eta_1, \ldots, \eta_{N-1})^\sT$, $\bu = (u_1, \ldots, u_n)^\sT$, and 
\begin{align}
\tQ_{a,b}=&~ \frac{1}{d}\<\tbtheta_a,\tbtheta_b\>\, ,\;\;\;\;\; &a&,b\in [N-1]\, ,\label{eq:tQdef}\\
\tJ_{i,a} =&~ \frac{1}{\sqrt{d}}\varphi\left(\frac{1}{\sqrt{d}}\<\tbx_i,\tbtheta_a\>\right) \, ,\;\;\;\;\; &a&\in [N-1],\, i\in [n]\, ,\label{eq:tJdef}\\
\tH_{ij} =&~ \frac{1}{d}\<\tbx_i,\tbx_j\>\, ,\;\;\;\;\; &i&,j\in [n]\, .\label{eq:tHdef}
\end{align}
Further, we have $\bE_1 = (E_{1,ia})_{i\in [n], a\in [N-1]} \in \R^{n \times N}$, where 
\begin{align}
E_{1,ia} =&~ \frac{1}{\sqrt{d}}\left[\varphi\Big(\frac{1}{\sqrt{d}}\<\tbx_i,\tbtheta_a\>+\frac{1}{\sqrt{d}}u_i\eta_a\Big)-\varphi\Big(\frac{1}{\sqrt{d}}\<\tbx_i,\tbtheta_a\>\Big)-\frac{\ob_1}{\sqrt{d}}u_i\eta_a\right]\\
=&~ \frac{1}{\sqrt{d}}\left[\varphi_{\perp}\Big(\frac{1}{\sqrt{d}}\<\tbx_i,\tbtheta_a\>+\frac{1}{\sqrt{d}}u_i\eta_a\Big)-\varphi_{\perp}\Big(\frac{1}{\sqrt{d}}\<\tbx_i,\tbtheta_a\>\Big)\right]\,,
\end{align}
where $\vphi_\perp(x) \equiv \vphi(x) - \ob_1 x$.

\noindent
{\bf Step 2. Perturbation bound for the Schur complement. }

Denote
\begin{align}
\omega_1 =&~ \Big(-\xi+s_1+s_2\|\overline \btheta_N\|_2^2/d- \overline \bA_{\cdot, N}^{\sT}(\overline \bB-\xi \id_{M-1})^{-1}\overline \bA_{\cdot, N} \Big)^{-1}, \label{eqn:def_omega_1_in_Stieltjes_proof}\\
\omega_2 =&~ \Big(-\xi+s_1+s_2- \overline  \bA_{\cdot, N}^{\sT}( \tbB+\bDelta-\xi\id_{M-1})^{-1} \overline \bA_{\cdot, N} \Big)^{-1}. \label{eqn:def_omega_2_in_Stieltjes_proof}
\end{align}
Note we have $\overline m_{1, d} = \psi_{1, d} \E[\omega_1]$. Combining Lemma \ref{lem:stieltjes_first_reformulation}, \ref{lem:concentration_of_E1_in_Stieltjes_proof}, and \ref{lem:concentration_A_vector_in_Stieltjes_proof} below, we have 
\[
\vert \omega_1 - \omega_2 \vert \le O_{d}(1) \cdot \Big\vert \|\overline \btheta_N\|_2^2/d - 1 \Big\vert + O_d(1) \cdot \| \overline  \bA_{\cdot, N}\|_2^2  \cdot \| \bE_1 \|_{\op} = o_{d, \P}(1).
\]
Moreover, by Lemma \ref{lem:stieltjes_first_reformulation}, $\vert \omega_1 - \omega_2 \vert$ is deterministically bounded by $2/ \xi_0$. This gives
\begin{equation}\label{eqn:interpolation_omega_2_m}
\vert \overline m_{1, d} - \psi_{1, d} \E[\omega_2]\vert \le \psi_{1, d} \E[\vert \omega_1 - \omega_2 \vert] = o_{d}(1). 
\end{equation}

\begin{lemma}\label{lem:stieltjes_first_reformulation}
Using the definitions of $\omega_1$ and $\omega_2$ as in  Eq. (\ref{eqn:def_omega_1_in_Stieltjes_proof}) and (\ref{eqn:def_omega_2_in_Stieltjes_proof}), for $\Im \xi \ge \xi_0$, we have
\[
\begin{aligned}
\vert \omega_1 - \omega_2 \vert \le  \Big[  s_2 \vert \|\overline \btheta_N\|_2^2/d - 1 \vert / \xi_0^2 + 2 \| \overline  \bA_{\cdot, N}\|_2^2 \| \bE_1 \|_{\op}  / \xi_0^4 \Big] \wedge [ 2 / \xi_0 ].
\end{aligned}
\]
\end{lemma}

\begin{proof}[Proof of Lemma \ref{lem:stieltjes_first_reformulation}] Note that
\[
\Im ( -\omega_1^{-1}) \ge \Im \xi + \Im ( \overline \bA_{\cdot, N}^{\sT}(\overline \bB-\xi \id_{M-1})^{-1}\overline \bA_{\cdot, N} ) \ge \Im \xi > \xi_0. 
\]
Hence we have  $\vert \omega_1 \vert \le 1 / \xi_0$, and, using a similar argument, $\vert \omega_2 \vert \le 1/ \xi_0$. Hence we get the bound $\vert \omega_1 - \omega_2 \vert \le 2 / \xi_0$. 

Denote
\[
\omega_{1.5} = \Big(-\xi+s_1+s_2 - \overline \bA_{\cdot, N}^{\sT}(\overline \bB-\xi \id_{M-1})^{-1}\overline \bA_{\cdot, N} \Big)^{-1}, \\
\]
we get
\[
\begin{aligned}
\vert \omega_1 - \omega_{1.5}\vert= s_2 \Big \vert \omega_1 (\|\overline \btheta_N\|_2^2/d - 1) \omega_{1.5} \Big\vert \le s_2 \Big\vert \|\overline \btheta_N\|_2^2/d - 1 \Big\vert/\xi_0^2. 
\end{aligned}
\]
Moreover, we have 
\[
\begin{aligned}
\vert \omega_{1.5} - \omega_2 \vert =&~ \vert \omega_{1.5}\omega_2 \overline  \bA_{\cdot, N}^{\sT}[ ( \tbB+\bDelta-\xi\id_{M-1})^{-1} - ( \tbB+\bDelta + \bE_0 -\xi\id_{M-1})^{-1} ] \overline \bA_{\cdot, N} \vert \\
=&~\vert \omega_{1.5}\omega_2 \overline  \bA_{\cdot, N}^{\sT} ( \tbB+\bDelta-\xi\id_{M-1})^{-1} \bE_0 ( \tbB+\bDelta + \bE_0 -\xi\id_{M-1})^{-1} \overline \bA_{\cdot, N} \vert\\
\le&~ (1/ \xi_0^2) \cdot \| \overline \bA_{\cdot, N}\|_2^2 (1/ \xi_0^2) \| \bE_0 \|_{\op} \le 2 \| \bE_1 \|_{\op} \|  \overline \bA_{\cdot, N}\|_2^2 / \xi_0^4. 
\end{aligned}
\]
This proves the lemma. 
\end{proof}

\begin{lemma}\label{lem:concentration_of_E1_in_Stieltjes_proof}
Under the assumptions of Lemma \ref{lemma:KeyFiniteD}, we have
\begin{align}
\|\bE_1\|_{\op} = O_{d, \P}(\Poly(\log d) / d^{1/2}). 
\end{align}
\end{lemma}
\begin{proof}[Proof of Lemma \ref{lem:concentration_of_E1_in_Stieltjes_proof}]
Define $\bz_i =\tbtheta_i$ for $i\in [N-1]$,  $\bz_i = \tbx_{i-N+1}$ for $N\le i\le M -1$, and 
$\zeta_i = \eta_i$ for $i\in [N-1]$,  $\zeta_i = u_{i-N+1}$ for $N\le i\le M -1$. Consider the symmetric matrix $\bE\in\reals^{(M-1)\times (M-1)}$
with $E_{ii}=0$, and
\begin{align}
E_{ij} =&~\frac{1}{\sqrt{d}}\left[\varphi_{\perp}\Big(\frac{1}{\sqrt{d}}\<\bz_i,\bz_j\>+\frac{1}{\sqrt{d}}\zeta_i\zeta_j\Big)-\varphi_{\perp}\Big(\frac{1}{\sqrt{d}}\<\bz_i,\bz_j\>\Big)\right]\, .
\end{align}
Since $\bE_1$ is a sub-matrix of $\bE$, we have $\|\bE_1\|_{\op}\le \|\bE\|_{\op}$. By the intermediate value theorem 
\begin{align}
\bE =&~ \frac{1}{\sqrt{d}}\bXi \bF_1\bXi+\frac{1}{2 d}\bXi^2\bF_2\bXi^2\, ,\\
\bXi \equiv&~ \diag(\zeta_1,\dots,\zeta_{M-1})\, ,\\
F_{1,ij} \equiv&~ \frac{1}{\sqrt{d}}\varphi'_{\perp}\Big(\frac{1}{\sqrt{d}}\<\bz_i,\bz_j\>\Big)\, \bfone_{i\neq j},\\
F_{2,ij} \equiv&~ \frac{1}{\sqrt{d}}\varphi''_{\perp}\big(\tz_{ij}\big)\, \bfone_{i\neq j}\, , \;\;\;\;\; \tz_{ij}\in \Big[\frac{1}{\sqrt{d}}\<\bz_i,\bz_j\>,\frac{1}{\sqrt{d}}\<\bz_i,\bz_j\>+\frac{1}{\sqrt{d}}\zeta_i\zeta_j\Big]\, .
\end{align}
Hence we get
\[
\| \bE \|_{\op} \le (\| \bF_1 \|_{\op}/\sqrt d) \| \bXi \|_{\op}^2 + (\| \bF_2 \|_{\op} / d) \| \bXi \|_{\op}^4. 
\]
Note that $\varphi''_{\perp}(x) = \varphi''(x)$ is a polynomial with some fixed degree $\bar k$. Therefore we have
\[
\begin{aligned}
\E\{\|\bF_2\|_F^2 \}  =&~[ M(M-1) / d] \cdot \E[\varphi''_\perp(\tz_{12})^2] \le O_d(d) \cdot \E[ (1 + \vert \tz_{12} \vert)^{2 \bar k}] \\
\le&~ O_d(d) \cdot \Big\{ \E\big[(1 + \vert \<\bz_i,\bz_j\> / \sqrt d \vert)^{2 \bar k}\big] + \E\big[(1 + \vert \<\bz_i,\bz_j\> + \zeta_i \zeta_j / \sqrt d \vert )^{2 \bar k}\big] \Big\} = O_d(d). 
\end{aligned}
\]
Moreover, by the fact that $\vphi_\perp'$ is a polynomial with $\E[\vphi_\perp'(G)] = 0$, and by Theorem 1.7 in \cite{fan2019spectral}, we have $\|\bF_1\|_{\op} = O_{d, \P}(1)$. By the concentration bound for $\chi$-squared random variable, we get $\| \bXi \|_{\op} = O_{d, \P}(\sqrt{\log d})$. Therefore, we have
\[
\| \bE \|_{\op} \le O_{d, \P}(d^{-1/2}) O_{d, \P}(\Poly(\log d)) + O_{d, \P}(d^{-1/2}) O_{d, \P}(\Poly(\log d)) = O_{d, \P}(\Poly(\log d) / d^{-1/2}). 
\]
This proves the lemma. 
\end{proof}

\begin{lemma}\label{lem:concentration_A_vector_in_Stieltjes_proof}
Under the assumptions of Lemma \ref{lemma:KeyFiniteD}, we have
\begin{align}
\|\overline \bA_{\cdot,N}\|_{2} = O_{d, \P}(1). 
\end{align}
\end{lemma}
\begin{proof}[Proof of Lemma \ref{lem:concentration_A_vector_in_Stieltjes_proof}]
Recall the definition of $\overline \bA_{\cdot, N}$ as in Eq. (\ref{eq:AcolDef}). Denote $\ba_1 = s_2 \bfeta \| \overline \btheta_N \|_2 / d \in \R^{N-1}$, and $\ba_2 = \vphi(\bu \|\overline \btheta_N\|_2 / \sqrt d) / \sqrt d \in \R^n$, where $\bfeta \sim \normal(\bzero, \id_{N-1})$, and $\bu \sim \normal(\bzero, \id_n)$. Then $\overline \bA_{\cdot,N} = (\ba_1; \ba_2) \in \R^{n + N - 1}$. 

For $\ba_1$, note we have $\| \ba_1 \|_2 = \vert s_2\vert\cdot \| \bfeta \|_2 \| \overline \btheta_N \|_2 / d$ where $\bfeta \sim \normal(\bzero, \id_{N-1})$ and $\overline \btheta_N \sim \normal(\bzero, \id_d)$ are independent. Hence we have 
\[
\E[\| \ba_1 \|_2^2] =  s_2^2 \E[\| \bfeta \|_2^2 \| \overline \btheta_N \|_2^2] / d^2 = O_{d}(1). 
\]
For $\ba_2$, note $\vphi$ is a polynomial with some fixed degree $\bar k$, hence we have 
\[
\E[\| \ba_2 \|_2^2] = \E[ \vphi(u_i \| \overline \btheta_N \|_2 / \sqrt d)^2] = O_{d}(1). 
\]
This proves the lemma. 
\end{proof}

\noindent
{\bf Step 3. Simplification using Sherman-Morrison-Woodbury formula. }

Notice that $\bDelta$ is a matrix with rank at most two. Indeed
\begin{align}
\bDelta =&~ \bU\bM\bU^{\sT} \in \R^{(M-1) \times (M-1)} ,\;\;\;\;\; \bU = \frac{1}{\sqrt{d}}\left[\begin{matrix}\bfeta & \bzero\\
\bzero &\bu\end{matrix} \right] \in \R^{(M - 1) \times 2} ,\;\;\; \bM = \left[\begin{matrix} s_2 & \ob_1\\
\ob_1 & t_2\end{matrix} \right] \in \R^{2 \times 2} .\label{eqn:def_M_lemma_leave_one_out}
\end{align}
Since we assumed $\bq \in \cQ$ so that $\vert s_2 t_2 \vert \le \ob_1^2 / 2$, the matrix $\bM$ is invertible with $\| \bM^{-1} \|_{\op} \le C$.

Recall the definition of $\omega_2$ in Eq. (\ref{eqn:def_omega_2_in_Stieltjes_proof}). By the Sherman-Morrison-Woodbury formula, we get
\begin{align}\label{eqn:def_omega2_lemma_leave_one_out}
\omega_2 =&~ \Big(-\xi+s_1+s_2- v_1 + \bv_2^\sT (\bM^{-1} + \bV_3)^{-1} \bv_2 \Big)^{-1}, 
\end{align}
where
\begin{align}
v_1 =&~  \overline \bA_{\cdot, N}^{\sT}(\tbB-\xi\id_{M-1})^{-1} \overline  \bA_{\cdot, N}, & \bv_2 =&~ \bU^{\sT}(\tbB-\xi\id_{M-1})^{-1} \overline  \bA_{\cdot, N}, & \bV_3  =&~ \bU^{\sT}(\tbB-\xi\id_{M-1})^{-1}\bU. \label{eqn:def_V3_lemma_leave_one_out}
\end{align}
We define
\begin{align}
\overline v_1 =&~ s_2^2  \overline m_{1,d}+( \ob_1^2 + \ob_\star^2) \overline m_{2,d}, &  \overline \bv_2 =&~ \begin{bmatrix}s_2 \overline  m_{1,d}\\
\ob_1 \overline m_{2,d}\end{bmatrix},&  \overline\bV_3  =&~ \begin{bmatrix} \overline m_{1,d} & 0\\
0 & \overline  m_{2,d}\end{bmatrix},\label{eqn:def_barV3_lemma_leave_one_out}
\end{align}
and
\begin{align}
\omega_3 = \Big(-\xi+s_1+s_2-  \overline v_1 + \overline \bv_2^\sT (\bM^{-1} +  \overline \bV_3)^{-1} \overline \bv_2 \Big)^{-1}. \label{eq:LOO-3}
\end{align}
By  auxiliary Lemmas \ref{lem:stieltjes_second_reformulation}, \ref{lem:concentration_v_in_Stieltjes_proof}, and \ref{lem:Reverse_SMW_in_Stieltjes_proof} below, we get
\[
\E[ \vert \omega_2 - \omega_3 \vert] = o_d(1),
\]
Combining with Eq. (\ref{eqn:interpolation_omega_2_m}) we get
\[
\vert \overline m_{1, d} - \psi_{1,d} \omega_3 \vert = o_{d}(1). 
\]
Elementary algebra simplifying Eq.~\eqref{eq:LOO-3} gives $\psi_{1, d} \omega_3 = \sFone(\overline m_{1,d}, \overline m_{2,d};\xi;\bq,\psi_{1,d},\psi_{2,d}, \ob_1, \ob_\star)$. This proves Eq. (\ref{eqn:1_KeyFiniteD}) in Lemma \ref{lemma:KeyFiniteD}. Eq. (\ref{eqn:2_KeyFiniteD}) follows by the same argument (exchanging $N$ and $n$). In the rest of this section, we prove auxiliary Lemmas \ref{lem:stieltjes_second_reformulation}, \ref{lem:concentration_v_in_Stieltjes_proof}, and \ref{lem:Reverse_SMW_in_Stieltjes_proof}.

\begin{lemma}\label{lem:stieltjes_second_reformulation}
Using the formula of $\omega_2$ and $\omega_3$ as in Eq. (\ref{eqn:def_omega2_lemma_leave_one_out}) and (\ref{eq:LOO-3}), for $\Im \xi \ge \xi_0$, we have 
\[
\begin{aligned}
\vert \omega_2 - \omega_3 \vert \le&~  O_d(1) \cdot \Big\{ \Big[  \vert v_1 - \overline v_1 \vert + \| \overline \bv_2 \|_{2}^2 \| (\bM^{-1} + \bV_3)^{-1} \|_{\op} \| (\bM^{-1} + \overline \bV_3)^{-1} \|_{\op} \| \bV_3 - \overline \bV_3 \| _{\op} \\
& + (\| \bv_2 \|_2 + \| \overline \bv_2 \|_2) \| (\bM^{-1} + \bV_3)^{-1} \|_{\op} \| \bv_2 - \overline \bv_2 \|_2  \Big] \wedge 1 \Big\}.
\end{aligned}
\]
\end{lemma}

\begin{proof}[Proof of Lemma \ref{lem:stieltjes_second_reformulation}]
Denote
\[
\omega_{2.5} = \Big(-\xi+s_1+s_2- \overline v_1 + \bv_2^\sT (\bM^{-1} + \bV_3)^{-1} \bv_2 \Big)^{-1}
\]
We have 
\[
\begin{aligned}
\vert \omega_2 - \omega_{2.5} \vert =&~ \vert \omega_{2} (v_1 - \overline v_1) \omega_{2.5}\vert \le \vert v_1 - \overline v_1 \vert / \xi_0^2. 
\end{aligned}
\]
Moreover, we have 
\[
\begin{aligned}
\vert \omega_{2.5} - \omega_3\vert \le&~ (1/\xi_0^2) \vert \bv_2^\sT (\bM^{-1} + \bV_3)^{-1} \bv_2 - \overline \bv_2^\sT (\bM^{-1} + \overline \bV_3)^{-1} \overline \bv_2\vert \\
\le&~  (1/\xi_0^2) \Big\{ (\| \bv_2 \|_2 + \| \overline \bv_2 \|_2) \| (\bM^{-1} + \bV_3)^{-1} \|_{\op} \| \bv_2 - \overline \bv_2 \|_2 \\
& + \| \overline \bv_2 \|_2^2 \| (\bM^{-1} + \bV_3)^{-1} \|_{\op} \| (\bM^{-1} + \overline \bV_3)^{-1} \|_{\op} \| \bV_3 - \overline \bV_3 \|_{\op} \Big\}. 
\end{aligned}
\]
Combining with $\vert \omega_2 - \omega_3 \vert \le \vert \omega_2 \vert + \vert \omega_3 \vert \le O_d(1)$ proves the lemma.  
\end{proof}

\begin{lemma}\label{lem:concentration_v_in_Stieltjes_proof}
Under the assumptions of Lemma \ref{lemma:KeyFiniteD}, we have (following the notations of Eq. (\ref{eqn:def_V3_lemma_leave_one_out}) and (\ref{eqn:def_barV3_lemma_leave_one_out}))
\begin{align}
\| \overline \bv_2 \|_{2} =&~ O_{d}(1),\\
\vert v_1 - \overline v_1 \vert =&~ o_{d, \P}(1), \label{eq:hRh}\\
\| \bv_2 - \overline \bv_2 \|_2 =&~o_{d, \P}(1), \label{eq:hRU}\\
\| \bV_3 - \overline \bV_3 \|_{\op}  =&~ o_{d, \P}(1). \label{eq:URU}
\end{align}
\end{lemma}

\begin{proof}[Proof of Lemma \ref{lem:concentration_v_in_Stieltjes_proof}]

The first bound is because (see Lemma \ref{lemma:BasicProperties} for the boundedness of $\overline m_{1, d}$ and $\overline m_{2, d}$)
\[
\| \overline \bv_2 \|_{2}  \le \vert s_2\vert\cdot \vert \overline m_{1, d} \vert + \vert \ob_1\vert\cdot \vert \overline m_{2, d}\vert\le (\psi_1 + \psi_2) (\vert s_2\vert + \vert \ob_1 \vert)/\xi_0 = O_d(1). 
\]
In the following, we limit ourselves to proving Eq.~\eqref{eq:hRh}, since Eq. (\ref{eq:hRU}) and (\ref{eq:URU})  follow by similar arguments. 
 
Recall the definition of $\tbB$ as in Eq. (\ref{eq:Decomposition_B_tilde}). Let $\bR \equiv (\tbB-\xi\id_{M-1})^{-1}$. Then we have $\|\bR\|_{\op}\le 1/ \xi_0$. Define $\ba, \bh$ as
\[
\begin{aligned}
\ba =&~ \overline \bA_{\cdot, N} =  \Big[ \frac{1}{d} s_2\bfeta^\sT\|\overline \btheta_N\|_2, \frac{1}{\sqrt{d}} \varphi\Big(\frac{1}{\sqrt{d}} \bu^\sT \|\overline \btheta_N\|_2\Big) \Big]^\sT, \\
\bh =&~ \Big[ \frac{1}{\sqrt d} s_2\bfeta^\sT , \frac{1}{\sqrt{d}} \varphi( \bu^\sT ) \Big]^\sT. \\
\end{aligned}
\]
Then by the definition of $v_1$ in Eq. (\ref{eqn:def_V3_lemma_leave_one_out}), we have $v_1 = \ba^\sT \bR \ba$. Note we have 
\[
\| \bh - \ba\|_2 \le ( s_2 \|\bfeta \|_2 +  \| \vphi'(\bu \odot \bxi) \|_2)  \cdot \vert \| \overline \btheta_N \|_2 / \sqrt d - 1\vert / \sqrt d,
\]
for some $\bxi = (\xi_1, \ldots, \xi_n)^\sT$ with $\xi_i$ between $\| \overline \btheta_N \|_2 / \sqrt d$ and $1$. Since $\vert \| \overline \btheta_N \|_2 / \sqrt d - 1\vert = O_{d, \P}(\sqrt{\log d} / \sqrt d)$, $\| \bfeta \|_2 = O_{d, \P}(\sqrt d)$, and $\| \vphi'(\bu \cdot \bxi)\|_2 = O_{d, \P}(\Poly(\log d)\cdot \sqrt d)$, we have
\[
\| \bh - \ba \|_{2} = o_{d, \P}(1). 
\]
By Lemma \ref{lem:concentration_A_vector_in_Stieltjes_proof} we have $\| \ba \|_2 = O_{d, \P}(1)$ and hence $\| \bh \|_{2} = O_{d, \P}(1)$. Combining all these bounds, we have 
\begin{align}\label{eqn:bound_v_hRh}
\vert v_1 - \bh^\sT \bR \bh \vert = \vert \ba^\sT \bR \ba - \bh^\sT \bR \bh \vert \le (\| \ba \|_2 + \| \bh\|_2) \| \bh - \ba \|_2 \| \bR \|_{\op} = o_{d, \P}(1).
\end{align}

Denote by $\bD$ the covariance matrix of $\bh$. Since $\bh$ has independent elements, $\bD$ a diagonal matrix with $\max_{i} D_{ii}= \max_i \Var(h_i)\le C/d$. Since $\E[\bh] = \bzero$, we have
\begin{align}
\E\{\bh^{\sT}\bR\bh\vert \bR\} =&~ \Tr(\bD\bR). 
\end{align}
We next compute $\Var(\bh^{\sT}\bR\bh\vert \bR)$. By a similar calculation of Lemma \ref{lem:variance_calculations}, we have (for a complex matrix, denote $\bR^\sT$ to be the transpose of $\bR$, and $\bR^*$ to be the conjugate transpose of $\bR$)
\[
\begin{aligned}
&\Var(\bh^{\sT}\bR\bh\vert \bR) 
= \sum_{i=1}^{M-1} \vert R_{ii} \vert^2 (\E[h_i^4] - 3 \E[h_i^2]^2) +  \Tr(\bD \bR^\sT \bD \bR^*) + \Tr(\bD \bR \bD \bR^*). \\
\end{aligned}
\]
Note that we have $\max_i [\E[h_i^4] - 3 \E[h_i^2]^2] = O_d(1/ d^2)$, so that
\[
\sum_{i=1}^{M-1} \vert R_{ii} \vert^2 (\E[h_i^4] - 3 \E[h_i^2]^2) \le O_d(1/d^2) \cdot \| \bR \|_F^2 \le O_d(1/d) \| \bR \|_{\op}^2 = O_d(1/d). 
\]
Moreover, we have
\[
\vert \Tr(\bD \bR^\sT \bD \bR^*) + \Tr(\bD \bR \bD \bR^*)\vert \le \|\bD \bR \|_F^2 + \| \bD \bR \|_F \| \bD \bR^* \|_F \le 2 \| \bD \|_{\op}^2 \| \bR \|_F^2 = O_d(1/d),
\]
which gives
\[
\Var(\bh^{\sT}\bR\bh\vert \bR) = O_d(1/d), 
\]
and therefore
\begin{align}\label{eqn:bound_gRg}
\vert \bh^{\sT}\bR\bh - \Tr(\bD\bR) \vert = O_{d, \P}(d^{-1/2}). 
\end{align}

Combining Eq. (\ref{eqn:bound_gRg}) and (\ref{eqn:bound_v_hRh}), we obtain
\begin{equation}\label{eqn:bound_v1_trdr}
\begin{aligned}
\vert v_1 - \Tr(\bD\bR) \vert  \le&~\vert \ba^\sT \bR \ba - \bh^{\sT}\bR\bh \vert + \vert \bh^{\sT}\bR\bh - \Tr(\bD\bR) \vert = o_{d, \P}(1). 
\end{aligned}
\end{equation}
Finally, notice that
\[
\Tr(\bD \bR) = s_2^2 \Tr_{[1, N-1]}((\tbB-\xi\id_{M-1})^{-1}) / d + (\ob_1^2 + \ob_\star^2) \Tr_{[N, M-1]}((\tbB-\xi\id_{M-1})^{-1}) / d. 
\]
By Lemmas \ref{eq:BD-Diff}, partial Stieltjes transforms are stable with respect to deleting one row and one column of the same index. By Lemma \ref{lem:perturbation_dimension_Stieltjes} (which will be stated and proved later), partial Stieltjes transforms are stable with respect to small changes of the dimension $d$. Moreover, by Lemma \ref{lemma:Concentration_Stieltjes}, partial Stieltjes transforms concentrate tightly around their mean. As a consequence of all these lemmas (Lemma \ref{eq:BD-Diff}, \ref{lem:perturbation_dimension_Stieltjes}, \ref{lemma:Concentration_Stieltjes}), we have 
\[
\begin{aligned}
\vert \Tr_{[1, N-1]}((\tbB-\xi\id_{M-1})^{-1}) / d  - \overline m_{1, d} \vert =&~ o_{d, \P}(1), \\
\vert \Tr_{[N, M-1]}((\tbB-\xi\id_{M-1})^{-1}) / d  - \overline m_{2, d} \vert =&~ o_{d, \P}(1),
\end{aligned}
\]
so that 
\[
\vert \Tr(\bD \bR) - \overline v_1 \vert  = o_{d, \P}(1). 
\]
Combining with Eq. (\ref{eqn:bound_v1_trdr}) proves Eq. (\ref{eq:hRh}). 
\end{proof}

The following lemma is the analog of Lemma B.7 and B.8 in \cite{cheng2013spectrum}. 
\begin{lemma}\label{lem:Reverse_SMW_in_Stieltjes_proof}
Under the assumptions of Lemma \ref{lemma:KeyFiniteD}, we have (using the definitions in Eq. (\ref{eqn:def_M_lemma_leave_one_out}), (\ref{eqn:def_V3_lemma_leave_one_out}) and (\ref{eqn:def_barV3_lemma_leave_one_out}) )
\begin{align}
\| (\bM^{-1} + \bV_3)^{-1} \|_{\op} =&~ O_{d, \P}(1), \\
\| (\bM^{-1} + \overline \bV_3)^{-1}\|_{\op} =&~O_{d}(1).
\end{align}
\end{lemma}
\begin{proof}[Proof of Lemma \ref{lem:Reverse_SMW_in_Stieltjes_proof}]~ 

\noindent
{\bf Step 1. Bounding $\| (\bM^{-1} + \bV_3)^{-1} \|_{\op}$. } By Sherman-Morrison-Woodbury formula, we have 
\[
\begin{aligned}
&(\bM^{-1} + \bV_3)^{-1} = (\bM^{-1} + \bU^{\sT}(\tbB-\xi\id_{M-1})^{-1}\bU)^{-1}\\
=&~ \bM - \bM \bU^\sT (\tbB - \xi \id_{M-1}+ \bU \bM \bU^\sT)^{-1} \bU \bM.
\end{aligned}
\]
Note we have $\| \bM \|_{\op} = O_d(1)$, and $\| (\tbB - \xi \id_{M-1}+ \bU \bM \bU^\sT)^{-1} \|_{\op} \le 1/ \xi_0 = O_d(1)$. Therefore, by the concentration of $\| \bfeta \|_2 / \sqrt d$ and $\| \bu \|_2 / \sqrt d$, we have 
\[
(\bM^{-1} + \bV_3)^{-1} = O_d(1) \cdot (1 +  \| \bU \|_{\op}^2) = O_d(1) (1 + \| \bfeta \|_2 /\sqrt d + \| \bu \|_2 / \sqrt d) = O_{d, \P}(1). 
\]

\noindent
{\bf Step 2. Bounding $\| (\bM^{-1} + \overline \bV_3)^{-1} \|_{\op}$. } Define $\bG = \bM^{1/2} \bV_3 \bM^{1/2}$ and $\overline \bG = \bM^{1/2} \overline \bV_3 \bM^{1/2}$. By Lemma \ref{lem:concentration_v_in_Stieltjes_proof}, we have 
\begin{align}\label{eqn:C13_eq1}
\| \bG - \overline \bG \|_{\op} = o_{d, \P}(1). 
\end{align}
By the bound $\|(\bM^{-1} + \bV_3)^{-1}\|_{\op} = O_{d, \P}(1)$, we get
\begin{align}\label{eqn:C13_eq2}
\| (\id_2 + \bG )^{-1}\|_{\op} = \| \bM^{-1/2} (\bM^{-1} + \bV_3 )^{-1} \bM^{-1/2}\|_{\op} \le \| (\bM^{-1} + \bV_3)^{-1} \| \cdot \| \bM^{- 1/2} \|_{\op}^2 = O_{d, \P}(1).
\end{align}
Note we have 
\[
(\id_2 + \overline \bG)^{-1} -(\id_2 +  \bG)^{-1} = (\id_2 + \overline \bG)^{-1} (\bG - \overline \bG) (\id_2 + \bG)^{-1}, 
\]
so that 
\[
(\id_2 + \overline \bG)^{-1} = \{\id_2 -  (\bG - \overline \bG) (\id_2 + \bG)^{-1}\} (\id_2 +  \bG)^{-1}. 
\]
Combining with Eq. (\ref{eqn:C13_eq1}) and (\ref{eqn:C13_eq2}), we get
\[
\| (\id_2 + \overline \bG)^{-1} \|_{\op} \le \| \id_2 -  (\bG - \overline \bG) (\id_2 + \bG)^{-1} \|_{\op} \| (\id_2 +  \bG)^{-1} \|_{\op} = O_{d, \P}(1) = O_d(1). 
\]
The last equality holds because $\| (\id_2 + \overline \bG)^{-1} \|_{\op}$ is deterministic. Hence we have 
\[
\| (\bM^{-1} + \overline \bV_3)^{-1} \|_{\op} = \| \bM^{1/2} (\id_2 + \overline \bG)^{-1}  \bM^{1/2}  \|_{\op} \le \| (\id_2 + \overline \bG)^{-1} \|_{\op} \| \bM^{1/2} \|_{\op}^2 = O_{d}(1). 
\]
This proves the lemma. 
\end{proof}

The following lemma shows that, the partial resolvents are stable with respect to small changes of the dimension $d$.
%
\begin{lemma}\label{lem:perturbation_dimension_Stieltjes}
Follow the assumptions of Lemma \ref{lemma:KeyFiniteD}. Let $\overline \bX = (\overline \bx_1, \ldots, \overline \bx_n)^\sT \in \R^{n \times d}$ with $(\overline \bx_i = (\overline x_{i, 1}, \ldots, \overline x_{i, d})^\sT )_{i \in [n]} \sim_{iid} \normal(\bzero, \id_d)$, and $\overline \bTheta = (\overline \btheta_1, \ldots, \overline \btheta_N)^\sT \in \R^{N \times d}$ with $(\overline \btheta_a = (\overline \theta_{a, 1}, \ldots, \overline \theta_{a, d})^\sT)_{a \in [N]} \sim_{iid} \normal(\bzero, \id_d)$. Let $\underline \bX = (\underline \bx_1, \ldots, \underline \bx_n )^\sT \in \R^{n \times (d-1)}$ with $\underline \bx_i = \Restrict \overline \bx_i = (\overline x_{i, 1}, \ldots, \overline x_{i, d-1})^\sT \in \R^{d-1}$ and $\underline \bTheta = (\underline \btheta_1, \ldots, \underline \btheta_N)^\sT \in \R^{N \times (d-1)}$ with $\underline \btheta_a = \Restrict \overline \btheta_a = (\overline \theta_{a, 1}, \ldots, \overline \theta_{a, d-1})^\sT \in \R^{d-1}$, where $\Restrict = (\id_{d-1}, \bzero_{(d-1) \times 1}) \in \R^{(d - 1) \times d}$ is the matrix form of the operator that restricts a vector to its $1$'st to $(d-1)$'th canonical coordinates. Denote 
\begin{align}
\overline \bJ \equiv&~ \frac{1}{\sqrt{d}}\vphi \Big(\frac{1}{\sqrt{d}}\overline \bX\, \overline \bTheta^{\sT}\Big), & \underline \bJ \equiv&~ \frac{1}{\sqrt{d}}\vphi \Big(\frac{1}{\sqrt d}\underline \bX\underline \bTheta^{\sT}\Big), \\
\overline\bQ \equiv&~ \frac{1}{d} \overline \bTheta\, \overline \bTheta^{\sT}\, & \underline \bQ \equiv&~ \frac{1}{d}  \underline \bTheta \underline \bTheta^{\sT}\, ,\\
\overline\bH \equiv&~ \frac{1}{d}\overline \bX\, \overline\bX^{\sT}\, & \underline \bH \equiv&~ \frac{1}{d} \underline \bX \underline \bX^{\sT} \, ,
\end{align}
as well as the block matrix $\overline \bA, \underline \bA \in\reals^{M\times M}$, $M=N+n$, defined by
\begin{align}
\overline \bA =&~ \begin{bmatrix}
s_1\id_N+s_2\overline \bQ & \overline \bJ^\sT \\
\overline \bJ & t_1\id_n+t_2\overline \bH
\end{bmatrix}\, & 
\underline \bA =&~ \begin{bmatrix}
s_1\id_N+s_2\underline  \bQ & \underline \bJ^\sT  \\
\underline \bJ & t_1\id_n+t_2 \underline \bH
\end{bmatrix}\,.
\end{align}
Then for any $\xi \in \C_+$ with $\Im \xi \ge \xi_0 > 0$, we have 
\begin{align}
\frac{1}{d}\E \Big\vert \Tr_{[1, N]}[(\overline \bA - \xi \id_M)^{-1}] -\Tr_{[1, N]}[(\underline \bA - \xi \id_M)^{-1}] \Big\vert =&~ o_{d}(1), \label{lem:perturbation_dimension_1}\\
\frac{1}{d}\E \Big\vert \Tr_{[N+1, M]}[(\overline \bA - \xi \id_M)^{-1} - \Tr_{[N+1, M]}[(\underline \bA - \xi \id_M)^{-1}] \Big\vert =&~ o_{d}(1). \label{lem:perturbation_dimension_2}
\end{align}
%
\end{lemma}

\begin{proof}[Proof of Lemma \ref{lem:perturbation_dimension_Stieltjes}]~

\noindent
{\bf Step 1. The Schur complement. }

We denote $\overline \bA_{ij}$ and $\underline \bA_{ij}$ for $i,j\in[2]$ to be 
\[
\begin{aligned}
\overline \bA =&~ \begin{bmatrix}
\overline \bA_{11} & \overline \bA_{12}\\
\overline \bA_{21} & \overline \bA_{22}\\
\end{bmatrix} = \begin{bmatrix}
s_1 \id_N + s_2 \overline \bQ & \overline \bJ^\sT \\
\overline \bJ & t_1 \id_n + t_2 \overline \bH\\
\end{bmatrix},&\underline \bA =&~ \begin{bmatrix}
\underline \bA_{11} & \underline \bA_{12}\\
\underline \bA_{21} & \underline \bA_{22}\\
\end{bmatrix} = \begin{bmatrix}
s_1\id_N+s_2\underline  \bQ & \underline \bJ^\sT  \\
\underline \bJ & t_1\id_n+t_2 \underline \bH
\end{bmatrix}.
\end{aligned}
\]
Define
\[
\begin{aligned}
\overline \omega =&~ \frac{1}{d}\Tr_{[1, N]}[(\overline \bA - \xi \id_M)^{-1}], & \underline \omega =&~ \frac{1}{d}\Tr_{[1, N]}[(\underline \bA - \xi \id_M)^{-1}],
\end{aligned}
\]
and
\[
\begin{aligned}
\overline \bOmega =&~ \Big(\overline \bA_{11} - \xi \id_N - \overline \bA_{12}( \overline \bA_{22} - \xi \id_n)^{-1} \overline \bA_{21} \Big)^{-1}, \\
\underline \bOmega =&~ \Big(\underline \bA_{11} - \xi \id_N - \underline \bA_{12}( \underline \bA_{22} - \xi \id_n)^{-1} \underline \bA_{21} \Big)^{-1}. \\
\end{aligned}
\]
Then we have 
\[
\begin{aligned}
\overline \omega =&~ \frac{1}{d}\Tr(\overline \bOmega), & \underline \omega =&~ \frac{1}{d}\Tr(\underline  \bOmega). 
\end{aligned}
\]
Define 
\[
\begin{aligned}
\bOmega_1 =&~ \Big(\underline \bA_{11} - \xi \id_N - \overline \bA_{12}( \overline \bA_{22} - \xi \id_n)^{-1} \overline \bA_{21} \Big)^{-1}, \\
\bOmega_2 =&~ \Big(\underline \bA_{11} - \xi \id_N - \underline \bA_{12}( \overline \bA_{22} - \xi \id_n)^{-1} \overline \bA_{21} \Big)^{-1}, \\
\bOmega_3 =&~ \Big(\underline \bA_{11} - \xi \id_N - \underline \bA_{12}( \overline \bA_{22} - \xi \id_n)^{-1} \underline \bA_{21} \Big)^{-1}, \\
\end{aligned}
\]
Then it's easy to see that $\| \overline \bOmega \|_{\op}, \| \bOmega_1 \|_{\op}, \| \bOmega_2 \|_{\op}, \|  \bOmega_3 \|_{\op}, \| \underline \bOmega \|_{\op} \le 1/ \xi_0$. 

Calculating their difference, we have
\[
\begin{aligned}
\Big\vert \frac{1}{d}\Tr(\overline \bOmega) - \frac{1}{d}\Tr(\bOmega_1) \Big\vert =&~ \Big\vert \frac{1}{d} \Tr( \overline \bOmega (\underline \bA_{11} - \overline \bA_{11}) \underline \bOmega ) \Big\vert \le O_d(1) \cdot \frac{1}{d} \| \underline \bA_{11} - \overline \bA_{11} \|_{\star}, \\
\Big\vert \frac{1}{d}\Tr(\bOmega_1) - \frac{1}{d}\Tr(\bOmega_2) \Big\vert \le&~ O_d(1) \cdot \frac{1}{d} \| (\underline \bA_{12} - \overline \bA_{12} ) ( \overline \bA_{22} - \xi \id_n)^{-1} \overline \bA_{21} \|_{\star}, \\
\Big\vert \frac{1}{d}\Tr(\bOmega_2) - \frac{1}{d}\Tr(\bOmega_3) \Big\vert \le&~ O_d(1) \cdot \frac{1}{d}  \| (\underline \bA_{12} - \overline \bA_{12} ) ( \overline \bA_{22} - \xi \id_n)^{-1} \underline \bA_{21} \|_{\star}, \\
\Big\vert \frac{1}{d}\Tr(\bOmega_3) - \frac{1}{d}\Tr(\underline \bOmega) \Big\vert \le&~ O_d(1) \cdot \frac{1}{d}  \| \underline \bA_{12} ( \overline \bA_{22} - \xi \id_n)^{-1}(\overline \bA_{22} - \underline \bA_{22}) ( \underline \bA_{22} - \xi \id_n)^{-1} \underline \bA_{21} \|_{\star}. \\
\end{aligned}
\]

\noindent
{\bf Step 2. Bounding the differences. }

First, we have 
\[
\overline \bA_{11} - \underline \bA_{11} = s_2 (\overline \bQ - \underline \bQ) = s_2 ( \overline \theta_{a d} \overline \theta_{b d} / d )_{a, b \in [N]} = s_2 \bfeta \bfeta / d, 
\]
where $\bfeta = (\overline \theta_{1d}, \ldots, \overline \theta_{Nd})^\sT \sim \normal(\bzero, \id_N)$. This gives 
\[
\| \overline \bA_{11} - \underline \bA_{11}\|_{\star} / d = s_2 \| \bfeta \|_2^2 / d^2 = o_{d, \P}(1),
\]
and therefore 
\[
\Big\vert \frac{1}{d}\Tr(\overline \bOmega) - \frac{1}{d}\Tr(\bOmega_1) \Big\vert = o_{d, \P}(1). 
\]

By Theorem 1.7 in \cite{fan2019spectral}, and by the fact that $\vphi$ is a polynomial with $\E[\vphi(G)] = 0$, we have 
\[
\begin{aligned}
\| \overline \bA_{12} \|_{\op} =&~ \| \overline \bJ \|_{\op} = O_{d, \P}(1),& \| \underline \bA_{12} \|_{\op} =&~ O_{d, \P}(1). 
\end{aligned}
\]
It is also easy to see that
\[
\| ( \overline \bA_{22} - \xi \id_n)^{-1} \|_{\op}, \| ( \underline \bA_{22} - \xi \id_n)^{-1} \|_{\op} \le 1/ \xi_0 = O_{d}(1). 
\]
Moreover, we have 
\[
\overline \bA_{22} - \underline \bA_{22} = t_2 ( \overline \bH - \underline \bH) = t_2 ( \overline x_{id} \overline x_{jd} / d )_{i, j \in [n]} = t_2 \bu \bu / d, 
\]
where $\bu = (\overline x_{1d}, \ldots, \overline x_{nd})^\sT \sim \normal(\bzero, \id_n)$. This gives 
\[
\begin{aligned}
&\| \underline \bA_{12} ( \overline \bA_{22} - \xi \id_n)^{-1}(\overline \bA_{22} - \underline \bA_{22}) ( \underline \bA_{22} - \xi \id_n)^{-1} \underline \bA_{21} \|_{\star}/ d\\
 \le&~ t_2 \| \underline \bA_{12} ( \overline \bA_{22} - \xi \id_n)^{-1} \bu \|_2 \| \underline \bA_{12} ( \underline \bA_{22} - \xi \id_n)^{-1} \bu\|_2 / d^2 \\
  \le&~ t_2 \| \underline \bA_{12}\|_{\op}^2 \|( \overline \bA_{22} - \xi \id_n)^{-1}\|_{\op}^2 \| \bu \|_2^2 / d^2 \\
 =&~ O_{d, \P}(1) \cdot \| \bu \|_2^2 / d^2 = o_{d, \P}(1),
 \end{aligned}
\]
and therefore 
\[
\Big\vert \frac{1}{d}\Tr(\bOmega_3) - \frac{1}{d}\Tr(\underline \bOmega) \Big\vert = o_{d, \P}(1). 
\]

By Lemma \ref{lem:concentration_of_E1_in_Stieltjes_proof}, defining
\[
\bE = \overline \bA_{12} - \underline \bA_{12} - \ob_1 \bu \bfeta^\sT / d, 
\]
we have $\| \bE \|_{\op} = O_d(\Poly(\log d) / \sqrt d)$. Therefore, we get
\[
\begin{aligned}
&\| (\underline \bA_{12} - \overline \bA_{12} ) ( \overline \bA_{22} - \xi \id_n)^{-1} \overline \bA_{21} \|_{\star}/ d\\
\le&~\| (\ob_1 \bu \bfeta^\sT / d) ( \overline \bA_{22} - \xi \id_n)^{-1} \overline \bA_{21} \|_{\star}/ d + \| \bE ( \overline \bA_{22} - \xi \id_n)^{-1} \overline \bA_{21} \|_{\star}/ d\\
\le&~ \ob_1 \| \bfeta\|_2  \| (\overline \bA_{22} - \xi \id_n)^{-1}\|_{\op} \| \overline \bA_{21}\|_{\op} \| \bu \|_2 /d^2 + \| \bE \|_{\op} \| ( \overline \bA_{22} - \xi \id_n)^{-1} \|_{\op} \|\overline \bA_{21} \|_{\op} \\
=&~ o_{d, \P}(1),
\end{aligned}
\]
and therefore
\[
\Big\vert \frac{1}{d}\Tr(\bOmega_1) - \frac{1}{d}\Tr(\bOmega_2) \Big\vert, \Big\vert \frac{1}{d}\Tr(\bOmega_2) - \frac{1}{d}\Tr(\bOmega_3) \Big\vert = o_{d, \P}(1). 
\]
Combining all these bounds establishes Eq. (\ref{lem:perturbation_dimension_1}). Finally, Eq. (\ref{lem:perturbation_dimension_2}) can be shown using the same argument. 
\end{proof}

\subsection{Equivalence between Gaussian and sphere version of Stieltjes transforms}\label{subsec:equivalence_G_S}

In this subsection, we show that the Stieltjes transform of matrix $\bA$ as defined in Eq. (\ref{eqn:matrix_A}) and that of matrix $\overline \bA$ as defined in Eq. (\ref{eqn:matrix_A_bar}) share the same asymptotics. For the reader's convenience, we restate the definitions of these two matrices here. 

Let $(\overline \btheta_a)_{a\in [N]}\sim_{iid} \normal(0,\id_d)$, $(\overline \bx_i)_{i \in
 [n]}\sim_{iid} \normal(0,\id_d)$. We denote by $\overline \bTheta\in \reals^{N\times d}$ the matrix whose $a$-th row is given by $\overline\btheta_a$, and by $\overline\bX\in\reals^{n\times d}$ the matrix whose $i$-th row is given by $\overline \bx_i$. We denote by $\bTheta\in \reals^{N\times d}$ the matrix whose $a$-th row is given by $\btheta_a = \sqrt d \cdot \overline\btheta_a /\| \overline \btheta_a \|_2 $, and by $\bX\in\reals^{n\times d}$ the matrix whose $i$-th row is given by $\bx_i = \sqrt  d \cdot \overline \bx_i / \| \overline \bx_i\|_2$. Then we have $(\bx_i)_{i \in [n]} \sim_{iid} \Unif(\S^{d-1}(\sqrt d))$ and $(\btheta_a)_{a \in [N]} \sim_{iid} \Unif(\S^{d-1}(\sqrt d))$ independently. 

We consider activation functions $\sigma, \vphi: \reals\to\reals$ with $\vphi(x) = \sigma(x) - \E_{G \sim \normal(0, 1)}[\sigma(G)]$. We define the following matrices (where $\ob_1$ is the first Hermite coefficients of $\sigma$)
\begin{align}
\overline \bJ \equiv&~ \frac{1}{\sqrt{d}}\vphi \Big(\frac{1}{\sqrt{d}}\overline \bX\, \overline \bTheta^{\sT}\Big), & \bZ \equiv&~ \frac{1}{\sqrt{d}}\sigma \Big(\frac{1}{\sqrt d}\bX\bTheta^{\sT}\Big), \\
\overline \bJ_1 \equiv&~ \frac{\ob_1}{d} \overline \bX\, \overline \bTheta^\sT, & \bZ_1 \equiv&~ \frac{\ob_1}{d}  \bX \bTheta^\sT  \,, \\
\overline\bQ \equiv&~ \frac{1}{d} \overline \bTheta\, \overline \bTheta^{\sT}, & \bQ \equiv&~ \frac{1}{d}  \bTheta\bTheta^{\sT}\, ,\\
\overline\bH \equiv&~ \frac{1}{d}\overline \bX\, \overline\bX^{\sT}, & \bH \equiv&~ \frac{1}{d} \bX\bX^{\sT} \, ,
\end{align}
as well as the block matrix $\overline \bA, \bA \in\reals^{M\times M}$, $M=N+n$, defined by
\begin{align}
\overline \bA =&~ \left[\begin{matrix}
s_1\id_N+s_2\overline \bQ & \overline \bJ^\sT  + p \overline \bJ_1^{\sT}\\
\overline \bJ + p \overline \bJ_1 & t_1\id_n+t_2\overline \bH
\end{matrix}\right], & 
\bA =&~ \left[\begin{matrix}
s_1\id_N+s_2 \bQ & \bZ^\sT  + p  \bZ_1^{\sT}\\
\bZ + p \bZ_1 & t_1\id_n+t_2 \bH
\end{matrix}\right]\,,
\end{align}
and the Stieltjes transforms $\overline M_d(\bxi; \bq)$ and $M_d(\bxi; 
\bq)$, defined by
\begin{align}\label{eqn:two_Stieltjes_transforms}
\overline M_{d}(\xi; \bq) =&~ \frac{1}{d} \Tr[(\overline \bA - \xi \id_M)^{-1}], &~~~~ M_{d}(\xi; \bq) =&~ \frac{1}{d} \Tr[(\bA - \xi \id_M)^{-1}].
\end{align}
The readers could keep in mind: a quantity with an overline corresponds to the case when features and data are Gaussian, while a quantity without overline usually corresponds to the case when features and data are on the sphere. 

\begin{lemma}\label{lem:Gaussian_to_sphere_statement}
Let $\sigma$ be a fixed polynomial. Let $\vphi(x) = \sigma(x) - \E_{G \sim \normal(0, 1)}[\sigma(G)]$. Consider the linear regime of Assumption \ref{ass:linear}. For any fixed $\bq \in \cQ$ and for any $\xi_0 > 0$, we have 
\[
\E\Big[ \sup_{\Im \xi \ge \xi_0} \vert \overline M_{d}(\xi; \bq) - M_{d}(\xi; \bq) \vert \Big] = o_{d}(1). 
\]
\end{lemma}

\begin{proof}[Proof of Lemma \ref{lem:Gaussian_to_sphere_statement}]~

\noindent
{\bf Step 1. Show that the resolvent is stable with respect to nuclear norm perturbation. }

We define
\[
\Delta(\bA, \overline \bA, \xi) = M_{d}(\xi; \bq) - \overline M_{d}(\xi; \bq). 
\]
Then we have deterministically
\[
\begin{aligned}
\vert \Delta(\bA, \overline \bA, \xi) \vert \le \vert M_{d}(\xi; \bq)  \vert + \vert \overline M_{d}(\xi; \bq) \vert \le 4 (\psi_1 + \psi_2) / \Im \xi. 
\end{aligned}
\]
Moreover, we have 
\[
\begin{aligned}
\vert \Delta(\bA, \overline \bA, \xi) \vert =&~ \vert \Tr( (\bA - \xi \id)^{-1} (\bA - \overline \bA) (\overline \bA - \xi \id)^{-1} ) \vert / d\\
\le&~ \| (\bA - \xi \id)^{-1} (\overline \bA - \xi \id)^{-1} \|_{\op} \| \bA- \overline\bA \|_{\star}/d \\
\le&~ \| \bA - \overline \bA \|_{\star} / ( d (\Im \xi)^2).
\end{aligned}
\]
Therefore, if we can show $\| \bA - \overline \bA \|_{\star} / d = o_{d, \P}(1)$, then $\E[\sup_{\Im \xi \ge \xi_0}\vert \Delta(\bA, \overline \bA, \xi)\vert] = o_d(1)$.

\noindent
{\bf Step 2. Show that $\| \bA - \overline \bA \|_\star / d = o_{d, \P}(1)$. }

Denote $\bZ_0 = \E_{G \sim \normal(0, 1)}[\sigma(G)] \ones_n \ones_N^\sT / \sqrt d$ and $\bZ_\star = \vphi(\bX \bTheta^\sT / \sqrt d) / \sqrt d$. Then we have $\bZ = \bZ_0 + \bZ_\star$, and 
\[
\begin{aligned}
\bA - \overline \bA = &s_2 \begin{bmatrix}
\bQ - \overline \bQ & \bzero \\
\bzero & \bzero 
\end{bmatrix}
+
t_2 \begin{bmatrix}
\bzero & \bzero \\
\bzero & \bH - \overline \bH
\end{bmatrix}
+
p \begin{bmatrix}
\bzero & \bZ_1^\sT - \overline \bJ_1^\sT \\
\bZ_1 - \overline \bJ_1 & \bzero 
\end{bmatrix}\\
&+ 
\begin{bmatrix}
\bzero & \bZ_\star^\sT - \overline \bJ^\sT \\
\bZ_\star - \overline \bJ & \bzero 
\end{bmatrix}
+
\begin{bmatrix}
\bzero & \bZ_0^\sT \\
\bZ_0 & \bzero 
\end{bmatrix}.
\end{aligned}
\]
Since $\bq = (s_1, s_2, t_1, t_2, p)$ is fixed, we have
\begin{align*}
\frac{1}{d}\| \bA - \overline \bA \|_\star \le &C \left[\frac{1}{\sqrt
    d} \| \overline  \bQ - \bQ \|_F + \frac{1}{\sqrt d} \| \overline
  \bH - \bH \|_F  + \frac{1}{\sqrt d} \| \overline  \bJ_1 - \bZ_1 \|_F
  + \frac{1}{\sqrt d}\| \overline \bJ - \bZ_\star \|_F \right.\\
&\left.+ \frac{1}{d} \left\| \begin{bmatrix}\bzero & \bZ_0^\sT\\ \bZ_0 & \bzero \end{bmatrix} \right\|_\star \right]. 
\end{align*}

The nuclear norm of the term involving $\bZ_0$ can be easily bounded by 
\[
\begin{aligned}
\frac{1}{d} \left\| \begin{bmatrix}\bzero & \bZ_0^\sT\\ \bZ_0 & \bzero \end{bmatrix} \right\|_\star =&~ \frac{1}{d^{3/2}}\vert \E_{G \sim \normal(0, 1)}[\sigma(G)] \vert \cdot \left\| \begin{bmatrix}\bzero& \ones_N \ones_n^\sT \\ \ones_n \ones_N^\sT& \bzero \end{bmatrix} \right\|_{\star} 
   = o_d(1). 
\end{aligned}
\]
For term $\overline \bH - \bH$, denoting $\bD_\bx = \diag(\sqrt{d}/\| \overline \bx_1\|_2, \ldots, \sqrt d / \|\overline \bx_n \|_2 )$, we have 
\[
\| \overline \bH - \bH \|_{F} / \sqrt d \le \| \overline \bH - \bH \|_{\op} \le \| \id_n - \bD_\bx \|_{\op} \| \overline \bH \|_{\op} (1 + \| \bD_\bx \|_{\op}) = o_{d, \P}(1),  
\]
where we used the fact that $\| \bD_\bx - \id_n \|_{\op} = o_{d, \P}(1)$ and $\| \bH \|_{\op} = O_{d, \P}(1)$. Similar argument shows that 
\[
\| \overline \bQ - \bQ \|_F / \sqrt d = o_{d, \P}(1), ~~~~~ \| \overline \bJ_1 - \bZ_1 \|_F / \sqrt d = o_{d, \P}(1). 
\]

\noindent
{\bf Step 3. Bound for $\| \overline \bJ - \bZ_\star \|_F / \sqrt d$. }

Define $\overline \bZ_\star = \vphi(\bD_\bx \overline \bX\, \overline \bTheta^\sT / \sqrt d) / \sqrt d$. Define $r_i = \sqrt d / \| \overline \bx_i \|_2$. We have (for $\zeta_{ia}$ between $r_i$ and $1$)
\[
\begin{aligned}
\overline \bZ_\star - \overline \bJ  =&~ \Big( \vphi(r_i \< \overline \bx_i, \overline \btheta_a\> / \sqrt d) / \sqrt d - \vphi(\< \overline \bx_i, \overline \btheta_a\> / \sqrt d) / \sqrt d \Big)_{i \in [n], a \in [N]}\\
=&~ \Big( (r_i - 1) (\< \overline \bx_i, \overline \btheta_a\> / \sqrt d ) \vphi'(\zeta_{ia} \< \overline \bx_i, \overline \btheta_a\> / \sqrt d)  / \sqrt d \Big)_{i \in [n], a \in [N]} \\
=&~ (\bD_\bx - \id_n) \bar \vphi(\bXi \odot ( \overline \bX\, \overline \bTheta^\sT /\sqrt d) ) / \sqrt d,
\end{aligned}
\]
where $\bXi = (\zeta_{ia})_{i \in [n], a \in [N]}$, 
and $\bar \vphi(x) = x \vphi'(x)$ (so $\bar \vphi$ is a polynomial). It is easy to see that 
\[
\| \bD_\bx - \id_n \|_{\op} = \max_i\vert r_i - 1\vert = O_{d, \P}(\sqrt{\log d }/ \sqrt d), ~~~~ \| \bXi \|_{\max} = O_{d, \P}(1), ~~~~ \| \overline \bX\, \overline \bTheta^\sT/\sqrt d\|_{\max} = O_{d, \P}(\sqrt{\log d}).
\]
Therefore, we have (denoting $\deg(\vphi)$ to be the degree of polynomial $\vphi$, and $C(\vphi)$ to be a constant that only depends on $\vphi$)
\[
\begin{aligned}
&\| \overline \bZ_\star - \overline \bJ \|_{F}/\sqrt d  = \| (\bD_\bx - \id_n) \bar \vphi(\bXi \odot ( \overline \bX\, \overline \bTheta^\sT /\sqrt d) ) \|_F / d  \\
\le&~ \| \bD_\bx - \id_n \|_{\op} \| \bar \vphi(\bXi \odot ( \overline \bX\, \overline \bTheta^\sT/\sqrt d) )  \|_F /  d \\
 \le&~ C(\vphi) \cdot \| \bD_\bx - \id_n \|_{\op} (1 + \| \bXi \|_{\max} \| \overline \bX\, \overline \bTheta^\sT/\sqrt d \|_{\max})^{\deg(\vphi)} = O_{d, \P}((\log d)^{\deg(\vphi) + 1} / \sqrt d) = o_{d, \P}(1). 
\end{aligned}
\]
This proves the lemma. 
\end{proof}

\subsection{Proof of Proposition \ref{prop:Stieltjes}}\label{subsec:proof_Stieltjes}

\noindent
{\bf Step 1. Polynomial activation function $\sigma$. }

First we consider the case when $\sigma$ is a fixed polynomial with $\E_{G \sim \normal(0, 1)}[\sigma(G) G] \neq 0$. Let $\vphi(u) = \sigma(u) - \E[\sigma(G)]$, and let $\overline \bbm_d \equiv (\overline m_{1,d}, \overline m_{2,d})$ (whose definition is given by Eq. (\ref{eqn:m_Gaussian_definition}) and (\ref{eqn:M_Gaussian_definition})),  and recall that $\psi_{1,d} \to \psi_1$ and $\psi_{2,d}\to \psi_2$ as $d\to\infty$. By Lemma \ref{lemma:KeyFiniteD}, together with the continuity of $\sFone, \sFtwo$ with respect to $\psi_{1},\psi_2$, we have, for any $\xi_0 > 0$, there exists $C=C(\xi_0, \bq, \psi_1,\psi_2, \vphi)$ and $\err(d) \to 0$ such that for all $\xi \in \C_+$ with $\Im \xi \ge \xi_0$, 
\begin{align}\label{eqn:bbmd_approximate_fixed_point}
\big\|\overline \bbm_{d} - \bsF(\overline \bbm_{d};\xi)\big\|_2 \le  C \cdot \err(d). 
\end{align}

By Lemma \ref{lemma:PropertiesMap}, there exists $\xi_0 = \xi_0(\bq, \psi_1, \psi_2, \vphi) > 0$, such that for any $\xi \in \C_+$ with $\Im \xi \ge \xi_0$, $\bsF(\, \cdot\,; \xi)$ a continuous mapping from $\disk(2 \psi_1 / \xi_0) \times \disk(2 \psi_2 / \xi_0)$ to itself and has a unique fixed point $\bbm(\xi)$ in the same domain. By Lemma \ref{lemma:BasicProperties} (a), we have $\overline \bbm_d(\xi) \in \disk(\psi_1 / \xi_0) \times \disk(\psi_2 / \xi_0)$. Combining the above facts with Eq. (\ref{eqn:bbmd_approximate_fixed_point}), we have 
\[
\| \overline \bbm_d(\xi) - \bbm(\xi)\|_2 = o_d(1), ~~ \forall \xi\in \C_+, \Im\xi \ge \xi_0. 
\]
By the property of Stieltjes transform as in Lemma \ref{lemma:BasicProperties} $(c)$, we have 
\[
\|\overline \bbm_d(\xi) - \bbm(\xi) \|_2 = o_d(1), ~~~ \forall \xi\in\complex_+.
\]
By the concentration result of Lemma \ref{lemma:Concentration_Stieltjes}, for $\overline M_d(\xi) = d^{-1} \Tr[(\overline \bA - \xi \id_M)^{-1}]$, we also have 
\begin{equation}\label{eqn:resolvent_to_transform_sphere}
\E\vert \overline M_{d}(\xi) - m(\xi) \vert = o_d(1), ~~~ \forall \xi \in \C_+. 
\end{equation}

Then we use Lemma \ref{lem:Gaussian_to_sphere_statement} to transfer this property from $\overline M_d$ to $M_d$. Recall the definition of resolvent $M_d(\xi; \bq)$ in sphere case in Eq. (\ref{eqn:Stieltjes_A}).  Combining Lemma \ref{lem:Gaussian_to_sphere_statement} with Eq. (\ref{eqn:resolvent_to_transform_sphere}), we have
\begin{equation}
\E \vert M_{d}(\xi) - m(\xi) \vert  = o_d(1), ~~~ \forall \xi \in \C_+. 
\end{equation}

\noindent
{\bf Step 2. General activation function $\sigma$ satisfying Assumption \ref{ass:activation}. }

Next consider the case of a general function $\sigma$ as in the theorem statement satisfying Assumption \ref{ass:activation}. Fix $\eps>0$ and let $\tsigma$ is a polynomial be such that $\|\sigma-\tsigma\|_{L^2(\tau_d)}\le \eps$, where $\tau_d$ is the marginal distribution of $\< \bx, \btheta\>/\sqrt d$ for $\bx, \btheta \sim_{iid} \Unif(\S^{d-1}(\sqrt d))$. In order to construct such a polynomial, consider the expansion of $\sigma$ in the orthogonal basis of Hermite polynomials 
\begin{align}
\sigma(x) = \sum_{k=0}^{\infty}\frac{\ob_k}{k!}\He_k(x)\, .
\end{align}
Since this series converges in $L^2(\mu_G)$, we can choose $\ok<\infty$ such that, letting
$\tsigma(x) =  \sum_{k=0}^{\ok}(\ob_k/k!)\He_k(x)$, we have $\|\sigma-\tsigma\|_{L^2(\mu_G)}^2 \le \eps/2$. By Lemma \ref{lemma:square_integrable} (cf. Eq.~\eqref{eq:ConvergenceExpMud}) we therefore 
have $\|\sigma-\tsigma\|_{L^2(\tau_d)}^2 \le \eps$ for all $d$ large enough. 

Write $\ob_k(\tsigma) = \E[\tsigma(G) \He_k(G)]$ and $\ob_\star(\tsigma)^2 = \sum_{k = 2}^{\ok} \ob_k^2 / k!$. Notice that, by construction we have $\ob_0(\tsigma)=\ob_0(\sigma)$, $\ob_1(\tsigma)=\ob_1(\sigma)$ and $\vert \ob_\star (\tsigma)^2 -\ob_\star (\sigma)^2 \vert \le  \eps$. Let $\tm_{1,d}, \tm_{2,d}$ be the Stieltjes transforms associated to activation $\tsigma$, and $\tm_1,\tm_2$ be the solution of the 
corresponding fixed point equation \eqref{eq:FixedPoint} (with $\ob_\star =\ob_\star(\tsigma)$ and $\ob_1= \ob_1(\tsigma)$), and $\tilde m = \tilde m_1 + \tilde m_2$. Denoting by $\tilde \bA$ the matrix obtained by replacing the $\sigma$ in $\bA$ to be $\tsigma$, and $\tilde M_d(\xi) = (1/d) \Tr[(\tilde \bA - \xi \id)^{-1}]$. Step 1 of this proof implies 
\begin{align}\label{eqn:tilde_M_m_convergnece}
\E \vert \tilde M_{d}(\xi) - \tilde m(\xi)\vert = o_{d}(1), ~~~ \forall \xi \in \C_+.
\end{align}
Further, by continuity of the solution of the fixed point equation with respect to $\ob_\star, \ob_1$ when $\Im \xi \ge \xi_0$ for some large $\xi_0$ (as stated in Lemma \ref{lemma:PropertiesMap}), we have for $\Im \xi \ge \xi_0$, 
\begin{align}\label{eqn:tilde_m_m_convergnece}
\vert \tilde m(\xi) -  m(\xi) \vert \le C(\xi, \bq) \eps. 
\end{align}
Eq. (\ref{eqn:tilde_m_m_convergnece}) also holds for any $\xi \in \C_+$, by the property of Stieltjes transform as in Lemma \ref{lemma:BasicProperties} $(c)$.

Moreover, we have (for $C$ independent of $d$, $\sigma$, $\tilde \sigma$ and $\eps$, but depend on $\xi$ and $\bq$)
\[
\begin{aligned}
&\E \Big[ \Big\vert M_{d}(\xi)-\tilde M_{d}(\xi)\Big\vert \Big] \le \frac{1}{d}\E \Big[ \Big\vert \Tr [(\bA-\xi\id)^{-1} (\tilde \bA-\bA) (\tilde \bA-\xi\id)^{-1}]\Big\vert \Big]\\ 
\le&~ \frac{1}{d}\E\Big[ \|(\tilde \bA-\xi\id)^{-1}(\bA-\xi\id)^{-1} \|_{\op} \|\tilde \bA-\bA \|_\star \Big] \\
\le&~ [1/ (\xi_0^2d)] \cdot \E [ \|\tilde \bA-\bA\|_\star ] \le [1/(\xi_0^2 \sqrt d)] \cdot \E\{\|\tilde \bA-\bA\|_F^2\}^{1/2}  \le C(\xi, \bq) \cdot \|\sigma-\tsigma\|_{L^2(\tau_{d})}.
\end{aligned}
\]
Therefore 
\begin{align}\label{eqn:tilde_M_M_convergnece}
\limsup_{d\to\infty}  \E[ \vert M_{d}(\xi) - \tilde M_d(\xi) \vert ] \le C(\xi, \bq) \eps, ~~ \forall \xi \in \C_+.
\end{align}
Combining Eq. (\ref{eqn:tilde_M_m_convergnece}), (\ref{eqn:tilde_m_m_convergnece}), and (\ref{eqn:tilde_M_M_convergnece}), we obtain
\[
\limsup_{d \to \infty} \E \vert M_{d}(\xi) - m(\xi)\vert \le C(\xi, \bq) \eps, ~~~ \forall \xi \in \C_+.
\]
Taking $\eps \to 0$ proves Eq. (\ref{eqn:weak_version_M_m_convergence}).

\noindent
{\bf Step 3. Uniform convergence in compact sets (Eq. (\ref{eqn:strongest_in_compact})). }

Note $m_d(\xi; \bq) = \E[M_d(\xi; \bq)]$ is an analytic function on $\C_+$. By Lemma \ref{lemma:BasicProperties} (c), for any compact set $\Omega \subseteq \C_+$, we have 
\begin{align}\label{eqn:uniform_Md_eq0}
\lim_{d \to \infty} \Big[ \sup_{\xi \in \Omega} \vert \E[M_d(\xi; \bq)] - m(\xi; \bq) \vert\Big] = 0.
\end{align}

In the following, we show the concentration of $M_d(\xi; \bq)$ around its expectation uniformly in the compact set $\Omega \subset \C_+$. Define $L = \sup_{\xi \in \Omega} (1/(\Im \xi)^2)$. Since $\Omega \subset \C_+$ is a compact set, we have $L < \infty$, and $M_d(\xi; \bq)$ (as a function of $\xi$) is $L$-Lipschitz on $\Omega$. Moreover, for any $\eps > 0$, there exists a finite set $\cN(\eps, \Omega) \subseteq \C_+$ which is an $\eps / L$-covering of $\Omega$. That is, for any $\xi \in \Omega$, there exists $\xi_\star \in \cN(\eps, \Omega)$ such that $\vert \xi - \xi_\star \vert \le \eps /L$. Since $M_d(\xi; \bq)$ (as a function of $\xi$) is $L$-Lipschitz on $\Omega$, we have 
\begin{equation}\label{eqn:uniform_Md_eq1}
\begin{aligned}
\sup_{\xi \in \Omega} \inf_{\xi_\star \in \cN(\eps, \Omega)}\vert M_d(\xi; \bq) - M_d(\xi_\star; \bq)\vert \le& \eps, \\
\sup_{\xi \in \Omega} \inf_{\xi_\star \in \cN(\eps, \Omega)}\vert \E[M_d(\xi; \bq)] - \E[M_d(\xi_\star; \bq)]\vert \le& \eps. 
\end{aligned}
\end{equation}
By the concentration of $M_d(\xi_\star; \bq)$ to its expectation (which is the spherical version of Lemma \ref{lemma:Concentration_Stieltjes}), we have 
\[
\vert M_d(\xi_\star; \bq) - \E[M_d(\xi_\star; \bq)] \vert = o_{d, \P}(1), 
\]
and since $\cN(\eps, \Omega)$ is a finite set, we have
\begin{align}
\sup_{\xi_\star \in \cN(\eps, \Omega)} \vert M_d(\xi_\star; \bq) - \E[M_d(\xi_\star; \bq)] \vert = o_{d, \P}(1). 
\end{align}
This high probability bound will become an expectation bound by the uniform boundedness of $M_d(\xi;\bq)$ for $\xi$ in any compact domain. That is, we have 
\begin{align}\label{eqn:uniform_Md_eq2}
\E\Big[\sup_{\xi_\star \in \cN(\eps, \Omega)} \vert M_d(\xi_\star; \bq) - \E[M_d(\xi_\star; \bq)] \vert \Big] = o_{d}(1). 
\end{align}
Combining Eq. (\ref{eqn:uniform_Md_eq0}), (\ref{eqn:uniform_Md_eq1}) and (\ref{eqn:uniform_Md_eq2}), we have 
\[
\E\Big[\sup_{\xi \in \Omega} \vert M_d(\xi; \bq) - m(\xi; \bq) \vert \Big] \le \eps + o_{d}(1). 
\]
Letting $\eps \to 0$ proves Eq. (\ref{eqn:strongest_in_compact}). This concludes the proof of Proposition \ref{prop:Stieltjes}.


\section{Proof of Proposition \ref{prop:expression_for_log_determinant}}\label{sec:PropoIntegrateXi}

In Section \ref{subsec:preliminary_log_determinant} we state and prove some lemmas that are used in the proof of Proposition \ref{prop:expression_for_log_determinant}. We prove Proposition \ref{prop:expression_for_log_determinant} in Section \ref{subsec:prove_proposition_log_determinant}.

\subsection{Properties of the Stieltjes transforms and the log determinant}\label{subsec:preliminary_log_determinant}

%

The first lemma concerns the behavior of the partial Stieltjes transforms $m_1$ and $m_2$ when $\Im \xi \to \infty$. 

\begin{lemma}\label{lem:asymptotics_m_large_xi}
For $\xi \in \C_+$ and $\bq \in \cQ$ (c.f. Eq. (\ref{eqn:definition_of_cQ})), let $m_1(\xi; \bq), m_2(\xi; \bq)$ be defined as the analytic continuation of solution of Eq. (\ref{eq:FixedPoint}) as defined in Proposition \ref{prop:Stieltjes}. Denote $\xi = \xi_r + \imagunit K$ for some fixed $\xi_r \in \R$. Then we have 
\[
\lim_{K \to \infty} \vert m_1(\xi; \bq) \xi + \psi_1 \vert = 0, ~~ \lim_{K \to \infty} \vert m_2(\xi; \bq) \xi + \psi_2 \vert = 0.
\]
\end{lemma}

\begin{proof}[Proof of Lemma \ref{lem:asymptotics_m_large_xi}] Define $\overline m_1 = - \psi_1 / \xi$, $\overline m_2 = -\psi_2 / \xi$, $\overline \bbm = (\overline m_1, \overline m_2)^\sT$, and $\bbm = (m_1, m_2)^\sT$. Let $\sFone, \sFtwo$ be defined as in Eq. (\ref{eq:Fdef}), $\bsF$ be defined as in Eq. (\ref{eqn:def_bsF}), and $\sHone$ defined as in Eq. (\ref{eq:Hdef}). By simple calculus we can see that 
\[
\lim_{K \to \infty} \sHone(\overline \bbm)  = s_2. 
\]
This gives
\[
\xi [\overline m_1 - \sFone(\overline \bbm; \xi)] = \psi_1 \frac{s_1 + \sHone(\overline \bbm)}{\xi - s_1 - \sHone(\overline \bbm)} \to 0, ~~~ \text{ as } K \to \infty. 
\]
As a result, we have $\xi \| \overline \bbm - \bsF(\overline \bbm; \xi) \|_2 \to 0$ as $K \to \infty$. Moreover, by Lemma \ref{lemma:PropertiesMap}, there exists sufficiently large $\xi_0$, so that for any $\Im \xi = K \ge \xi_0$, $\bsF(\bbm; \xi)$ is $1/2$-Lipschitz on domain $\bbm \in \disk(2 \psi_1 / \xi_0) \times \disk(2 \psi_2 / \xi_0)$. Therefore, for $\Im \xi = K \ge \xi_0$, we have (note we have $\bbm = \bsF(\bbm; \xi)$)
\[
\begin{aligned}
\| \overline \bbm - \bbm  \|_2 =&~ \| \bsF(\overline \bbm; \xi) - \bsF(\bbm; \xi) +  \overline \bbm - \bsF(\overline \bbm; \xi)\|_2 \\
\le&~  \| \bsF(\overline \bbm; \xi) - \bsF(\bbm; \xi) \|_2 + \| \overline \bbm - \bsF(\overline \bbm; \xi) \|_2 \\
\le&~ \| \overline \bbm - \bbm  \|_2 / 2 + \| \overline \bbm - \bsF(\overline \bbm; \xi) \|_2, 
\end{aligned}
\]
so that 
\[
\xi \| \overline \bbm - \bbm  \|_2 \le 2 \xi \| \overline \bbm - \bsF(\overline \bbm; \xi) \|_2 \to 0, ~~~ \text{ as } K \to \infty. 
\]
This proves the lemma. 
\end{proof}

The next lemma concerns the behavior of the log-determinants when $\Im \xi \to \infty$. 

\begin{lemma}\label{lem:large_imaginary_behavior_G}
Follow the notations and settings of Proposition \ref{prop:expression_for_log_determinant}. For any fixed $\bq$, we have
\begin{align}
\lim_{K \to \infty} \sup_{d \ge 1} \E \vert G_d(\imagunit K; \bq) - (\psi_1 + \psi_2)\Log( -\imagunit K) \vert  = 0, \label{eqn:large_imaginary_behavior_prelimit}\\
\lim_{K \to \infty} \vert g(\imagunit K; \bq) - (\psi_1 + \psi_2)\Log(- \imagunit K) \vert = 0. 
\end{align}
\end{lemma}

\begin{proof}[Proof of Lemma \ref{lem:large_imaginary_behavior_G}]~ 

\noindent
{\bf Step 1. Asymptotics of $G_d(\imagunit K; \bq )$. } First we look at the real part. We have 
\[
\begin{aligned}
& \Big\vert \Re \Big[ \frac{1}{M}\sum_{i=1}^M \Log (\lambda_i(\bA) - \imagunit K) - \Log(- \imagunit K) \Big] \Big\vert \\
=&~ \frac{1}{2 M}\sum_{i=1}^M \log (1 + \lambda_i(\bA)^2 / K^2) \le \frac{1}{2 M K^2}\sum_{i=1}^M  \lambda_i(\bA)^2  =\frac{\| \bA \|_F^2}{2 M K^2}. 
\end{aligned}
\]
For the imaginary part, we have 
\[
\begin{aligned}
& \Big\vert \Im \Big[ \frac{1}{M}\sum_{i=1}^M \Log (\lambda_i(\bA) - \imagunit K) - \Log(- \imagunit K) \Big] \Big\vert  \\
=&~ \Big\vert \frac{1}{M}\sum_{i=1}^M \arctan(\lambda_i(\bA) / K) \Big\vert \le \frac{1}{ M K}\sum_{i=1}^M \vert \lambda_i(\bA) \vert  \le \frac{\| \bA \|_F}{ M^{1/2} K}. 
\end{aligned}
\]
Combining the bound of the real part and the imaginary part, we have 
\[
\begin{aligned}
\E \Big\vert \frac{1}{M}\sum_{i=1}^M \Log (\lambda_i(\bA) - \imagunit K) - \Log(- \imagunit K) \Big\vert 
\le&~ \frac{\E[\| \bA \|_F^2]}{2 MK^2} + \frac{\E[\| \bA \|_F^2]^{1/2}}{ M^{1/2} K}.
\end{aligned}
\]
Note that
\[
\begin{aligned}
\frac{1}{M }\E[\| \bA \|_F^2]  \le \frac{1}{M}\Big( \E \| s_1 \id_N + s_2 \bQ \|_F^2 + \E \| t_1 \id_n + t_2 \bH \|_F^2 +  2 \E[\| \bZ + p \bZ_1 \|_F^2]\Big) = O_d(1).
\end{aligned}
\]
This proves Eq. (\ref{eqn:large_imaginary_behavior_prelimit}).

\noindent
{\bf Step 2. Asymptotics of $g(\imagunit K; \bq )$. } 

Recall the definition of $\Xi$ as given in Eq. (\ref{eqn:log_determinant_variation}). Define
\begin{equation}
\begin{aligned}
\Xi_1(z_1, z_2; \bq) \equiv&~~ \log[(s_2 z_1 + 1)(t_2 z_2 + 1) - \ob_1^2 (1 + p)^2 z_1 z_2] - \ob_\star^2 z_1 z_2 + s_1 z_1 +  t_1 z_2, \\
\Xi_2(\xi, z_1, z_2) \equiv&~~  - \psi_1 \log (z_1 / \psi_1) - \psi_2 \log (z_2 / \psi_2)  - \xi (z_1 + z_2) - \psi_1 - \psi_2.
\end{aligned}
\end{equation}
Then we have 
\begin{equation}\label{eqn:large_imaginary_behavior_G_0}
\Xi(\xi, z_1, z_2; \bq) = \Xi_1(z_1, z_2; \bq) + \Xi_2(\xi, z_1, z_2). 
\end{equation}

It is easy to see that for any fixed $\bq$, we have 
\[
\lim_{z_1, z_2 \to 0} \Xi_1(z_1, z_2, \bq) = 0. 
\]
By Lemma \ref{lem:asymptotics_m_large_xi}, we have $\lim_{K \to \infty} m_1(\imagunit K) = 0$ and $\lim_{K \to \infty} m_2(\imagunit K) = 0$ (for notational simplicity, here and below, we suppressed the argument $\bq$ in $m_1$ and $m_2$), which gives
\begin{equation}\label{eqn:large_imaginary_behavior_G_1}
\lim_{K \to \infty} \Xi_1(m_1(\imagunit K), m_2(\imagunit K), \bq) = 0. 
\end{equation}

Moreover, we have 
\[
\begin{aligned}
&\vert \Xi_2(\imagunit K, m_1(\imagunit K), m_2(\imagunit K)) - \Xi_2(\imagunit K, \imagunit \psi_1 / K, \imagunit \psi_2 / K) \vert\\
\le&~ \psi_1 \vert \log(- \imagunit K m_1(\imagunit K) / \psi_1) \vert + \psi_2 \vert \log(- \imagunit K m_2(\imagunit K) / \psi_2) \vert + \vert \imagunit K m_1(\imagunit K) + \psi_1 \vert + \vert \imagunit K m_2(\imagunit K) + \psi_2 \vert. 
\end{aligned}
\]
By Lemma \ref{lem:asymptotics_m_large_xi} again, we have
\[
\lim_{K \to \infty} \vert \imagunit K m_1(\imagunit K) + \psi_1 \vert = \lim_{K \to \infty} \vert \imagunit K m_2(\imagunit K) + \psi_2 \vert = 0,
\]
and hence
\begin{equation}\label{eqn:large_imaginary_behavior_G_2}
\begin{aligned}
&\lim_{K \to \infty} \vert \Xi_2(\imagunit K, m_1(\imagunit K), m_2(\imagunit K)) - \Xi_2(\imagunit K, \imagunit \psi_1 / K, \imagunit \psi_2 / K) \vert = 0. 
\end{aligned}
\end{equation}
Combining Eq. (\ref{eqn:large_imaginary_behavior_G_0}), (\ref{eqn:large_imaginary_behavior_G_1}) and (\ref{eqn:large_imaginary_behavior_G_2}), we get
\[
\lim_{K \to \infty}\Big\vert \Xi(\xi, m_1(\imagunit K), m_2(\imagunit K); \bq) - \Xi_2(\imagunit K, \imagunit \psi_1 / K, \imagunit \psi_2 / K) \Big\vert = 0. 
\]
Finally, by the definition of $g$ as in Eq. (\ref{eqn:formula_g}) and noting that we have $\Xi_2(\imagunit K, \imagunit \psi_1 / K, \imagunit \psi_2 / K) = (\psi_1 + \psi_2) \Log (-\imagunit K)$, this proves the lemma. 
\end{proof}

Next, we give some uniform upper bounds on the difference of derivatives of $G_d$ and $g$.

\begin{lemma}\label{lem:bounded_derivatives_G}
Follow the notations and settings of Proposition \ref{prop:expression_for_log_determinant}. For fixed $u \in \R_+$, we have
\[
\begin{aligned}
\limsup_{d \to \infty} \sup_{\bq \in \R^5} \E \| \nabla_\bq G_d(\imagunit u; \bq) - \nabla_\bq g(\imagunit u; \bq)\|_2 <& \infty,\\
\limsup_{d \to \infty} \sup_{\bq \in \R^5} \E \| \nabla_\bq^2 G_d(\imagunit u; \bq) - \nabla_\bq^2 g(\imagunit u; \bq) \|_{\op} <& \infty,\\
\limsup_{d \to \infty} \sup_{\bq \in \R^5} \E \| \nabla_{\bq}^3 G_d(\imagunit u; \bq) - \nabla_\bq^3 g(\imagunit u; \bq)\|_{\op}  <& \infty.
\end{aligned}
\]
\end{lemma}

\begin{proof}[Proof of Lemma \ref{lem:bounded_derivatives_G}]

Define $\bq = (s_1, s_2, t_1, t_2, p) \equiv (q_1, q_2, q_3, q_4, q_5)$, and 
\[
\bS_1 = \begin{bmatrix}\id_N & \bzero\\ \bzero& \bzero \end{bmatrix},  ~~~ \bS_2 = \begin{bmatrix}\bQ& \bzero\\ \bzero& \bzero \end{bmatrix}, ~~~ \bS_3 = \begin{bmatrix} \bzero & \bzero \\ \bzero & \id_n \end{bmatrix}, ~~~ \bS_4 = \begin{bmatrix} \bzero & \bzero\\ \bzero& \bH \end{bmatrix}, ~~~ \bS_5 = \begin{bmatrix}\bzero &  \bZ_1^\sT\\ \bZ_1& \bzero\end{bmatrix}.
\]
Then by the bound on the operator norm of Wishart matrix \cite{Guionnet}, for any fixed $k \in \N$, we have
\[
\begin{aligned}
\limsup_{d \to \infty}\sup_{i \in [5]}\E[\| \bS_i \|_{\op}^{2k}] < \infty. 
\end{aligned}
\]
Moreover, define $\bR = (\bA - \imagunit u \id_M)^{-1}$. Then we have almost surely $\sup_{\bq} \| \bR \|_{\op} \le 1/u$. 

Therefore
\[
\begin{aligned}
\sup_{\bq}\E \vert \partial_{q_i} G_d(\imagunit u; \bq) \vert =&~ \sup_{\bq} \frac{1}{d} \E \vert \Tr( \bR \bS_i ) \vert  \le \sup_{\bq} \frac{1}{u} \E[\|\bS_i \|_{\op}] = O_d(1), \\
\sup_{\bq}\E \vert \partial_{q_i, q_j}^2 G_d(\imagunit u; \bq) \vert =&~ \sup_{\bq} \frac{1}{d} \E \vert \Tr( \bR \bS_i \bR \bS_j ) \vert  \le \sup_{\bq} \frac{1}{u^2} (\E[\|\bS_i \|_{\op}^2] \E\| \bS_j \|_{\op}^2])^{1/2} = O_d(1), \\
\sup_{\bq}\E \vert \partial_{q_i, q_j, q_l}^3 G_d(\imagunit u; \bq) \vert =&~ \sup_{\bq} \frac{1}{d} \Big[ \E \vert \Tr( \bR \bS_i \bR \bS_j \bR \bS_l) \vert   + \E \vert \Tr( \bR \bS_i \bR \bS_l \bR \bS_j) \vert \Big] \\
 \le&~ 2 \sup_{\bq} \frac{1}{u^3} \Big[ \E[\|\bS_i \|_{\op}^4] \E[\|\bS_j \|_{\op}^4] \E[ \|\bS_l \|_{\op}^4]\Big]^{1/4} = O_d(1). \\
\end{aligned}
\]
Similarly we can show that for fixed $u > 0$, we have $\sup_{\bq \in \R^5}  \| \nabla_\bq^j g(\imagunit u; \bq) \| < \infty$ for $j = 1, 2, 3$. The lemma holds by the following inequality, 
\[
\limsup_{d \to \infty} \sup_{\bq \in \R^5} \E \| \nabla_{\bq}^j G_d(\imagunit u; \bq) - \nabla_\bq^j g(\imagunit u; \bq)\| \le \limsup_{d \to \infty} \sup_{\bq \in \R^5} \Big[ \E \|\nabla_{\bq}^j G_d(\imagunit u; \bq) \| + \| \nabla_\bq^j g(\imagunit u; \bq) \| \Big] < \infty
\]
for $j = 1, 2, 3$. 
\end{proof}

Finally, we show that the derivatives of a function in a region can be upper bounded by the function value and the second derivatives of the function in the region.
\begin{lemma}\label{lem:Taylor_bound}
Let $f \in C^2([a, b])$. Then we have 
\[
\sup_{x \in [a, b]} \vert f'(x) \vert \le \Big \vert \frac{f(a) - f(b)}{a - b} \Big\vert + \frac{1}{2} \sup_
{x\in [a, b]} \vert f''(x) \vert \cdot \vert a - b\vert. 
\]
As a consequence, letting $f \in C^2(\ball^d(\bzero, 2 r))$ where $\ball^d(\bzero, r) = \{\bx \in \R^d: \| \bx \|_2 \le r \}$, we have 
\[
\sup_{\bx \in \ball(\bzero, r)} \| \nabla f(\bx)\|_{2} \le   r^{-1}\sup_{\bx \in \ball(\bzero, 2 r)} \vert f(\bx) \vert +  2 r \sup_{\bx \in \ball(\bzero, 2 r)} \| \nabla^2 f(\bx) \|_{\op}. 
\]
\end{lemma}
The proof of Lemma \ref{lem:Taylor_bound} is elementary and simply follows from Taylor expansion. 

%

\subsection{Proof of Proposition \ref{prop:expression_for_log_determinant}}\label{subsec:prove_proposition_log_determinant}

By the expression of $\Xi$ in Eq. (\ref{eqn:log_determinant_variation}), we have 
\[
\begin{aligned}
\partial_{z_1} \Xi(\xi, z_1, z_2; \bq) =&~ \frac{s_2(t_2 z_2 + 1) - \ob_1^2 (1 + p)^2 z_2}{(s_2 z_1 + 1)(t_2 z_2 + 1) - \ob_1^2 (1 + p)^2 z_1 z_2} - \ob_\star^2 z_2 + s_1 - \psi_1/ z_1  - \xi,\\
\partial_{z_2} \Xi(\xi, z_1, z_2; \bq) =&~ \frac{t_2(s_2 z_1 + 1) - \ob_1^2 (1 + p)^2 z_1}{(s_2 z_1 + 1)(t_2 z_2 + 1) - \ob_1^2 (1 + p)^2 z_1 z_2} - \ob_\star^2 z_1 + s_2 - \psi_2/ z_2  - \xi. 
\end{aligned}
\]
By fixed point equation (\ref{eq:FixedPoint}) with $\sFone, \sFtwo$ defined in (\ref{eq:Fdef}), we obtain that 
\[
\nabla_{(z_1, z_2)}\Xi(\xi, z_1, z_2; \bq) \vert_{(z_1, z_2) = (m_1(\xi; \bq), m_2(\xi; \bq))} = \bzero. 
\]
As a result, by the definition of $g$ given in Eq. (\ref{eqn:formula_g}), and by formula for implicit differentiation, we have
\[
\frac{\de \phantom{\xi}}{\de \xi}g(\xi; \bq) = - m(\xi; \bq). 
\]
Hence, for any $\xi \in \C_+$ and $K \in \R$ and compact continuous path $\phi(\xi, \imagunit K)$ that connects $\xi$ and $\imagunit K$, we have
\begin{align}\label{eqn:G_and_M_0}
g(\xi; \bq) - g(\imagunit K; \bq) = \int_{\phi(\xi, \imagunit K)} m(\eta; \bq) \de \eta.
\end{align}
By Proposition \ref{prop:G_and_M}, for any $\xi \in \C_+$ and $K \in \R$, we have
\begin{align}\label{eqn:G_and_M_1}
G_d(\xi; \bq) - G_d(\imagunit K; \bq ) =  \int_{\phi(\xi, \imagunit K)} M_{d}(\eta; \bq) \de \eta. 
\end{align}
Combining Eq. (\ref{eqn:G_and_M_1}) with Eq. (\ref{eqn:G_and_M_0}), we get
\begin{align}\label{eqn:G_and_M_4}
\E[\vert G_d(\xi; \bq)- g(\xi; \bq) \vert] \le \int_{\phi(\xi, \imagunit K)} \E \vert M_{d}(\eta; \bq) - m(\eta; \bq) \vert \de \eta  + \E \vert G_d(\imagunit K; \bq) - g(\imagunit K; \bq) \vert. 
\end{align}

By Proposition \ref{prop:Stieltjes}, we have 
\begin{align}\label{eqn:G_and_M_2}
\lim_{d \to \infty} \int_{\phi(\xi, \imagunit K)} \E \vert M_{d}( \eta; \bq) - m( \eta; \bq)\vert \de \eta = 0. 
\end{align}
By Lemma \ref{lem:large_imaginary_behavior_G}, we have 
\begin{align}\label{eqn:G_and_M_3}
\lim_{K \to \infty} \sup_{d \ge d_0}\E \vert G_d(\imagunit K; \bq) - g(\imagunit K; \bq) \vert = 0. 
\end{align}
Combining Eq. (\ref{eqn:G_and_M_4}), (\ref{eqn:G_and_M_2}) and (\ref{eqn:G_and_M_3}), we get Eq. (\ref{eqn:expression_for_log_determinant}). 

For fixed $\xi \in \C_+$, define $E_d(\bq) = G_d(\xi, \bq) - g(\xi; \bq)$. By Lemma \ref{lem:Taylor_bound}, we have
\begin{equation}\label{eqn:bound_function_second_derivative}
\sup_{\bq \in \ball(\bzero, \eps)} \| \nabla E_d(\bq)\|_{2} \le \eps^{-1}\sup_{\bq \in \ball(\bzero, 2 \eps)} \vert E_d(\bq) \vert + 2 \eps \sup_{\bq \in \ball(\bzero, 2 \eps)} \| \nabla^2 E_d(\bq) \|_{\op}. 
\end{equation}
By Eq. (\ref{eqn:expression_for_log_determinant}) and Lemma \ref{lem:bounded_derivatives_G}, by the covering number argument (similar to the Step 3 in Section \ref{subsec:proof_Stieltjes}), we get that for any compact region $\cQ_\star$, there is
\[
\lim_{d \to \infty} \E\Big[\sup_{\bq \in \cQ_\star} \vert E_d(\bq) \vert \Big] = 0. 
\]
Taking $\cQ_\star = \ball(\bzero, 2 \eps)$, by Eq. (\ref{eqn:bound_function_second_derivative}) and by Lemma \ref{lem:bounded_derivatives_G} again, there exists some constant $C$, such that
\[
\limsup_{d \to \infty} \E\Big[ \sup_{\bq \in \ball(\bzero, \eps)} \| \nabla E_d(\bq)\|_{2}\Big] \le C \eps. 
\]
Sending $\eps \to 0$ gives Eq. (\ref{eqn:convergence_of_derivatives}). By similar argument we get Eq. (\ref{eqn:convergence_of_second_derivatives}). This finishes the proof of Proposition \ref{prop:expression_for_log_determinant}.

\section{Proof of Theorem \ref{thm:ridgeless_limit}, \ref{thm:overparametrized_limit}, and \ref{thm:large_sample_limit}}\label{sec:simplification}

\subsection{Proof of Theorem \ref{thm:ridgeless_limit}}

To prove this theorem, we just need to show that 
\[
\begin{aligned}
\lim_{\olambda \to 0} \cuB(\ratio, \psi_1, \psi_2, \olambda) =&~ \cuB_{\rless}(\ratio, \psi_1, \psi_2), \\
\lim_{\olambda \to 0} \cuV(\ratio, \psi_1, \psi_2, \olambda) =&~ \cuV_{\rless}(\ratio, \psi_1, \psi_2). \\
\end{aligned}
\]
More specifically, we just need to show that, the formula for $\chi$ defined in Eq. (\ref{eqn:definition_chi_main_formula}) as $\olambda \to 0$ coincides with the formula for $\chi$ defined in Eq. (\ref{eq:chidef}). By the relationship of $\chi$ and $m_1 m_2$ as per Eq. (\ref{eqn:chi_m0_relationship}), we just need to show the lemma below. 


\begin{lemma}\label{lem:limiting_behavior_stieltjes}
Let Assumption \ref{ass:activation} and \ref{ass:linear} hold. For fixed $\xi \in \C_+$, let $m_1(\xi; \psi_1, \psi_2)$ and $m_2(\xi; \psi_1, \psi_2)$ be defined by
\begin{equation}\label{eqn:another_def_m1m2}
\begin{aligned}
m_1(\xi; \psi_1, \psi_2) =&~ \lim_{d \to \infty, N/d \to \psi_1, n/d \to \psi_2 }\frac{1}{d} \E\{ \Tr_{[1, N]}[([\bzero, \bZ^\sT; \bZ, \bzero]- \xi \id_M)^{-1}]\},\\
m_2(\xi; \psi_1, \psi_2) =&~ \lim_{d \to \infty, N/d \to \psi_1, n/d \to \psi_2 }\frac{1}{d} \E\{ \Tr_{[N+1, M]}[([\bzero, \bZ^\sT; \bZ, \bzero]- \xi \id_M)^{-1}]\}. \\
\end{aligned}
\end{equation}
By Proposition \ref{prop:Stieltjes} this is equivalently saying $m_1(\xi; \psi_1, \psi_2), m_2(\xi; \psi_1, \psi_2)$ is the analytic continuation of solution of Eq. (\ref{eq:FixedPoint}) as defined in Proposition \ref{prop:Stieltjes}, when $\bq = \bzero$. Defining $\psi = \min(\psi_1, \psi_2)$, we have 
\begin{equation}\label{lem:limiting_behavior_formula}
\lim_{u \to 0} [m_1(\imagunit u; \psi_1, \psi_2) m_2(\imagunit u; \psi_1, \psi_2)] = - \frac{[(\psi \ratio^2 - \ratio^2 - 1)^2 + 4\ratio^2 \psi]^{1/2} + (\psi \ratio^2 - \ratio^2 - 1) }{2 \ob_\star^2 \ratio^2}. 
\end{equation}
\end{lemma}

\begin{proof}[Proof of Lemma \ref{lem:limiting_behavior_stieltjes}]~ 
In the following, we consider the case $\psi_2 > \psi_1$. The proof for the case $\psi_2 < \psi_1$ is the same, and the case $\psi_1 = \psi_2$ is simpler. By Proposition \ref{prop:Stieltjes}, $m_1 = m_1(\imagunit u) = m_1(\imagunit u; \psi_1, \psi_2)$ and $m_2 = m_2(\imagunit u) = m_2(\imagunit u; \psi_1, \psi_2)$ must satisfy Eq. (\ref{eq:FixedPoint}) for $\xi = \imagunit u$ and $\bq = \bzero$. A reformulation for Eq. (\ref{eq:FixedPoint}) for $\bq = \bzero$ yields
\begin{align}
\frac{ - \ob_1^2 m_1 m_2}{1 - \ob_1^2 m_1 m_2} - \ob_\star^2 m_1 m_2 - \psi_1 - \imagunit u \cdot m_1 =&~ 0, \label{eqn:equation_for_m1_m2_one}\\
\frac{ - \ob_1^2 m_1 m_2}{1 - \ob_1^2 m_1 m_2} - \ob_\star^2 m_1 m_2 - \psi_2 - \imagunit u \cdot m_2 =&~ 0. \label{eqn:equation_for_m1_m2_two}
\end{align}
Defining $m_0(\imagunit u) = m_1(\imagunit u) m_2(\imagunit u)$. Then $m_0$ must satisfy the following equation
\[
- u^2 m_0 = \Big( \frac{ - \ob_1^2 m_0}{1 - \ob_1^2 m_0} - \ob_\star^2 m_0 - \psi_1 \Big) \Big( \frac{ - \ob_1^2 m_0}{1 - \ob_1^2 m_0} - \ob_\star^2 m_0 - \psi_2 \Big). 
\]
Note we must have $\vert m_0(\imagunit u) \vert \le \vert m_1(\imagunit u) \vert \cdot \vert m_2(\imagunit u) \vert \le \psi_1 \psi_2 /u^2$, and hence $\vert u^2 m_0 \vert = O_{u}(1)$ (as $u \to 0$). This implies that 
\[
\frac{ - \ob_1^2 m_0}{1 - \ob_1^2 m_0} - \ob_\star^2 m_0  = O_u(1),
\]
and hence $m_0 = O_u(1)$. Taking the difference between Eq. (\ref{eqn:equation_for_m1_m2_one}) and (\ref{eqn:equation_for_m1_m2_two}), we get
\begin{equation}\label{eqn:m2_m1_difference}
m_2 - m_1 = -(\psi_2 - \psi_1)/(\imagunit u). 
\end{equation}
This implies one of $m_1$ and $m_2$ should be of order $1/u$ and the other one should be of order $u$, as $u \to 0$. 

By definition of $m_1$ and $m_2$ in Eq. (\ref{eqn:another_def_m1m2}), we have
\[
\begin{aligned}
m_1(\imagunit u) =&~ \imagunit u \lim_{d \to \infty , N/d \to \psi_1, n/d \to \psi_2 }\frac{1}{d} \E\{ \Tr [(\bZ^\sT \bZ + u^2 \id_N)^{-1}]\},\\
m_2(\imagunit u) =&~ \imagunit u \lim_{d \to \infty , N/d \to \psi_1, n/d \to \psi_2 }\frac{1}{d} \E\{ \Tr [(\bZ \bZ^\sT + u^2 \id_N)^{-1}]\}. 
\end{aligned}
\]
When $\psi_2 > \psi_1$ (i.e., $n > N$), $(\bZ \bZ^\sT + u^2 \id_N)$ has $(n - N)$ number of eigenvalues that are $u^2$, and therefore we must have $m_2(\imagunit u) = \Omega_u(1/u)$. Hence $m_1(\imagunit u) = O_u(u)$. Moreover, when $u > 0$, $m_1(\imagunit u)$ and $m_2(\imagunit u)$ are purely imaginary and $\Im m_1(\imagunit u), \Im m_2(\imagunit u) > 0$. This implies that $m_0(\imagunit u)$ must be a real number which is non-positive. 


By Eq. (\ref{eqn:equation_for_m1_m2_one}) and $\lim_{u \to 0} \imagunit u \cdot m_1(\imagunit u) = 0$, all the accumulation points of $m_1(\imagunit u) m_2(\imagunit u)$ as $u \to 0$ should satisfy the quadratic equation
\[
\frac{ - \ob_1^2 m_\star}{1 - \ob_1^2 m_\star} - \ob_\star^2 m_\star - \psi_1 = 0. 
\]
Note that the above equation has only one non-positive solution, and $m_0(\imagunit u)$ for any $u > 0$ must be non-positive. Therefore $\lim_{u \to 0} m_1(\imagunit u) m_2(\imagunit u)$ must exists and be the non-positive solution of the above quadratic equation. The right hand side of Eq. (\ref{lem:limiting_behavior_formula}) gives the non-positive solution of the above quadratic equation. 
\end{proof}

\subsection{Proof of Theorem \ref{thm:overparametrized_limit}}

To prove this theorem, we just need to show that 
\[
\begin{aligned}
\lim_{\psi_1 \to \infty} \cuB(\ratio, \psi_1, \psi_2, \olambda) =&~ \cuB_{\wide}(\ratio,  \psi_2, \olambda), \\
\lim_{\psi_1 \to \infty} \cuV(\ratio, \psi_1, \psi_2, \olambda) =&~ \cuV_{\wide}(\ratio, \psi_2, \olambda). \\
\end{aligned}
\]
This follows by simple calculus and a lemma below. 

\begin{lemma}\label{lem:limiting_behavior_stieltjes_infinite_width}
Under the same condition of Lemma \ref{lem:limiting_behavior_stieltjes}, we have
\[
\begin{aligned}
&\lim_{\psi_1 \to \infty} [m_1(\imagunit (\psi_1 \psi_2 \ob_\star^2 \olambda)^{1/2}; \psi_1, \psi_2) m_2(\imagunit (\psi_1 \psi_2 \ob_\star^2 \olambda)^{1/2}; \psi_1, \psi_2)] \\
=&~ - \frac{[(\psi_2 \ratio^2 - \ratio^2 - (\olambda \psi_2 + 1))^2 + 4\ratio^2 \psi_2 (\olambda \psi_2 + 1)]^{1/2} + (\psi_2 \ratio^2 - \ratio^2 - (\olambda \psi_2 + 1)) }{2 \ob_\star^2 \ratio^2 (\olambda \psi_2 + 1)}. 
\end{aligned}
\]
\end{lemma}
The proof of this lemma is similar to the proof of Lemma \ref{lem:limiting_behavior_stieltjes}.

\subsection{Proof of Theorem \ref{thm:large_sample_limit}}

To prove this theorem, we just need to show that 
\[
\begin{aligned}
\lim_{\psi_2 \to \infty} \cuB(\ratio, \psi_1, \psi_2, \olambda) =&~ \cuB_{\lsamp}(\ratio,  \psi_1, \olambda), \\
\lim_{\psi_2 \to \infty} \cuV(\ratio, \psi_1, \psi_2, \olambda) =&~ 0. \\
\end{aligned}
\]
This follows by simple calculus and a lemma below (this lemma is symmetric to Lemma \ref{lem:limiting_behavior_stieltjes_infinite_width}). 

\begin{lemma}\label{lem:limiting_behavior_stieltjes_infinite_sample}
Under the same condition of Lemma \ref{lem:limiting_behavior_stieltjes}, we have
\[
\begin{aligned}
&\lim_{\psi_2 \to \infty} [m_1(\imagunit (\psi_1 \psi_2 \ob_\star^2 \olambda)^{1/2}; \psi_1, \psi_2) m_2(\imagunit (\psi_1 \psi_2 \ob_\star^2 \olambda)^{1/2}; \psi_1, \psi_2)] \\
=&~ - \frac{[(\psi_1 \ratio^2 - \ratio^2 - (\olambda \psi_1 + 1))^2 + 4\ratio^2 \psi_1 (\olambda \psi_1 + 1)]^{1/2} + (\psi_1 \ratio^2 - \ratio^2 - (\olambda \psi_1 + 1)) }{2 \ob_\star^2 \ratio^2 (\olambda \psi_1 + 1)}. 
\end{aligned}
\]
\end{lemma}

\section{Proof of Proposition \ref{prop:limiting_behavior_ridgeless} and \ref{prop:limiting_behavior_wide}}

\subsection{Proof of Proposition \ref{prop:limiting_behavior_ridgeless}}\label{sec:proof_limiting_behavior_ridgeless}

\noindent
{\bf Proof of Point (1). } When $\psi_1 \to 0$, we have $\chi = O(\psi_1)$, so that $\cuE_{1, \rless}  = - \psi_1 \psi_2 + O(\psi_1^2)$, $\cuE_{2, \rless} = O(\psi_1^2)$ and $\cuE_{0, \rless} = - \psi_1 \psi_2 + O(\psi_1^2)$. This proves Point (1).

\noindent
{\bf Proof of Point (2). } When $\psi_1 = \psi_2$, substituting the expression for $\chi$ into $\cuE_{0, \rless}$, we can see that $\cuE_{0, \rless}(\ratio, \psi_2, \psi_2) = 0$. We also see that $\cuE_{1, \rless}(\ratio, \psi_2, \psi_2) \neq 0$ and $\cuE_{2, \rless}(\ratio, \psi_2, \psi_2) \neq 0$. This proves Point (2).

\noindent
{\bf Proof of Point (3). } When $\psi_1 > \psi_2$, we have 
\[
\begin{aligned}
\lim_{\psi_1 \to \infty} \cuE_{0, \rless}(\ratio, \psi_1, \psi_2) / \psi_1 =&~ ( \psi_2 - 1 )\chi^3\ratio^6 +(1 - 3 \psi_2) \chi^2 \ratio^4  + 3 \psi_2 \chi \ratio^2 - \psi_2, \\
\lim_{\psi_1 \to \infty} \cuE_{1, \rless}(\ratio, \psi_1, \psi_2) / \psi_1 =&~ \psi_2 \chi \ratio^2 - \psi_2, \\
\lim_{\psi_1 \to \infty} \cuE_{2, \rless}(\ratio, \psi_1, \psi_2) / \psi_1 =&~ \chi^3 \ratio^6 - \chi^2 \ratio^4.\\
\end{aligned}
\]
This proves Point (3). 

\noindent
{\bf Proof of Point (4). } For $\psi_1 > \psi_2$, taking derivative of $\cuB_{\rless}$ and $\cuV_{\rless}$ with respect to $\psi_1$, we have 
\[
\begin{aligned}
\partial_{\psi_1} \cuB_{\rless}(\ratio, \psi_1, \psi_2) =&~ (\partial_{\psi_1} \cuE_{1, \rless} \cdot \cuE_{0, \rless} - \partial_{\psi_1} \cuE_{0, \rless} \cdot \cuE_{1, \rless}) / \cuE_{0, \rless}^2, \\
\partial_{\psi_1} \cuV_{\rless}(\ratio, \psi_1, \psi_2) =&~ (\partial_{\psi_1} \cuE_{2, \rless} \cdot \cuE_{0, \rless} - \partial_{\psi_1} \cuE_{0, \rless} \cdot \cuE_{2, \rless}) / \cuE_{0, \rless}^2. 
\end{aligned}
\]
It is easy to check that when $\psi_1 > \psi_2$, the functions $\partial_{\psi_1} \cuE_{1, \rless} \cdot \cuE_{0, \rless} - \partial_{\psi_1} \cuE_{0, \rless} \cdot \cuE_{1, \rless}$ and $\partial_{\psi_1} \cuE_{2, \rless} \cdot \cuE_{0, \rless} - \partial_{\psi_1} \cuE_{0, \rless} \cdot \cuE_{2, \rless}$ are functions of $\ratio$ and $\psi_2$, and are independent of $\psi_1$ (note when $\psi_1 > \psi_2$, $\chi$ is a function of $\psi_2$ and doesn't depend on $\psi_1$). 
Therefore, $\cuB_{\rless}(\ratio, \cdot, \psi_2)$ and $\cuV_{\rless}(\ratio, \cdot, \psi_2)$ as functions of $\psi_1$ must be strictly increasing, strictly decreasing, or staying constant on the interval $\psi_1 \in (\psi_2, \infty)$. However, we know $\cuB_{\rless}(\ratio,\psi_2, \psi_2) = \cuV_{\rless}(\ratio,\psi_2, \psi_2) = \infty$, and $\cuB_{\rless}(\ratio, \infty, \psi_2)$ and $\cuV_{\rless}(\ratio, \infty, \psi_2)$ are finite. Therefore, we must have that $\cuB_{\rless}$ and $\cuV_{\rless}$ are strictly decreasing on $\psi_1 \in (\psi_2, \infty)$.

%

\subsection{Proof of Proposition \ref{prop:limiting_behavior_wide}}\label{sec:proof_limiting_behavior_wide}

In Proposition \ref{prop:limiting_behavior_wide_restate} given by the following, we give a more precise description of the behavior of $\cuR_{\wide}$, which is stronger than Proposition \ref{prop:limiting_behavior_wide}. 

%

\begin{proposition}\label{prop:limiting_behavior_wide_restate}
Denote 
\[
\begin{aligned}
\overline \cuR_{\wide}(u, \rho, \psi_2) =&~ \frac{\psi_2 \rho + u^2}{(1 + \rho)(\psi_2 - 2 u \psi_2 + u^2 \psi_2 - u^2)},\\
\omega(\olambda, \ratio, \psi_2) =&~ - \frac{[( \psi_2\ratio^2 - \ratio^2 -\olambda\psi_2 -1)^2 + 4\psi_2\ratio^2(\olambda\psi_2 + 1)]^{1/2} +  (\psi_2\ratio^2 - \ratio^2 -\olambda\psi_2 -1)}{2(\olambda\psi_2 + 1)},\\
\omega_0(\ratio, \psi_2) =&~ - \frac{[( \psi_2\ratio^2 - \ratio^2 -1)^2 + 4\psi_2\ratio^2]^{1/2} +  (\psi_2\ratio^2 - \ratio^2  -1)}{2},\\
\omega_1(\rho, \psi_2) =&~ - \frac{(\psi_2 \rho - \rho - 1) + [(\psi_2 \rho  - \rho - 1)^2 + 4  \psi_2 \rho  ]^{1/2}}{2},\\
\rho_\star(\ratio, \psi_2) =&~ \frac{\omega_0^2 - \omega_0}{(1  - \psi_2) \omega_0 + \psi_2}, \\
\ratio_\star^2(\rho, \psi_2) =&~ \frac{\omega_1^2   -  \omega_1}{\omega_1 - \psi_2 \omega_1  + \psi_2},\\ 
\olambda_\star(\ratio, \psi_2, \rho) =&~ \frac{\ratio^2  \psi_2 -  \ratio^2  \omega_1 \psi_2 + \ratio^2 \omega_1 + \omega_1 - \omega_1^2  }{(\omega_1^2 - \omega_1) \psi_2}. \\
\end{aligned}
\]

Fix $\ratio, \psi_2 \in (0, \infty)$ and $\rho \in (0, \infty)$. Then the function $\olambda\mapsto \cuR_{\wide}(\rho,\ratio, \psi_2,\olambda)$ is either strictly increasing in $\olambda$, or strictly decreasing first and then strictly increasing. 

Moreover, For any $\rho < \rho_\star(\ratio, \psi_2)$, we have 
\[
\begin{aligned}
\argmin_{\olambda \ge 0} \cuR_{\wide}(\rho, \ratio, \olambda, \psi_2) =&~ 0,\\
\min_{\olambda \ge 0} \cuR_{\wide}(\rho, \ratio, \olambda, \psi_2) =&~  \overline \cuR_{\wide}(\omega_0(\ratio, \psi_2), \rho, \psi_2). \\
\end{aligned}
\]
For any $\rho \ge \rho_\star(\ratio, \psi_2)$, we have 
\[
\begin{aligned}
\argmin_{\olambda \ge 0} \cuR_{\wide}(\rho, \ratio, \olambda, \psi_2) =&~ \olambda_\star(\ratio, \psi_2, \rho),\\
\min_{\olambda \ge 0} \cuR_{\wide}(\rho, \ratio, \olambda, \psi_2) =&~ \overline \cuR_{\wide}(\omega_1(\rho, \psi_2), \rho, \psi_2).\\
\end{aligned}
\]
Minimizing over $\olambda$ and $\ratio$, we have
\[
\min_{\ratio, \olambda \ge 0} \cuR_{\wide}(\rho, \ratio, \olambda, \psi_2) =  \overline \cuR_{\wide}(\omega_1(\rho, \psi_2), \rho, \psi_2). \\
\]
The minimizer is achieved for any $\ratio^2 \ge \ratio_\star^2(\rho, \psi_2)$, and $\olambda = \olambda_\star(\ratio, \psi_2, \rho)$.
\end{proposition}

In the following, we prove Proposition \ref{prop:limiting_behavior_wide_restate}. It is easy to see that 
\[
\cuR_{\wide}(\rho, \ratio, \olambda, \psi_2) = \overline \cuR_{\wide}(\omega(\olambda, \ratio, \psi_2), \rho, \psi_2). 
\]
Hence we study the properties of $\overline \cuR_{\wide}$ first. 

\noindent
{\bf Step 1. Properties of the function $\overline \cuR_{\wide}$. } Calculating the derivative of $\overline \cuR_{\wide}$ with respect to $u$, we have
\[
\partial_u \overline \cuR_{\wide}(u, \rho, \psi_2) = - 2 \psi_2 [u^2 + (\psi_2 \rho  - \rho - 1) u - \psi_2 \rho] / [(1+\rho)(\psi_2 - 2 u \psi_2 + u^2 \psi_2 - u^2)^2].
\]
Note the equation 
\[
u^2 + (\psi_2 \rho  - \rho - 1) u - \psi_2 \rho = 0
\]
has one negative and one positive solution, and $\omega_1$ is the negative solution of the above equation. Therefore, when $u \le \omega_1$, $\overline \cuR_{\wide}$ will be strictly decreasing in $u$; when $0 \ge u \ge \omega_1$, $\overline \cuR_{\wide}$ will be strictly increasing in $u$. Therefore, we have 
\[
\argmin_{u \in (-\infty, 0]} \overline \cuR_{\wide}(u, \rho, \psi_2) = \omega_1(\rho, \psi_2). 
\]

\noindent
{\bf Step 2. Properties of the function $\cuR_{\wide}$. } For fixed $(\ratio, \rho, \psi_2)$, we look at the minimizer over $\olambda$ of the function $\cuR_{\wide}(\rho, \ratio, \olambda, \psi_2) = \overline \cuR_{\wide}(\omega(\olambda, \ratio, \psi_2), \rho, \psi_2)$. The minimum $\min_{\olambda \ge 0}\cuR_{\wide}(\rho, \ratio, \olambda, \psi_2)$ could be different from the minimum  $\min_{u \in (- \infty, 0]} \overline \cuR_{\wide}(u, \rho, \psi_2)$, since $\argmin_{u \in (-\infty, 0]} \overline \cuR_{\wide}(u, \rho, \psi_2) = \omega_1(\rho, \psi_2)$ may not be achievable by $\omega(\olambda, \ratio, \psi_2)$ when $\olambda \ge 0$. 

One observation is that $\omega(\cdot, \psi_2, \ratio)$ as a function of $\olambda$ is always negative and increasing. 
\begin{lemma}\label{lem:nu_kappa_derivative_non_negative}
Let 
\[
\omega(\olambda, \psi_2, \ratio) = - \frac{[( \psi_2\ratio^2 - \ratio^2 -\olambda\psi_2 -1)^2 + 4\psi_2\ratio^2(\olambda\psi_2 + 1)]^{1/2} +  (\psi_2\ratio^2 - \ratio^2 -\olambda\psi_2 -1)}{2(\olambda\psi_2 + 1)}. 
\]
Then for any $\psi_2 \in (0, \infty)$, $\ratio \in (0, \infty)$ and $\olambda > 0$, we have 
\[
\begin{aligned}
\omega(\olambda, \psi_2, \ratio) <& 0,\\
\partial_{\olambda} \omega(\olambda, \psi_2, \ratio) >& 0. 
\end{aligned}
\]
\end{lemma}
Let us for now admit this lemma holds. When $\rho$ is such that $\omega_1 > \omega_0$ (i.e. $\rho < \rho_\star(\ratio, \psi_2)$), we can choose $\olambda = \olambda_\star(\ratio, \psi_2, \rho) > 0$ such that $\omega(\olambda, \ratio, \psi_2) = \omega(\olambda_\star, \ratio, \psi_2) = \omega_1(\rho, \psi_2)$, and then $\cuR_{\wide}(\rho, \ratio, \olambda_\star(\ratio, \psi_2, \rho), \psi_2)  = \overline \cuR_{\wide}(\omega_1(\rho, \psi_2), \rho, \psi_2)$ gives the minimum of $\cuR_{\wide}$ optimizing over $\olambda \in [0, \infty)$. When $\rho$ is such that $\omega_1 < \omega_0$ (i.e. $\rho > \rho_\star(\ratio, \psi_2)$), there is not a $\olambda$ such that $\omega(\olambda, \ratio, \psi_2) = \omega_1(\rho, \psi_2)$ holds. Therefore, the best we can do is to take $\olambda = 0$, and then $\cuR_{\wide}(\rho, \ratio, 0, \psi_2) = \overline \cuR_{\wide}(\omega_0(\rho, \psi_2), \rho, \psi_2)$ gives the minimum of $\cuR_{\wide}$ optimizing over $\olambda \in [0, \infty)$. 

Finally, when we minimize $\cuR_{\wide}(\rho, \ratio, \olambda, \psi_2)$  jointly over $\ratio$ and $\olambda$, note that as long as $\ratio^2 \ge \ratio_\star^2$, we can choose $\olambda = \olambda_\star(\ratio, \psi_2, \rho) > 0$ such that $\omega(\olambda, \ratio, \psi_2) =\omega(\olambda_\star, \ratio, \psi_2) = \omega_1(\rho, \psi_2)$, and then $\cuR_{\wide}(\rho, \ratio, \olambda_\star(\ratio, \psi_2, \rho), \psi_2)  = \overline \cuR_{\wide}(\omega_1(\rho, \psi_2), \rho, \psi_2)$ gives the minimum of $\cuR_{\wide}$ optimizing over $\olambda \in [0, \infty)$ and $\ratio \in (0, \infty)$. This proves Proposition \ref{prop:limiting_behavior_wide_restate}.

%

In the following, we prove Lemma \ref{lem:nu_kappa_derivative_non_negative}. 

\begin{proof}[Proof of Lemma \ref{lem:nu_kappa_derivative_non_negative}] 
It is easy to see that $\omega(\olambda, \psi_2, \ratio) < 0$. In the following, we show $\partial_{\olambda} \omega(\olambda, \psi_2, \ratio) > 0$. 

\noindent
{\bf Step 1. When $\psi_2 \ge 1$. } We have 
\[
\omega = - \frac{[( \psi_2\ratio^2 - \ratio^2 -\olambda\psi_2 -1)^2 + 4\psi_2\ratio^2(\olambda\psi_2 + 1)]^{1/2} +  (\psi_2\ratio^2 - \ratio^2 -\olambda\psi_2 -1)}{2(\olambda\psi_2 + 1)}.
\]
Then we have
\[
\begin{aligned}
\partial_{\olambda} \omega  = \frac{(\psi_2 - 1)[(\olambda\psi_2 - \psi_2\ratio^2 + \ratio^2 + 1)^2 + 4  \psi_2  \ratio^2(\olambda\psi_2 + 1)]^{1/2} + (\olambda\psi_2^2 + \olambda\psi_2 + (\psi_2 - 1)^2 \ratio^2+ \psi_2  + 1)}{2 \psi_2^2  \olambda [\olambda^2\psi_2^2(\olambda\psi_2 - \psi_2\ratio^2 + \ratio^2 + 1)^2 + 4\olambda^2\psi_2^3\ratio^2(\olambda\psi_2 + 1)]^{1/2}(\olambda\psi_2 + 1)^2}
\end{aligned}
\]
It is easy to see that, when $\olambda>0$ and $\psi_2 > 1$, both the denominator and numerator is positive, so that $\partial_{\olambda} \omega > 0$.

\noindent
{\bf Step 2. When $\psi_2 < 1$. } Note $\omega$ is the negative solution of the quadratic equation
\[
(\olambda\psi_2 + 1) \omega^2 + (\psi_2\ratio^2 - \ratio^2 -\olambda\psi_2 -1) \omega - \psi_2\ratio^2 = 0.
\]
Differentiating the quadratic equation with respect to $\olambda$, we have 
\[
\psi_2 \omega^2 + 2 (\olambda\psi_2 + 1) \omega \partial_{\olambda} \omega- \psi_2 \omega + (\psi_2\ratio^2 - \ratio^2 -\olambda\psi_2 -1) \partial_{\olambda} \omega = 0,
\]
which gives
\[
\partial_{\olambda} \omega = (\psi_2 \omega - \psi_2 \omega^2) / [2 (\olambda\psi_2 + 1) \omega + \psi_2\ratio^2 - \ratio^2 -\olambda\psi_2 -1] = (\psi_2 \omega - \psi_2 \omega^2) /  [ (\olambda\psi_2 + 1) (2 \omega - 1) + (\psi_2-1)\ratio^2].  
\]
We can see that, since $\omega < 0$, when $\psi_2 < 1$, both the denominator and numerator is negative. This proves $\partial_{\olambda} \omega > 0$ when $\psi_2 < 1$. 
\end{proof}

\section*{Acknowledgements}

This work was partially supported by grants NSF CCF-1714305, IIS-1741162, and ONR
N00014-18-1-2729.

\bibliographystyle{amsalpha}
\bibliography{paper_sm_v22.bbl}

\newcommand{\etalchar}[1]{$^{#1}$}
\providecommand{\bysame}{\leavevmode\hbox to3em{\hrulefill}\thinspace}
\providecommand{\MR}{\relax\ifhmode\unskip\space\fi MR }
\providecommand{\MRhref}[2]{%
  \href{http://www.ams.org/mathscinet-getitem?mr=#1}{#2}
}
\providecommand{\href}[2]{#2}
\begin{thebibliography}{GMMM19}

\bibitem[ADH{\etalchar{+}}19]{arora2019fine}
Sanjeev Arora, Simon~S Du, Wei Hu, Zhiyuan Li, and Ruosong Wang,
  \emph{Fine-grained analysis of optimization and generalization for
  overparameterized two-layer neural networks}, arXiv:1901.08584 (2019).

\bibitem[AGZ09]{Guionnet}
Greg~W. Anderson, Alice Guionnet, and Ofer Zeitouni, \emph{{An introduction to
  random matrices}}, Cambridge University Press, 2009.

\bibitem[AM15]{alaoui2015fast}
Ahmed Alaoui and Michael~W Mahoney, \emph{Fast randomized kernel ridge
  regression with statistical guarantees}, Advances in Neural Information
  Processing Systems, 2015, pp.~775--783.

\bibitem[AOY19]{araujo2019mean}
Dyego Ara{\'u}jo, Roberto~I Oliveira, and Daniel Yukimura, \emph{A mean-field
  limit for certain deep neural networks}, arXiv:1906.00193 (2019).

\bibitem[AS17]{advani2017high}
Madhu~S Advani and Andrew~M Saxe, \emph{High-dimensional dynamics of
  generalization error in neural networks}, arXiv:1710.03667 (2017).

\bibitem[AZLL18]{allen2018learning}
Zeyuan Allen-Zhu, Yuanzhi Li, and Yingyu Liang, \emph{Learning and
  generalization in overparameterized neural networks, going beyond two
  layers}, arXiv:1811.04918 (2018).

\bibitem[AZLS18]{allen2018convergence}
Zeyuan Allen-Zhu, Yuanzhi Li, and Zhao Song, \emph{A convergence theory for
  deep learning via over-parameterization}, arXiv:1811.03962 (2018).

\bibitem[Bac13]{bach2013sharp}
Francis Bach, \emph{Sharp analysis of low-rank kernel matrix approximations},
  Conference on Learning Theory, 2013, pp.~185--209.

\bibitem[Bac17a]{bach2017breaking}
\bysame, \emph{Breaking the curse of dimensionality with convex neural
  networks}, The Journal of Machine Learning Research \textbf{18} (2017),
  no.~1, 629--681.

\bibitem[Bac17b]{bach2017equivalence}
\bysame, \emph{On the equivalence between kernel quadrature rules and random
  feature expansions}, The Journal of Machine Learning Research \textbf{18}
  (2017), no.~1, 714--751.

\bibitem[BHMM19]{belkin2018reconciling}
Mikhail Belkin, Daniel Hsu, Siyuan Ma, and Soumik Mandal, \emph{Reconciling
  modern machine learning and the bias-variance trade-off}, Proceedings of the
  National Academy of Sciences \textbf{116} (2019), 15849--15854.

\bibitem[BHX19]{belkin2019two}
Mikhail Belkin, Daniel Hsu, and Ji~Xu, \emph{Two models of double descent for
  weak features}, arXiv:1903.07571, 2019.

\bibitem[BLLT20]{bartlett2019benign}
Peter~L Bartlett, Philip~M Long, G{\'a}bor Lugosi, and Alexander Tsigler,
  \emph{Benign overfitting in linear regression}, Proceedings of the National
  Academy of Sciences (2020).

\bibitem[BMM18]{belkin2018understand}
Mikhail Belkin, Siyuan Ma, and Soumik Mandal, \emph{To understand deep learning
  we need to understand kernel learning}, arXiv:1802.01396, 2018.

\bibitem[BRT19]{belkin2018does}
Mikhail Belkin, Alexander Rakhlin, and Alexandre~B Tsybakov, \emph{Does data
  interpolation contradict statistical optimality?}, The 22nd International
  Conference on Artificial Intelligence and Statistics, 2019, pp.~1611--1619.

\bibitem[BS10]{bai2010spectral}
Zhidong Bai and Jack~W Silverstein, \emph{Spectral analysis of large
  dimensional random matrices}, vol.~20, Springer, 2010.

\bibitem[CB18a]{chizat2018note}
Lenaic Chizat and Francis Bach, \emph{A note on lazy training in supervised
  differentiable programming}, arXiv:1812.07956 (2018).

\bibitem[CB18b]{chizat2018global}
\bysame, \emph{On the global convergence of gradient descent for
  over-parameterized models using optimal transport}, Advances in neural
  information processing systems, 2018, pp.~3036--3046.

\bibitem[Chi11]{chihara2011introduction}
Theodore~S Chihara, \emph{An introduction to orthogonal polynomials}, Courier
  Corporation, 2011.

\bibitem[CS13]{cheng2013spectrum}
Xiuyuan Cheng and Amit Singer, \emph{The spectrum of random inner-product
  kernel matrices}, Random Matrices: Theory and Applications \textbf{2} (2013),
  no.~04, 1350010.

\bibitem[Dan17]{daniely2017sgd}
Amit Daniely, \emph{Sgd learns the conjugate kernel class of the network},
  Advances in Neural Information Processing Systems, 2017, pp.~2422--2430.

\bibitem[DFS16]{daniely2016toward}
Amit Daniely, Roy Frostig, and Yoram Singer, \emph{Toward deeper understanding
  of neural networks: The power of initialization and a dual view on
  expressivity}, Advances In Neural Information Processing Systems, 2016,
  pp.~2253--2261.

\bibitem[DHM89]{devore1989optimal}
Ronald~A DeVore, Ralph Howard, and Charles Micchelli, \emph{Optimal nonlinear
  approximation}, Manuscripta mathematica \textbf{63} (1989), no.~4, 469--478.

\bibitem[DJ89]{donoho1989projection}
David~L Donoho and Iain~M Johnstone, \emph{Projection-based approximation and a
  duality with kernel methods}, The Annals of Statistics (1989), 58--106.

\bibitem[DL19]{dou2019training}
Xialiang Dou and Tengyuan Liang, \emph{Training neural networks as learning
  data-adaptive kernels: Provable representation and approximation benefits},
  arXiv:1901.07114 (2019).

\bibitem[DLL{\etalchar{+}}18]{du2018gradientb}
Simon~S Du, Jason~D Lee, Haochuan Li, Liwei Wang, and Xiyu Zhai, \emph{Gradient
  descent finds global minima of deep neural networks}, arXiv:1811.03804
  (2018).

\bibitem[DM16]{deshpande2016sparse}
Yash Deshpande and Andrea Montanari, \emph{Sparse pca via covariance
  thresholding}, Journal of Machine Learning Research \textbf{17} (2016),
  1--41.

\bibitem[DZPS18]{du2018gradient}
Simon~S Du, Xiyu Zhai, Barnabas Poczos, and Aarti Singh, \emph{Gradient descent
  provably optimizes over-parameterized neural networks}, arXiv:1810.02054
  (2018).

\bibitem[EF14]{costas2014spherical}
Costas Efthimiou and Christopher Frye, \emph{Spherical harmonics in p
  dimensions}, World Scientific, 2014.

\bibitem[EK10]{el2010spectrum}
Noureddine El~Karoui, \emph{The spectrum of kernel random matrices}, The Annals
  of Statistics \textbf{38} (2010), no.~1, 1--50.

\bibitem[FM19]{fan2019spectral}
Zhou Fan and Andrea Montanari, \emph{The spectral norm of random inner-product
  kernel matrices}, Probability Theory and Related Fields \textbf{173} (2019),
  no.~1-2, 27--85.

\bibitem[GAAR18]{garriga2018deep}
Adri{\`a} Garriga-Alonso, Laurence Aitchison, and Carl~Edward Rasmussen,
  \emph{Deep convolutional networks as shallow gaussian processes},
  arXiv:1808.05587 (2018).

\bibitem[GJS{\etalchar{+}}19]{geiger2019scaling}
Mario Geiger, Arthur Jacot, Stefano Spigler, Franck Gabriel, Levent Sagun,
  St{\'e}phane d'Ascoli, Giulio Biroli, Cl{\'e}ment Hongler, and Matthieu
  Wyart, \emph{Scaling description of generalization with number of parameters
  in deep learning}, arXiv:1901.01608 (2019).

\bibitem[GMMM19]{ghorbani2019linearized}
Behrooz Ghorbani, Song Mei, Theodor Misiakiewicz, and Andrea Montanari,
  \emph{Linearized two-layers neural networks in high dimension},
  arXiv:1904.12191 (2019).

\bibitem[HJ15]{hazan2015steps}
Tamir Hazan and Tommi Jaakkola, \emph{Steps toward deep kernel methods from
  infinite neural networks}, arXiv:1508.05133 (2015).

\bibitem[HMRT19]{hastie2019surprises}
Trevor Hastie, Andrea Montanari, Saharon Rosset, and Ryan~J Tibshirani,
  \emph{Surprises in high-dimensional ridgeless least squares interpolation},
  arXiv:1903.08560 (2019).

\bibitem[HMS18]{helton2018applications}
J~William Helton, Tobias Mai, and Roland Speicher, \emph{Applications of
  realizations (aka linearizations) to free probability}, Journal of Functional
  Analysis \textbf{274} (2018), no.~1, 1--79.

\bibitem[HTF09]{Hastie}
Trevor Hastie, Robert Tibshirani, and Jerome Friedman, \emph{{The elements of
  statistical learning}}, Springer, 2009.

\bibitem[JGH18]{jacot2018neural}
Arthur Jacot, Franck Gabriel, and Cl{\'e}ment Hongler, \emph{Neural tangent
  kernel: Convergence and generalization in neural networks}, Advances in
  neural information processing systems, 2018, pp.~8571--8580.

\bibitem[JMM19]{javanmard2019analysis}
Adel Javanmard, Marco Mondelli, and Andrea Montanari, \emph{Analysis of a
  two-layer neural network via displacement convexity}, arXiv:1901.01375
  (2019).

\bibitem[KLS18]{kobak2018implicit}
Dmitry Kobak, Jonathan Lomond, and Benoit Sanchez, \emph{{Implicit ridge
  regularization provided by the minimum-norm least squares estimator when $
  n\ll p$}}, arXiv:1805.10939 (2018).

\bibitem[LBN{\etalchar{+}}17]{lee2017deep}
Jaehoon Lee, Yasaman Bahri, Roman Novak, Samuel~S Schoenholz, Jeffrey
  Pennington, and Jascha Sohl-Dickstein, \emph{Deep neural networks as gaussian
  processes}, arXiv:1711.00165 (2017).

\bibitem[LL18]{li2018learning}
Yuanzhi Li and Yingyu Liang, \emph{Learning overparameterized neural networks
  via stochastic gradient descent on structured data}, Advances in Neural
  Information Processing Systems, 2018, pp.~8157--8166.

\bibitem[LLC18]{louart2018random}
Cosme Louart, Zhenyu Liao, and Romain Couillet, \emph{A random matrix approach
  to neural networks}, The Annals of Applied Probability \textbf{28} (2018),
  no.~2, 1190--1248.

\bibitem[LR18]{liang2018just}
Tengyuan Liang and Alexander Rakhlin, \emph{Just interpolate: Kernel"
  ridgeless" regression can generalize}, arXiv:1808.00387 (2018).

\bibitem[MMM19]{mei2019mean}
Song Mei, Theodor Misiakiewicz, and Andrea Montanari, \emph{Mean-field theory
  of two-layers neural networks: dimension-free bounds and kernel limit},
  arXiv:1902.06015 (2019).

\bibitem[MMN18]{mei2018mean}
Song Mei, Andrea Montanari, and Phan-Minh Nguyen, \emph{A mean field view of
  the landscape of two-layer neural networks}, Proceedings of the National
  Academy of Sciences \textbf{115} (2018), no.~33, E7665--E7671.

\bibitem[MP11]{marinucci2011random}
Domenico Marinucci and Giovanni Peccati, \emph{Random fields on the sphere:
  representation, limit theorems and cosmological applications}, vol. 389,
  Cambridge University Press, 2011.

\bibitem[MRH{\etalchar{+}}18]{matthews2018gaussian}
Alexander G de~G Matthews, Mark Rowland, Jiri Hron, Richard~E Turner, and
  Zoubin Ghahramani, \emph{Gaussian process behaviour in wide deep neural
  networks}, arXiv:1804.11271 (2018).

\bibitem[MRSY19]{montanari2019maxmargin}
Andrea Montanari, Feng Ruan, Youngtak Sohn, and Jun Yan, \emph{The
  generalization error of max-margin linear classifiers: High-dimensional
  asymptotics in the overparametrized regime}, arXiv preprint arXiv:1911.01544
  (2019).

\bibitem[MVS19]{muthukumar2019harmless}
Vidya Muthukumar, Kailas Vodrahalli, and Anant Sahai, \emph{Harmless
  interpolation of noisy data in regression}, arXiv:1903.09139 (2019).

\bibitem[Nea96]{neal1996priors}
Radford~M Neal, \emph{Priors for infinite networks}, Bayesian Learning for
  Neural Networks, Springer, 1996, pp.~29--53.

\bibitem[Ngu19]{nguyen2019mean}
Phan-Minh Nguyen, \emph{Mean field limit of the learning dynamics of multilayer
  neural networks}, arXiv:1902.02880 (2019).

\bibitem[NXB{\etalchar{+}}18]{novak2018bayesian}
Roman Novak, Lechao Xiao, Yasaman Bahri, Jaehoon Lee, Greg Yang, Jiri Hron,
  Daniel~A Abolafia, Jeffrey Pennington, and Jascha Sohl-Dickstein,
  \emph{Bayesian deep convolutional networks with many channels are gaussian
  processes}.

\bibitem[OS19]{oymak2019towards}
Samet Oymak and Mahdi Soltanolkotabi, \emph{Towards moderate
  overparameterization: global convergence guarantees for training shallow
  neural networks}, arXiv:1902.04674 (2019).

\bibitem[PW17]{pennington2017nonlinear}
Jeffrey Pennington and Pratik Worah, \emph{Nonlinear random matrix theory for
  deep learning}, Advances in Neural Information Processing Systems, 2017,
  pp.~2637--2646.

\bibitem[RJBVE19]{rotskoff2019neuron}
Grant Rotskoff, Samy Jelassi, Joan Bruna, and Eric Vanden-Eijnden, \emph{Neuron
  birth-death dynamics accelerates gradient descent and converges
  asymptotically}, International Conference on Machine Learning, 2019,
  pp.~5508--5517.

\bibitem[RR08]{rahimi2008random}
Ali Rahimi and Benjamin Recht, \emph{Random features for large-scale kernel
  machines}, Advances in neural information processing systems, 2008,
  pp.~1177--1184.

\bibitem[RR17]{rudi2017generalization}
Alessandro Rudi and Lorenzo Rosasco, \emph{Generalization properties of
  learning with random features}, Advances in Neural Information Processing
  Systems, 2017, pp.~3215--3225.

\bibitem[RVE18]{rotskoff2018neural}
Grant~M Rotskoff and Eric Vanden-Eijnden, \emph{Neural networks as interacting
  particle systems: Asymptotic convexity of the loss landscape and universal
  scaling of the approximation error}, arXiv:1805.00915 (2018).

\bibitem[RZ18]{rakhlin2018consistency}
Alexander Rakhlin and Xiyu Zhai, \emph{Consistency of interpolation with
  laplace kernels is a high-dimensional phenomenon}, arXiv:1812.11167 (2018).

\bibitem[SS19]{sirignano2019mean}
Justin Sirignano and Konstantinos Spiliopoulos, \emph{Mean field analysis of
  neural networks: A central limit theorem}, Stochastic Processes and their
  Applications (2019).

\bibitem[{Sze}39]{szego1939orthogonal}
{Szeg\H{o}, Gabor}, \emph{Orthogonal polynomials}, vol.~23, American
  Mathematical Soc., 1939.

\bibitem[VW18]{vempala2018gradient}
Santosh Vempala and John Wilmes, \emph{Gradient descent for one-hidden-layer
  neural networks: Polynomial convergence and sq lower bounds},
  arXiv:1805.02677 (2018).

\bibitem[Wil97]{williams1997computing}
Christopher~KI Williams, \emph{Computing with infinite networks}, Advances in
  neural information processing systems, 1997, pp.~295--301.

\bibitem[YRC07]{yao2007early}
Yuan Yao, Lorenzo Rosasco, and Andrea Caponnetto, \emph{On early stopping in
  gradient descent learning}, Constructive Approximation \textbf{26} (2007),
  no.~2, 289--315.

\bibitem[ZBH{\etalchar{+}}16]{zhang2016understanding}
Chiyuan Zhang, Samy Bengio, Moritz Hardt, Benjamin Recht, and Oriol Vinyals,
  \emph{Understanding deep learning requires rethinking generalization},
  arXiv:1611.03530 (2016).

\bibitem[ZCZG20]{zou2018stochastic}
Difan Zou, Yuan Cao, Dongruo Zhou, and Quanquan Gu, \emph{Stochastic gradient
  descent optimizes over-parameterized deep relu networks}, Machine Learning
  \textbf{109} (2020), 467--492.

\end{thebibliography}

\clearpage

\appendix

\section{Technical background}
\label{sec:Background}

In this section we introduce the technical background which will be useful for the proofs in the next sections.
In particular, we will use decompositions in (hyper-)spherical harmonics on the  $\S^{d-1}(\sqrt{d})$ and in orthogonal polynomials
on the real line. All of the properties listed below are classical: we will however prove a few facts that are slightly less standard. 
We refer the reader to \cite{costas2014spherical,szego1939orthogonal,chihara2011introduction,ghorbani2019linearized} for further information on these topics. Expansions in spherical harmonics have been used in the past in the statistics literature, for instance in \cite{donoho1989projection,bach2017breaking}.

\subsection{Functional spaces over the sphere}

For $d \ge 1$, we let $\S^{d-1}(r) = \{\bx \in \R^{d}: \| \bx \|_2 = r\}$ denote the sphere with radius $r$ in $\reals^d$.
We will mostly work with the sphere of radius $\sqrt d$, $\S^{d-1}(\sqrt{d})$ and will denote by $\gamma_d$  the uniform probability measure on $\S^{d-1}(\sqrt d)$. 
All functions in the following are assumed to be elements of $ L^2(\S^{d-1}(\sqrt d) ,\gamma_d)$, with scalar product and norm denoted as $\<\,\cdot\,,\,\cdot\,\>_{L^2}$
and $\|\,\cdot\,\|_{L^2}$:
\begin{align}
\<f,g\>_{L^2} \equiv \int_{\S^{d-1}(\sqrt d)} f(\bx) \, g(\bx)\, \gamma_d(\de \bx)\,.
\end{align}

For $\ell\in\integers_{\ge 0}$, let $\tilde{V}_{d,\ell}$ be the space of homogeneous harmonic polynomials of degree $\ell$ on $\reals^d$ (i.e. homogeneous
polynomials $q(\bx)$ satisfying $\Delta q(\bx) = 0$), and denote by $V_{d,\ell}$ the linear space of functions obtained by restricting the polynomials in $\tilde{V}_{d,\ell}$
to $\S^{d-1}(\sqrt d)$. With these definitions, we have the following orthogonal decomposition
\begin{align}
L^2(\S^{d-1}(\sqrt d) ,\gamma_d) = \bigoplus_{\ell=0}^{\infty} V_{d,\ell}\, . \label{eq:SpinDecomposition}
\end{align}
The dimension of each subspace is given by
\begin{align}
\dim(V_{d,\ell}) = B(d, \ell) = \frac{2 \ell + d - 2}{\ell} { \ell + d - 3 \choose \ell - 1} \, .
\end{align}
For each $\ell\in \integers_{\ge 0}$, the spherical harmonics $\{ Y_{\ell j}^{(d)}\}_{1\le j \le B(d, \ell)}$ form an orthonormal basis of $V_{d,\ell}$:
\[
\<Y^{(d)}_{ki}, Y^{(d)}_{sj}\>_{L^2} = \delta_{ij} \delta_{ks}.
\]
Note that our convention is different from the more standard one, that defines the spherical harmonics as functions on $\S^{d-1}(1)$.
It is immediate to pass from one convention to the other by a simple scaling. We will drop the superscript $d$ and write $Y_{\ell, j} = Y_{\ell, j}^{(d)}$ whenever clear from the context.

We denote by $\proj_k$  the orthogonal projections to $V_{d,k}$ in $L^2(\S^{d-1}(\sqrt d),\gamma_d)$. This can be written in terms of spherical harmonics as
\begin{align}
\proj_k f(\bx) \equiv&~ \sum_{l=1}^{B(d, k)} \< f, Y_{kl}\>_{L^2} Y_{kl}(\bx). 
\end{align}
Then for a function $f \in L^2(\S^{d-1}(\sqrt d))$, we have 
\[
f(\bx) = \sum_{k = 0}^\infty \proj_k f(\bx) = \sum_{k = 0 }^\infty \sum_{l = 1}^{B(d, k)} \< f, Y_{kl}\>_{L^2} Y_{kl}(\bx). 
\]

\subsection{Gegenbauer polynomials}
\label{sec:Gegenbauer}

The $\ell$-th Gegenbauer polynomial $Q_\ell^{(d)}$ is a polynomial of degree $\ell$. Consistently
with our convention for spherical harmonics, we view $Q_\ell^{(d)}$ as a function $Q_{\ell}^{(d)}: [-d,d]\to \reals$. The set $\{ Q_\ell^{(d)}\}_{\ell\ge 0}$ forms an orthogonal basis on $L^2([-d,d], \tilde \tau_d)$ (where $\tilde \tau_d$ is the distribution of $\<\bx_1, \bx_2\>$ when $\bx_1, \bx_2 \sim_{i.i.d.} \Unif(\S^{d-1}(\sqrt d))$), satisfying the normalization condition:
\begin{align}
\< Q^{(d)}_k, Q^{(d)}_j \>_{L^2(\tilde \tau_d)} = \frac{1}{B(d,k)}\, \delta_{jk} \, .  \label{eq:GegenbauerNormalization}
\end{align}
In particular, these polynomials are normalized so that  $Q_\ell^{(d)}(d) = 1$. 
As above, we will omit the superscript $d$ when clear from the context (write it as $Q_\ell$ for notation simplicity).

Gegenbauer polynomials are directly related to spherical harmonics as follows. Fix $\bv\in\S^{d-1}(\sqrt{d})$ and 
consider the subspace of  $V_{\ell}$ formed by all functions that are invariant under rotations in $\reals^d$ that keep $\bv$ unchanged.
It is not hard to see that this subspace has dimension one, and coincides with the span of the function $Q_{\ell}^{(d)}(\<\bv,\,\cdot\,\>)$.

We will use the following properties of Gegenbauer polynomials
\begin{enumerate}
\item For $\bx, \by \in \S^{d-1}(\sqrt d)$
\begin{align}
\< Q_j^{(d)}(\< \bx, \cdot\>), Q_k^{(d)}(\< \by, \cdot\>) \>_{L^2(\S^{d-1}(\sqrt d), \gamma_d)} = \frac{1}{B(d,k)}\delta_{jk}  Q_k^{(d)}(\< \bx, \by\>).  \label{eq:ProductGegenbauer}
\end{align}
\item For $\bx, \by \in \S^{d-1}(\sqrt d)$
\begin{align}
Q_k^{(d)}(\< \bx, \by\> ) = \frac{1}{B(d, k)} \sum_{i =1}^{ B(d, k)} Y_{ki}^{(d)}(\bx) Y_{ki}^{(d)}(\by). \label{eq:GegenbauerHarmonics}
\end{align}
\end{enumerate}
Note in particular that property 2 implies that --up to a constant-- $Q_k^{(d)}(\< \bx, \by\> )$ is a representation of the projector onto 
the subspace of degree-$k$ spherical harmonics
\begin{align}
(\proj_k f)(\bx) = B(d,k) \int_{\S^{d-1}(\sqrt{d})} \, Q_k^{(d)}(\< \bx, \by\> )\,  f(\by)\, \gamma_d(\de\by)\, .\label{eq:ProjectorGegenbauer}
\end{align}
For a function $\sigma \in L^2([-\sqrt d, \sqrt d], \tau_d)$ (where $\tau_d$ is the distribution of $\< \bx_1, \bx_2 \> / \sqrt d$ when $\bx_1, \bx_2 \sim_{iid} \Unif(\S^{d-1}(\sqrt d))$), denoting its spherical harmonics coefficients $\lambda_{d, k}(\sigma)$ to be 
\begin{align}\label{eqn:technical_lambda_sigma}
\lambda_{d, k}(\sigma) = \int_{[-\sqrt d , \sqrt d]} \sigma(x) Q_k^{(d)}(\sqrt d x) \tau_d(x),
\end{align}
then we have the following equation holds in $L^2([-\sqrt d, \sqrt d],\tau_d)$ sense
\begin{equation}\label{eqn:sigma_G_decomposition}
\sigma(x) = \sum_{k = 0}^\infty \lambda_{d, k}(\sigma) B(d, k) Q_k^{(d)}(\sqrt d x). 
\end{equation}

\subsection{Hermite polynomials}

The Hermite polynomials $\{\He_k\}_{k\ge 0}$ form an orthogonal basis of $L^2(\reals,\mu_G)$, where $\mu_G(\de x) = e^{-x^2/2}\de x/\sqrt{2\pi}$ 
is the standard Gaussian measure, and $\He_k$ has degree $k$. We will follow the classical normalization (here and below, expectation is with respect to
$G\sim\normal(0,1)$):
\begin{align}
\E\big\{\He_j(G) \,\He_k(G)\big\} = k!\, \delta_{jk}\, .
\end{align}
As a consequence, for any function $\sigma \in L^2(\reals,\mu_G)$, we have the decomposition
\begin{align}\label{eqn:sigma_He_decomposition}
\sigma(x) = \sum_{k=1}^{\infty}\frac{\mu_k(\sigma )}{k!}\, \He_k(x)\, ,\;\;\;\;\;\; \mu_k(\sigma) \equiv \E\big\{\sigma(G)\, \He_k(G)\}\, .
\end{align}

The Hermite polynomials can be obtained as high-dimensional limits of the Gegenbauer polynomials introduced in the previous section. Indeed, 
the Gegenbauer polynomials (up to a $\sqrt d$ scaling in domain) are constructed by Gram-Schmidt orthogonalization of the monomials $\{x^k\}_{k\ge 0}$ with respect to the measure 
$\tau_d$, while Hermite polynomial are obtained by Gram-Schmidt orthogonalization with respect to $\mu_G$. Since $\tau_d\Rightarrow \mu_G$
(here $\Rightarrow$ denotes weak convergence),
it is immediate to show that, for any fixed integer $k$, 
\begin{align}
\lim_{d \to \infty} \Coeff\{ Q_k^{(d)}( \sqrt d x) \, B(d, k)^{1/2} \} = \Coeff\left\{ \frac{1}{(k!)^{1/2}}\,\He_k(x) \right\}\, .\label{eq:Gegen-to-Hermite}
\end{align}
Here and below, for $P$ a polynomial, $\Coeff\{ P(x) \}$ is  the vector of the coefficients of $P$. As a consequence, for any fixed integer $k$, we have 
\begin{align}\label{eqn:relationship_mu_lambda}
\mu_k(\sigma) = \lim_{d \to \infty} \lambda_{d, k}(\sigma) (B(d, k) k!)^{1/2},
\end{align}
where $\mu_k(\sigma)$ and $\lambda_{d, k}(\sigma)$ are given in Eq. (\ref{eqn:sigma_He_decomposition}) and (\ref{eqn:technical_lambda_sigma}).

\clearpage

\section{Proof of Proposition \ref{prop:G_and_M}}\label{sec:proof_prop_G_and_M}

We can see Eq. (\ref{eqn:connection_G_M}) is trivially implied by the definition of $G_d$ and $M_d$ as in Eq. (\ref{eqn:Stieltjes_A}). To prove Eq. (\ref{eqn:Psi_j_and_G_d}), it is enough to prove the following equations: for $u \in \R$, we have
\begin{equation}\label{eqn:connection_G_Psi}
\begin{aligned}
\partial_p G_d(\imagunit u; \bzero) =&~ \frac{2}{d}  \Tr\Big( ( u^2 \id_N + \bZ^\sT  \bZ)^{-1}  \bZ_1^\sT \bZ \Big), \\
\partial_{s_1, t_1}^2 G_d(\imagunit u; \bzero) =&~ - \frac{1}{d} \Tr\Big( ( u^2 \id_N + \bZ^\sT  \bZ)^{-2} \bZ^\sT \bZ  \Big),\\
\partial_{s_1, t_2}^2 G_d(\imagunit u; \bzero) =&~ - \frac{1}{d} \Tr\Big( ( u^2 \id_N + \bZ^\sT  \bZ)^{-2} \bZ^\sT \bH \bZ  \Big),\\
\partial_{s_2, t_1}^2 G_d(\imagunit u; \bzero) =&~ -  \frac{1}{d} \Tr\Big( ( u^2 \id_N + \bZ^\sT  \bZ)^{-1} \bQ ( u^2 \id_N + \bZ^\sT  \bZ)^{-1} \bZ^\sT \bZ  \Big),\\
\partial_{s_2, t_2}^2 G_d(\imagunit u; \bzero) =&~ - \frac{1}{d}  \Tr\Big( ( u^2 \id_N + \bZ^\sT  \bZ)^{-1} \bQ ( u^2 \id_N + \bZ^\sT  \bZ )^{-1}  \bZ^\sT \bH \bZ  \Big). 
\end{aligned}
\end{equation}

Now we prove Eq. (\ref{eqn:connection_G_Psi}). For any fixed $\bq \in \R^5$, $\xi \in \C_+$ and a fixed instance $\bA(\bq)$, the determinant can be represented as $\det(\bA(\bq) - \xi \id_M) = r(\bq, \xi) \exp(\imagunit \theta(\bq, \xi))$ for $\theta(\bq, \xi) \in (-\pi, \pi]$. Without loss of generality, we assume for this fixed $\bq$ and $\xi$, we have $\theta(\bq, \xi) \neq \pi$, and then $\Log(\det(\bA(\bq) - \xi \id_M)) = \log r(\bq, \xi) + \imagunit \theta(\bq, \xi)$ (when $\theta(\bq, \xi) = \pi$, we use another definition of $\Log$ notation, and the proof is the same). For this $\bq$, $\xi$, and $\bA(\bq)$, there exists some integer $k = k(\bq, \xi) \in \N$, such that 
\[
\sum_{i = 1}^M \Log(\lambda_i(\bA(\bq)) - \xi) = \Log \det(\bA(\bq) - \xi \id_M) + 2 \pi \imagunit k(\bq, \xi). 
\]
Moreover, the set of eigenvalues of $\bA(\bq) - \xi \id_M$ and $\det(\bA(\bq) - \xi \id_M)$ are continuous with respect to $\bq$. Therefore, for any perturbation $\Delta \bq$ with $\| \Delta \bq \|_2 \le \eps$ and $\eps$ small enough, we have $k(\bq + \Delta \bq, \xi) = k(\bq, \xi)$. As a result, we have 
\[
\partial_{q_i} \Big[ \sum_{i = 1}^M \Log(\lambda_i(\bA(\bq)) - \xi)\Big] = \partial_{q_i} \Log\Big[ \det(\bA(\bq) - \xi \id_M) \Big] = \Tr\Big[ ( \bA(\bq)  - \xi \id_M) ^{-1} \partial_{q_i} \bA(\bq) \Big]. 
\]
Moreover, $\bA(\bq)$ (defined as in Eq. (\ref{eqn:matrix_A})) is a linear matrix function of $\bq$, which gives $\partial_{q_i, q_j} \bA(\bq) = \bzero$. Hence we have 
\[
\begin{aligned}
&\partial_{q_i, q_j}^2 \Big[ \sum_{i = 1}^M \Log(\lambda_i(\bA(\bq)) - \xi)\Big] \\
=&~ \partial_{q_i, q_j}^2 \Log\Big[ \det(\bA(\bq) - \xi \id_M) \Big] \\
=&~ \partial_{q_j}  \Tr \Big[ ( \bA(\bq)  - \xi \id_M) ^{-1} \partial_{q_i} \bA(\bq) \Big] \\
=&~ -\Tr\Big[ ( \bA(\bq)  - \xi \id_M) ^{-1} \partial_{q_j} \bA(\bq) ( \bA(\bq)  - \xi \id_M) ^{-1} \partial_{q_i} \bA(\bq) \Big]. 
\end{aligned}
\]
Note
\begin{align*}
\partial_{s_1} \bA(\bzero) =&~ \begin{bmatrix}
\id_N & \bzero \\
\bzero & \bzero
\end{bmatrix}, 
 & \partial_{s_2} \bA(\bzero) =&~ \begin{bmatrix}
\bQ & \bzero \\
\bzero & \bzero
\end{bmatrix}, \\
\partial_{t_1} \bA(\bzero) =&~ \begin{bmatrix}
\bzero & \bzero \\
\bzero & \id_n
\end{bmatrix}, 
&\partial_{t_2} \bA(\bzero) =&~ \begin{bmatrix}
\bzero & \bzero \\
\bzero & \bH
\end{bmatrix}, 
&\partial_{p} \bA(\bzero) =&~ \begin{bmatrix}
\bzero & \bZ_1^\sT \\
\bZ_1 & \bzero
\end{bmatrix},
\end{align*}
and using the formula for block matrix inversion, we have
\[
( \bA(\bzero)  - \imagunit u \id_M)^{-1} = \begin{bmatrix}
(- \imagunit u \id_N - \imagunit \bZ^\sT \bZ / u)^{-1} &  (u^2 \id_N + \bZ^\sT \bZ)^{-1}\bZ^\sT  \\
\bZ (u^2 \id_N + \bZ^\sT \bZ )^{-1}  & (- \imagunit u \id_n - \imagunit \bZ \bZ^\sT / u)^{-1}
\end{bmatrix}. 
\]
With simple algebra, we can show that Eq. (\ref{eqn:connection_G_Psi}) holds.

\clearpage

\section{Additional Proofs in Section \ref{sec:decomposition}}\label{sec:additional_decomposition}

\subsection{Proof of Lemma \ref{lem:decomposition1} and Lemma \ref{lem:randomize_beta}}

\begin{proof}[Proof of Lemma \ref{lem:decomposition1}]~ 
We define the sequence $(\normf_{d, k}^2)_{k \ge 2}$ to be the coefficients of Gegenbauer expansion of $\Sigma_d$:
\[
\Sigma_d(x / \sqrt d) = \sum_{k = 2}^\infty \normf_{d, k}^2 Q_k^{(d)}(\sqrt d x). 
\]
In the expansion, the zeroth and first order coefficients are $0$, because, according to  Assumption \ref{ass:ground_truth}, 
\[
\E_{\bx \sim \Unif(\S^{d-1}(\sqrt d))}[\Sigma_d(x_1/\sqrt d)] = \E_{\bx \sim \Unif(\S^{d-1}(\sqrt d))}[\Sigma_d(x_1/\sqrt d) x_1] = 0.
\] 

To check point (1), we have $\Sigma_d(1) =  \sum_{k = 2}^\infty \normf_{d, k}^2 Q_k^{(d)}(d) = \sum_{k = 2}^\infty \normf_{d, k}^2$, and by Assumption \ref{ass:ground_truth} we have $\lim_{d \to \infty} \Sigma_d(1) = \normf_\star^2$, so that (1) holds. 

To check point (2), defining $(\bbeta_{d, k})_{k \ge 2}$ and $g_d^{\sNL}(\bx)$ accordingly, we have
\[
\begin{aligned}
\E_{\bbeta}[g_d^{\sNL}(\bx_1) g_d^{\sNL}(\bx_2)] =&~ \E_\bbeta \Big[ \Big(\sum_{k \ge 2} \sum_{l \in [B(d, k)]}  (\bbeta_{d, k})_l Y_{kl}^{(d)}(\bx_1) \Big) \Big( \sum_{k \ge 2} \sum_{l \in [B(d, k)]}  (\bbeta_{d, k})_l Y_{k l}^{(d)}(\bx_2) \Big) \Big] \\
=&~ \sum_{k \ge 2} \normf_{d, k}^2 Y_{kl}^{(d)}(\bx_1) Y_{k l}^{(d)}(\bx_2) / B(d, k) = \sum_{k \ge 2} \normf_{d, k}^2 Q_k^{(d)}(\< \bx_1, \bx_2\>) = \Sigma_d(\< \bx_1, \bx_2\> / d). 
\end{aligned}
\]
This proves Lemma \ref{lem:decomposition1}
\end{proof}

\begin{proof}[Proof of Lemma \ref{lem:randomize_beta}]~ 
With a little abuse of notations, let us define
\[
\cE(\bbeta_{d, 1}, f_d^{\sNL}, \bX, \bTheta, \beps) \equiv \Big\vert R_\RF(f_d, \bX, \bTheta, \lambda) - \Big[ \normf_1^2 (1 - 2 \Psi_1 + \Psi_2) + (\normf_\star^2 + \tau^2) \Psi_3 + \normf_\star^2 \Big] \Big\vert. 
\]
For any orthogonal matrix $\cO \in \R^{d \times d}$, it is easy to see that, there exists a transformation $\cT_{\cO}$ that acts on $f_d^{\sNL}$ with $f_d^{\sNL} \stackrel{d}{=} \cT_{\cO}[f_d^{\sNL}]$, such that for any fixed $\bbeta_{d, 1}$, $\bX$, $\bTheta$, $\beps$ and $f_d^{\sNL}$, we have
\[
\cE(\bbeta_{d, 1}, f_d^{\sNL}, \bX, \bTheta, \beps) = \cE(\cO \bbeta_{d, 1}, \cT_\cO[f_d^{\sNL}], \bX \cO^\sT, \bTheta \cO^\sT, \beps). 
\]
Moreover, note that $\bX$, $\bTheta$, $\beps$, and $f_d^{\sNL}$ are mutually independent, $\bX \stackrel{d}{=} \bX \cO^\sT$, $\bTheta \stackrel{d}{=} \bTheta \cO^\sT$, and $f_d^{\sNL} \stackrel{d}{=} \cT_{\cO}[f_d^{\sNL}]$. Then, for any fixed $\bbeta_{d, 1}$, we have 
\[
\cE(\cO \bbeta_{d, 1}, \cT_\cO[f_d^{\sNL}], \bX \cO^\sT, \bTheta \cO^\sT, \beps) \stackrel{d}{=} \cE(\cO \bbeta_{d, 1}, f_d^{\sNL}, \bX, \bTheta, \beps)
\]
where the randomness is given by $(\bX, \bTheta, \beps, f_d^{\sNL})$. As a result, for any $\bbeta_{d, 1} \in \S^{d-1}(\normf_{d, 1})$ and $\cO$ orthogonal matrix, we have 
\[
\E_{\bX, \bTheta, \beps, f_d^{\sNL}} [ \cE(\bbeta_{d, 1})] = \E_{\bX, \bTheta, \beps, f_d^{\sNL}} [ \cE(\cO \bbeta_{d, 1})]. 
\]
This immediately proves the lemma. 
\end{proof}

\subsection{Proof of Lemma \ref{lem:constant_term_one} and Lemma \ref{lem:constant_term_two}}\label{subsec:bounds_key_constant_term}

To prove Lemma \ref{lem:constant_term_one} and \ref{lem:constant_term_two}, first we state a lemma that reformulate $A_1, A_2$ and $B_\alpha$ using Sherman-Morrison-Woodbury formula. 

\begin{lemma}[Simplifications using Sherman-Morrison-Woodbury formula]\label{lem:SMW_calculations}
Use the same definitions and assumptions as Proposition \ref{prop:decomposition} and Lemma \ref{lem:reformulation}. For $\bM \in \R^{N \times N}$, define
\begin{align}
L_1 =&~  \frac{1}{\sqrt d} \lambda_{d, 0}(\sigma) \Tr [ \ones_N \ones_n^\sT  \bZ \Res ],\label{eqn:L1_in_SMW_calculations}\\
L_2(\bM) =&~ \frac{1}{d} \Tr [ \Res \bM \Res \bZ^\sT \ones_n \ones_n^\sT \bZ ]. \label{eqn:L2_in_SMW_calculations}
\end{align}
We then have
\begin{align}
L_1 =&~ 1 - \frac{K_{12} + 1}{K_{11} (1 - K_{22}) + (K_{12} + 1)^2}, \label{eqn:L2_in_SMW_calculations_reformulation_1} \\
L_2(\bM) =&~ \psi_2 \frac{G_{11} (1 - K_{22})^2 + G_{22}(K_{12} + 1)^2 + 2 G_{12} (K_{12} + 1) (1 - K_{22})}{(K_{11} (1 - K_{22}) + (K_{12} + 1)^2)^2},\label{eqn:L2_in_SMW_calculations_reformulation_2}
\end{align}
where
\[
\begin{aligned}
K_{11} =&~  \bT_1^\sT \bE_0^{-1} \bT_1, \\
K_{12} =&~  \bT_1^\sT \bE_0^{-1} \bT_2, \\
K_{22} =&~  \bT_2^\sT \bE_0^{-1} \bT_2,\\
G_{11} =&~  \bT_1^\sT \bE_0^{-1} \bM \bE_0^{-1} \bT_1, \\
G_{12} =&~  \bT_1^\sT \bE_0^{-1} \bM \bE_0^{-1} \bT_2, \\
G_{22} =&~  \bT_2^\sT \bE_0^{-1} \bM \bE_0^{-1} \bT_2, \\
\end{aligned}
\]
and
\[
\begin{aligned}
\vphi_d(x) =&~ \sigma(x) -  \lambda_{d, 0}(\sigma), \\
\bJ =&~ \frac{1}{\sqrt d}\vphi_d\Big(\frac{1}{\sqrt d}\bX \bTheta^\sT \Big), \\
\bE_0 =&~ \bJ^\sT \bJ + \psi_1 \psi_2 \lambda \id_N, \\
\bT_1 =&~ \psi_2^{1/2}  \lambda_{d, 0}(\sigma) \ones_N, \\
\bT_2 =&~ \frac{1}{\sqrt n}\bJ^\sT \ones_n.
\end{aligned}
\]
\end{lemma}

\begin{proof}[Proof of Lemma \ref{lem:SMW_calculations}]~

\noindent
{\bf Step 1. Term $L_1$. } 

Note we have (denoting $ \lambda_{d, 0} =  \lambda_{d, 0}(\sigma)$)
\[
\bZ =  \lambda_{d, 0} \ones_n \ones_N^\sT / \sqrt d + \bJ. 
\]
Hence we have (denoting $\bT_2 =  \bJ^\sT \ones_n / \sqrt {n}$)
\[
\begin{aligned}
L_1 =&~ \Tr\Big[ \lambda_{d, 0} \ones_N \ones_n^\sT ( \lambda_{d, 0} \ones_n \ones_N^\sT/ \sqrt d  +\bJ) [( \lambda_{d, 0} \ones_n \ones_N^\sT / \sqrt d + \bJ)^\sT ( \lambda_{d, 0} \ones_n \ones_N^\sT / \sqrt d + \bJ) + \psi_1 \psi_2 \lambda \id_N]^{-1}\Big]/ \sqrt d \\
=&~  \Tr\Big[ (\psi_2  \lambda_{d, 0}^2 \ones_N \ones_N^\sT + \psi_2^{1/2} \lambda_{d, 0} \ones_N \bT_2^\sT)[\psi_2  \lambda_{d, 0}^2 \ones_N^\sT \ones_N^\sT +\psi_2^{1/2}  \lambda_{d, 0} \ones_N \bT_2^\sT + \psi_2^{1/2}  \lambda_{d, 0} \bT_2 \ones_N^\sT + \bJ^\sT \bJ + \psi_1 \psi_2 \lambda \id_N]^{-1}\Big] . \\
\end{aligned}
\]
Define
\[
\begin{aligned}
\bE =&~\bZ^\sT \bZ + \psi_1 \psi_2 \lambda \id_N = \bE_0 + \bF_1 \bF_2^\sT, \\
\bE_0 =&~ \bJ^\sT \bJ + \psi_1 \psi_2 \lambda \id_N, \\
\bF_1 =&~ (\bT_1, \bT_1, \bT_2), \\
\bF_2 =&~ (\bT_1, \bT_2, \bT_1), \\
\bT_1 =&~ \psi_2^{1/2}  \lambda_{d, 0} \ones_N, \\
\bT_2 =&~ \bJ^\sT \ones_n / \sqrt {n}.
\end{aligned}
\]
By the Sherman-Morrison-Woodbury formula, we have
\[
\bE^{-1} = \bE_0^{-1} - \bE_0^{-1} \bF_1 (\id_3 + \bF_2^\sT \bE_0^{-1} \bF_1)^{-1} \bF_2^\sT \bE_0^{-1}. 
\]
Then we have
\[
\begin{aligned}
L_1 =&~  \Tr\Big[ ( \bT_1 \bT_1^\sT + \bT_1 \bT_2^\sT)(\bE_0^{-1} - \bE_0^{-1} \bF_1 (\id_3 + \bF_2^\sT \bE_0^{-1} \bF_1)^{-1} \bF_2^\sT \bE_0^{-1})\Big] \\
=&~  (  \bT_1^\sT \bE_0^{-1} \bT_1 - \bT_1^\sT\bE_0^{-1} \bF_1 (\id_3 + \bF_2^\sT \bE_0^{-1} \bF_1)^{-1} \bF_2^\sT \bE_0^{-1} \bT_1) \\
&+ (  \bT_2^\sT \bE_0^{-1}\bT_1 - \bT_2^\sT \bE_0^{-1} \bF_1 (\id_3 + \bF_2^\sT \bE_0^{-1} \bF_1)^{-1} \bF_2^\sT \bE_0^{-1} \bT_1) \\
=&~  (  K_{11} - [K_{11}, K_{11}, K_{12}] (\id_3 + \bK)^{-1} [K_{11}, K_{12}, K_{11}]^\sT) \\
&+(  K_{12} - [K_{12}, K_{12}, K_{22}] (\id_3 + \bK)^{-1} [K_{11}, K_{12}, K_{11}]^\sT) \\
=&~  [K_{11}, K_{11}, K_{12}] (\id_3 + \bK)^{-1} [1, 0, 0]^\sT \\
&+[K_{12}, K_{12}, K_{22}] (\id_3 + \bK)^{-1} [1, 0, 0]^\sT \\
=&~ (K_{12}^2 + K_{12} + K_{11} - K_{11} K_{22})/(K_{12}^2 + 2 K_{12} + K_{11} - K_{11} K_{22} + 1) \\
=&~ 1 - (K_{12} + 1)/[K_{11} (1 - K_{22}) + (K_{12} + 1)^2], 
\end{aligned}
\]
where
\[
\begin{aligned}
K_{11} =&~  \bT_1^\sT \bE_0^{-1} \bT_1 = \psi_2  \lambda_{d, 0}^2 \ones_N^\sT (\bJ^\sT \bJ + \psi_1 \psi_2 \lambda \id_N)^{-1} \ones_N, \\
K_{12} =&~  \bT_1^\sT \bE_0^{-1} \bT_2 =  \lambda_{d, 0} \ones_N^\sT (\bJ^\sT \bJ + \psi_1 \psi_2 \lambda \id_N)^{-1} \bJ^\sT \ones_n / \sqrt { d}, \\
K_{22} =&~  \bT_2^\sT \bE_0^{-1} \bT_2 = \ones_n^\sT \bJ (\bJ^\sT \bJ + \psi_1 \psi_2 \lambda \id_N)^{-1} \bJ^\sT \ones_n / n,\\
\bK =&~ \begin{bmatrix}
K_{11}& K_{11}& K_{12}\\
K_{12}& K_{12}& K_{22}\\
K_{11}& K_{11}& K_{12}
\end{bmatrix}.
\end{aligned}
\]
This prove Eq. (\ref{eqn:L2_in_SMW_calculations_reformulation_1}). 

\noindent
{\bf Step 2. Term $L_2(\bM)$. } We have
\[
\begin{aligned}
\bZ^\sT \ones_n \ones_n^\sT \bZ / d =&~( \lambda_{d, 0} \ones_n \ones_N^\sT / \sqrt d + \bJ)^\sT\ones_n \ones_n^\sT ( \lambda_{d, 0} \ones_n \ones_N^\sT / \sqrt d + \bJ) / d\\
=&~ \psi_2^2  \lambda_{d, 0}^2 \ones_N \ones_N^\sT + \psi_2 \bT_2 \cdot \sqrt{\psi_2}  \lambda_{d, 0} \ones_N^\sT +  \psi_2  \sqrt{\psi_2}  \lambda_{d, 0} \ones_N \bT_2^\sT + \psi_2 \bT_2 \bT_2^\sT = \psi_2 (\bT_1 + \bT_2) (\bT_1 + \bT_2)^\sT. 
\end{aligned}
\]
As a result, we have
\[
\begin{aligned}
L_2(\bM) =&~ \psi_2 \cdot (\bT_1+ \bT_2)^\sT \bE^{-1} \bM \bE^{-1} (\bT_1+ \bT_2)   \\
=&~ \psi_2 \cdot  (\bT_1+ \bT_2)^\sT (\id_N - \bE_0^{-1} \bF_1 (\id_3 + \bF_2^\sT \bE_0^{-1} \bF_1)^{-1} \bF_2^\sT) \\
& \cdot (\bE_0^{-1} \bM \bE_0^{-1}) (\id_N - \bF_2 (\id_3 + \bF_1^\sT \bE_0^{-1} \bF_2)^{-1} \bF_1^\sT  \bE_0^{-1}) (\bT_1+ \bT_2).  \\
\end{aligned}
\]
Simplifying this formula using simple algebra proves Eq. (\ref{eqn:L2_in_SMW_calculations_reformulation_2}).
\end{proof}

\begin{proof}[Proof of Lemma \ref{lem:constant_term_one}]~ 

\noindent
{\bf Step 1. Term $A_1$. } By Lemma \ref{lem:SMW_calculations}, we get
\begin{align}\label{eqn:A1_SMW}
A_1 =&~ 1 - (K_{12} + 1)/(K_{11} (1 - K_{22}) + (K_{12} + 1)^2), 
\end{align}
where
\[
\begin{aligned}
K_{11} =&~  \bT_1^\sT \bE_0^{-1} \bT_1 = \psi_2  \lambda_{d, 0}^2 \ones_N^\sT (\bJ^\sT \bJ +  \psi_1 \psi_2\lambda \id_N)^{-1} \ones_N, \\
K_{12} =&~  \bT_1^\sT \bE_0^{-1} \bT_2 =  \lambda_{d, 0} \ones_N^\sT (\bJ^\sT \bJ + \psi_1\psi_2 \lambda \id_N)^{-1} \bJ^\sT \ones_n / \sqrt { d}, \\
K_{22} =&~  \bT_2^\sT \bE_0^{-1} \bT_2 = \ones_n^\sT \bJ (\bJ^\sT \bJ + \psi_1\psi_2 \lambda \id_N)^{-1} \bJ^\sT \ones_n / n. 
\end{aligned}
\]

\noindent
{\bf Step 2. Term $A_2$. }

Note that we have 
\[
A_2 =  \Tr( (\bZ^\sT \bZ +   \psi_1\psi_2 \lambda \id_N)^{-1}\bU_0 (\bZ^\sT \bZ + \psi_1\psi_2 \lambda \id_N)^{-1} \bZ^\sT \ones_n \ones_n^\sT \bZ ) / d, 
\]
where
\begin{equation}\label{eqn:constant_term_one_1}
\bU_0 =   \lambda_{d, 0}(\sigma)^2 \ones_N \ones_N^\sT = \bT_1 \bT_1^\sT / \psi_2. 
\end{equation}
By Lemma \ref{lem:SMW_calculations}, we have 
\begin{align}\label{eqn:S20_first_simplifications}
A_2 =&~ \psi_2 [G_{11} (1 - K_{22})^2 + G_{22}(K_{12} + 1)^2 + 2 G_{12} (K_{12} + 1) (1 - K_{22})] /(K_{11} (1 - K_{22}) + (K_{12} + 1)^2)^2,
\end{align}
where
\[
\begin{aligned}
G_{11} =&~  \bT_1^\sT \bE_0^{-1} \bU_0 \bE_0^{-1} \bT_1 = K_{11}^2 / \psi_2, \\
G_{12} =&~  \bT_1^\sT \bE_0^{-1} \bU_0 \bE_0^{-1} \bT_2 = K_{11} K_{12} / \psi_2, \\
G_{22} =&~  \bT_2^\sT \bE_0^{-1} \bU_0 \bE_0^{-1} \bT_2 = K_{12}^2 / \psi_2. \\
\end{aligned}
\]
We can simplify $S_{20}$ in Eq. (\ref{eqn:S20_first_simplifications}) further, and get
\begin{align}\label{eqn:A2_SMW}
A_2 = (K_{11} (1 -  K_{22}) + K_{12}^2 + K_{12})^2 / (K_{11} (1 - K_{22}) + (K_{12} + 1)^2)^2. 
\end{align}

\noindent
{\bf Step 3. Combining $A_1$ and $A_2$}

By Eq. (\ref{eqn:A1_SMW}) and (\ref{eqn:A2_SMW}), we have 
\[
A = 1 - 2 A_1+ A_2 = (K_{12} + 1)^2 / (K_{11} (1 - K_{22}) + (K_{12} + 1)^2)^2 \ge 0.
\]
For term $K_{12}$, we have 
\[
\vert K_{12} \vert \le  \lambda_{d, 0} \| (\bJ^\sT \bJ + \psi_1\psi_2 \lambda \id_N)^{-1} \bJ^\sT \|_{\op} \| \ones_n \ones_N^\sT  / \sqrt { d} \|_{\op} = O_{d}(\sqrt d). 
\]
For term $K_{11}$, we have 
\[
 K_{11}  \ge \psi_2  \lambda_{d, 0}^2 N \lambda_{\min}((\bJ^\sT \bJ + \psi_1\psi_2 \lambda \id_N)^{-1}) = \Omega_d(d) / (\| \bJ^\sT \bJ \|_{\op} + \psi_1 \psi_2 \lambda). 
\]
For term $K_{22}$, we have 
\[
\begin{aligned}
 1 \ge 1 - K_{22}  =&~ \ones_n^\sT (\id_n - \bJ (\bJ^\sT \bJ +  \psi_1 \psi_2 \lambda \id_N)^{-1} \bJ^\sT) \ones_n / n \ge 1 - \lambda_{\max}( \bJ (\bJ^\sT \bJ + \psi_1 \psi_2 \lambda \id_N)^{-1} \bJ^\sT)\\
 \ge&~  \psi_1 \psi_2 \lambda / (\psi_1 \psi_2 \lambda + \| \bJ^\sT \bJ \|_{\op}) > 0. 
 \end{aligned}
\]
As a result, we have 
\[
1/ (K_{11} (1 - K_{22}) + (K_{12} + 1)^2)^2 = O_d(d^{-2}) \cdot (1 + \| \bJ \|_{\op}^8),  
\]
and hence
\[
A = O_d(1/d)\cdot (1 + \| \bJ \|_{\op}^8)
\]
Lemma \ref{lem:concentration_operator_general_sphere} in Section \ref{subsec:preliminary_spherical_random_matrix} provides an upper bound on the operator norm of $\| \bJ \|_{\op}$, which gives $\| \bJ \|_{\op} = O_{d, \P}(\exp\{C (\log d)^{1/2}\})$ (note $\bJ$ can be regarded as a sub-matrix of $\bK$ in Lemma \ref{lem:concentration_operator_general_sphere}, so that $\| \bJ \|_{\op} \le \| \bK \|_{\op}$). Using this bound, we get
\[
A = o_{d, \P}(1). 
\]
It is easy to see that $0 \le A \le 1$. Hence the high probability bound translates to an expectation bound. This proves the lemma. 
\end{proof}

\begin{proof}[Proof of Lemma \ref{lem:constant_term_two}] For notation simplicity, we prove this lemma under the case when $\cA = \{ \alpha \}$ which is a singleton. We denote $B = B_\alpha$. The proof can be directly generalized to the case for arbitrary set $\cA$. 

By Lemma \ref{lem:SMW_calculations} (when applying Lemma \ref{lem:SMW_calculations}, we change the role of $N$ and $n$, and the role of $\bTheta$ and $\bX$; this can be done because the role of $\bTheta$ and $\bX$ is symmetric), we have
\begin{equation}\label{eqn:expression_B_constant_term_two}
\begin{aligned}
B =&~ \psi_2 \frac{G_{11} (1 - K_{22})^2 + G_{22}(K_{12} + 1)^2 + 2 G_{12} (K_{12} + 1) (1 - K_{22})}{(K_{11} (1 - K_{22}) + (K_{12} + 1)^2)^2},
\end{aligned}
\end{equation}
where
\[
\begin{aligned}
K_{11} =&~  \bT_1^\sT \bE_0^{-1} \bT_1 = \psi_2  \lambda_{d, 0}(\sigma)^2 \ones_N^\sT (\bJ^\sT \bJ + \psi_1\psi_2  \lambda \id_N)^{-1} \ones_N, \\
K_{12} =&~  \bT_1^\sT \bE_0^{-1} \bT_2 =  \lambda_{d, 0}(\sigma) \ones_N^\sT (\bJ^\sT \bJ + \psi_1 \psi_2 \lambda \id_N)^{-1} \bJ^\sT \ones_n / \sqrt { d}, \\
K_{22} =&~  \bT_2^\sT \bE_0^{-1} \bT_2 = \ones_n^\sT \bJ (\bJ^\sT \bJ + \psi_1 \psi_2 \lambda \id_N)^{-1} \bJ^\sT \ones_n / n,\\
G_{11} =&~  \bT_1^\sT \bE_0^{-1} \bM \bE_0^{-1} \bT_1 = \psi_2  \lambda_{d, 0}(\sigma)^2 \ones_N^\sT (\bJ^\sT \bJ + \psi_1 \psi_2 \lambda \id_N)^{-1} \bM (\bJ^\sT \bJ + \psi_1 \psi_2 \lambda \id_N)^{-1} \ones_N, \\
G_{12} =&~  \bT_1^\sT \bE_0^{-1} \bM \bE_0^{-1} \bT_2 = \lambda_{d, 0}(\sigma) \ones_N^\sT (\bJ^\sT \bJ + \psi_1 \psi_2 \lambda \id_N)^{-1} \bM (\bJ^\sT \bJ + \psi_1 \psi_2 \lambda \id_N)^{-1} \bJ^\sT \ones_n / \sqrt { d}, \\
G_{22} =&~  \bT_2^\sT \bE_0^{-1} \bM \bE_0^{-1} \bT_2 =  \ones_n^\sT \bJ (\bJ^\sT \bJ + \psi_1 \psi_2 \lambda \id_N)^{-1} \bM (\bJ^\sT \bJ + \psi_1 \psi_2 \lambda \id_N)^{-1} \bJ^\sT \ones_n / n. \\
\end{aligned}
\]

Note we have shown in the proof of Lemma \ref{lem:constant_term_one} that 
\[
\begin{aligned}
K_{11}  =&~ \Omega_d(d) / (\psi_1 \psi_2 \lambda + \| \bJ \|_{\op}^2), \\
K_{12} =&~ O_{d}(\sqrt d),\\ 
1 \ge 1 - K_{22} \ge&~ \psi_1 \psi_2 \lambda / (\psi_1 \psi_2 \lambda + \| \bJ \|_{\op}^2), \\
1/(K_{11} (1 - K_{22}) + (K_{12} + 1)^2)^2 =&~ O_d(d^{-2}) \cdot ( 1 \vee \| \bJ \|_{\op}^8). 
\end{aligned}
\]
Lemma \ref{lem:concentration_operator_general_sphere} provides an upper bound on the operator norm of $\| \bJ \|_{\op}$, which gives $\| \bJ \|_{\op} = O_{d, \P}(\exp\{C (\log d)^{1/2}\})$. Using this bound, we get for any $\eps > 0$
\[
\begin{aligned}
(1 - K_{22})^2 / (K_{11} (1 - K_{22}) + (K_{12} + 1)^2)^2 =&~ O_{d, \P}(d^{-2 + \eps}), \\
(K_{12} + 1)^2 / (K_{11} (1 - K_{22}) + (K_{12} + 1)^2)^2 =&~ O_{d, \P}(d^{-1+ \eps}), \\
\vert (K_{12} + 1) (1 - K_{22})\vert / (K_{11} (1 - K_{22}) + (K_{12} + 1)^2)^2 =&~ O_{d, \P}(d^{-3/2 + \eps}). 
\end{aligned}
\]
Since all the quantities above are deterministically bounded by a constant, these high probability bounds translate to expectation bounds. 

Moreover, we have
\[
\begin{aligned}
\E[ G_{11}^2]^{1/2} \le&~ \psi_2  \lambda_{d, 0}(\sigma)^2  ( \psi_1 \psi_2 \lambda)^{-2} \E[ \| \bM \|_{\op}^2]^{1/2} \| \ones_N \ones_N^\sT\|_{\op} = O_d(d), \\
\E[G_{22}^2]^{1/2} \le&~ O_d(1) \cdot \E[  \| \bM \|_{\op}^2]^{1/2} \| \ones_n \ones_n^\sT / n \|_{\op} = O_d(1),\\
\E[G_{12}^2]^{1/2} \le&~O_d(1) \cdot  \lambda_{d, 0}(\sigma) \E[\| \bM \|_{\op}^2]^{1/2} \| \ones_n \ones_N^\sT/ \sqrt { d} \|_{\op} = O_d(d^{1/2}). \\ 
\end{aligned} 
\]
Plugging in the above bounds into Equation (\ref{eqn:expression_B_constant_term_two}),  we have
\[ 
\E[\vert B \vert] = o_d(1). 
\]
This proves the lemma. 
\end{proof}

\subsection{Some auxiliary lemmas}

We denote by $\mu_d$ the probability law of $\<\bx_1,\bx_2\>/\sqrt{d}$ when $\bx_1,\bx_2\sim_{iid}\normal(\bzero,\id_d)$.
Note that $\mu_d$ is symmetric, and $\int x^2 \mu_d(\de x) = 1$.  
By the central limit theorem, $\mu_d$ converges weakly to $\mu_G$  as $d\to\infty$, where $\mu_G$ is the standard Gaussian measure. In fact, we have the following  stronger convergence result. 
\begin{lemma}\label{lemma:Sub-Exp}
For any $\lambda\in [- \sqrt{d}/2, \sqrt{d}/2]$, we have 
\begin{align}
\int e^{\lambda x}\, \mu_d(\de x) \le e^{\lambda^2}\, .\label{eq:ExponentialTailMud}
\end{align}
Further, let $f:\reals\to\reals$ be a continuous function such that $|f(x)|\le c_0\exp(c_1|x|)$ for some constants $c_0,c_1<\infty$. Then
\begin{align}
\lim_{d\to\infty}\int f(x)\, \mu_d(\de x) = \int f(x)\, \mu_G(\de x)\, . \label{eq:ConvergenceExpMudGaussian}
\end{align}
\end{lemma}

\begin{proof}[Proof of Lemma \ref{lemma:Sub-Exp}]~
In order to prove Eq.~\eqref{eq:ExponentialTailMud}, we note that the left hand side is given by
\begin{align*}
\E\big\{e^{\lambda\<\bx_1,\bx_2\>/\sqrt{d}}\big\} =&~ \frac{1}{(2\pi)^d}
\int\exp\Big\{-\frac{1}{2}\|\bx_1\|_2^2 -\frac{1}{2}\|\bx_2\|_2^2+\frac{\lambda}{\sqrt{d}}\<\bx_1,\bx_2\>\Big\}
\de\bx_1\de\bx_2\\
=&~ \left[\det\left(\begin{matrix}
1 & -\lambda/\sqrt{d}\\
-\lambda/\sqrt{d} & 1
\end{matrix}\right)\right]^{-d/2}= \Big(1-\frac{\lambda^2}{d}\Big)^{-d/2}\\
\le&~ e^{\lambda^2}\, ,
\end{align*}
where the last inequality holds for $\vert \lambda\vert \le \sqrt d / 2$ using the fact that $(1-x)^{-1}\le e^{2x}$ for $x\in [0,1/4]$.

In order to prove \eqref{eq:ConvergenceExpMudGaussian}, let $X_d\sim\mu_d$, and $G\sim \normal(0,1)$. Since 
$\mu_d$ converges weakly to $\normal(0,1)$, we can construct such
random variables so that $X_d\to G$ almost surely. Hence $f(X_d)\to f(G)$ almost surely. However
$\vert f(X_d) \vert \le c_0\exp(c_1 \vert X_d\vert )$ which is a uniformly integrable family by the previous point, implying $\E f(X_d) \to \E f(G)$ as claimed.
\end{proof}

The next several lemmas establish general bounds on the operator norm of random kernel matrices which is of independent interest. 

\begin{lemma}\label{lem:concentration_operator_general_gaussian}
Let $\sigma: \R \to \R$ be an activation function satisfying Assumption \ref{ass:activation}, i.e., $\vert \sigma(u) \vert, \vert \sigma'(u) \vert \le c_0 e^{c_1 \vert u \vert}$ for some constants $c_0, c_1 \in (0, \infty)$. Let $(\overline \bz_i)_{i \in [M]} \sim_{iid} \normal(\bzero, \id_d)$. Assume $0 < 1/c_2 \le M /d \le c_2 < \infty$ for some constant $c_2\in (0, \infty)$. 
Consider the random matrix $\overline \bR \in \R^{M\times M}$ defined by
\begin{align}\label{eqn:overlineR_general_Gaussian}
\overline R_{ij} = \ones_{i \neq j}\cdot \sigma(\< \overline \bz_i, \overline \bz_j\> / \sqrt d) / \sqrt d.
\end{align}
Then there exists a constant $C$ depending uniquely on $c_0, c_1, c_2$, and a sequence of numbers $(\overline \eta_d)_{d \ge 1}$ with $\vert \overline \eta_d \vert \le C \exp\{ C (\log d)^{1/2} \}$, such that 
\begin{align}
\| \overline \bR - \overline \eta_d \ones_M \ones_M^\sT / \sqrt d \|_{\op} = O_{d, \P}(\exp\{ C (\log d)^{1/2} \}). 
\end{align}
\end{lemma}

\begin{proof}[Proof of Lemma \ref{lem:concentration_operator_general_gaussian}]~
By Lemma \ref{lemma:Sub-Exp} and Markov inequality, we have, for any $i\neq j$ and all $0\le t\le\sqrt{d}$,
\begin{align}
\P\Big( \<\overline \bz_i,\overline \bz_j\> / \sqrt d \ge t \Big) \le e^{-t^2/4}\, .
\end{align}
Hence
\begin{equation}\label{eqn:bound_zi_zj_max}
\begin{aligned}
&\P\Big( \max_{1 \le i<j\le M} \Big \vert \frac{1}{\sqrt{d}} \<\overline \bz_i,\overline \bz_j\>\Big \vert \ge 16\sqrt{\log M} \Big)\\
 \le&~ \frac{M^2}{2} \max_{1 \le i<j \le M} \P \Big(\Big\vert \frac{1}{\sqrt{d}} \<\overline \bz_i, \overline \bz_j\>\Big\vert \ge 16\sqrt{\log M}\Big) \le M^2 \exp\{-4(\log M)\}\le \frac{1}{M^2}\, .
\end{aligned}
\end{equation}
We define $\tilde \sigma: \R \to \R$ as follows: for $\vert u \vert \le \overline x \equiv 16\sqrt{\log d}$, define $\tilde \sigma(u) \equiv \sigma(u) e^{- c_1 \vert \overline x \vert} / c_0$; for $u > \overline x$, define $\tilde \sigma(u) = \tilde \sigma(\overline x)$; for $u < - \overline x$, define $\tilde \sigma(u) = \tilde \sigma(- \overline x)$. Then $\tilde \sigma$ is a $1$-bounded-Lipschitz function on $\R$. 
Define $\tilde \eta_d = \E_{\overline \bx, \overline \by \sim \normal(\bzero, \id_d)}[\tilde \sigma(\< \overline \bx, \overline \by\> / \sqrt d)]$ and $\overline \eta_d = \tilde \eta_d c_0 e^{c_1 \vert \overline x \vert} $. Since we have $\vert  \tilde \eta_d \vert \le \max_u \vert \tilde \sigma(u) \vert \le 1$, we have 
\begin{align}\label{eqn:overlineeta_d_general_Gaussian}
\vert \overline \eta_d \vert = O_d(\exp\{C (\log d)^{1/2}\}). 
\end{align}
Moreover, we define $\overline \bK, \tilde \bK \in \R^{M \times M}$ by
\begin{equation}\label{eqn:overlineK_general_Gaussian}
\begin{aligned}
\tilde K_{ij} = \ones_{i \neq j}\cdot (\tilde \sigma(\< \overline \bz_i, \overline \bz_j\> / \sqrt d) - \tilde \eta_d) / \sqrt d,\\
\overline K_{ij} = \ones_{i \neq j}\cdot (\sigma(\< \overline \bz_i, \overline \bz_j\> / \sqrt d) - \overline \eta_d) / \sqrt d. 
\end{aligned}
\end{equation}
By \cite[Lemma 20]{deshpande2016sparse}, there exists a constant $C$ such that 
\[
\P(\| \tilde \bK \|_{\op}\ge C )\le C e^{-d/C}.
\]
Note that \cite[Lemma 20]{deshpande2016sparse} considers one specific choice of
$\tilde \sigma$, but the proof applies unchanged to any $1$-Lipschitz function with zero expectation under the measure $\mu_d$, where $\mu_d$ is the distribution of $\< \overline \bx, \overline \by\>/\sqrt d$ for $\overline \bx, \overline \by \sim \normal(\bzero, \id_d)$. 

Defining the event $\mathcal G \equiv\{ \vert \<\overline \bz_i, \overline\bz_j\>/\sqrt{d} \vert \le 16\sqrt{\log d}, \, \forall 1 \le i<j\le M\}$, we have
\begin{align}\label{eqn:overlineK_probability_bound}
\P\Big( \| \overline \bK \|_{\op} \ge C\, c_0 e^{c_1\vert \overline x \vert} \Big)\le \P\Big( \| \overline \bK \|_{\op} \ge C\, c_0 e^{c_1 \vert \overline x \vert}; \mathcal G \Big) + \P(\mathcal G^c) \le \P\Big( \| \tilde \bK \|_{\op } \ge C \Big) + \frac{1}{M^2} = o_d(1). 
\end{align}
By Eq. (\ref{eqn:overlineR_general_Gaussian}) and (\ref{eqn:overlineK_general_Gaussian}), we have
\[
\overline \bR = \overline \bK - \overline \eta_d \id_M  / \sqrt d + \overline \eta_d \ones_M \ones_M^\sT / \sqrt d.
\]
By Eq. (\ref{eqn:overlineK_probability_bound}) and (\ref{eqn:overlineeta_d_general_Gaussian}), we have 
\[
\| \overline \bR - \overline \eta_d \ones_M \ones_M^\sT / \sqrt d \|_{\op} = \| \overline \bK - \overline \eta_d \id_M / \sqrt d \|_{\op} \le \| \overline \bK \|_{\op} + \overline \eta_d / \sqrt d = O_{d, \P}(\exp\{C (\log d)^{1/2} \}).
\]
This completes the proof.
\end{proof}

\begin{lemma}\label{lem:operator_interpolation_gaussian_sphere}
Let $\sigma: \R \to \R$ be an activation function satisfying Assumption \ref{ass:activation}, i.e., $\vert \sigma(u) \vert, \vert \sigma'(u) \vert \le c_0 e^{c_1 \vert u \vert}$ for some constants $c_0, c_1 \in (0, \infty)$. Let $(\overline \bz_i)_{i \in [M]} \sim_{iid} \normal(\bzero, \id_d)$. Assume $0 < 1/c_2 \le M /d \le c_2 < \infty$ for some constant $c_2\in (0, \infty)$.  Define $\bz_i = \sqrt d \cdot \overline \bz_i / \| \overline \bz_i \|_2$. Consider two random matrices $\bR, \overline \bR \in \R^{M\times M}$ defined by
\[
\begin{aligned}
\overline R_{ij} =&~ \ones_{i \neq j}\cdot \sigma(\< \overline \bz_i, \overline \bz_j\> / \sqrt d) / \sqrt d,\\
R_{ij} =&~ \ones_{i \neq j}\cdot \sigma(\< \bz_i, \bz_j\> / \sqrt d) / \sqrt d.
\end{aligned}
\]
Then there exists a constant $C$ depending uniquely on $c_0, c_1, c_2$, such that 
\[
\| \overline \bR - \bR \|_{\op} = O_{d, \P}(\exp\{ C (\log d)^{1/2} \}).
\]
\end{lemma}

\begin{proof}[Proof of Lemma \ref{lem:operator_interpolation_gaussian_sphere}] In the proof of this lemma, we assume $\sigma$ has continuous derivatives. In the case when $\sigma$ is only weak differentiable, the proof is the same, except that we need to replace the mean value theorem to its integral form. 

Define $r_i = \sqrt d / \| \overline \bz_i \|_2$, and
\[
\tilde R_{ij} = \ones_{i \neq j}\cdot \sigma(r_i \< \overline \bz_i, \overline \bz_j\> / \sqrt d) / \sqrt d.
\]
By the concentration of $\chi$-squared distribution, it is easy to see that 
\[
\max_{i \in [M]} \vert r_i - 1 \vert = O_{d, \P}((\log d)^{1/2} / d^{1/2}). 
\]
Moreover, we have (for $\zeta_i$ between $r_i$ and $1$)
\[
\vert \overline R_{ij} - \tilde R_{ij}\vert \le  \vert \sigma'(\zeta_i \< \overline \bz_i, \overline \bz_j\> / \sqrt d)\vert \cdot  \vert \< \overline \bz_i, \overline \bz_j\> / \sqrt d \vert  \cdot \vert r_i - 1 \vert / \sqrt d. 
\]
By Eq. (\ref{eqn:bound_zi_zj_max}), we have 
\[
\begin{aligned}
\max_{i \neq j \in [M]} [ \< \overline \bz_i, \overline \bz_j \> / \sqrt d ]=&~ O_{d, \P}(( \log d)^{1/2}), \\
\max_{i \neq j \in [M]} [ \zeta_i \< \overline \bz_i, \overline \bz_j \> / \sqrt d ]=&~ O_{d, \P}(( \log d)^{1/2}).
\end{aligned}
\]
Moreover by the assumption that $\vert \sigma'(u)\vert \le c_0 e^{c_1 \vert u \vert}$, we have
\[
\max_{i \neq j \in [M]} \vert \sigma'(\zeta_i \< \overline \bz_i, \overline \bz_j\> / \sqrt d)\vert \cdot  \vert \< \overline \bz_i, \overline \bz_j\> / \sqrt d \vert = O_{d, \P}( \exp\{ C (\log d)^{1/2}\}). 
\]
This gives 
\[
\max_{i \neq j \in [M]}\vert \overline R_{ij} - \tilde R_{ij}\vert = O_{d, \P}( \exp\{ C (\log d)^{1/2}\} / d). 
\]
Using similar argument, we can show that 
\[
\max_{i \neq j \in [M]}\vert  R_{ij} - \tilde R_{ij}\vert = O_{d, \P}( \exp\{ C (\log d)^{1/2}\} / d), 
\]
which gives
\[
\max_{i \neq j \in [M]}\vert  R_{ij} - \overline R_{ij}\vert = O_{d, \P}( \exp\{ C (\log d)^{1/2}\} / d). 
\]
This gives 
\[
\| \bR - \overline \bR \|_{\op} \le \| \bR - \overline \bR \|_F \le d \cdot \max_{i \neq j \in [M]}\vert  R_{ij} - \overline R_{ij}\vert = O_{d, \P}( \exp\{ C (\log d)^{1/2}\}). 
\]
This proves the lemma. 
\end{proof}

\begin{lemma}\label{lem:concentration_operator_general_sphere}
Let $\sigma: \R \to \R$ be an activation function satisfying Assumption \ref{ass:activation}, i.e., $\vert \sigma(u) \vert, \vert \sigma'(u) \vert \le c_0 e^{c_1 \vert u \vert}$ for some constants $c_0, c_1 \in (0, \infty)$. Let $(\bz_i)_{i \in [M]} \sim_{iid} \Unif(\S^{d-1}(\sqrt d))$. Assume $0 < 1/c_2 \le M /d \le c_2 < \infty$ for some constant $c_2\in (0, \infty)$. Define $\lambda_{d, 0} = \E_{\bz_i, \bz_2 \sim \Unif(\S^{d-1}(\sqrt d))}[\sigma(\< \bz_1, \bz_2\> / \sqrt d)]$, and $\vphi_d(u) = \sigma(u) - \lambda_{d, 0}$. Consider the random matrix $\bK \in \R^{M \times M}$ with 
\[
K_{ij} = \ones_{i \neq j}\cdot \frac{1}{\sqrt d}\vphi_d\Big(\frac{1}{\sqrt d}\< \bz_i,  \bz_j\> \Big).
\]
Then there exists a constant $C$ depending uniquely on $c_0, c_1, c_2$, such that
\[
\| \bK \|_{\op} \le O_{d, \P}(\exp\{ C (\log d)^{1/2} \}). 
\] 
\end{lemma}

\begin{proof}[Proof of Lemma \ref{lem:concentration_operator_general_sphere}] 
We construct $(\bz_i)_{i \in [M]}$ by normalizing a collection of independent Gaussian random vectors. Let $(\overline \bz_i)_{i \in [M]} \sim_{iid} \normal(\bzero, \id_d)$ and denote $\bz_i = \sqrt d \cdot \overline \bz_i / \| \overline \bz_i \|_2$ for $i \in [M]$. Then we have $(\bz_i)_{i \in [M]} \sim_{iid} \Unif(\S^{d-1}(\sqrt d))$. 

Consider two random matrices $\overline \bR, \bR \in \R^{M\times M}$ defined by
\[
\begin{aligned}
\overline R_{ij} =&~ \ones_{i \neq j}\cdot \sigma(\< \overline \bz_i, \overline \bz_j\> / \sqrt d) / \sqrt d, \\
R_{ij} =&~ \ones_{i \neq j}\cdot \sigma(\< \bz_i, \bz_j\> / \sqrt d) / \sqrt d. \\
\end{aligned}
\]
By Lemma \ref{lem:concentration_operator_general_gaussian}, there exists a sequence $(\overline \eta_d)_{d \ge 0}$ with $\vert \overline \eta_d \vert \le C \exp\{ C (\log d)^{1/2} \}$, such that 
\[
\| \overline \bR - \overline \eta_d \ones_M \ones_M^\sT / \sqrt d \|_{\op} = O_{d, \P}( \exp\{ C (\log d)^{1/2} \}). 
\]
Moreover, by Lemma \ref{lem:operator_interpolation_gaussian_sphere}, we have, 
\[
\| \overline \bR - \bR \|_{\op} \le O_{d, \P}( \exp\{ C (\log d)^{1/2} \}),
\]
which gives, 
\[
\| \bR - \overline \eta_d \ones_M \ones_M^\sT / \sqrt d\|_{\op} = O_{d, \P}( \exp\{ C (\log d)^{1/2} \}). 
\]
Note we have 
\[
\bR = \bK + \lambda_{d, 0} \ones_M \ones_M^\sT / \sqrt d - \lambda_{d, 0} \id_M / \sqrt d. 
\]
Moreover, note that $\lim_{d \to \infty} \lambda_{d, 0} = \E_{G \sim \normal(0, 1)}[\sigma(G)]$ so that $\sup_d \vert \lambda_{d, 0} \vert \le C$. Therefore, denoting $\kappa_d = \lambda_{d, 0} - \overline \eta_d$, we have 
\begin{equation}\label{eqn:operator_bound_K_plus_etad}
\begin{aligned}
\| \bK + \kappa_d \ones_M \ones_M^\sT / \sqrt d \|_{\op} =&~ \| \bR - \overline \eta_d \ones_M \ones_M^\sT / \sqrt d + \lambda_{d, 0} \id_M / \sqrt d \|_{\op} \\
\le &~ \| \bR - \overline \eta_d \ones_M \ones_M^\sT / \sqrt d\|_{\op} + \lambda_{d, 0} / \sqrt d = O_{d, \P}(\exp\{ C (\log d)^{1/2} \}). 
\end{aligned}
\end{equation}

Notice that
\[
\vert \ones_M^\sT \bK \ones_M / M \vert \le \frac{C}{M^{3/2} }  \Big \vert \sum_{i \neq j }\vphi_d(\< \bz_i, \bz_j\>/ \sqrt d)\Big \vert \le \frac{C}{M}  \sum_{i=1}^M \Big \vert \sum_{j: j\neq i }\vphi_d(\< \bz_i, \bz_j\>/ \sqrt d) / \sqrt {M} \Big \vert \equiv \frac{C}{M}  \sum_{i=1}^M \vert V_i\vert, 
\]
where
\[
V_i =  \frac{1}{\sqrt {M}}\sum_{j: j\neq i }\vphi_d(\< \bz_i, \bz_j\>/ \sqrt d). 
\]
Note $\E[\vphi_d(\< \bz_i, \bz_j\>/ \sqrt d)] = 0$ for $i \neq j$ so that $\E[\vphi_d(\< \bz_i, \bz_{j_1}\>/ \sqrt d) \vphi_d(\< \bz_i, \bz_{j_2}\>/ \sqrt d)] = 0$ for $i, j_1, j_2$ distinct. Calculating the second moment, we have
\[
\sup_{i \in [M]}\E[V_i^2] = \sup_{i \in [M]} \E\Big[ \Big( \sum_{j: j\neq i }\vphi_d(\< \bz_i, \bz_j\>/ \sqrt d) / \sqrt {M} \Big)^2\Big] = \sup_{i \in [M]} \frac{1}{M} \sum_{j: j\neq i }  \E[ \vphi_d(\< \bz_i, \bz_j\>/ \sqrt d)^2 ]  = O_d(1). 
\]
Therefore, we have 
\[
\E[( \ones_M^\sT \bK \ones_M / M)^2] \le \frac{C^2}{M^2} \sum_{i, j =1}^M \E [\vert V_i\vert \cdot \vert V_j\vert] \le \frac{C^2}{M^2} \sum_{i, j =1}^M \E [(V_i^2 + V_j^2) / 2 ]  \le C^2 \sup_{i \in [M]} \E[V_i^2] = O_d(1). 
\]
This gives 
\[
\vert \ones_M^\sT \bK \ones_M / M \vert = O_{d, \P}(1). 
\]
Combining this equation with Eq. (\ref{eqn:operator_bound_K_plus_etad}), we get
\[
\begin{aligned}
&\| \kappa_d \ones_M \ones_M^\sT  / \sqrt d \|_{\op} = \vert \< \ones_M, (\kappa_d \ones_M \ones_M^\sT / \sqrt d) \ones_M \> / M  \vert \\
\le&~ \vert \<\ones_M, (\bK + \kappa_d \ones_M \ones_M^\sT / \sqrt d) \ones_M \> / M \vert + \vert \ones_M^\sT \bK \ones_M / M \vert\\
\le&~ \| \bK + \kappa_d \ones_M \ones_M^\sT / \sqrt d \|_{\op} + \vert \ones_M^\sT \bK \ones_M / M \vert = O_{d, \P}( \exp\{ C (\log d)^{1/2} \}),
\end{aligned}
\]
and hence
\[
\| \bK \|_{\op} \le \| \bK + \kappa_d \ones_M \ones_M^\sT / \sqrt d \|_{\op} + \| \kappa_d \ones_M \ones_M^\sT / \sqrt d \|_{\op} = O_{d, \P}( \exp\{ C (\log d)^{1/2} \}).
\]
This proves the lemma. 
\end{proof}

\subsection{The decomposition of kernel inner product matrices}\label{subsec:preliminary_spherical_random_matrix}

The following lemma (Lemma \ref{lem:gegenbauer_identity}) is a reformulation of Proposition 3 in \cite{ghorbani2019linearized}. We present it in a stronger form, but it can be easily derived from the proof of Proposition 3 in \cite{ghorbani2019linearized}. This lemma was first proved in \cite{el2010spectrum} in the Gaussian case.
(Notice that the second estimate ---on $Q_k(\bTheta \bX^\sT)$---  follows
by applying the first one whereby $\bTheta$ is replaced by $\bW = [\bTheta^{\sT}|\bX^{\sT}]^{\sT}$
\begin{lemma}\label{lem:gegenbauer_identity}
Let $\bTheta = (\btheta_1, \ldots, \btheta_N)^\sT \in \R^{N \times d}$ with $(\btheta_a)_{a\in [N]} \sim_{iid} \Unif(\S^{d- 1}(\sqrt d))$ and $\bX = (\bx_1, \ldots, \bx_n)^\sT \in \R^{n \times d}$ with $(\bx_i)_{i\in [n]} \sim_{iid} \Unif(\S^{d- 1}(\sqrt d))$. Assume $1/c \le n / d, N / d \le c$ for some constant $c \in (0, \infty)$. Then 
\begin{align}
  \E\Big[ \sup_{k \ge 2} \| Q_k(\bTheta \bTheta^\sT) - \id_N \|_{\op}^2 \Big]&= o_d(1)\, ,\label{eq:QTT}\\
  \E\Big[ \sup_{k \ge 2} \| Q_k(\bTheta \bX^\sT) \|_{\op}^2 \Big] &= o_d(1). \label{eq:QTX}
\end{align}
\end{lemma}
Notice that the second estimate ---on $Q_k(\bTheta \bX^\sT)$---  follows
by applying the first one ---Eq.~\eqref{eq:QTT}--- whereby $\bTheta$ is replaced by
$\bW = [\bTheta^{\sT}|\bX^{\sT}]^{\sT}$, and we use
$\| Q_k(\bTheta \bX^\sT) \|_{\op} \le \| Q_k(\bW \bW^\sT)-\id_{N+n} \|_{\op}$. 

The following lemma (Lemma \ref{lem:decomposition_of_kernel_matrix}) can be easily derived from Lemma \ref{lem:gegenbauer_identity}. Again, this lemma was first proved in \cite{el2010spectrum} in the Gaussian case. 
\begin{lemma}\label{lem:decomposition_of_kernel_matrix}
Let $\bTheta = (\btheta_1, \ldots, \btheta_N)^\sT \in \R^{N \times d}$ with $(\btheta_a)_{a\in [N]} \sim_{iid} \Unif(\S^{d- 1}(\sqrt d))$. Let activation function $\sigma$ satisfies Assumption \ref{ass:activation}. Assume $1/c \le N / d \le c$ for some constant $c \in (0, \infty)$. Denote 
\[
\bU = \Big(\E_{\bx \sim \Unif(\S^{d-1}(\sqrt d))}[\sigma(\< \btheta_a, \bx\> / \sqrt d) \sigma(\< \btheta_b, \bx \> / \sqrt d)] \Big)_{a, b \in [N]} \in \R^{N \times N}. 
\]
Then we can rewrite the matrix $\bU$ to be
\[
\bU =  \lambda_{d, 0}(\sigma)^2 \ones_N \ones_N^\sT + \ob_1^2 \bQ  + \ob_\star^2 (\id_N + \bDelta),
\]
with $\bQ = \bTheta \bTheta^\sT / d$ and $\E[\| \bDelta \|_{\op}^2] = o_{d}(1)$. 
\end{lemma}

\subsection{A lemma on the variance of the quadratic form}

\begin{lemma}\label{lem:variance_calculations}
Let $\bA \in \R^{n \times N}$ and $\bB \in \R^{n \times n}$. Let $\bg = (g_1, \ldots, g_n)^\sT$ with $g_i \sim_{iid} \P_g$, $\E_g[g] = 0$, and $\E_g[g^2] = 1$. Let $\bh = (h_1, \ldots, h_N)^\sT$ with $h_i \sim_{iid} \P_h$, $\E_h[h] = 0$, and $\E_h[h^2] = 1$. Further we assume that $\bh$ is independent of $\bg$. Then we have 
\[
\begin{aligned}
\Var(\bg^\sT \bA \bh) =&~ \| \bA \|_F^2, \\
\Var(\bg^\sT \bB \bg) =&~ \sum_{i = 1}^n B_{ii}^2 (\E[g^4] - 3) + \| \bB \|_F^2 + \Tr(\bB^2). 
\end{aligned}
\]
\end{lemma}

\begin{proof}[Proof of Lemma \ref{lem:variance_calculations}] ~

\noindent
{\bf Step 1. Term $\bg^\sT \bA \bh$. }
Calculating the expectation, we have 
\[
\E[\bg^\sT \bA \bh] = 0. 
\]
Hence we have 
\[
\Var(\bg^\sT \bA \bh) = \E[\bg^\sT \bA \bh \bh^\sT \bA^\sT \bg] = \E[\Tr(\bg\bg^\sT \bA \bh \bh^\sT \bA^\sT)] = \Tr(\bA \bA^\sT) = \| \bA \|_F^2. 
\]

\noindent
{\bf Step 2. Term $\bg^\sT \bB \bg$. }
Calculating the expectation, we have
\[
\E[\bg^\sT \bB \bg] = \E[\Tr(\bB \bg \bg^\sT)] = \Tr(\bB). 
\]
Hence we have 
\[
\begin{aligned}
&\Var(\bg^{\sT}\bB\bg) =\Big\{ \sum_{i_1, i_2, i_3, i_4} \E[g_{i_1} B_{i_1 i_2} g_{i_2} g_{i_3} B_{i_3 i_4} g_{i_4}] \Big\}- \Tr(\bB)^2\\
=&~\Big\{\Big( \sum_{i_1 =  i_2 = i_3 = i_4} + \sum_{i_1 = i_2 \neq i_3 = i_4}  + \sum_{i_1 = i_3 \neq i_2 = i_4}  +  \sum_{i_1 = i_4 \neq i_2 =  i_3}  \Big) \E[g_{i_1} B_{i_1 i_2} g_{i_2} g_{i_3} B_{i_3 i_4} g_{i_4}] \Big\} - \Tr(\bB)^2\\
=&~\sum_{i=1}^{n} B_{ii}^2 \E[g^4] + \sum_{i \neq j} B_{ii} B_{jj} + \sum_{i \neq j}  (B_{ij} B_{ij} + B_{ij} B_{ji}) - \Tr(\bB)^2\\
=&~\sum_{i=1}^{n} B_{ii}^2 (\E[g^4] - 3) +  \Tr(\bB^\sT\bB) + \Tr(\bB^2). \\
\end{aligned}
\]
This proves the lemma. 
\end{proof}

\clearpage

\section{Proof of Lemma \ref{lem:formula_derivatives_g}}\label{sec:proof_lemma_formula_derivatives_g}

\begin{proof}[Proof of Lemma \ref{lem:formula_derivatives_g}]
For fixed $\xi \in \C_+$ and $\bq \in \R^5$, by the fixed point equation satisfied by $m_1, m_2$ (c.f. Eq. (\ref{eq:FixedPoint})), we see that $(m_1(\xi; \bq), m_2(\xi; \bq))$ is a stationary point of function $\Xi(\xi, \cdot, \cdot; \bq)$. Using the formula for implicit differentiation, we have 
\[
\begin{aligned}
\partial_p g(\xi; \bq) =&~ \partial_p \Xi(\xi, z_1, z_2; \bq) \vert_{(z_1, z_2) = (m_1(\xi; \bq), m_2(\xi; \bq))}, \\
\partial_{s_1, t_1}^2 g(\xi; \bq) =&~ \bH_{1, 3} - \bH_{1, [5, 6]}\bH_{[5,6], [5,6]}^{-1}\bH_{[5, 6], 3}, \\
\partial_{s_1, t_2}^2 g(\xi; \bq) =&~ \bH_{1, 4} - \bH_{1, [5, 6]}\bH_{[5,6], [5,6]}^{-1}\bH_{[5, 6], 4}, \\
\partial_{s_2, t_1}^2 g(\xi; \bq) =&~ \bH_{2, 3} - \bH_{2, [5, 6]}\bH_{[5,6], [5,6]}^{-1}\bH_{[5, 6], 3}, \\
\partial_{s_2, t_2}^2 g(\xi; \bq) =&~ \bH_{2, 4} - \bH_{2, [5, 6]}\bH_{[5,6], [5,6]}^{-1}\bH_{[5, 6], 4}, \\
\end{aligned}
\]
where we have, for $\bu = (s_1, s_2, t_1, t_2, z_1, z_2)^\sT$
\[
\bH = \nabla_{\bu}^2 \Xi(\xi, z_1, z_2; \bq) \vert_{(z_1, z_2) = (m_1(\xi; \bq), m_2(\xi; \bq))}. 
\]
Basic algebra completes the proof. 
\end{proof}

\clearpage

\section{Proof sketch for Theorem \ref{thm:training_asymptotics}}\label{sec:training_proof_sketch}

In this section, we sketch the calculations of Theorem \ref{thm:training_asymptotics}. We assume $\psi_{1, d} \equiv N / d = \psi_1$ and $\psi_{2, d} \equiv n / d = \psi_2$ are constants independent of $d$. Recall that the definitions of two useful resolvent matrix $\Res$ and $\oRes$ are 
\[
\begin{aligned}
\Res =&~ (\bZ^\sT \bZ + \lambda \psi_1 \psi_2 \id_N)^{-1}, \\
\oRes =&~ (\bZ \bZ^\sT + \lambda \psi_1 \psi_2 \id_n)^{-1}.
\end{aligned}
\]

\noindent
{\bf Step 1. The expectation of regularized training error. }

By Eq. (\ref{eqn:regularized_training_loss}), the regularized training error of random features regression gives
\[
\begin{aligned}
L_{\RF}(f_d, \bX, \bTheta, \lambda) =&~ \min_{\ba}\Big[ \frac{1}{n} \sum_{i=1}^n \Big( y_i - \sum_{j = 1}^N a_j \sigma(\<\btheta_j, \bx_i \> / \sqrt d) \Big)^2 + \lambda \psi_1 \| \ba \|_2^2\Big]\\
=&~ \min_{\ba}\Big[ \frac{1}{n} \| \by - \sqrt d  \bZ \ba \| ^2 + \lambda \psi_1 \| \ba \|_2^2\Big]\\
=&~ \frac{1}{n} \| \by - \bZ  \Res \bZ^\sT \by \| ^2 + \lambda \psi_1 \|   \Res  \bZ^\sT \by\|_2^2 / d \\
=&~ \frac{1}{n} \Big[ \| \by \|_2^2 - \by^\sT \bZ  \Res \bZ^\sT \by   \Big]. 
\end{aligned}
\]
Its expectation with respect to $f_d^{\sNL}$ (that satisfies Assumption \ref{ass:ground_truth}), $\bbeta_1 \sim \Unif(\S^{d-1}(\normf_{d, 1}))$ and $\beps$ gives
\[
\begin{aligned}
&\E_{\bbeta, \beps}[ L_{\RF}(f_d, \bX, \bTheta, \lambda)] \\
=&~ \frac{1}{n} \Big[ \E_{\bbeta, \beps}[ \| \by \|_2^2] - \E_{\bbeta, \beps}[\by^\sT \bZ  \Res \bZ^\sT \by]   \Big] \\
=&~ \E_{\bbeta}[\| f_d \|_{L^2}^2] + \tau^2 - \frac{1}{n} \E_{\bbeta} \Big[ \boldf^\sT \bZ  \Res \bZ^\sT \boldf \Big]- \frac{1}{n} \E_{\beps} \Big[ \beps^\sT \bZ  \Res \bZ^\sT \beps \Big]\\
=&~ \E_{\bbeta}[\| f_d \|_{L^2}^2] + \tau^2 - \frac{1}{n} \E_{\bbeta}\Big[ \Big( \sum_{k = 0}^\infty \bY_{\bx, k} \bbeta_k \Big)^\sT \bZ \Res \bZ^\sT \Big( \sum_{k = 0}^\infty \bY_{\bx, k} \bbeta_k \Big)\Big]  - \frac{\tau^2}{n} \Tr\Big( \Res \bZ^\sT \bZ \Big)\\
=&~ \sum_{ k = 0}^\infty \normf_k^2 + \tau^2  - \frac{1}{n} \sum_{k = 0}^\infty \normf_k^2 \Tr\Big( \Res \bZ^\sT Q_k(\bX \bX^\sT) \bZ \Big)  - \frac{\tau^2}{n} \Tr\Big( \Res \bZ^\sT \bZ\Big)\\
\end{aligned}
\]
It can be shown that the coefficients before $\normf_0^2$ is asymptotically vanishing, and by Lemma \ref{lem:gegenbauer_identity}, we have $\E[\sup_{k \ge 2} \| \bQ_k(\bX \bX^\sT) - \id_n \|_{\op}^2] = o_d(1)$. Hence we get
\[
\begin{aligned}
\E_{\bbeta, \beps}[ L_{\RF}(f_d, \bX, \bTheta, \lambda)] = &~ \normf_1^2 \Big\{ 1 - \frac{1}{n}  \Tr\Big( \Res \bZ^\sT \bH \bZ \Big) \Big\}  + ( \normf_\star^2 + \tau^2) \cdot \Big\{ 1 -  \frac{1}{n} \Tr\Big( \Res \bZ^\sT \bZ \Big) \Big\} + o_{d, \P}(1). 
\end{aligned}
\]

Using the fact that 
\[
\Res \bZ^\sT = \bZ^\sT \oRes,
\]
we have 
\[
\begin{aligned}
\E_{\bbeta, \beps}[L_{\RF}(f_d, \bX, \bTheta, \lambda)] \stackrel{\cdot}{=}&~ \normf_1^2 \cdot \frac{\psi_1 \lambda}{d} \Tr ( \oRes \bH ) + ( \normf_\star^2 + \tau^2) \cdot \frac{\psi_1 \lambda}{d}  \Tr ( \oRes) + o_{d, \P}(1). \\
\end{aligned}
\]

\noindent
{\bf Step 2. The norm square of minimizers. }

We have 
\[
\begin{aligned}
\| \ba \|_2^2 =&~ \| \by^\sT \bZ  \Res \|_2^2 / d = \by^\sT \bZ  \Res^2 \bZ^\sT \by / d,
\end{aligned}
\]
so that 
\[
\begin{aligned}
\E_{\bbeta, \beps}[ \| \ba \|_2^2] =&~  \E_{\bbeta}[\boldf^\sT \bZ  \Res^2 \bZ^\sT \boldf] / d + \psi_1 \E_\beps[\beps^\sT \bZ  \Res^2 \bZ^\sT \beps] / d \\
=&~  \E_{\bbeta}\Big[ \Big( \sum_{k = 0}^\infty \bY_{\bx, k} \bbeta_k \Big)^\sT \bZ  \Res^2 \bZ^\sT \Big( \sum_{k = 0}^\infty \bY_{\bx, k} \bbeta_k \Big) \Big] / d +  \tau^2 \Tr\Big( \Res^2 \bZ^\sT \bZ \Big) / d \\
=&~  \sum_{k = 0}^\infty \normf_k^2\cdot \Tr \Big( \Res^2 \bZ^\sT Q_k(\bX \bX^\sT) \bZ \Big) / d +  \tau^2  \Tr\Big( \Res^2 \bZ^\sT \bZ \Big) / d \\
=&~  \normf_1^2 \Tr \Big( \Res^2 \bZ^\sT \bH \bZ \Big) / d + (\normf_\star^2 +  \tau^2) \cdot \Tr\Big( \Res^2 \bZ^\sT \bZ \Big) / d  + o_{d, \P}(1). \\
\end{aligned}
\]

\noindent
{\bf Step 3. The derivatives of the log determinant. }

Define $\bq = (s_1,s_2, t_1,t_2, p) \in \R^5$, and introduce a block matrix $\bA\in\R^{M\times M}$ with $M=N+n$, defined by
\begin{align}\label{eqn:matrix_A_bis}
\bA = \left[\begin{matrix}
s_1\id_N+s_2\bQ & \bZ^\sT  + p \bZ_1^{\sT}\\
\bZ + p \bZ_1 & t_1\id_n+t_2\bH
\end{matrix}\right]\,.
\end{align}
For any $\xi \in \C_+$, we consider the quantity
\[
G_d(\xi; \bq) = \frac{1}{d} \sum_{i = 1}^M \log(\lambda_i(\bA(\bq)) - \xi). 
\]
With simple algebra, we can show that
%
\begin{equation}\label{eqn:connection_G_Psi_appendix}
\begin{aligned}
\partial_{t_1} G_d(\imagunit u; \bzero) =&~ \frac{\imagunit u}{d} \Tr\Big( ( u^2 \id_n + \bZ \bZ^\sT )^{-1} \Big),\\
\partial_{t_2} G_d(\imagunit u; \bzero) =&~ \frac{\imagunit u}{d} \Tr\Big( ( u^2 \id_n + \bZ \bZ^\sT )^{-1} \bH  \Big),\\
\partial_{s_1, t_1}^2 G_d(\imagunit u; \bzero) =&~ - \frac{1}{d} \Tr\Big( ( u^2 \id_N + \bZ^\sT  \bZ)^{-2} \bZ^\sT \bZ  \Big),\\
\partial_{s_1, t_2}^2 G_d(\imagunit u; \bzero) =&~ - \frac{1}{d} \Tr\Big( ( u^2 \id_N + \bZ^\sT  \bZ)^{-2} \bZ^\sT \bH \bZ  \Big).\\
\end{aligned}
\end{equation}
Hence, we have
\[
\begin{aligned}
\E[L_{\RF}(f_d, \bX, \bTheta, \lambda)] =&~ - \normf_1^2 \cdot \imagunit \Big( \frac{ \psi_1 \lambda}{\psi_2} \Big)^{1/2} \partial_{t_2} \E[G_d(\imagunit (\lambda \psi_1 \psi_2)^{1/2}; \bzero)] \\
&~ - (\normf_\star^2 + \tau^2) \cdot \imagunit \Big( \frac{ \psi_1 \lambda}{\psi_2} \Big)^{1/2} \partial_{t_1} \E[ G_d(\imagunit (\lambda \psi_1 \psi_2)^{1/2}; \bzero)] + o_{d}(1),  \\
\end{aligned}
\]
and
\[
\begin{aligned}
\E [ \| \ba \|_2^2] =&~ -  \normf_1^2 \partial_{s_1, t_2}^2 \E[ G_d(\imagunit (\lambda \psi_1 \psi_2)^{1/2}; \bzero)] - (\normf_\star^2 +  \tau^2) \cdot \partial_{s_1, t_1}^2 \E[G_d(\imagunit (\lambda \psi_1 \psi_2)^{1/2}; \bzero)] + o_{d}(1). 
\end{aligned}
\]
By Lemma \ref{lem:bounded_derivatives_G}, we get
\[
\begin{aligned}
\E[L_{\RF}(f_d, \bX, \bTheta, \lambda)] =&~ - \normf_1^2 \cdot \imagunit \Big( \frac{ \psi_1 \lambda}{\psi_2} \Big)^{1/2} \partial_{t_2} g(\imagunit (\lambda \psi_1 \psi_2)^{1/2}; \bzero) \\
&~ - (\normf_\star^2 + \tau^2) \cdot \imagunit \Big( \frac{ \psi_1 \lambda}{\psi_2} \Big)^{1/2} \partial_{t_1} g(\imagunit (\lambda \psi_1 \psi_2)^{1/2}; \bzero) + o_{d}(1),  \\
\end{aligned}
\]
and
\[
\begin{aligned}
\E_{\bbeta, \beps} [ \| \ba \|_2^2] =&~ -  \normf_1^2 \partial_{s_1, t_2}^2 g(\imagunit (\lambda \psi_1 \psi_2)^{1/2}; \bzero) - (\normf_\star^2 +  \tau^2) \cdot \partial_{s_1, t_1}^2 g(\imagunit (\lambda \psi_1 \psi_2)^{1/2}; \bzero) + o_{d, \P}(1), 
\end{aligned}
\]
where $g$ is given in Eq. (\ref{eqn:formula_g}). The derivatives of $g$ can be obtained by differentiating Eq. (\ref{eqn:log_determinant_variation}) and using Daskin's theorem. The theorem then follows by simple calculus.

\end{document}